\RequirePackage{fix-cm}
\documentclass[smallextended]{svjour3}       
\smartqed  
\usepackage{graphicx}

\usepackage{amssymb}
\usepackage{amsmath}
 \usepackage{mathrsfs}
\usepackage{hyperref}
\usepackage{diagbox}

\usepackage{lipsum,color,xcolor}
 \usepackage{amsfonts}
 \usepackage{algorithm}
 \usepackage{algorithmic}

\begin{document}

\title{Tensor
robust principal component analysis via the tensor nuclear over Frobenius norm
}

\titlerunning{TNF for TRPCA}        

\author{Huiwen Zheng         \and
        Yifei Lou  \and
        Guoliang Tian \and
        Chao Wang
}


\institute{Huiwen Zheng \at
              Department of Statistics and Data Science, Southern University of Science and Technology, Shenzhen 518005, China \\
              \email{12031217@mail.sustech.edu.cn}           
           \and
           Yifei Lou \at
              Department of Mathematics and School of Data Sciences and Society, University of North Carolina at Chapel Hill, Chapel Hill, NC, 27599 USA\\
              \email{yflou@unc.edu}           
           \and
           Guoliang Tian \at
              Department of Statistics and Data Science, Southern University of Science and Technology, Shenzhen 518005, China \\
              \email{tiangl@sustech.edu.cn}           
           \and
           Chao Wang \at
              Department of Statistics and Data Science, Southern University of Science and Technology, Shenzhen 518005, Guangdong Province, China  \\
             National Centre for Applied Mathematics Shenzhen, Shenzhen 518055, Guangdong Province, China  \\
              \email{wangc6@sustech.edu.cn}   
}

\date{Received: date / Accepted: date}

\maketitle

\begin{abstract}
We address the problem of tensor robust principal component analysis (TRPCA), which entails decomposing a given tensor into the sum of a low-rank tensor and a sparse tensor. By leveraging the tensor singular value decomposition (t-SVD), we introduce the ratio of the tensor nuclear norm to the tensor Frobenius norm (TNF) as a nonconvex approximation of the tensor's tubal rank in TRPCA. Additionally, we utilize the traditional $\ell_1$ norm to identify the sparse tensor. For brevity, we refer to the combination of TNF and $\ell_1$ as simply TNF. Under a series of incoherence conditions,  
we prove that a pair of tensors serves as a local minimizer of the proposed TNF-based TRPCA model if one tensor is sufficiently low in rank and the other tensor is sufficiently sparse. 
In addition, we propose replacing the $\ell_1$ norm with the ratio of the $\ell_1$ and Frobenius norm for tensors, the latter denoted as the $\ell_F$ norm. We refer to the combination of TNF and $\ell_1/\ell_F$ as the TNF$+$ model in short.
To solve both TNF and TNF$+$  models, we employ the alternating direction method of multipliers (ADMM) and prove subsequential convergence under certain conditions. Finally, extensive experiments on synthetic data, real color images, and videos are conducted to demonstrate the superior performance of our proposed models in comparison to state-of-the-art methods in TRPCA.
\end{abstract}




\section{Introduction}
Over the past decade, there has been a significant rise in the volume of data, accompanied by 
a notable shift towards multidimensional data, as opposed to traditional data confined to one or two dimensions. This trend presents various challenges regarding storage, transmission, and analysis. Tensors \cite{de2008tensor,kolda2009tensor}, representing multidimensional arrays, have emerged as crucial tools in numerous applications such as computer vision \cite{brandoni2020tensor,liu2021tensors,xia2024tensor}, signal processing \cite{7891546,miron2020tensor}, seismic imaging \cite{ely20155d,popa2021improved}, statistics \cite{guo2011tensor,zhao2015bayesian,kossaifi2020tensor}, and machine learning \cite{anandkumar2014tensor,chen2018sharing,8884203}.  Exploring low-dimensional structures within such complex data has gained increasing importance, particularly when these structures can be effectively modeled by certain low-rank properties.

Unlike matrices designed to handle two-dimensional (2D) data, the concept of tensor rank is not universally defined and can vary depending on the chosen tensor decomposition methods.
The CANDECOMP/PARAFAC (CP) rank \cite{kolda2009tensor} originates from the CP decomposition \cite{kiers2000towards}, determining the minimum number of rank-one decompositions to represent a given tensor. On the other hand, the Tucker rank \cite{gandy2011tensor}, derived from the Tucker decomposition \cite{tucker1966some}, is a vector wherein each element represents the rank of a matrix unfolded from the original tensor. Additionally, the tensor multi-rank \cite{gandy2011tensor} and tubal rank \cite{kilmer2013third} have emerged from tensor singular value decomposition (t-SVD) \cite{kilmer2011factorization}, analogous to the singular value decomposition (SVD) of matrices. Consequently, various surrogates of tensor rank have been proposed. For instance, Liu et al.~\cite{liu2012tensor} introduced the sum of the nuclear norm (SNN) based on the Tucker decomposition.  The concept of a matrix nuclear norm was extended to the tensor nuclear norm (TNN) for the t-SVD in \cite{semerci2014tensor}. Moreover, several nonconvex alternatives to TNN have been proposed \cite{yang2020multiview,YANG2022108311,zhao2022robust}. Jiang et al.~\cite{jiang2020multi} introduced the partial sum of the tubal nuclear norm (PSTNN), which calculates the partial sum of smaller singular values for every frontal slice after applying a discrete Fourier transformation (DFT). Xu et al.~\cite{xu2019laplace} incorporated the Laplace function into TNN, leading to a Laplace-based nonconvex surrogate. Qiu et al.~\cite{qiu2022fast} proposed a nonconvex alternating projection method with linear convergence, followed by an acceleration leveraging the properties of the tangent space of low-rank tensors. Recently, Yan and Guo \cite{yan2024tensor} considered using the $\ell_{p}$ quasi-norm ($0<p<1$) to impose sparse constraints on both the singular values and sparse components simultaneously, which is referred to as  the $p$-TRPCA model.

A conventional yet valuable tool in data analysis is principal component analysis (PCA) \cite{wold1987principal}, utilized for extracting dominant patterns from matrices. However, a well-known drawback of PCA is its susceptibility to sparse errors and outlier observations. To address this limitation, robust PCA (RPCA) \cite{candes2011robust} was introduced as the first polynomial-time algorithm with robust recovery guarantees. Subsequently, tensor robust principal component analysis (TRPCA) \cite{lu2016tensor} extended RPCA from matrices to tensors, allowing for the identification of low-rank tensors from sparsely corrupted entries. 
Specifically, TRPCA aims to decompose an observed tensor $\mathcal{X} \in \mathbb{R}^{n_{1} \times n_2 \times n_{3}}$ into $\mathcal{X}=\mathcal{L}_0+\mathcal{E}_0$, where $\mathcal{L}_0$ represents a low-rank tensor and $\mathcal{E}_0$ is a tensor containing only a small number of nonzero elements. Mathematically,  TRPCA can be formulated as the following optimization problem:
\begin{equation}
(\mathcal L_0,\mathcal E_0)=\arg\min \limits_{(\mathcal{L}, \mathcal{E})}\text{rank}(\mathcal{L}) +\lambda \|\mathcal{E}\|_{0}, \quad
\text{s.t.}\quad  \mathcal{X}=\mathcal{L}+\mathcal{E}, \label{equ:TRPCAfirst}
\end{equation}
where $\lambda > 0$ is a fixed parameter,  $\| \cdot \|_0$ denotes the number of nonzero elements, and $\text{rank}(\mathcal{L})$ represents  some type of tensor rank. 
However, both the rank and the $\ell_0$ minimization problems are NP-hard \cite{natarajan1995sparse}. 
Alternatively, convex or non-convex surrogate functions \cite{yang2020low,han2022tensor} to approximate the rank and $\ell_0$ penalties are used in TRPCA.

The $\ell_1$ norm serves as a convex relaxation of the $\ell_0$ norm and has found widespread application in statistics,  as highlighted in \cite{tibshirani1996regression} with the introduction of the least absolute shrinkage and selection operator (LASSO). However, Fan and Li \cite{fan2001variable} noted that the $\ell_1$ norm may not always be statistically optimal to yield the best estimation performance.  Consequently, various nonconvex penalties \cite{yang2015robust,gu2017tvscad} have been proposed, including the bridge penalty by Huang et al. \cite{huang2008asymptotic}, the logistic penalty by Nikolova et al. \cite{nikolova2008efficient}, the hard thresholding penalty function by Fan and Li \cite{fan2001variable}, the minimax concave penalty by Zhang~\cite{zhang2010nearly}.
In the context of tensor recovery problems, two convex relaxation methods are the $\ell_1$ norm \cite{lu2019tensor,qin2021robust,gao2023tensor} and the $\|\cdot\|_{2,1}$ functional \cite{zhou2017outlier}. Nonconvex penalties include the $\ell_p$ regularization  \cite{li2018low} for low-rank tensor recovery problems.

All the aforementioned nonconvex surrogates of tensor rank come with internal parameters that significantly influence the model's performance. Motivated by the remarkable performance of the ratio of the $\ell_1$-norm and the $\ell_2$-norm for sparse signal recovery \cite{rahimi2019scale,wang2020accelerated,wang2021limited,wang2022minimizing}, we propose a parameter-free regularization technique utilizing the ratio of the tensor nuclear norm and the Frobenius norm (TNF) to approximate the tensor tubal rank. 
Specifically, in the TRPCA problem \eqref{equ:TRPCAfirst}, we utilize the TNF regularization to enforce the low rankness while using the $\ell_1$-norm for sparsity.
For brevity, we refer to the combination of TNF and $\ell_1$ as simply TNF. Following a set of incoherence conditions formulated in  the TNN-based TRPCA model \cite{lu2019tensor}, 
we prove that a pair of tensors serves as a local minimizer of the proposed TNF-based TRPCA model if one tensor is sufficiently low in rank and the other tensor is sufficiently sparse. 
In addition, we propose replacing the $\ell_1$ norm with the ratio of the $\ell_1$ and Frobenius norm for tensors, the latter denoted as the $\ell_F$ norm. We refer to the combination of TNF and $\ell_1/\ell_F$ as the TNF$+$ model in short. 

Computationally, we devise efficient algorithms based on the alternating direction method of multipliers (ADMM) \cite{boyd2011distributed} to solve for both TNF and TNF$+$ models. We also establish their subsequential convergence under certain conditions. Extensive experiments conducted on synthetic and real image data confirm the superiority of our proposed methods over state-of-the-art approaches. The key contributions of our work are summarized as follows:
\begin{itemize}

\item[$\bullet$] We propose two novel models (TNF and TNF$+$) for TRPCA.
\item[$\bullet$] We present an exact recovery theory of TNF under incoherence conditions.  
\item[$\bullet$] We adopt ADMM to solve the proposed models along with convergence analysis.
\end{itemize}

We organize the rest of the paper as follows. 
We introduce our proposed model TNF and its properties in Sect.~\ref{sect:modelsPCA}. The algorithm development with convergence analysis is presented in Algorithm~\ref{sec:algorithm}.  A variant model, called TNF$+$, along with an algorithm and its convergence is discussed in Sect.~\ref{sect:RPCAL}. We conduct numerical experiments in Sect.~\ref{sect:experimentsPCA}, including synthetic and real data to show the superiority of our proposed models. Finally, we conclude the paper in Sect.~\ref{sect:concludePCA}.

\section {TNF-based TRPCA model and recovery guarantee}\label{sect:modelsPCA}
In this section, we introduce basic tensor notations and discuss a non-convex regularization method, which involves the ratio of the tensor nuclear norm to the Frobenius norm (TNF), aimed at approximating the tensor tubal rank~\cite{zheng2024scale}. We adapt the TNF regularization to the TRPCA problem and establish a recovery guarantee of using TNF and $\ell_1$ to identify a low-rank tensor and a sparse tensor, respectively. 

\subsection{Notations and preliminary}
We provide an overview of necessary notations and definitions used throughout this paper, as summarized in Table~\ref{TRPCAnotation}. The field of real numbers and complex numbers are denoted as R and C, respectively. Considering a tensor $\mathcal{A}\in \mathbb{R}^{n_1\times n_2\times n_3}$, 
we denote ${\overline{\mathcal{A}}}$  as the tensor after applying the fast Fourier Transform (FFT) to the tensor $\mathcal{A}$ along the third (tubal) dimension, i.e., $\overline{\mathcal{A}}=\operatorname{fft}\left(\mathcal{A}, [\ ], 3\right)$ via the MATLAB command ``$\mathrm{fft}$''. We compute $\mathcal{A}$ via $\mathcal{A}=\operatorname{ifft}\left(\overline{\mathcal{A}}, [\ ], 3\right)$. 
Following the work \cite{lu2019tensor},  we define the tensor Frobenius norm and the tensor nuclear norm (TNN) as follows,
\begin{align}
&\|\mathcal{A}\|_{F}^2 =\frac{1}{n_{3}}\sum_{i=1}^{n_3}\|\overline{\mathbf{A}}^{(i)}\|_{F}^2  = \frac{1}{n_{3}} \sum_{i=1}^{n_{3}}\sum_{j=1}^{n_{(2)} } \sigma_{ij}^2,\\
\label{eq:tnn_tsn}
&\|\mathcal{A}\|_{\rm{*}} :=\frac{1}{n_{3}}\sum_{i=1}^{n_{3}}\left\|\overline{\mathbf{A}}^{(i)}\right\|_{*}=\frac{1}{n_{3}}\sum_{i=1}^{n_{3}}\sum_{j=1}^{ n_{(2)} } \sigma_{ij},
\end{align}
where $\overline{\mathbf{A}}^{(i)}$ is the $i$-th frontal slice of $\overline{\mathcal{A}}$, $\sigma_{ij}$ is the $j$-th singular value of $\overline{\mathbf{A}}^{(i)}$, and $n_{(2)}=\min{\{n_1,n_2\}}$. We define $n_{(1)}=\max{\{n_1,n_2\}}$.

Our models and algorithms are built on the t-SVD algebraic
framework \cite{kilmer2011factorization}. Please refer to Definition~\ref{def:tsvd} for t-SVD and other related concepts in Appendix~\ref{app:t_svd}.

\begin{definition} [tensor singular value decomposition: t-SVD \cite{kilmer2011factorization}]\label{def:tsvd}
Let ${\mathcal{A} \in \mathbb{R}^{n_{1} \times n_{2} \times n_{3}}}$, then the t-SVD of $\mathcal{A} $ is given by 
\begin{equation}
 \mathcal{A}=\mathcal{U} * \mathcal{S} * \mathcal{V}^{*},
\end{equation} 
where ${\mathcal{U} \in \mathbb{R}^{n_{1} \times n_{1} \times n_{3}}, \mathcal{V} \in \mathbb{R}^{n_{2} \times n_{2} \times n_{3}}}$ are orthogonal tensors, and ${\mathcal{S} \in \mathbb{R}^{n_{1} \times n_{2} \times n_{3}}}$ is an ${f}$-diagonal tensor. 
\end{definition}
Analogous to skinny matrix SVD, the skinny t-SVD requires the tensor tubal rank, which is defined in Definition~\ref{def:tar}.
\begin{definition}[tensor tubal rank \cite{lu2019tensor}]
\label{def:tar}
    For a  tensor ${\mathcal{A} \in \mathbb{R}^{n_{1} \times n_{2} \times n_{3}}}$, its tubal rank, denoted as $\operatorname{rank}_{t}(\mathcal{A})$ defined as the number of nonzero singular tubes of $\mathcal{S}$, where $\mathcal{S}$ comes from the t-SVD of $\mathcal{A}$, i.e. $\mathcal{A}=\mathcal{U} * \mathcal{S} * \mathcal{V}^{*}$. 
\end{definition}
For tensor ${\mathcal{A} \in \mathbb{R}^{n_{1} \times n_{2} \times n_{3}}}$ with tubal rank $r<n_{(2)}$, the \textit{skinny t-SVD} of $\mathcal{A}$ is defined by
$\mathcal{A}=\mathcal{U} * \mathcal{S} * \mathcal{V}^{*}$, where $\mathcal{U} \in \mathbb{R}^{n_{1} \times r \times n_{3}}, \mathcal{S} \in \mathbb{R}^{r \times r \times n_{3}}, \mathcal{V} \in \mathbb{R}^{n_{2} \times r \times n_{3}}$ with $\mathcal{U}^**\mathcal{U}=\mathcal{I}$ and $\mathcal{V}^**\mathcal{V}=\mathcal{I}$.

\begin{table}[]
\caption{Summary of main notations in the paper}
\scalebox{0.7}{
\begin{tabular}{llll}
\hline
Notation & Description & Notation & Description\\
\hline
$\boldsymbol{a} \in \mathbb{C}^n_{1}$       & vector    & $\boldsymbol{a}_{i}$ & the $i$-th entry of a vector $\boldsymbol{a}$  \\
$\mathbf{A} \in \mathbb{C}^{n_{1} \times n_{2}}$        & matrix    
& $\mathbf{A}_{i:}$ & the $i$-th row of $\mathbf{A}$ \\
  $\mathbf{A}_{:j}$ & the $j$-th column of $\mathbf{A}$ &
  $a_{ij}$ or $\mathbf{A}_{ij}$ & the $(i, j)$-th entry of $\mathbf{A}$\\
$\mathbf{I} \in \mathbb{C}^{n_{1} \times n_{2}}$        & identity matrix  & $\mathcal{A} \in \mathbb{C}^{n_{1} \times n_{2} \times n_3}$  & tensor \\
$\mathcal{A}_{i::}$ &  the $i$-th horizontal slice of $\mathcal{A}$ &  $\mathcal{A}_{:j:}$ &  the $j$-th  lateral slice of $\mathcal{A}$ \\
$\mathcal{A}_{::k}$ or $\mathbf{A}^{(k)}$ &  the $k$-th  frontal slice of $\mathcal{A}$ &
$\mathcal{A}_{:jk}$ &  the $(i, j)$-th tubal fiber of $\mathcal{A}$ \\
$\mathcal{O}$ & zero tensor &  $\|\boldsymbol{a}\|_{2}=\sqrt{\sum_{i}\left|\boldsymbol{a}_{i}\right|^{2}}$ & the $\ell_2$ norm of vector $\boldsymbol{a}$ \\
$\|\boldsymbol{a}\|_{1}=\sum_{i}\left|\boldsymbol{a}_{i}\right|$ & the $\ell_1$ norm of vector $\boldsymbol{a}$ & $\|\boldsymbol{a}\|_{\infty}=\max_{i}\left|\boldsymbol{a}_{i}\right|$ & the infinity norm of vector $\boldsymbol{a}$\\
$\mathbf{A}^*$ &  the conjugate transpose of $\mathbf{A}$ & $\operatorname{Tr}(\cdot)$ &  the matrix trace \\
$\sigma_{i}(\mathbf{A})$ & the $i$-th singular value of $\mathbf{A}$ & $\|\mathbf{A}\|=\max_{i} \sigma_{i}(\mathbf{A})$ & the
spectral norm of $\mathbf{A}$ \\
$\|\mathbf{A}\|_{*}=\sum_{i} \sigma_{i}(\mathbf{A})$ & the nuclear norm of $\mathbf{A}$ &
$\|\mathbf{A}\|_F=\sqrt{\sum_{ij}\left|a_{ij}\right|^{2}}$ & the Frobenius norm of $\mathbf{A}$
\\
$\|\mathbf{A}\|_1=\sum_{ij}\left|a_{ij}\right|$ & the $\ell_1$ norm of $\mathbf{A}$ &
$\|\mathbf{A}\|_{\infty}=\max_{ij}|a_{ij}|$ & the infinity norm of $\mathbf{A}$\\
$\langle \mathbf{A}, \mathbf{B}\rangle=\operatorname{Tr}\left(\mathbf{A}^{*} \mathbf{B}\right)$& the inner product of $\mathbf{A}$ and $\mathbf{B}$ &
$\|\mathcal{A}\|_1=\sum_{ijk}|a_{ijk}|$ & the $\ell_1$ norm of $\mathcal{A}$\\
$\|\mathcal{A}\|_F=\sqrt{\sum_{ijk}\left|a_{ijk}\right|^{2}}$ & the Frobenius norm of $\mathcal{A}$ & $\|\mathcal{A}\|_{\infty}=\max_{ijk}|a_{ijk}|$ & the infinity norm of $\mathcal{A}$\\
$\operatorname{conj}(\mathcal{A})$ & the complex conjugate of $\mathcal{A}$ & 
$\langle\mathcal{A}, \mathcal{B}\rangle=\sum_{k=1}^{n_{3}}\left\langle \mathbf{A}^{(k)}, \mathbf{B}^{(k)}\right\rangle$ & the inner product between $\mathcal{A}$ and $\mathcal{B}$ \\
\hline
\end{tabular} } \label{TRPCAnotation}
\end{table}

\subsection{TNF-based TRPCA model}
 Zheng et al.~\cite{zheng2024scale} proposed the ratio of the tensor nuclear norm and the Frobenius norm (TNF) as a nonconvex surrogate of tensor tubal rank for the tensor completion problem. The TNF regularization is defined as 
\begin{equation}
\label{eq:tnf}
    \|\mathcal{A}\|_{\rm{TNF}}:=\tfrac{\|\mathcal{A}\|_{*}}{\|\mathcal{A}\|_{F}}= \frac{\sum_{i=1}^{n_{3}}\sum_{j=1}^{n_{(2)} } \sigma_{ij}}{\sqrt{n_3\sum_{i=1}^{n_{3}}\sum_{j=1}^{n_{(2)}} \sigma_{ij}^2}}.  
\end{equation}
TNF effectively enforces a low-rank structure of the tensor, analogous to the $\ell_{1}/\ell_2$ model \cite{rahimi2019scale,wang2020accelerated} applied to the vector formed by stacking all singular values $\{\sigma_{ij}\}$.

This paper focuses on the TRPCA problem, which aims to decompose a given tensor into the sum of a low-rank tensor and a sparse tensor. We employ the TNF regularization for low rankness and the standard $\ell_1$ norm for sparsity. In short, 
the proposed model is formulated as 
\begin{equation}  \label{con:PCAa}
\min \limits_{\mathcal{L}, \mathcal{E}}\|\mathcal{L}\|_{\rm{TNF}} +\lambda \|\mathcal{E}\|_{1} \quad
\text{s.t.}\quad  \mathcal{X}=\mathcal{L}+\mathcal{E},
\end{equation}
where $\mathcal{X}\in \mathbb{R}^{n_{1} \times n_{2} \times n_{3}}$ is a given third-order tensor. Throughout the remainder of the paper, we refer to this TNF-based TRPCA model \eqref{con:PCAa} briefly as ``TNF''. 

\subsection{Recovery guarantee}
We establish a recovery theory for the proposed model \eqref{con:PCAa} to identify the two tensors (a low-rank tensor and a sparse tensor) under tensor incoherence conditions. Some notations are required to present the conditions in Definition~\ref{defi:incoherence} and Theorem~\ref{thm:TRPCA1}. The column basis is denoted as ${\overrightarrow{\boldsymbol{e}}_{i}}$, which is a tensor of size ${n_{1} \times}{1 \times n_{3}}$ with its ${(i, 1,1)}$-th entry set to 1 and the remaining entries set to 0.  The nonzero entry one only appears at the first frontal slice of ${\overrightarrow{\boldsymbol{e}}_{i}}$. Naturally its conjugate transpose ${\overrightarrow{\boldsymbol{e}}_{i}^{*}}$ is called row basis. The tube basis, denoted as ${\dot{\boldsymbol{e}}_{k}}$, is a tensor of size ${1 \times 1 \times n_{3}}$ with its ${(1,1, k)}$-th entry set to 1 and the reset set to 0.
We define $\mathbf{e}_{ijk}:=\overrightarrow{\boldsymbol{e}}_i * \dot{\boldsymbol{e}}_k * \overrightarrow{\boldsymbol{e}}_j^*\in \mathbb{R}^{n_1\times n_2\times n_3}$,  which is a unit tensor with the only non-zero entry at $(i,j,k)$ being to 1. 

\begin{definition}  \label{defi:incoherence}
{\rm (tensor incoherence conditions \cite{lu2019tensor})} For a low-rank tensor $\mathcal{L}_0 \in \mathbb{R}^{n_1\times n_2\times n_3}$, we assume ${\operatorname{rank}_{t}\left(\mathcal{L}_{0}\right)=r}$ and its skinny t-SVD is $\mathcal{L}_{0}=\mathcal{U} * \mathcal{S} * \mathcal{V}^{*}$, where $\mathcal{U} \in \mathbb{R}^{n_{1} \times r \times n_{3}}, \mathcal{S} \in \mathbb{R}^{r \times r \times n_{3}}, \mathcal{V} \in \mathbb{R}^{n_{2} \times r \times n_{3}}$. We say $\mathcal{L}_{0}$ satisfies the tensor incoherence conditions with parameter ${\mu>0}$ if 
\begin{equation}
\begin{aligned}
\max _{i=1, \ldots, n_1}\left\|\mathcal{U}^* * \overrightarrow{\boldsymbol{e}}_i\right\|_F &\leq \sqrt{\tfrac{\mu r}{n_1n_3}}, \\
\max _{j=1, \ldots, n_2}\left\|\mathcal{V}^* * \overrightarrow{\boldsymbol{e}}_j\right\|_F &\leq \sqrt{\tfrac{\mu r}{n_2n_3}},   \\
\left\|\mathcal{U} * \mathcal{V}^*\right\|_{\infty} &\leq \sqrt{\tfrac{\mu r}{n_1 n_2 n_3^{2}}},
\end{aligned} \label{condition:incoherence}
\end{equation}
where $\overrightarrow{\boldsymbol{e}}_i$ and $\overrightarrow{\boldsymbol{e}}_j$ are column basis of size $n_1\times 1\times n_3$ and $n_2\times1\times n_3,$ respectively. 
\end{definition} 

We present our main theoretical result regarding the recovery guarantee of the TNF regularization and the $\ell_1$ norm for finding a low-rank tensor and a sparse tensor, respectively. 
\begin{theorem}
Suppose $\mathcal{L}_{0} \in \mathbb{R}^{n_{1} \times n_{2} \times n_{3}}$ with tubal rank $r$ obeys the tensor incoherence conditions \eqref{condition:incoherence} with parameter $\mu$.  Suppose that the support
$\boldsymbol{\Omega}$ of $\mathcal{E}_0$ is uniformly distributed among 
all sets of cardinality $m=2\gamma n_1n_2n_3$, where $\gamma=\mathbb{P}(\operatorname{sgn}(\mathcal{E}_0)=1)=\mathbb{P}(\operatorname{sgn}(\mathcal{E}_0)=-1)$. If the parameter $\lambda$ in the TNF model \eqref{con:PCAa}
is selected within the interval 
$$\left[\max\left(\tfrac{2\sqrt{6r}n_1n_2n_3^2}{n_1n_2n_3^2\sqrt{1-2\gamma}\|\mathcal{L}_{0}\|_F-2\sqrt{6}},\sqrt{\tfrac{\log (n_{(1)}n_3)}{n_{(1)}n_3^2}}\right), \left(\tfrac{1}{4}\sqrt{\tfrac{n_{(1)}}{2\mu r}}-\sqrt{rn_3}\right)n_1n_2n_3^{3/2}\right]$$ 
with sufficiently large $n_1, n_2, n_3$, there exists a positive constant $c_0$ such that with 
probability at least $1-2(n_{(1)}n_3)^{-c_0}$, 
$(\mathcal L_0, \mathcal E_0)$ is a local minimum of \eqref{con:PCAa}, provided that
\begin{equation}
r \leq\min\left(n_{(2)}, \tfrac{c_{r} n_{(2)}n_3}{\mu\left(\log \left(n_{(1)} n_{3}\right)\right)^{2}}\right)\quad \text{and} \quad \gamma \leq \tfrac 1 2 - \tfrac{c_\gamma\mu r \log \left(n_{(1)} n_{3}\right)}{n_{(2)}n_3}.
\label{inequality:rankTRPCA}
\end{equation}
 \label{thm:TRPCA1}
\end{theorem}
Theorem~\ref{thm:TRPCA1} implies that for $\mathcal L_0$ with sufficiently low rank (its tubal rank is upper bounded) and  $\mathcal E_0$ with sufficiently sparse (its cardinality is upper bounded), the pair $(\mathcal L_0, \mathcal E_0)$ is a local minimum of the proposed TNF model \eqref{con:PCAa} with high probability under some certain conditions. In addition,  sufficiently large $n_1,n_2,n_3$ can ensure
that $\sqrt{\tfrac{24}{1-2\gamma}}/\|\mathcal L_0\|_F<n_1n_2n_3^2$, and $\tfrac{ 1}{4}\sqrt{\tfrac{n_{(1)}}{2\mu r}}>\sqrt{rn_3}$ such that the interval for $\lambda$ is well-defined. 
Its proof is given in the supplementary material. 

\section{Algorithmic developments} \label{sec:algorithm}
In this section, we employ the alternating direction method of multipliers (ADMM) to solve the proposed model \eqref{con:PCAa}, accompanied by analyses of its complexity and convergence.

\subsection {Numerical algorithm}\label{sect:RPCA}
We introduce an auxiliary variable $\mathcal{H}$ and design a specific splitting scheme that reformulates \eqref{con:PCAa} into 
\begin{equation}  
	{ \begin{array}{lc} \vspace{1ex}
\underset{\mathcal{L}, \mathcal{H}, \mathcal{E}}{\operatorname{min}} \  \tfrac{\|\mathcal{L}\|_{*}}{\|\mathcal{H}\|_{F}}\ +\lambda \|\mathcal{E}\|_{1} \\ 
\text{s.t.}\  \mathcal{X}=\mathcal{L}+\mathcal{E},\ \mathcal{H}=\mathcal{L}. \end{array} }
\end{equation}
The corresponding augmented Lagrangian function is expressed as 
\begin{equation}
    \begin{aligned}
    L_{1}\left(\mathcal{L},\mathcal{H},\mathcal{E},\mathcal{Y},\mathcal{Z}\right) =& \tfrac{\|\mathcal{L}\|_{*}}{\|\mathcal{H}\|_{F}}+\lambda\|\mathcal{E}\|_{1}+\tfrac{\mu_{1}}{2}\left\|\mathcal{L}-\mathcal{H}\right\|_{F}^{2}
          +\left\langle\mathcal{Y}, \mathcal{L}-\mathcal{H}\right\rangle \\ &\quad +\tfrac{\mu_{2}}{2}\|\mathcal{L}+\mathcal{E}-\mathcal{X}\|_{F}^{2}+\left\langle\mathcal{Z}, \mathcal{L}+\mathcal{E}-\mathcal{X}\right\rangle, \label{eq:trpca_lagra}
    \end{aligned}
\end{equation}
where $\mathcal{Y}, \mathcal{Z}$ are dual variables and $\mu_1, \mu_2$ are positive parameters. In the ADMM scheme, we   alternatively  update the variables $\mathcal{L}$, $\mathcal{H}$,
$\mathcal{E}$,  $\mathcal{Y}$, and $\mathcal{Z}$ by
\begin{equation}
    \left\{\begin{array}{l}
\mathcal{L}^{(k+1)} =\arg \min\limits _{\mathcal{L}} L_{1}\left(\mathcal{L},\mathcal{H}^{(k)},\mathcal{E}^{(k)},\mathcal{Y}^{(k)},\mathcal{Z}^{(k)} \right),\\
\mathcal{H}^{(k+1)}=\arg \min \limits_{\mathcal{H}} L_{1}\left(\mathcal{L}^{(k+1)},\mathcal{H},\mathcal{E}^{(k)},\mathcal{Y}^{(k)},\mathcal{Z}^{(k)} \right),\\
\mathcal{E}^{(k+1)}=\arg \min\limits _{\mathcal{E}} L_{1}\left(\mathcal{L}^{(k+1)},\mathcal{H}^{(k+1)},\mathcal{E},\mathcal{Y}^{(k)},\mathcal{Z}^{(k)}\right),\\
\mathcal{Y}^{(k+1)}=\mathcal{Y}^{(k)}+\mu_{1}\left(\mathcal{L}^{(k+1)}-\mathcal{H}^{(k+1)}\right),\\
\mathcal{Z}^{(k+1)}=\mathcal{Z}^{(k)}+\mu_{2}\left(\mathcal{L}^{(k+1)}+\mathcal{E}^{(k+1)}-\mathcal{X}\right).
\end{array}\right.\label{eq:TRPCA_ADMM}
\end{equation}

The $\mathcal{L}$-subproblem in \eqref{eq:TRPCA_ADMM} can be rewritten as 
\begin{equation}
   \min_{\mathcal{L}}  \left\{
  \tfrac{\|\mathcal{L}\|_{*}}{\|\mathcal{H}^{(k)}\|_{F}}+\tfrac{\mu_{1}}{2}\left\|\mathcal{L}-\mathcal{H}^{(k)}+\tfrac{\mathcal{Y}^{(k)}}{\mu_1}\right\|_{F}^{2}
           +\tfrac{\mu_{2}}{2}\left\|\mathcal{L}+\mathcal{E}^{(k)}-\mathcal{X}+\tfrac{\mathcal{Z}^{(k)}}{\mu_2}\right\|_{F}^{2}\right\}, \label{Lfirst}
\end{equation}
which has a closed-form solution by the tensor singular value thresholding (t-SVT)
\cite{lu2019tensor}. Specifically given a tensor $\mathcal A$ with its t-SVD ${\mathcal{A}=\mathcal{U} * \mathcal{S} * \mathcal{V}^{*}}$, the t-SVT operator is defined by
$${\mathcal{D}_{\tau}(\mathcal{A}):=\mathcal{U} * \mathcal{S}_{\tau} * \mathcal{V}^{*}, }$$
where $\tau>0$ and ${\mathcal{S}_{\tau}}$ is an 
tensor that satisfies $\overline{\mathcal{S}}_{\tau}=\max \{\overline{\mathcal{S}}-\tau, 0\}.$ 
Hence, we have the $\mathcal L$-update as follows,
\begin{equation}
    \mathcal{L}^{(k+1)} = \mathcal{D}_{\tau^{(k+1)}}\left(\tfrac{1}{\mu_{1}+\mu_{2}}\left(\mu_{1} \mathcal{H}^{(k)}+\mu_{2}\left(\mathcal{X}-\mathcal{E}^{(k)}\right)\right)-\tfrac{\mathcal{Y}^{(k)}+\mathcal{Z}^{(k)}}{\mu_{1}+\mu_{2}}\right),
    \label{con:TRPCALk}
\end{equation}
with $\tau^{(k+1)} = \tfrac{1}{(\mu_1+\mu_2)\|\mathcal{H}^{(k)}\|_F}$.

\medskip
The $\mathcal{H}$-subproblem of \eqref{eq:TRPCA_ADMM}  can be expressed as 
\begin{equation}
\begin{aligned}
\mathcal{H}^{(k+1)}=\arg \min _{\mathcal{H}}\left\{\tfrac{\rho^{(k+1)}}{\| \mathcal{H}\|_{F}}+ 
\tfrac{\mu_{1}}{2}\left\| \mathcal{H}- \mathcal{K}^{(k)}\right\|_{F}^{2}\right\},
\end{aligned} \label{con:newH}
\end{equation}
where a scalar $\rho^{(k+1)}=\|\mathcal{L}^{(k+1)}\|_{*}$ and a tensor $\mathcal{K}^{(k)}= \mathcal{L}^{(k+1)}+\tfrac{\mathcal{Y}^{(k)}}{\mu_{1}}$. Following the work of \cite{rahimi2019scale}, we  derive  the  closed-form solution to the problem \eqref{con:newH} given by 
\begin{equation}
\mathcal{H}^{(k+1)}= \begin{cases}\iota^{(k)} \mathcal{K}^{(k)} & \text { if } \mathcal{K}^{(k)} \neq \mathcal{O} \\ \mathcal{G}^{(k)} & \text { otherwise, }\end{cases} \label{con:newhupdate}  
\end{equation}
where $\mathcal{G}^{(k)}$ is a random tensor with its Frobenius norm being $ \sqrt[3]{\tfrac{\rho^{(k+1)}}{\mu_1}}$ and $\iota^{(k)} = \tfrac{1}{3}+\tfrac{1}{3}\left(C^{(k)}+\tfrac{1}{C^{(k)}}\right)$ with
$$ C^{(k)}=\sqrt[3]{\tfrac{27 E^{(k)}+2+\sqrt{(27 E^{(k)}+2)^{2}-4}}{2}} \text { and } E^{(k)}=\tfrac{\rho^{(k+1)}}{\mu_1 \|\mathcal{K}^{(k)}\|_{F}^{3}}. $$

Lastly,  the tensor $\mathcal{E}$-subproblem of \eqref{eq:TRPCA_ADMM}  can be equivalently expressed as minimizing the $\ell_1$ minimization elementwise, thus allowing for a closed-form solution  through a soft-thresholding operator, i.e.,
\begin{equation}
\mathcal{E}^{(k+1)}=\mathbf{shrink} \left(\mathcal{X}-\mathcal{L}^{(k+1)}-\mathcal{Z}^{(k)}/\mu_{2}, \ \lambda/\mu_2 \right),\label{e-step}
\end{equation}
where $\mathbf{shrink}(v, \rho) = \mathrm{sign}(v) \max \{|v|-\rho, 0\}. $ 

\subsection{Complexity}
We present the overall algorithm for solving problem \eqref{con:PCAa} in Algorithm~\ref{alg:ADMM2TRPCA}. 
Its primary computational complexity arises from updating $\mathcal{L}$ and $\mathcal{H}$. Specifically, in each iteration, updating $\mathcal{L}$ incurs a computational cost of $O((n_1n_2n_3(\log n_3+n_{(2)}))$, while updating $\mathcal{H}$ requires  $O(n_1n_2n_3n_{(2)})$. Consequently, the overall computational complexity of Algorithm~\ref{alg:ADMM2TRPCA} is 
$$O(n_1n_2(n_3\log n_3)+n_1n_2n_3n_{(2)}).$$ 

\begin{algorithm}[htb] 
\caption{ADMM for solving the TNF model \eqref{con:PCAa}. 
}
\label{alg:ADMM2TRPCA} 
\begin{algorithmic}[1] 
\REQUIRE  
  Observed data  $\mathcal{X}$, parameters: $\mu_1,\mu_2, \rm{kMax}, \epsilon$

\STATE{{\bf Initialization:} $(\mathcal {L}^{(0)}, \mathcal {E}^{(0)})$ by a TNN-based TRPCA model, $\mathcal {H}^{(0)}= \mathcal {L}^{(0)}, \mathcal Y^{(0)}=\mathcal Z^{(0)}=\mathcal{O},$ and $k=0$}
\WHILE {$k<\rm{kMax}$ or not converged}
\STATE $\tau^{(k+1)} = \tfrac{1}{(\mu_1+\mu_2)\|\mathcal{H}^{(k)}\|_F}$ 
\STATE $\mathcal{L}^{(k+1)} = \mathcal{D}_{\tau^{(k+1)}}\left(\tfrac{1}{\mu_{1}+\mu_{2}}\left(\mu_{1} \mathcal{H}^{(k)}+\mu_{2}\left(\mathcal{X}-\mathcal{E}^{(k)}\right)\right)-\tfrac{\mathcal{Y}^{(k)}+\mathcal{Z}^{(k)}}{\mu_{1}+\mu_{2}}\right)$
\STATE $\mathcal{E}^{(k+1)}=\mathbf{shrink} (\mathcal{X}-\mathcal{L}^{(k+1)}-\mathcal{Z}^{(k)}/\mu_{2}, \lambda/\mu_2 )$
\STATE $\mathcal{H}^{(k+1)}= \begin{cases}\iota^{(k)} (\mathcal{L}^{(k+1)}+\tfrac{\mathcal{Y}^{(k)}}{\mu_{1}}) & \text { if } \mathcal{L}^{(k+1)}+\tfrac{\mathcal{Y}^{(k)}}{\mu_{1}} \neq \mathbf{0} \\ \mathcal{G}^{(k)} & \text { otherwise }\end{cases}$
\STATE $\mathcal{Y}^{(k+1)}=\mathcal{Y}^{(k)}+\mu_{1}\left(\mathcal{L}^{(k+1)}-\mathcal{H}^{(k+1)}\right)$
\STATE $\mathcal{Z}^{(k+1)}=\mathcal{Z}^{(k)}+\mu_{2}\left(
\mathcal{L}^{(k+1)}+\mathcal{E}^{(k+1)}-\mathcal{X}\right)$
\STATE $k=k+1$ 
\STATE Check the convergence conditions \\
$\left\|\mathcal{L}^{(k+1)}-\mathcal{L}^{(k)}\right\|_{\infty} \leq \epsilon$, 
$\left\|\mathcal{E}^{(k+1)}-\mathcal{E}^{(k)}\right\|_{\infty} \leq \epsilon$, 
$\left\|\mathcal{H}^{(k+1)}-\mathcal{H}^{(k)}\right\|_{\infty} \leq \epsilon$,
$\left\|\mathcal{Y}^{(k+1)}-\mathcal{Y}^{(k)}\right\|_{\infty} \leq \epsilon$, 
$\left\|\mathcal{Z}^{(k+1)}-\mathcal{Z}^{(k)}\right\|_{\infty} \leq \epsilon$,
$\left\|\mathcal{L}^{(k+1)} + \mathcal{E}^{(k+1)} - \mathcal{X} \right\|_{\infty}\leq \epsilon$

\ENDWHILE
\RETURN $\hat{\mathcal{L}} = \mathcal{L}^{(k)}$ and $\hat{\mathcal{E}} = \mathcal{E}^{(k)}$ 
\end{algorithmic}
\end{algorithm}
\subsection{Convergence analysis}\label{sect:convPCA}
This section is devoted to the convergence analysis of our algorithm. Specifically, 
We show that the sequence generated by
Algorithm~\ref{alg:ADMM2TRPCA} has a subsequence convergent to a stationary point of  with the TNF model \eqref{con:PCAa} under the following two assumptions. 

A1: The sequence  $\{\mathcal{L}^{(k)}\}$ generated by \eqref{eq:TRPCA_ADMM} is bounded, so is its nuclear norm of $\mathcal{L}^{(k)}$, denoted by $\sup_k \{\|\mathcal{L}^{(k)}\|_{*}\}\leq M$.  
 
A2: The Frobenius norm of $\{\mathcal{H}^{(k)}\}$  has a uniform lower bound, i.e., there exists a positive constant $\delta$ such that $\|\mathcal{H}^{(k)}\|_F\geq \delta, \forall k. $

\begin{lemma}
 Under assumptions {\rm A1-A2} with sufficiently large parameters $\mu_1,\mu_2$, the sequence $\{\mathcal{Y}^{(k)}\}$ generated by {\rm \eqref{eq:TRPCA_ADMM}} satisfies
 \begin{equation}
    \left\|\mathcal{Y}^{(k+1)}-\mathcal{Y}^{(k)}\right\|_{F}^{2} \leq \tfrac{2 n_{(2)}}{\delta^{4}}\left\|\mathcal{L}^{(k+1)}-\mathcal{L}^{(k)}\right\|_F^2+\tfrac{4M^{2}}{\delta^{6}}\left\|\mathcal{H}^{(k+1)}-\mathcal{H}^{(k)}\right\|_{F}^{2}, \label{con:BsubY}
 \end{equation}
 where $M$ and $\delta$ are the constants defined in {\rm A1} and {\rm A2}, respectively. \label{lemma1P}
\end{lemma}

\begin{lemma}
Under assumptions {\rm A1-A2}, the augmented Lagrangian function \eqref{eq:trpca_lagra} of the sequence $\{\mathcal{L}^{(k)},\mathcal{H}^{(k)},\mathcal{E}^{(k)},\mathcal{Y}^{(k)},\mathcal{Z}^{(k)}\}$ generated by {\rm \eqref{eq:TRPCA_ADMM}} satisfies
\begin{equation}
\begin{aligned}
& L_{1}\left(\mathcal{L}^{(k+1)},\mathcal{H}^{(k+1)},\mathcal{E}^{(k+1)},\mathcal{Y}^{(k+1)},\mathcal{Z}^{(k+1)} \right)\\
\leq &L_{1}\left(\mathcal{L}^{(k)},\mathcal{H}^{(k)},\mathcal{E}^{(k)},\mathcal{Y}^{(k)},\mathcal{Z}^{(k)}\right)-c_{1} \| \mathcal{L}^{(k+1)}-\mathcal{L}^{(k)} \|_{F}^{2} \\
&-c_{2} \|\mathcal{H}^{(k+1)}-\mathcal{H}^{(k)} \|_{F}^{2} 
-c_{3} \|\mathcal{E}^{(k+1)}-\mathcal{E}^{(k)} \|_{F}^{2}
+c_{4} \|\mathcal{Z}^{(k+1)}-\mathcal{Z}^{(k)} \|_{F}^{2},
\end{aligned} \label{con:BsubL}
\end{equation}
with four positive constants $c_{1}, c_{2}, c_{3},c_4$. \label{lemma2P}
\end{lemma} 

\begin{lemma}
 Let $\mathcal{C}^{(k)}:=\left(\mathcal{L}^{(k)}, \mathcal{H}^{(k)}, \mathcal{E}^{(k)}, \mathcal{Y}^{(k)}, \mathcal{Z}^{(k)}\right)$ be the sequence generated by {\rm \eqref{eq:TRPCA_ADMM}}, then there exist a tensor $\mathcal{V}^{(k+1)}\in \partial L_{1}\left(\mathcal{C}^{(k+1)}\right)$
 and a constant $\kappa>0$ such that 
  \begin{equation}
 \begin{aligned}
\left\|\mathcal{V}^{(k+1)}\right\|_{F}^{2} &\leq \kappa\left\| \mathcal{C}^{(k+1)}-\mathcal{C}^{(k)}\right \|_{F}^{2}. 
 \end{aligned} \label{ine:vk}
 \end{equation} \label{lem:partialTRPCA}
\end{lemma}

\begin{theorem}
 Under assumptions {\rm A1-A2}, 
the sequence $$\mathcal{C}^{(k)}:=\left(\mathcal{L}^{(k)},\mathcal{H}^{(k)},\mathcal{E}^{(k)},\mathcal{Y}^{(k)},\mathcal{Z}^{(k)}\right)$$
generated by {\rm \eqref{eq:TRPCA_ADMM}} satisfies

{\rm \textrm{(i)}} The sequences \{$\mathcal{H}^{(k)}$\}, \{$\mathcal{E}^{(k)}\}$, \{$\mathcal{Y}^{(k)}\}$, and \{$\mathcal{Z}^{(k)}\}$ are bounded.

{\rm \textrm{(ii)}}The sequence $\{\mathcal{C}^{(k)}\}$ has a convergent subsequence. If $\lim\limits_{k\rightarrow +\infty}\|\mathcal{Z}^{(k+1)}-\mathcal{Z}^{(k)}\|_{F}=0$, this subsequence  converges to a critical point $\{\mathcal{L}^{\ast}, \mathcal{H}^{\ast}, \mathcal{E}^{\ast}, \mathcal{Y}^{\ast}, \mathcal{Z}^{\ast}\}$ with 
$\mathcal{O}\in \partial L_{1}(\mathcal{L}^{\ast}, \mathcal{H}^{\ast}, \mathcal{E}^{\ast}, \mathcal{Y}^{\ast}, \mathcal{Z}^{\ast})$,
where the zero tensor $\mathcal{O}$ is composed of five tensors, each of dimension ${n_1\times n_2\times n_3}$.
\label{proofthmP}
\end{theorem}

The proofs of Lemma~\ref{lemma1P}-Lemma~\ref{lem:partialTRPCA} and Theorem~\ref{proofthmP} are provided in the supplement. 
\begin{remark}
It is challenging to analyze the convergence of 
\eqref{eq:TRPCA_ADMM} due to the appearance of two Lagrangian multipliers, or so-called three-block ADMM \cite{chen2016direct}. Some existing works in the general optimization literature \cite{wang2019global} require an accompanying function, such as an objective function, merit function, or augmented Lagrangian function, that possesses properties such as being coercive, separable, or Lipschitz differentiable within a specific domain. However, none of these properties are satisfied for our TNF model.  Because \eqref{con:BsubL} includes a positive term  $\|\mathcal{Z}^{(k+1)}-\mathcal{Z}^{(k)} \|_{F}^{2}$ on the right-hand side, while the others are negative, we need to make an assumption about $\mathcal Z$ for the convergence analysis in Theorem~\ref{proofthmP}. This line of proof follows from two recent works \cite{du2021unifying,mu2020weighted}. 
\end{remark}

\section {A variant of the TNF-based TRPCA model}\label{sect:RPCAL}
This section introduces an alternative model based on TNF and $\ell_1/\ell_F$, along with an algorithm and its convergence analysis.

\subsection{The TNF$+$ model and its algorithm}
To mitigate the bias caused by the $\ell_1$ norm of $\mathcal E$ in \eqref{con:PCAa}, we propose utilizing $\ell_1/\ell_F$ to encourage sparsity of the tensor $\mathcal E$, thereby introducing a new model. The formulation of the second proposed model is given by
\begin{equation}  \label{con:PCAaL}
\min \limits_{\mathcal{L}, \mathcal{E}}\|\mathcal{L}\|_{\rm{TNF}} +\lambda \tfrac{\|\mathcal{E}\|_{1}}{\|\mathcal{E}\|_{F}} \quad
\text{s.t.}\quad  \mathcal{X}=\mathcal{L}+\mathcal{E},
\end{equation}
referred to as ``TNF$+$'' for the rest of the paper. Note that it is challenging to establish the recovery guarantee of the TNF$+$ model. The main difficulty lies in the two denominators in \eqref{con:PCAaL}, which change in opposite directions to satisfy the constraint $\mathcal{L}+\mathcal{E}=\mathcal{X}$, whereas TNF has only one fractional term. The analysis on the TNF$+$ model will be left to future work.

Similar to TNF, ADMM is employed to solve \eqref{con:PCAaL}. Specifically, we introduce two auxiliary variables $\mathcal{H}$ and $\mathcal{D}$ along with a specific splitting scheme that reformulates \eqref{con:PCAaL} into 
\begin{equation}  
{ \begin{array}{lc} \vspace{1ex}
\underset{\mathcal{L}, \mathcal{H}, \mathcal{E}, \mathcal{D}}{\operatorname{min}} \  \tfrac{\|\mathcal{L}\|_{*}}{\|\mathcal{H}\|_{F}}\ +\lambda \tfrac{\|\mathcal{E}\|_{1}}{\|\mathcal{D}\|_{F}} \\ 
\text{s.t.}\  \mathcal{X}=\mathcal{L}+\mathcal{E},\ \mathcal{H}=\mathcal{L},\ \mathcal{E}=\mathcal{D}. \end{array} }
\end{equation}
Its augmented Lagrangian function is written by
\begin{equation}
    \begin{aligned}
&L_{2}\left(\mathcal{L},\mathcal{H},\mathcal{E},\mathcal{D},\mathcal{Y},\mathcal{Z},\mathcal{U}\right) = \tfrac{\|\mathcal{L}\|_{*}}{\|\mathcal{H}\|_{F}}+\lambda\tfrac{\|\mathcal{E}\|_{1}}{\|\mathcal{D}\|_{F}}+\tfrac{\mu_{1}}{2}\left\|\mathcal{L}-\mathcal{H}\right\|_{F}^{2}
+\left\langle\mathcal{Y}, \mathcal{L}-\mathcal{H}\right\rangle \\
& +\tfrac{\mu_{2}}{2}\|\mathcal{L}+\mathcal{E}-\mathcal{X}\|_{F}^{2} 
+\left\langle\mathcal{Z},\mathcal{L}+\mathcal{E}-\mathcal{X}\right\rangle +\tfrac{\mu_{3}}{2}\left\|\mathcal{E}-\mathcal{D}\right\|_{F}^{2}
+\left\langle\mathcal{U}, \mathcal{E}-\mathcal{D}\right\rangle,
    \end{aligned} \label{eq:trpca_lagraL}
\end{equation}
with dual variables $\mathcal{Y}$, $\mathcal{Z}, \mathcal{U}$ and  positive parameters $\mu_1, \mu_2, \mu_3$. At each iteration, ADMM involves the following updates.
\begin{equation}
    \left\{\begin{array}{l}
\mathcal{L}^{(k+1)} =\arg \min\limits _{\mathcal{L}} L_{2}\left(\mathcal{L},\mathcal{H}^{(k)},\mathcal{E}^{(k)},\mathcal{D}^{(k)},\mathcal{Y}^{(k)},\mathcal{Z}^{(k)},\mathcal{U}^{(k)} \right),\\
\mathcal{H}^{(k+1)}=\arg \min \limits_{\mathcal{H}} L_{2}\left(\mathcal{L}^{(k+1)},\mathcal{H},\mathcal{E}^{(k)},\mathcal{D}^{(k)},\mathcal{Y}^{(k)},\mathcal{Z}^{(k)},\mathcal{U}^{(k)} \right),\\
\mathcal{E}^{(k+1)}=\arg \min\limits _{\mathcal{E}} L_{2}\left(\mathcal{L}^{(k+1)},\mathcal{H}^{(k+1)},\mathcal{E},\mathcal{D}^{(k)},\mathcal{Y}^{(k)},\mathcal{Z}^{(k)},\mathcal{U}^{(k)}\right),\\
\mathcal{D}^{(k+1)}=\arg \min\limits _{\mathcal{D}} L_{2}\left(\mathcal{L}^{(k+1)},\mathcal{H}^{(k+1)},\mathcal{E}^{(k+1)},\mathcal{D},\mathcal{Y}^{(k)},\mathcal{Z}^{(k)},\mathcal{U}^{(k)}\right),\\
\mathcal{Y}^{(k+1)}=\mathcal{Y}^{(k)}+\mu_{1}\left(\mathcal{L}^{(k+1)}-\mathcal{H}^{(k+1)}\right),\\
\mathcal{Z}^{(k+1)}=\mathcal{Z}^{(k)}+\mu_{2}\left(\mathcal{L}^{(k+1)}+\mathcal{E}^{(k+1)}-\mathcal{X}\right),\\
\mathcal{U}^{(k+1)}=\mathcal{U}^{(k)}+\mu_{3}\left(\mathcal{E}^{(k+1)}-\mathcal{D}^{(k+1)}\right).
\end{array}\right.\label{eq:TRPCA_ADMML}
\end{equation}

Since the $\mathcal{L}$-subproblem and the $\mathcal{H}$-subproblem
are the same as the ones in \eqref{eq:TRPCA_ADMM}, we use the same closed-form solutions for $\mathcal{L}^{(k+1)}$ and $\mathcal{H}^{(k+1)}$. 
The $\mathcal{E}$-subproblem of \eqref{eq:TRPCA_ADMML} can be expressed as
\begin{equation*}
\arg \min_{\mathcal{E}}  \left\{ \lambda
  \tfrac{\|\mathcal{E}\|_{1}}{\|\mathcal{D}^{(k)}\|_{F}}+\tfrac{\mu_{2}}{2}\left\|\mathcal{E}+\mathcal{L}^{(k+1)}-\mathcal{X}+\tfrac{\mathcal{Z}^{(k)}}{\mu_2}\right\|_{F}^{2}
 +\tfrac{\mu_{3}}{2}\left\|\mathcal{E}-\mathcal{D}^{(k)}+\tfrac{\mathcal{U}^{(k)}}{\mu_3}\right\|_{F}^{2}\right\}, \label{EfirstL}
\end{equation*}
which is equivalent to the $\ell_1$ minimization elementwise. Hence, it has a closed-form solution given by the soft-thresholding operator, i.e.,
\begin{equation}
\mathcal{E}^{(k+1)}=\mathbf{shrink} \left(\tfrac{\mu_2(\mathcal{X}-\mathcal{L}^{(k+1)})+\mu_3\mathcal{D}^{(k)}-\mathcal{Z}^{(k)}-\mathcal{U}^{(k)}}{\mu_{2}+\mu_{3}}, \ \tfrac{\lambda}{(\mu_2+\mu_3)\|\mathcal{D}^{(k)}\|_F}\right). \label{e-stepL}
\end{equation}

Lastly, the $\mathcal{D}$-subproblem of \eqref{eq:TRPCA_ADMML} can be expressed as
\begin{equation}
\mathcal{D}^{(k+1)}=  \arg \min_{\mathcal{E}}  \left\{ \lambda
  \tfrac{\|\mathcal{E}^{(k+1)}\|_{1}}{\|\mathcal{D}\|_{F}}+\tfrac{\mu_{3}}{2}\left\|\mathcal{D}-\mathcal{E}^{(k+1)}-\tfrac{\mathcal{U}^{(k)}}{\mu_3}\right\|_{F}^{2}
 \right\}. \label{DfirstL}
\end{equation}
Similar to $\mathcal{H}$-subproblem \eqref{con:newhupdate}, we  derive  the  closed-form solution of \eqref{DfirstL} to be
\begin{equation}
\mathcal{D}^{(k+1)}= \begin{cases}\zeta^{(k)}(\mathcal{E}^{(k+1)}+\tfrac{\mathcal{U}^{(k)}}{\mu_3})  & \text { if } \mathcal{E}^{(k+1)}+\tfrac{\mathcal{U}^{(k)}}{\mu_3} \neq \mathcal{O} \\ \mathcal{J}^{(k)} & \text { otherwise, }\end{cases} \label{con:newhupdateD}  
\end{equation}
where $\mathcal{J}^{(k)}$ is a random tensor with its Frobenius norm being $ \sqrt[3]{\tfrac{\beta^{(k+1)}}{\mu_1}}$ for $\beta^{(k+1)}=\lambda\|\mathcal{E}^{(k+1)}\|_{1}$ and $\zeta^{(k)} = \tfrac{1}{3}+\tfrac{1}{3}\left(B^{(k)}+\tfrac{1}{B^{(k)}}\right)$ for
$$ B^{(k)}=\sqrt[3]{\tfrac{27 A^{(k)}+2+\sqrt{(27 A^{(k)}+2)^{2}-4}}{2}} \text { and } A^{(k)}=\tfrac{\beta^{(k+1)}}{\mu_1 \|\mathcal{J}^{(k)}\|_{F}^{3}}.$$
We summarize the overall algorithm of ADMM for solving the problem \eqref{con:PCAaL} in Algorithm~\ref{alg:ADMM2LTRPCA}. Compared to Algorithm~\ref{alg:ADMM2TRPCA}, Algorithm~\ref{alg:ADMM2LTRPCA} incurs additional complexity due to the update of $\mathcal{D}$, which takes $O(n_1n_2n_3n_{(2)})$ and is of the same order as updating $\mathcal H$. Consequently, Algorithm~\ref{alg:ADMM2LTRPCA} exhibits equivalent complexity Algorithm~\ref{alg:ADMM2TRPCA}, that is,
$$O(n_1n_2(n_3\log n_3)+2n_1n_2n_3n_{(2)}).$$ 

\begin{algorithm}[htb] 
\caption{ The ADMM of the TNF$+$ model. 
}
\label{alg:ADMM2LTRPCA} 
\begin{algorithmic}[1] 
\REQUIRE  
  Observed data  $\mathcal{X}$, parameters: $\mu_1, \mu_2, \mu_3, \rm{kMax}, \epsilon$

\STATE{{\bf Initialization: } $(\mathcal {L}^{(0)}, \mathcal {E}^{(0)})$ by a TNN-based TRPCA model,   $\mathcal {H}^{(0)}= \mathcal {L}^{(0)}, \mathcal {D}^{(0)}= \mathcal {E}^{(0)},\mathcal Y^{(0)}=\mathcal Z^{(0)}=\mathcal U^{(0)}=\mathcal O,$ and $k = 0$. }
\WHILE {$k<\rm{kMax}$ or not converged}
\STATE $\tau^{(k+1)} = \tfrac{1}{(\mu_1+\mu_2)\|\mathcal{H}^{(k)}\|_F}$ 
\STATE $\mathcal{L}^{(k+1)} = \mathcal{D}_{\tau^{(k+1)}}\left(\tfrac{1}{\mu_{1}+\mu_{2}}\left(\mu_{1} \mathcal{H}^{(k)}+\mu_{2}\left(\mathcal{X}-\mathcal{E}^{(k)}\right)\right)-\tfrac{\mathcal{Y}^{(k)}+\mathcal{Z}^{(k)}}{\mu_{1}+\mu_{2}}\right)$
\STATE $\mathcal{H}^{(k+1)}= \begin{cases}\iota^{(k)} (\mathcal{L}^{(k+1)}+\tfrac{\mathcal{Y}^{(k)}}{\mu_{1}}) & \text { if } \mathcal{L}^{(k+1)}+\tfrac{\mathcal{Y}^{(k)}}{\mu_{1}} \neq \mathbf{0} \\ \mathcal{G}^{(k)} & \text { otherwise }\end{cases}$
\STATE $\mathcal{E}^{(k+1)}=\mathbf{shrink} \left(\tfrac{\mu_2(\mathcal{X}-\mathcal{L}^{(k+1)})+\mu_3\mathcal{D}^{(k)}-\mathcal{Z}^{(k)}-\mathcal{U}^{(k)}}{\mu_{2}+\mu_{3}}, \ \tfrac{\lambda}{(\mu_2+\mu_3)\|\mathcal{D}^{(k)}\|_F}\right)$
\STATE $\mathcal{D}^{(k+1)}= \begin{cases}\zeta^{(k)}(\mathcal{E}^{(k+1)}+\tfrac{\mathcal{U}^{(k)}}{\mu_3})  & \text { if } \mathcal{E}^{(k+1)}+\tfrac{\mathcal{U}^{(k)}}{\mu_3} \neq \mathcal{O} \\ \mathcal{J}^{(k)} & \text { otherwise }\end{cases}$
\STATE $\mathcal{Y}^{(k+1)}=\mathcal{Y}^{(k)}+\mu_{1}\left(\mathcal{L}^{(k+1)}-\mathcal{H}^{(k+1)}\right)$
\STATE $\mathcal{Z}^{(k+1)}=\mathcal{Z}^{(k)}+\mu_{2}\left(
\mathcal{L}^{(k+1)}+\mathcal{E}^{(k+1)}-\mathcal{X}\right)$
\STATE $\mathcal{U}^{(k+1)}=\mathcal{U}^{(k)}+\mu_{3}\left(\mathcal{E}^{(k+1)}-\mathcal{D}^{(k+1)}\right)$
\STATE $k=k+1$ 
\STATE Check the convergence conditions \\
$\left\|\mathcal{L}^{(k+1)}-\mathcal{L}^{(k)}\right\|_{\infty} \leq \epsilon$,  
$\left\|\mathcal{H}^{(k+1)}-\mathcal{H}^{(k)}\right\|_{\infty} \leq \epsilon$,
$\left\|\mathcal{E}^{(k+1)}-\mathcal{E}^{(k)}\right\|_{\infty} \leq \epsilon$,
$\left\|\mathcal{D}^{(k+1)}-\mathcal{D}^{(k)}\right\|_{\infty} \leq \epsilon$, 
$\left\|\mathcal{Y}^{(k+1)}-\mathcal{Y}^{(k)}\right\|_{\infty} \leq \epsilon$,
$\left\|\mathcal{Z}^{(k+1)}-\mathcal{Z}^{(k)}\right\|_{\infty} \leq \epsilon$,
$\left\|\mathcal{L}^{(k+1)}+\mathcal{E}^{(k+1)}-\mathcal{X}\right\|_{\infty} \leq \epsilon$

\ENDWHILE
\RETURN $\hat{\mathcal{L}} = \mathcal{L}^{(k)}$ and $\hat{\mathcal{E}} = \mathcal{E}^{(k)}$ 
\end{algorithmic}
\end{algorithm}

\subsection{Convergence for the TNF$+$ model}
\label{sec:conv-TRPCAl2}

We show that the sequence generated by
Algorithm~\ref{alg:ADMM2LTRPCA} has a subsequence convergent to a stationary point of \eqref{con:PCAaL} under the following two assumptions. 

A3: The sequence  $(\{\mathcal{L}^{(k)}\}, \{\mathcal{E}^{(k)}\})$ generated by \eqref{eq:TRPCA_ADMML} is bounded, so are the nuclear norm of $\mathcal{L}^{(k)}$ and the $\ell_1$ norm of $\{\mathcal{E}^{(k)}\}$, denoted by $\sup_k \{\|\mathcal{L}^{(k)}\|_{*}\}\leq M$ and $\sup_k \{\|\mathcal{E}^{(k)}\|_{1}\}\leq m$.  
 
A4: The Frobenius norm of $\{\mathcal{H}^{(k)}\}$ and $\{\mathcal{D}^{(k)}\}$  have uniform bounds, i.e., there exist positive constants $\delta_1$ and $\delta_2$ such that $\|\mathcal{H}^{(k)}\|_F\geq \delta_1, \forall k$ and $\|\mathcal{D}^{(k)}\|_F\geq \delta_2, \forall k$.

\begin{lemma}\label{lemma4}
 Under assumptions {\rm A3-A4} with sufficiently large parameters $\mu_1,\mu_2$, the sequence $\{\mathcal{Y}^{(k)}\}$ and $\{\mathcal{U}^{(k)}\}$ generated by {\rm \eqref{eq:TRPCA_ADMML}} satisfies
 \begin{equation}
    \left\|\mathcal{Y}^{(k+1)}-\mathcal{Y}^{(k)}\right\|_{F}^{2} \leq \tfrac{2 n_{(2)}}{\delta_1^{4}}\left\|\mathcal{L}^{(k+1)}-\mathcal{L}^{(k)}\right\|_F^2+\tfrac{4M^{2}}{\delta_1^{6}}\left\|\mathcal{H}^{(k+1)}-\mathcal{H}^{(k)}\right\|_{F}^{2}, \label{ine:yTRPCAL}
 \end{equation}
  \begin{equation}
    \left\|\mathcal{U}^{(k+1)}-\mathcal{U}^{(k)}\right\|_{F}^{2} \leq \tfrac{2 \lambda^2 n_1n_{2}n_3}{\delta_2^{4}}\left\|\mathcal{E}^{(k+1)}-\mathcal{E}^{(k)}\right\|_F^2+\tfrac{4\lambda^2 m^{2}}{\delta_2^{6}}\left\|\mathcal{D}^{(k+1)}-\mathcal{D}^{(k)}\right\|_{F}^{2}, \label{ine:uTRPCAL}
 \end{equation}
 where $M$, $m$, $\delta_1$ and $\delta_2$ are the constants defined in {\rm A3} and {\rm A4}, respectively. 
\end{lemma}

\begin{lemma}
Under assumptions {\rm A3-A4}, the augmented Lagrangian function \eqref{eq:trpca_lagraL} of the sequence $\{\mathcal{L}^{(k)},\mathcal{H}^{(k)},\mathcal{E}^{(k)},\mathcal{D}^{(k)},\mathcal{Y}^{(k)},\mathcal{Z}^{(k)},\mathcal{U}^{(k)}\}$ generated by {\rm \eqref{eq:TRPCA_ADMML}} satisfies
\begin{equation}
\begin{aligned}
& L_{2}\left(\mathcal{L}^{(k+1)},\mathcal{H}^{(k+1)},\mathcal{E}^{(k+1)},\mathcal{D}^{(k+1)},\mathcal{Y}^{(k+1)},\mathcal{Z}^{(k+1)},\mathcal{U}^{(k+1)} \right)\\
\leq &L_{2}\left(\mathcal{L}^{(k)},\mathcal{H}^{(k)},\mathcal{E}^{(k)},\mathcal{D}^{(k)},\mathcal{Y}^{(k)},\mathcal{Z}^{(k)},\mathcal{U}^{(k)}\right)-c_{5} \| \mathcal{L}^{(k+1)}-\mathcal{L}^{(k)} \|_{F}^{2} \\
& -c_{6} \|\mathcal{H}^{(k+1)}-\mathcal{H}^{(k)} \|_{F}^{2} 
-c_{7} \|\mathcal{E}^{(k+1)}-\mathcal{E}^{(k)} \|_{F}^{2}
-c_{8} \|\mathcal{D}^{(k+1)}-\mathcal{D}^{(k)} \|_{F}^{2}\\
& +c_{9} \|\mathcal{Z}^{(k+1)}-\mathcal{Z}^{(k)} \|_{F}^{2},
\end{aligned} 
\end{equation}
where $c_{5}, c_{6}, c_{7},c_8,c_9$ are positive constants. \label{lemmasdTRPCAL}
\end{lemma}

\begin{lemma}
 Let $\mathcal{C}^{(k)}:=\left(\mathcal{L}^{(k)}, \mathcal{H}^{(k)}, \mathcal{E}^{(k)},\mathcal{D}^{(k)},\mathcal{Y}^{(k)},\mathcal{Z}^{(k)},\mathcal{U}^{(k)}\right)$ be the sequence generated by {\rm \eqref{eq:TRPCA_ADMML}}, then there exist a tensor $\mathcal{W}^{(k+1)}\in \partial L_{2}\left(\mathcal{C}^{(k+1)}\right)$
 and a constant $\kappa_2>0$ such that 
  \begin{equation}
 \begin{aligned}
\left\|\mathcal{W}^{(k+1)}\right\|_{F}^{2} &\leq \kappa_2\left\| \mathcal{C}^{(k+1)}-\mathcal{C}^{(k)}\right \|_{F}^{2}. 
 \end{aligned} \label{ine:vk2}
 \end{equation} \label{lem:partialTRPCA2}
\end{lemma}

\begin{theorem}
 Under assumptions {\rm A3-A4}, 
the sequence $$\mathcal{C}^{(k)}:=\left(\mathcal{L}^{(k)},\mathcal{H}^{(k)},\mathcal{E}^{(k)},\mathcal{D}^{(k)},\mathcal{Y}^{(k)},\mathcal{Z}^{(k)},\mathcal{U}^{(k)}\right)$$ generated by {\rm \eqref{eq:TRPCA_ADMML}} satisfies

{\rm \textrm{(i)}} The sequences \{$\mathcal{H}^{(k)}$\}, \{$\mathcal{E}^{(k)}\}$,\{$\mathcal{D}^{(k)}\}$, \{$\mathcal{Y}^{(k)}\}$, \{$\mathcal{Z}^{(k)}\}$ and \{$\mathcal{U}^{(k)}\}$ are bounded. 

{\rm \textrm{(ii)}} 
The sequence $\{\mathcal{C}^{(k)}\}$ has a convergent subsequence. If $$\lim\limits_{k\rightarrow +\infty}\|\mathcal{Z}^{(k+1)}-\mathcal{Z}^{(k)}\|_{F}=0,$$  then this subsequence  converges to a 
 critical point $\{\mathcal{L}^{\ast}, \mathcal{H}^{\ast}, \mathcal{E}^{\ast}, \mathcal{D}^{\ast},\mathcal{Y}^{\ast}, \mathcal{Z}^{\ast},\mathcal{U}^{\ast}\}$, i.e., 
 $\mathcal{O}\in \partial L_{2}(\mathcal{L}^{\ast}, \mathcal{H}^{\ast}, \mathcal{E}^{\ast}, \mathcal{D}^{\ast}, \mathcal{Y}^{\ast}, \mathcal{Z}^{\ast},\mathcal{U}^{\ast})$ 
  where the zero tensor $\mathcal{O}$ is composed of seven tensors, each of  dimension ${n_1\times n_2\times n_3}$.
\label{prooflastTRPCAL}
\end{theorem}
We present the proofs of Lemma~\ref{lemma4} and Lemma~\ref{lemmasdTRPCAL} in the supplementary material while omitting the proofs of Lemma~\ref{lem:partialTRPCA2} and Theorem~\ref{prooflastTRPCAL} due to their similarity to the ones in the TNF model. 

\section{Experiments}\label{sect:experimentsPCA}
This section contains extensive experiments aimed at evaluating the performance of our proposed TNF and TNF$+$ models using both synthetic and real-world datasets. 
In the synthetic scenario, the observed data are generated as the sum of a low-rank tensor and a sparse tensor. For denoising experiments, we use real color images with manually added sparse noises. Furthermore, we employ surveillance videos for background modeling, wherein the video is decomposed into a low-rank background tensor and a sparse motion component. Although real-world data may not be strictly low rank, the underlying tensors can be predominantly approximately by their top singular values. Therefore, the proposed methodologies still yield satisfactory results. All experiments are conducted using MATLAB (R2023a) on the Windows 10 platform with an Intel Core i5-1135G7 2.40 GHz processor and 16 GB of RAM. 

\subsection {Synthetic data}
We generate each observation $\mathcal{X} \in \mathbb{R}^{n_1 \times n_2 \times n_{3}}$ by adding a low-rank tensor $\mathcal{L}_{0}$ and a sparse component $\mathcal{E}_{0}$ of the same dimension. Here $\mathcal{L}_{0}$ is obtained by the t-product of two tensors of smaller dimensions, i.e., $\mathcal{L}_{0}=\mathcal{P}*\mathcal{Q}$, where $\mathcal{P} \in \mathbb{R}^{n_1 \times r \times n_{3}}$ and $\mathcal{Q} \in \mathbb{R}^{r \times n_2 \times n_{3}}$ with $r\ll n_{(2)}$. The tubal rank of the resulting tensor $\mathcal{L}_{0}$ is at most $r$. The elements of tensor $\mathcal{P}$ are drawn from an independent and identically distributed (i.i.d.) Gaussian distribution $\mathcal{N}\left(0, 1/n_1\right)$, while the elements of $\mathcal{Q}$ are drawn from $\mathcal{N}\left(0, 1/n_2\right)$. As for the sparse components $\mathcal{E}_{0}$, we assume its support follows a Bernoulli distribution. Specifically, we randomly set the values of its entries to either $+1$ or $-1$, each with probability $\gamma$, and set to 0 with probability $1-2\gamma,$ where $2\gamma\in [0,1]$ is referred to as a sampling rate or a sparsity level.

We start with empirical evidence of convergence in the proposed TRPCA methodologies by considering a third-order tensor of dimensions $40 \times 40 \times 30$, with a tubal rank of 3 and a sampling rate of 0.2. The relative square errors of tensors $\mathcal{L}^{(k)}$ and $\mathcal{E}^{(k)}$ to the corresponding ground truth $\mathcal L_0$ and $\mathcal E_0$ at each iteration  
$k$ are depicted in Fig.~\ref{fig:empirical_conv}, showing that the errors of both models
reduce to less than $10^{-15}$ in about 100 iterations. Also, TNF$+$ has a faster convergence compared to TNF. 

 \begin{figure}
 		\begin{center}
 			\begin{tabular}{cc}
 				TNF & TNF$+$\\
 				\includegraphics[width=0.45\textwidth]{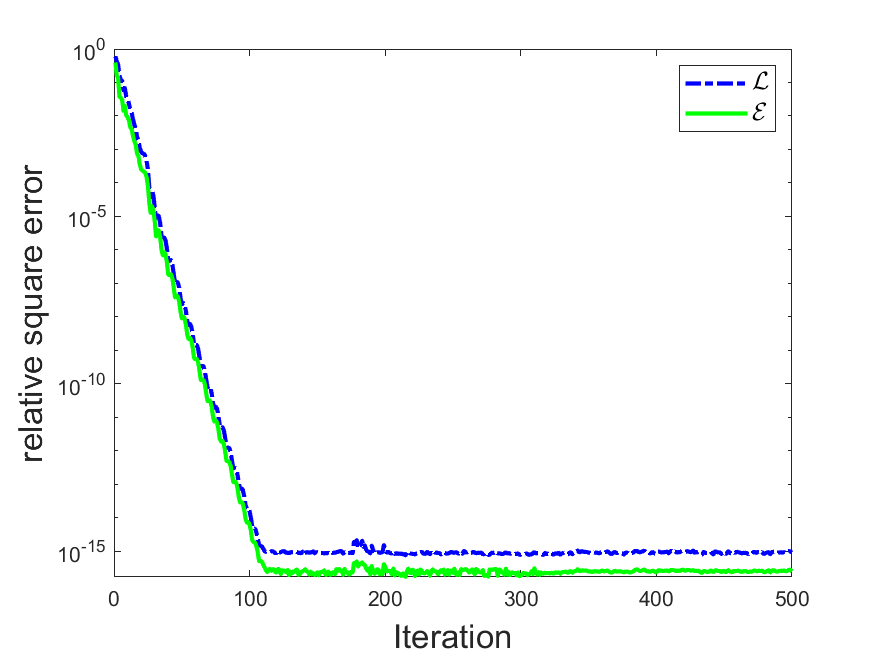}  &
 				\includegraphics[width=0.45\textwidth]{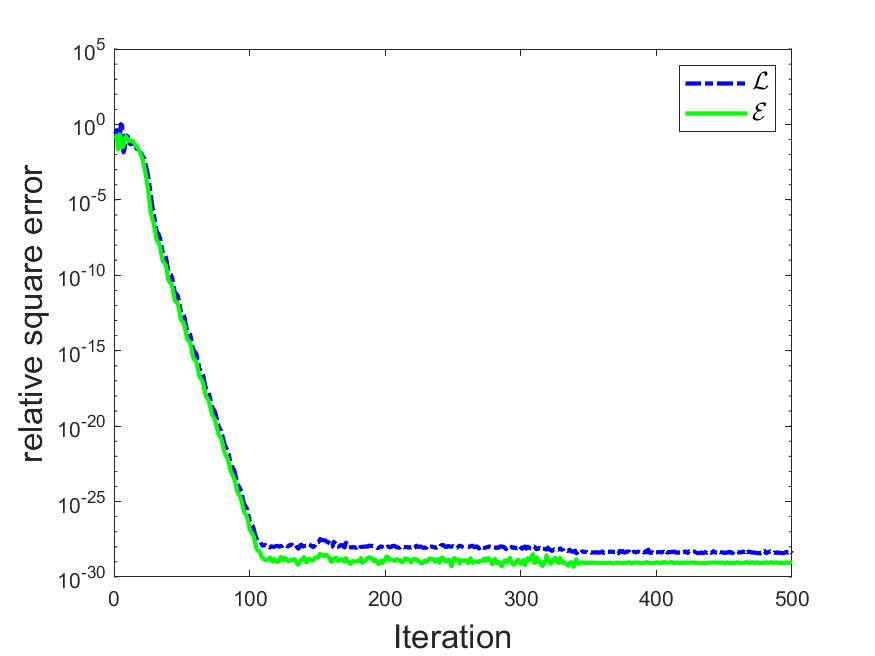} 
 			\end{tabular}
 		\end{center}
 		\caption{Empirical evidence on convergence in TRPCA by plotting
   the relative square errors between the current tensor $\mathcal L^{(k)}$ ($\mathcal E^{(k)}$) and the ground truth $\mathcal{L}_{0}$ ($\mathcal{E}_{0}$) with respect to the iteration index $k$ for TNF (left) and TNF$+$ (right) models. 
   } \label{fig:empirical_conv}
 	\end{figure} 

Next, we conduct a comparative study of our proposed TNF and TNF$+$ models with some existing works, including TNN \cite{lu2019tensor}, 
Laplace \cite{xu2019laplace}, $\mathrm{t}\mbox{-}S_{w, p}$ \cite{YANG2022108311}, and $p$-TRPCA \cite{yan2024tensor}. For our models, we set $\epsilon$ to $10^{-4}$ in both  Algorithm~\ref{alg:ADMM2TRPCA} and Algorithm~\ref{alg:ADMM2LTRPCA}. We gradually increase the values of $\mu_1$ and $\mu_2$ instead of fixed values for acceleration, as considered in \cite{lu2019tensor}.  Specifically, 
 we initialize $\mu_1=10^{-4}$ and $\mu_2=10^{-3}$ in TNF
while $\mu_1=10^{-4}$ and $\mu_2=\mu_3=10^{-3}$ in TNF$+$. 
For both TNF and TNF$+,$ we consider an increment factor of 1.1 in each iteration and a maximum cap of these parameters by
$10^{10}$. We set $\lambda=2\times 10^{-4}$ in TNF  for all synthetic experiments, although a finer tuning of $\lambda$ could potentially enhance our model's performance.
For TNF$+$ model, we set the value of $\lambda$ to $\frac{1}{\sqrt{n_{(1)}n_3}}=0.0289$, consistent with the $\lambda$ value used in the TNN model \cite{lu2019tensor}.
As for the competing methods (TNN, $\mathrm{t}\mbox{-}S_{w, p}$, and $p$-TRPCA), we employ the Matlab codes provided by the respective authors with default parameter settings. For the Laplace model, we adapt the code of the tensor completion model into the TRPCA model while setting $\epsilon$ to $10^{-6}$.

We employ success rates as a metric to assess recovery performance, which is defined as the ratio of successful trials to the total number of trials. Specifically, we conduct ten independent random trials for each combination of a predetermined tubal rank $(r)$ and sampling rate $(2\gamma)$.
A trial is deemed successful if the relative square error between the recovered tensor $\hat{\mathcal{L}}$ and the ground-truth tensor $\mathcal{L}_{0}$, denoted as $\frac{\|\hat{\mathcal{L}}-\mathcal{L}_{0}\|_{F}^{2}}{\|\mathcal{L}_{0}\|_{F}^{2}}$, is less than $10^{-3}$. The success rate is then calculated by dividing the number of successful trials by 10.
Finally, we adhere to the experimental setup in \cite{jiang2020multi} for using the tensor dimension of $40 \times 40 \times 30$. The tubal rank in $\mathcal{L}$ ranges from 1 to 19 with an increment of 2, while the sparsity in $\mathcal{E}$ varies from 0.05 to 0.5 with an increment of 0.05.

Each cell in Fig.~\ref{fig:TRPCAratio} illustrates the success rate corresponding to a combination of tubal rank and sparsity levels. Generally, successful recovery is more probable when the sparsity level or tubal rank is relatively low. Fig.~\ref{fig:TRPCAratio} showcases that our models outperform the state-of-the-art methods, particularly when the specified $\mathcal{L}$ rank is low.

                   

\begin{figure}
   \begin{center}
 			\begin{tabular}{ccc}
 				  TNN  & Laplace & $\mathrm{t}\mbox{-}S_{w, p}(0.9)$\\
                  \includegraphics[width=0.3\textwidth]{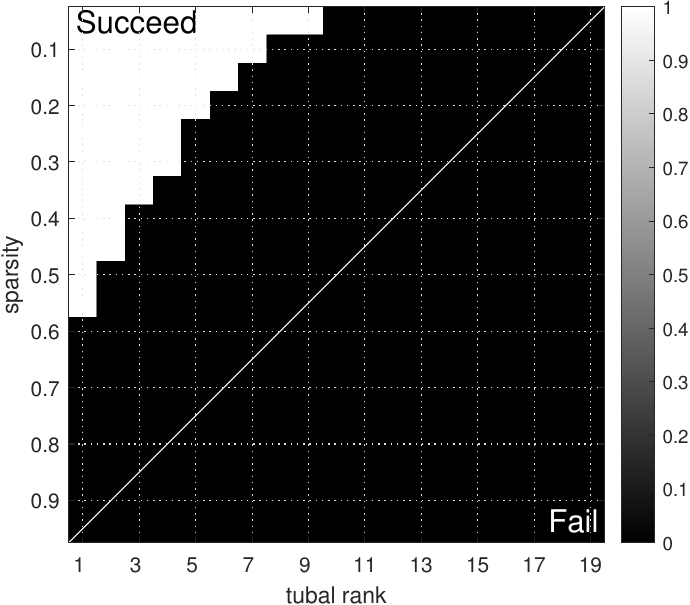}  &
 				\includegraphics[width=0.3\textwidth]{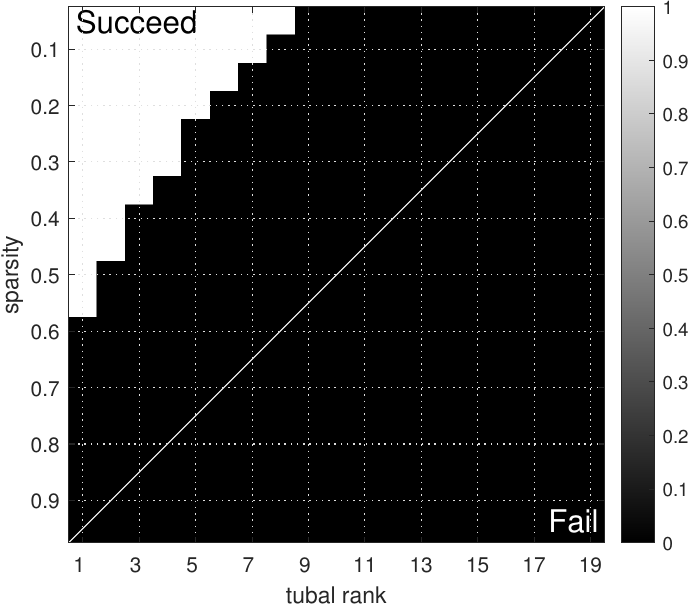} &
                 \includegraphics[width=0.3\textwidth]{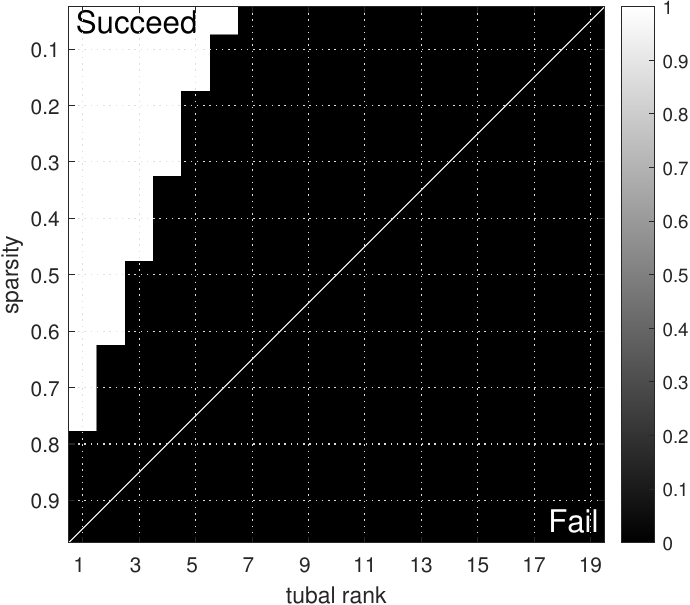}\\
                    $p$-TRPCA & TNF & TNF$+$\\
 				 \includegraphics[width=0.3\textwidth]{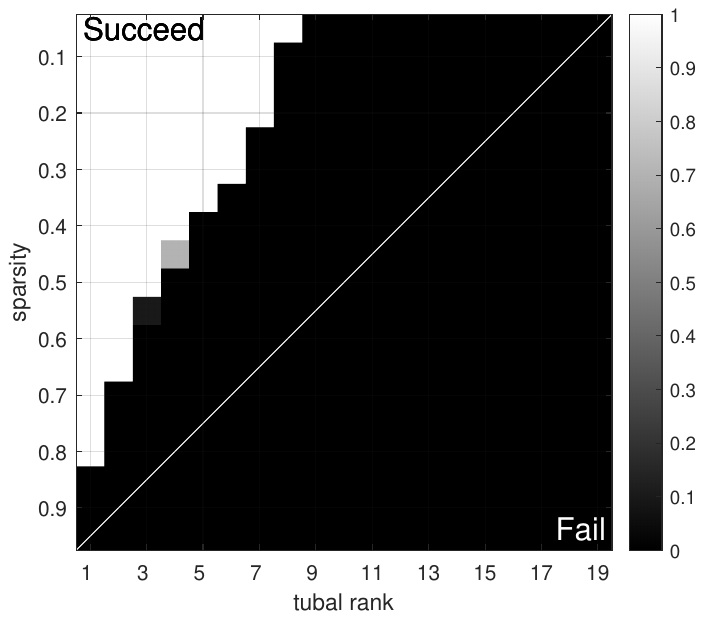} &
                      \includegraphics[width=0.3\textwidth]{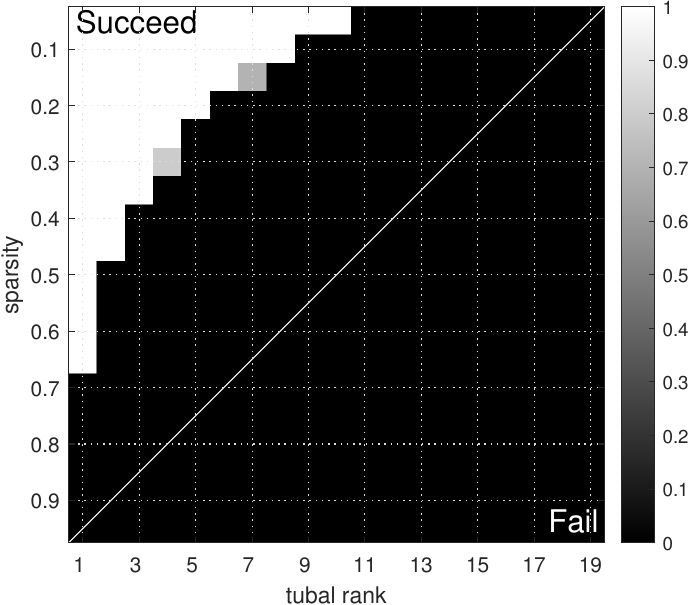} &
 				 \includegraphics[width=0.3\textwidth]{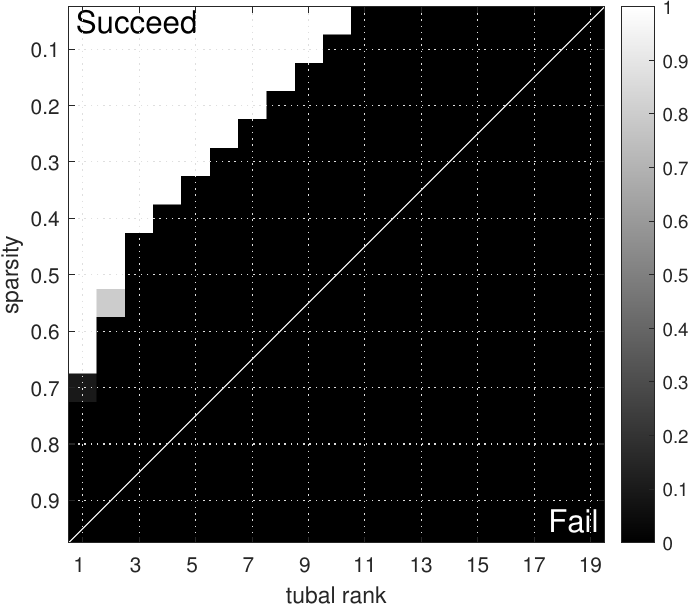} \\
                   
 			\end{tabular}
 		\end{center}
 		\caption{The success rates of various methods for the TRPCA problem with varying tubal ranks $(r)$ and sparsity levels $(\rho)$. Each cell represents the percentage of successful recoveries over ten independent realizations. White dashed lines have been added along the diagonal to facilitate comparison.
   } \label{fig:TRPCAratio}
 	\end{figure}
    
\subsection {Real-world data}
We perform experiments on real-world data comprising color images and videos. For image denoising tests, we employ the peak signal-to-noise ratio (PSNR) \cite{lu2019tensor} and the structural similarity index (SSIM) \cite{wang2004image} to quantitatively evaluate recovery performance. Additionally, we present background separation results using grayscale videos. Since the final results are evaluated using PSNR and SSIM, we have adopted the same termination criterion for all competing algorithms; see Algorithm~\ref{alg:ADMM2TRPCA} and Algorithm~\ref{alg:ADMM2LTRPCA}. 


\subsubsection{Color image denoising.}

We conduct image denoising experiments on five color images, labeled by ``boat'', ``houses'', ``seabeach'', ``bicycle'', and ``brook.'' These images can be obtained online\footnote{\url{http://r0k.us/graphics/kodak}\label{note2}}. Each image is corrupted by sparse noise, where 20$\%$ of the pixels randomly receive values in the range of 0 to 255, with the locations of the distorted pixels unspecified.
Both TNF and TNF$+$ models are applied for image denoising. In the TNF model, we select the best $\lambda$ value from the range $[4.5:0.5:6.5]\times 10^{-5}$ that achieves the highest PSNR. The initial values are set as $\mu_1=\mu_2=10^{-4}$. Conversely, for the TNF$+$ model, we choose $\lambda$ among the set $[1.6:0.4:2.8]\times 10^{-2}$ and initialize $\mu_1=10^{-4}, \mu_2=10^{-2}, \mu_3=10^{-4}$. 

We compare the TNF and TNF$+$ models with 
TNN \cite{lu2019tensor}, Laplace \cite{xu2019laplace}, $\mathrm{t}\mbox{-}S_{w, p}$ \cite{YANG2022108311}, and $p$-TRPCA \cite{yan2024tensor}. Quantitative evaluations in terms of PSNR and SSIM are presented in Table~\ref{table:TRPCA}, indicating improved performance with our proposed TNF regularization over state-of-the-art TRPCA methods for denoising sparse noise. 
Among the five color images, the best performance is achieved either by TNF or TNF$+$. Notably, the TNF model achieves the highest PSNR for the ``houses'', ``bicycle'', and ``brook'' images, while the TNF$+$ model appears to perform the best according to the SSIM metric. TNF and TNF$+$ achieve the top two performances in most cases. Specifically, the average PSNR of TNF and TNF$+$ is $28.1005$ and $28.0875$, respectively, both exceeding the rest of the methods by more than $0.2.$

We present visual recovery results in Fig.~\ref{fig:TRPCAreal}. Each image contains a zoomed region for ease of comparison. 
The noisy inputs are depicted in the second column of Fig.~\ref{fig:TRPCAreal}, exhibiting severe speckle artifacts. TNF and TNF$+$ provide results with fewer speckles, particularly noticeable in the zoomed region of the ``bicycle'' image. In the ``house'' and ``seabeach'' images, our methods better preserve the details of zoomed-in letters and trees, while TNN exhibits some blurring. Additionally, Laplace and $p$-TRPCA models retain artifacts from sparse noise. Moreover, our proposed model effectively removes noise without excessively smoothing the image, as compared to  $\mathrm{t}\mbox{-}S_{w, p}(0.9)$. This difference is particularly evident in the ``boat" and ``brook" images.

To analyze the efficiency of all approaches, we summarize the runtime of the algorithms in Table \ref{table:time_comparison}.
\begin{table}[h!]
\centering
\caption{Comparison of computation time.}\label{table:time_comparison}
\begin{tabular}{|c|c|c|c|c|c|c|}
\hline
\diagbox{Name}{Time} & TNN  & Laplace & $S_{wp}(0.9)$ &$p$-TRPCA & TNF & TNF$+$ \\ \hline
``boat''  & 13.7680  & 11.9506   & 8.6724 &36.3645  & 16.8582  & 16.8190   \\ \hline
``house''  & 13.5732  & 13.0155   & 8.1021 &36.4282  & 12.3516  & 16.8545   \\ \hline
``seabeach''  & 13.4972  & 12.2643   & 8.3995& 32.7584  & 12.2160  & 17.2248   \\ \hline
``bicycle''  & 14.0821  & 12.1993   & 8.4048&37.6260   & 13.3229  & 17.3746   \\ \hline
``brook''  & 14.2004  & 13.7930   & 7.9922& 34.4966  & 14.7146  & 17.1482   \\ \hline
\end{tabular}
\end{table}
The TNF algorithm generally runs faster than $p$-TRPCA and TNF$+$, but slower than Laplace and $S_{wp}(0.9)$.  Although TNF+ is slightly slower than TNF, it is still faster than PSTNN and $p$-TRPCA and is competitive with Laplace in some cases.

\begin{table}
\centering
\caption{Quantitative comparisons of denoising results. }
\label{table:TRPCA}
\scalebox{0.85}{
\begin{tabular}{llllllllll}
\hline
Image   & Index  & observed  & TNN  & Laplace  & $S_{wp}(0.9)$  & $p$-TRPCA & TNF & TNF$+$  \\ 
\hline
``boat'' & \begin{tabular}[c]{@{}l@{}}PSNR\\ SSIM\end{tabular} & \begin{tabular}[c]{@{}l@{}}15.5721\\ 0.4187\end{tabular} & \begin{tabular}[c]{@{}l@{}}28.6729\\ 0.9394\end{tabular} & \begin{tabular}[c]{@{}l@{}}29.1380\\ 0.9365\end{tabular}  & \begin{tabular}[c]{@{}l@{}}29.7413\\ 0.9492 \end{tabular}& \begin{tabular}[c]{@{}l@{}}29.5749\\ 0.9562\end{tabular} & \begin{tabular}[c]{@{}l@{}}29.9560\\ 0.9547\end{tabular} & \begin{tabular}[c]{@{}l@{}}\textbf{29.9658}\\ \textbf{0.9625}\end{tabular}                                                \\ \hline
``house''  & \begin{tabular}[c]{@{}l@{}}PSNR\\ SSIM\end{tabular} & \begin{tabular}[c]{@{}l@{}}15.4087\\ 0.6319\end{tabular} & \begin{tabular}[c]{@{}l@{}}24.9451\\ 0.9379\end{tabular} & \begin{tabular}[c]{@{}l@{}}25.2512\\ 0.9409\end{tabular}  & \begin{tabular}[c]{@{}l@{}}26.1093\\ 0.9463\end{tabular}& \begin{tabular}[c]{@{}l@{}}26.1068\\ 0.9354\end{tabular} & \begin{tabular}[c]{@{}l@{}}\textbf{26.3986}\\ 0.9515\end{tabular} & \begin{tabular}[c]{@{}l@{}} 26.1202\\ \textbf{0.9548}\end{tabular}                                                \\ \hline
``seabeach''& \begin{tabular}[c]{@{}l@{}}PSNR\\ SSIM\end{tabular} & \begin{tabular}[c]{@{}l@{}}16.1142\\ 0.3389\end{tabular} & \begin{tabular}[c]{@{}l@{}}31.8564\\ 0.9552\end{tabular} & \begin{tabular}[c]{@{}l@{}}32.6281\\ 0.9565\end{tabular} & \begin{tabular}[c]{@{}l@{}}33.5045\\ 0.9660\end{tabular}& \begin{tabular}[c]{@{}l@{}}32.7584\\ 0.9653\end{tabular} & \begin{tabular}[c]{@{}l@{}}33.2951\\ 0.9653\end{tabular} & \begin{tabular}[c]{@{}l@{}}\textbf{33.7234}\\ \textbf{0.9712}\end{tabular}                                                \\ \hline
``bicycle''  & \begin{tabular}[c]{@{}l@{}}PSNR\\ SSIM\end{tabular} & \begin{tabular}[c]{@{}l@{}}15.0582\\ 0.3607\end{tabular}  & \begin{tabular}[c]{@{}l@{}}24.2996\\ 0.9159\end{tabular} & \begin{tabular}[c]{@{}l@{}}24.5643\\ 0.9082\end{tabular} &\begin{tabular}[c]{@{}l@{}} 25.3992\\ 0.9324\end{tabular}& \begin{tabular}[c]{@{}l@{}}24.9952\\ 0.9174\end{tabular} & \begin{tabular}[c]{@{}l@{}}\textbf{25.6698}\\ 0.9350\end{tabular} & \begin{tabular}[c]{@{}l@{}}25.6359\\ \textbf{0.9471}\end{tabular}                                                 \\ \hline
``brook''  & \begin{tabular}[c]{@{}l@{}}PSNR\\ SSIM\end{tabular} & \begin{tabular}[c]{@{}l@{}}15.5316\\ 0.5614\end{tabular} & \begin{tabular}[c]{@{}l@{}}23.9839\\ 0.9013\end{tabular} & \begin{tabular}[c]{@{}l@{}}24.6561\\ 0.9074\end{tabular} & \begin{tabular}[c]{@{}l@{}}24.6090\\ 0.9076\end{tabular}& \begin{tabular}[c]{@{}l@{}}23.1515\\ 0.8896\end{tabular} & \begin{tabular}[c]{@{}l@{}}\textbf{25.1829}\\ 0.9263\end{tabular} & \begin{tabular}[c]{@{}l@{}}24.9921\\ \textbf{0.9331}\end{tabular} \\ \hline
average  & \begin{tabular}[c]{@{}l@{}}PSNR\\ SSIM\end{tabular} & \begin{tabular}[c]{@{}l@{}}15.5370\\ 0.4623\end{tabular} & \begin{tabular}[c]{@{}l@{}}26.7516\\ 0.9299\end{tabular} & \begin{tabular}[c]{@{}l@{}}27.2475\\ 0.9299\end{tabular} & \begin{tabular}[c]{@{}l@{}}27.8727\\ 0.9403\end{tabular}& \begin{tabular}[c]{@{}l@{}}27.5476\\ 0.9341\end{tabular} & \begin{tabular}[c]{@{}l@{}}\textbf{28.1005}\\ 0.9466\end{tabular} & \begin{tabular}[c]{@{}l@{}}28.0875\\ \textbf{0.9537}\end{tabular} \\ \hline
\end{tabular}}
\end{table}

\begin{figure}
	\begin{center}
		\begin{tabular}{cccccc}   
			\hspace{-5mm}(a)\hspace{-3mm} & \includegraphics[width=2.2cm]{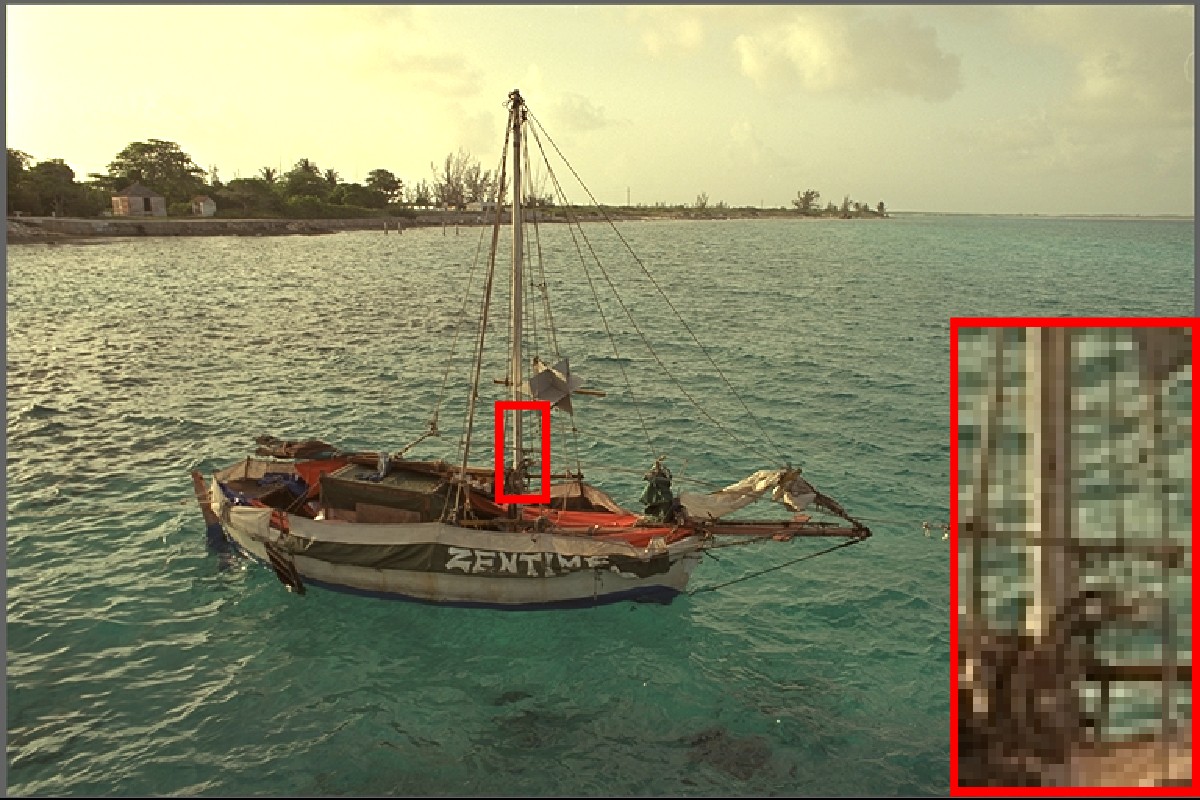} \hspace{-3.9mm} & 
			\includegraphics[width=2.2cm]{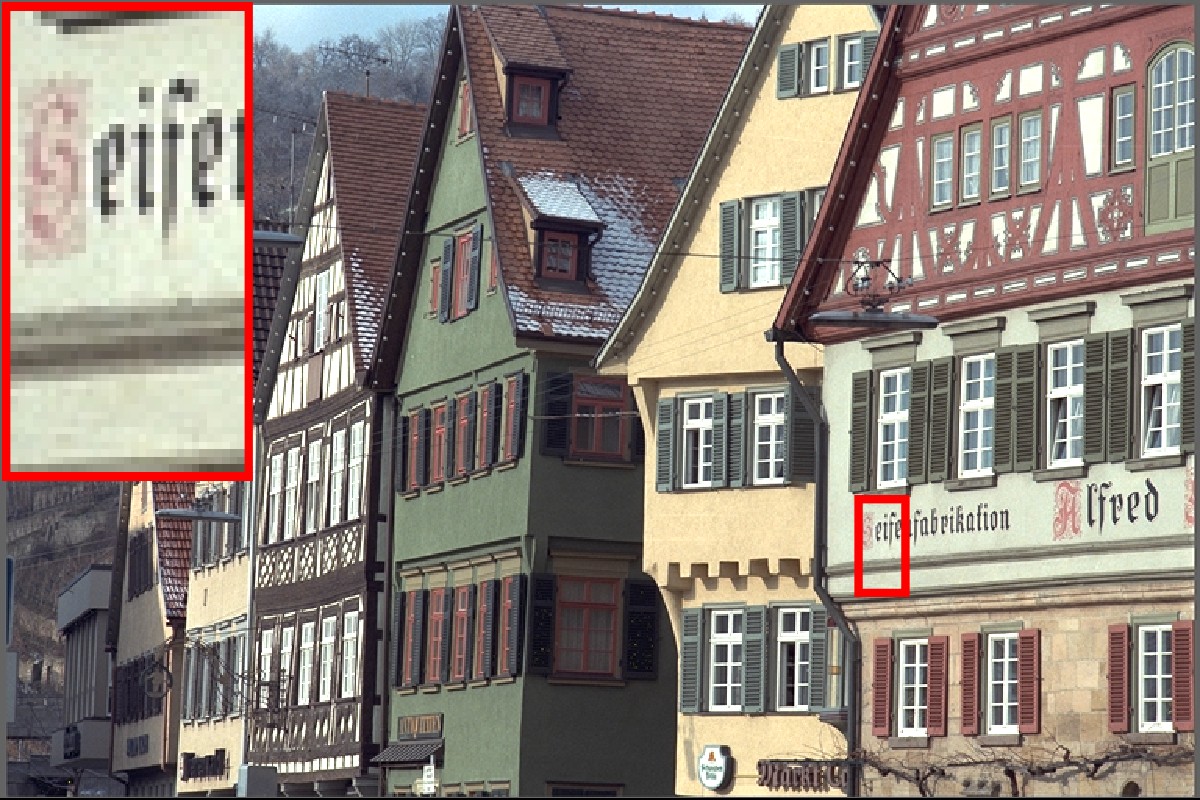} \hspace{-3.9mm} & 
			\includegraphics[width=2.2cm]{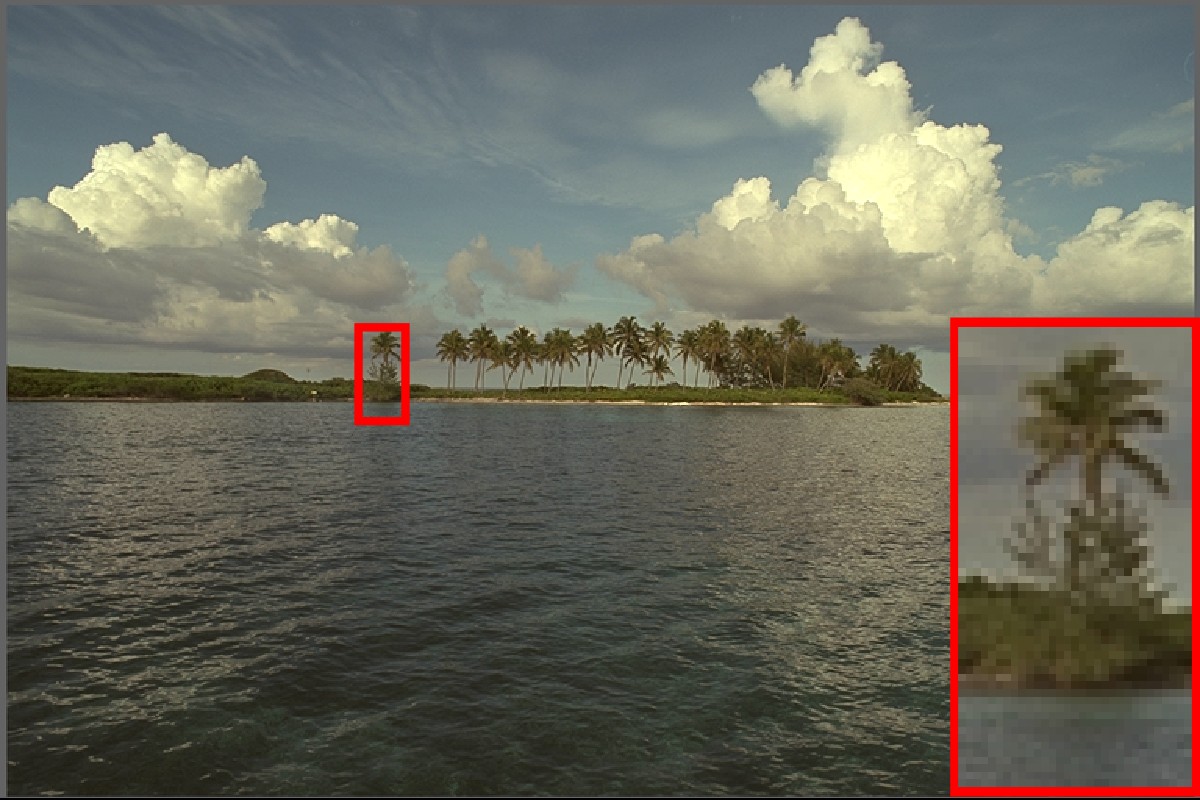} \hspace{-3.9mm} & 
			\includegraphics[width=2.2cm]{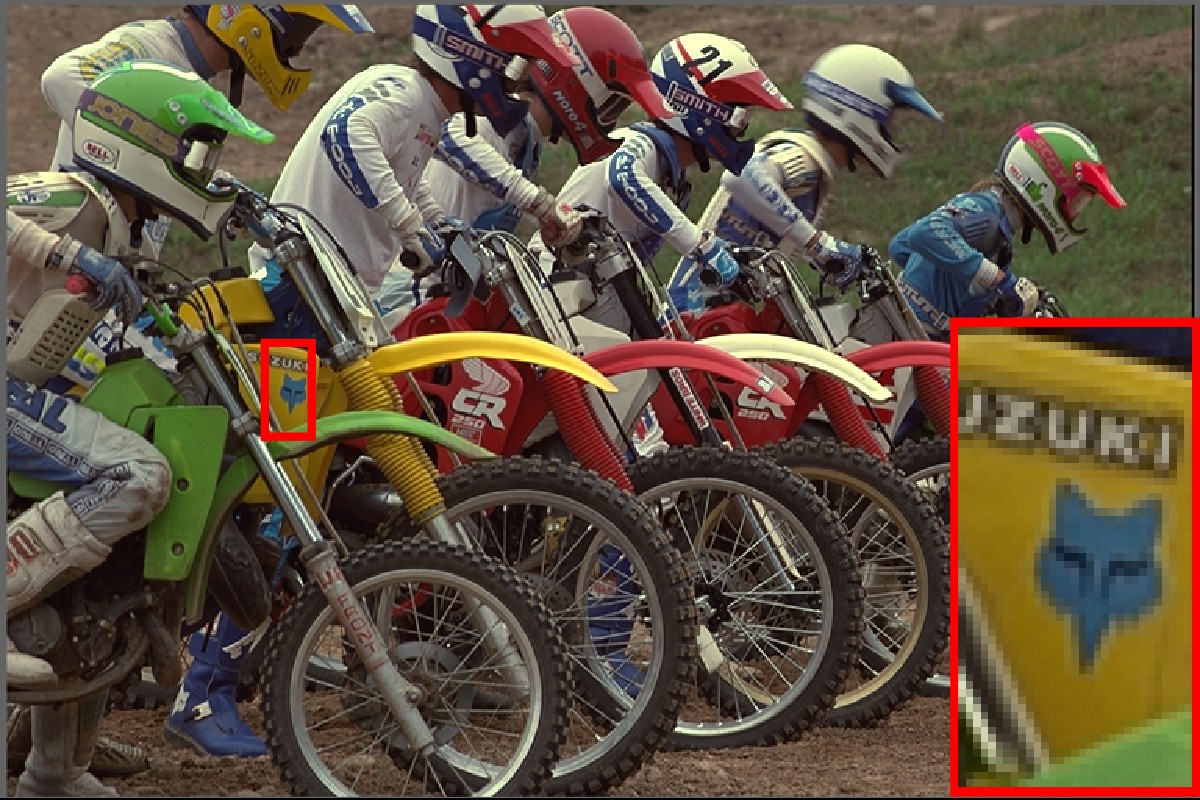} \hspace{-3.9mm} & 
			\includegraphics[width=2.2cm]{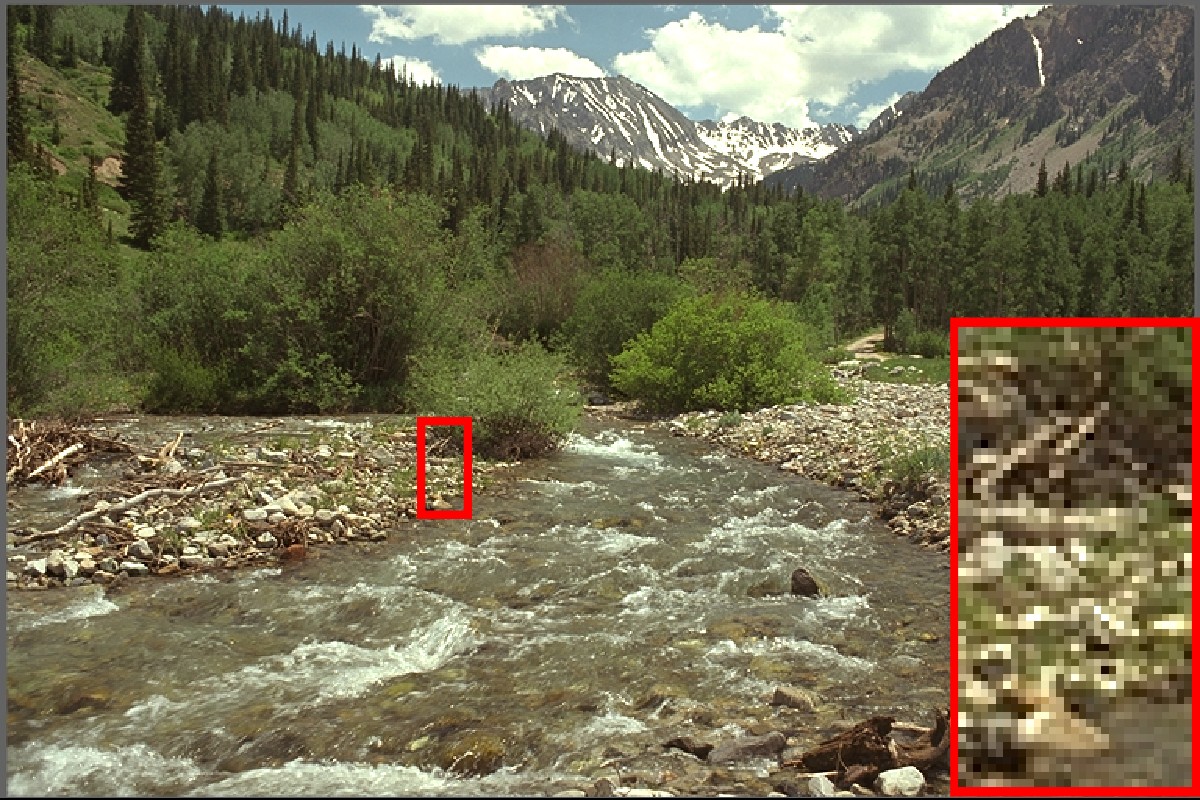} \hspace{-3.9mm} \\ 			
			\hspace{-5mm}(b)\hspace{-3mm} & \includegraphics[width=2.2cm]{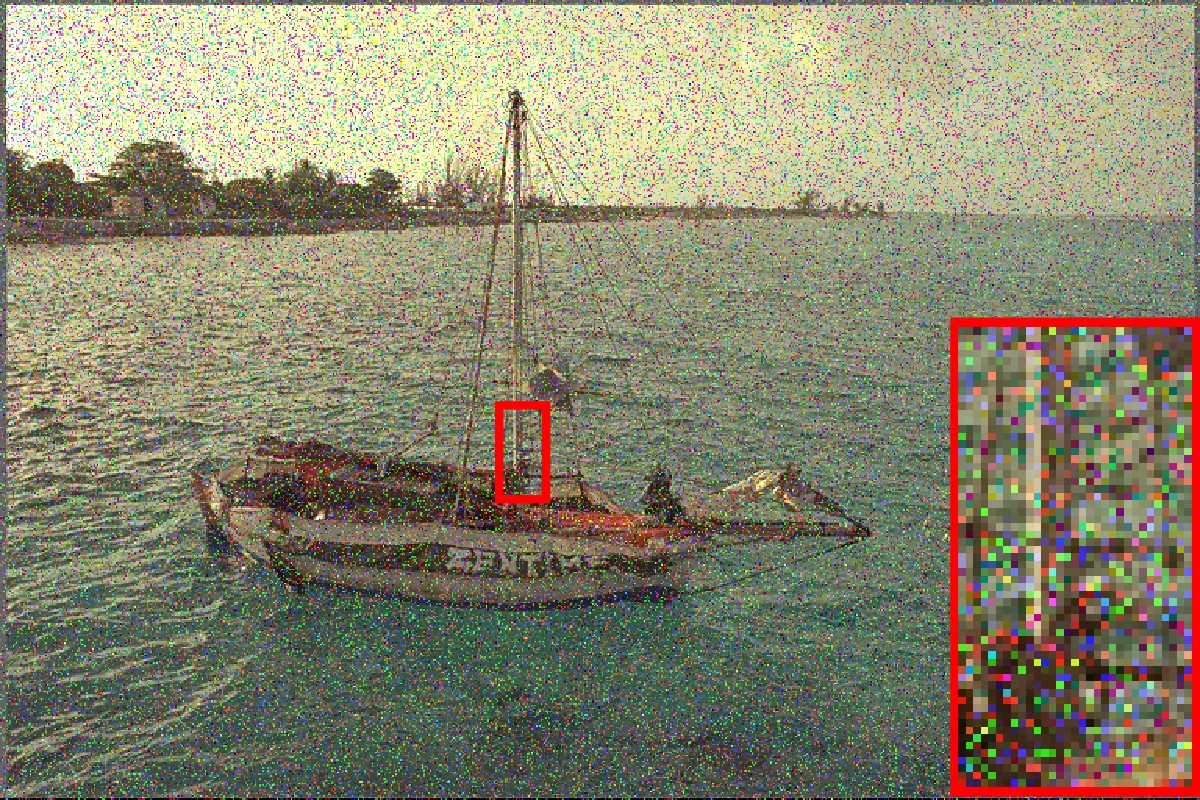} \hspace{-3.9mm} & 
			\includegraphics[width=2.2cm]{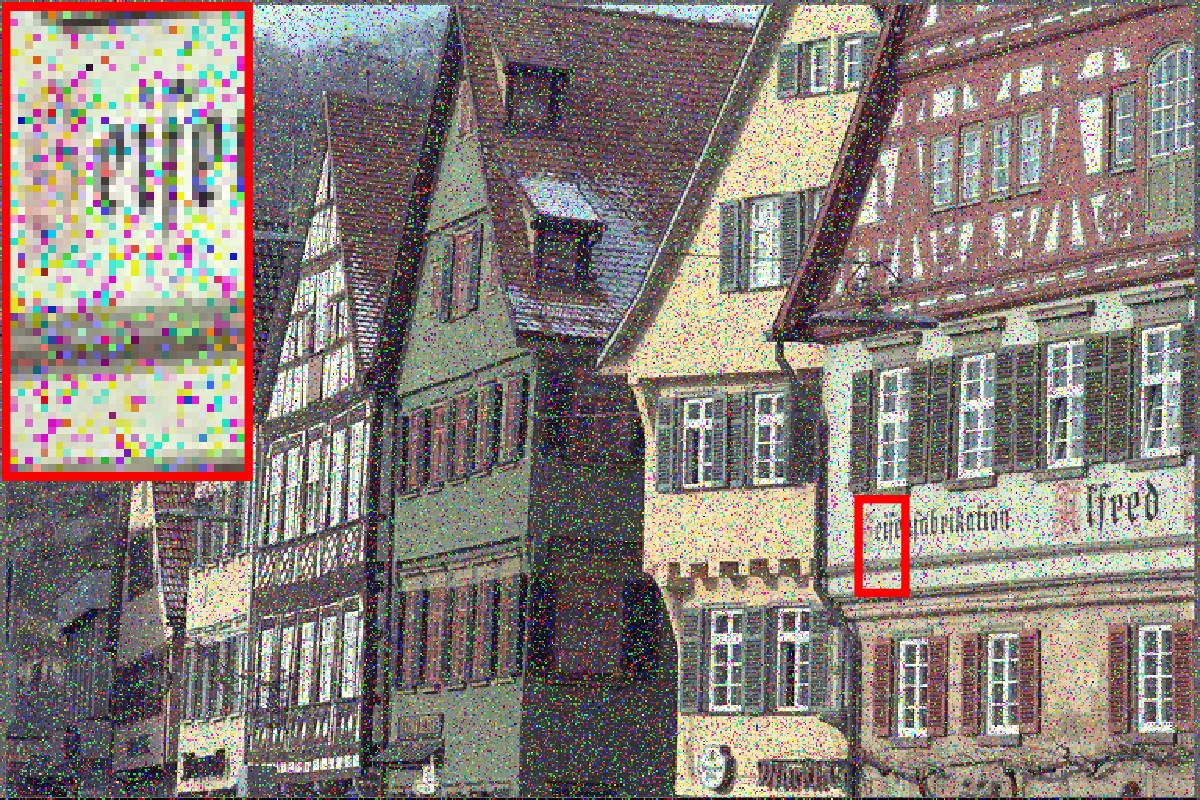} \hspace{-3.9mm} & 
			\includegraphics[width=2.2cm]{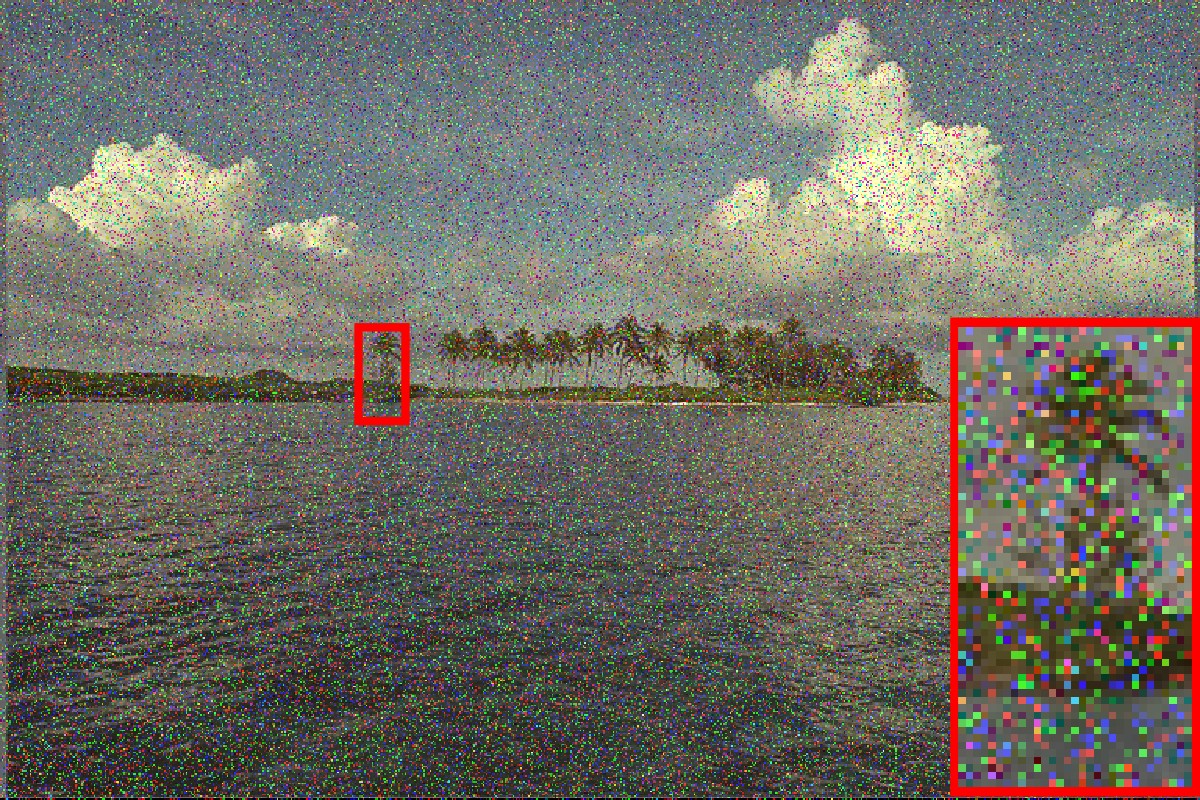} \hspace{-3.9mm} & 
			\includegraphics[width=2.2cm]{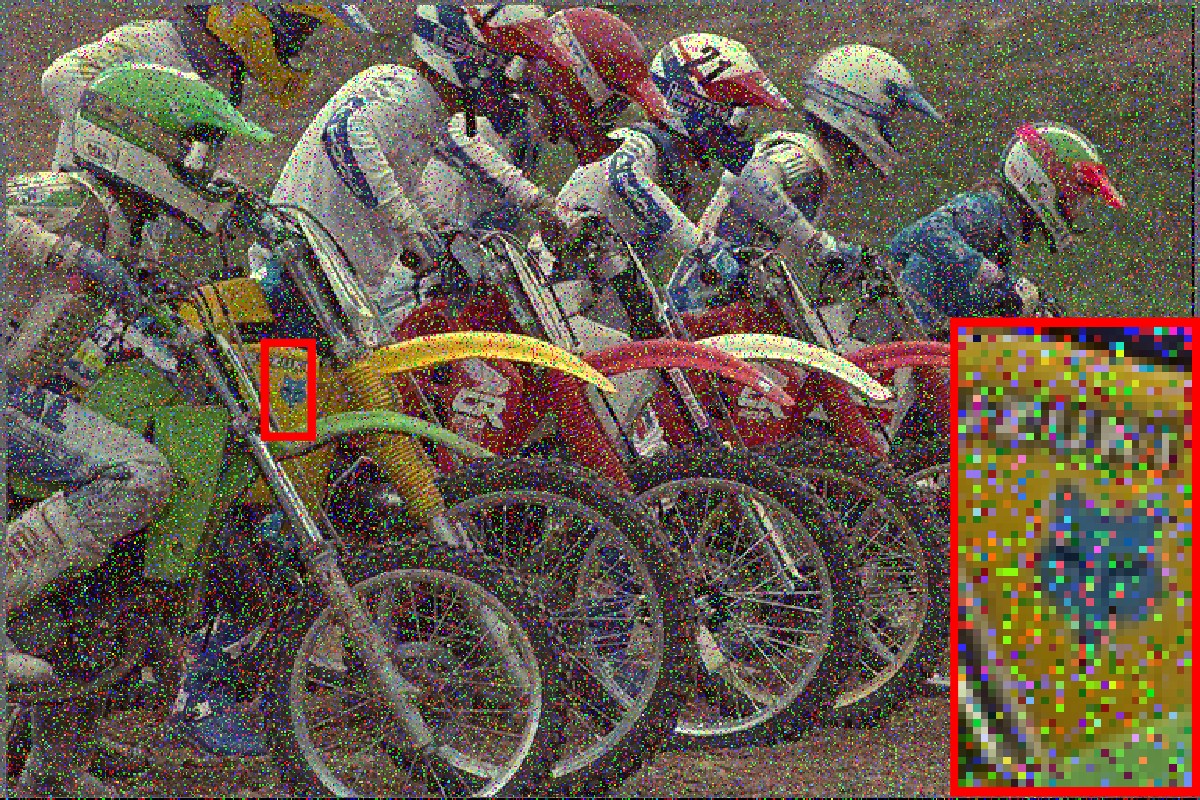} \hspace{-3.9mm} & 
			\includegraphics[width=2.2cm]{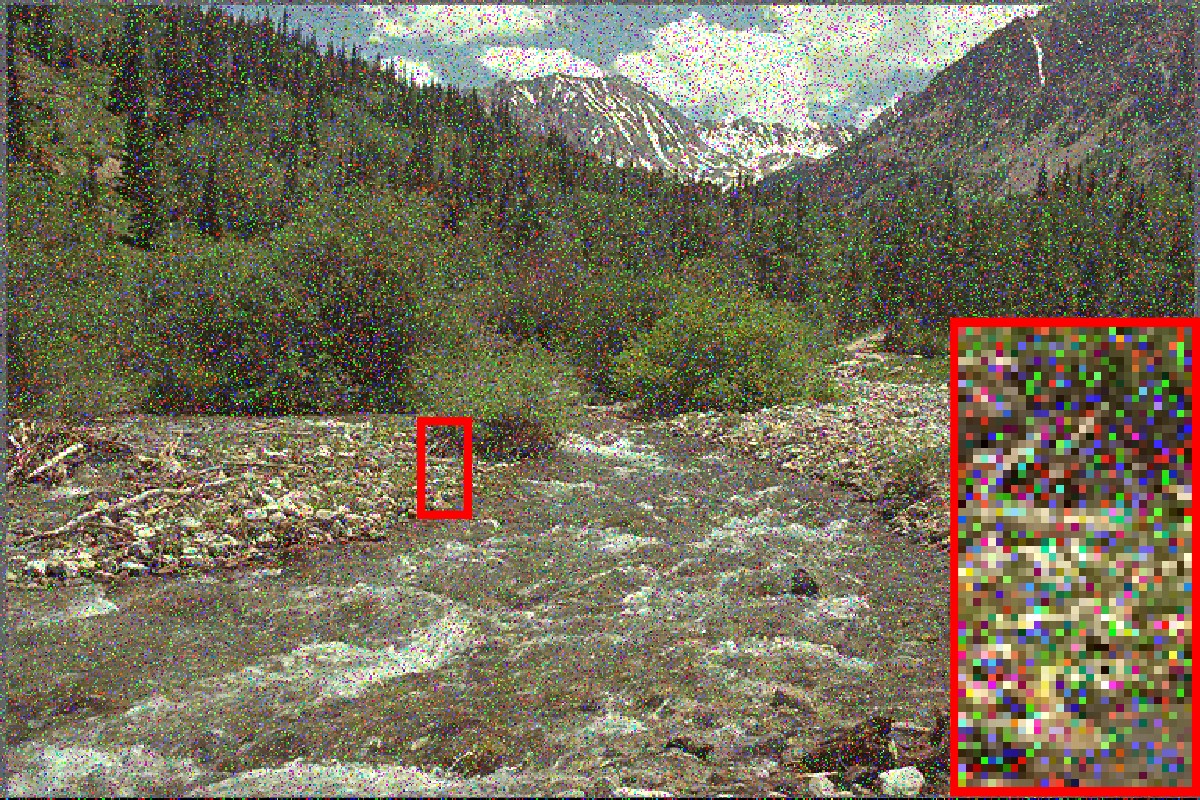} \hspace{-3.9mm} \\ 			
			\hspace{-5mm}(c)\hspace{-3mm} & \includegraphics[width=2.2cm]{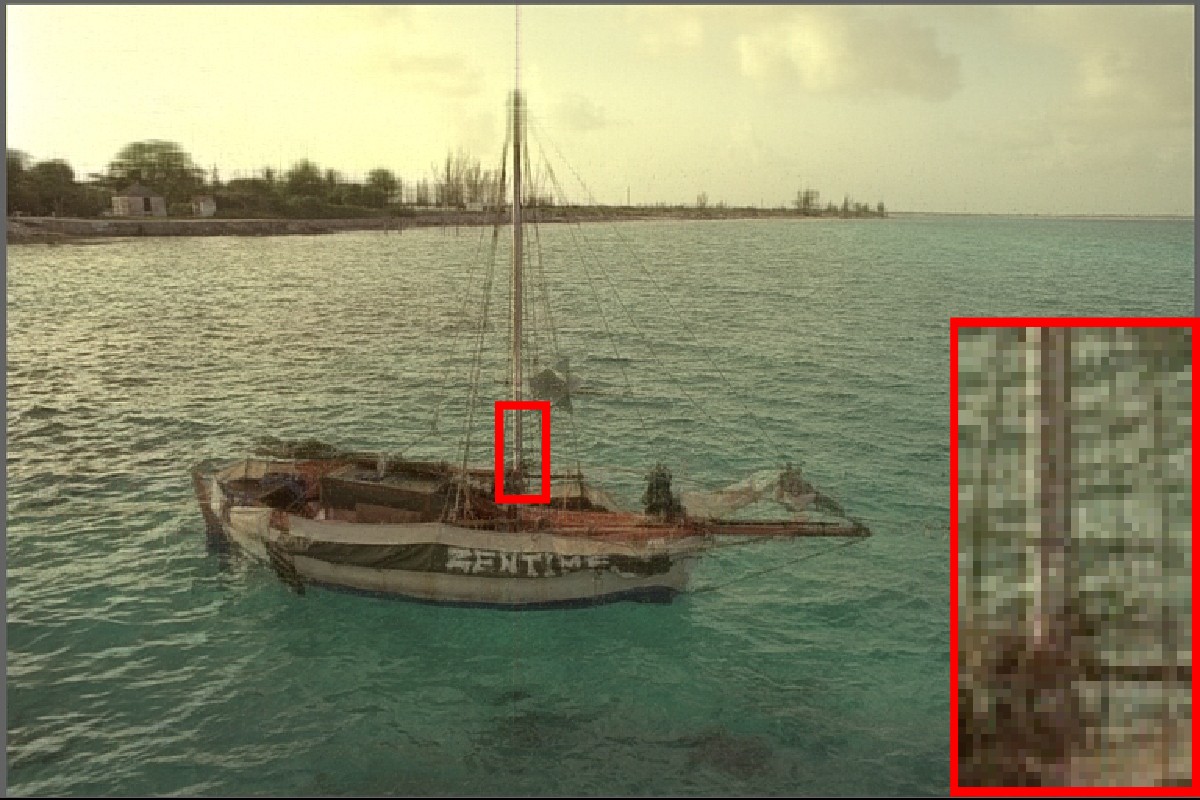} \hspace{-3.9mm} & 
			\includegraphics[width=2.2cm]{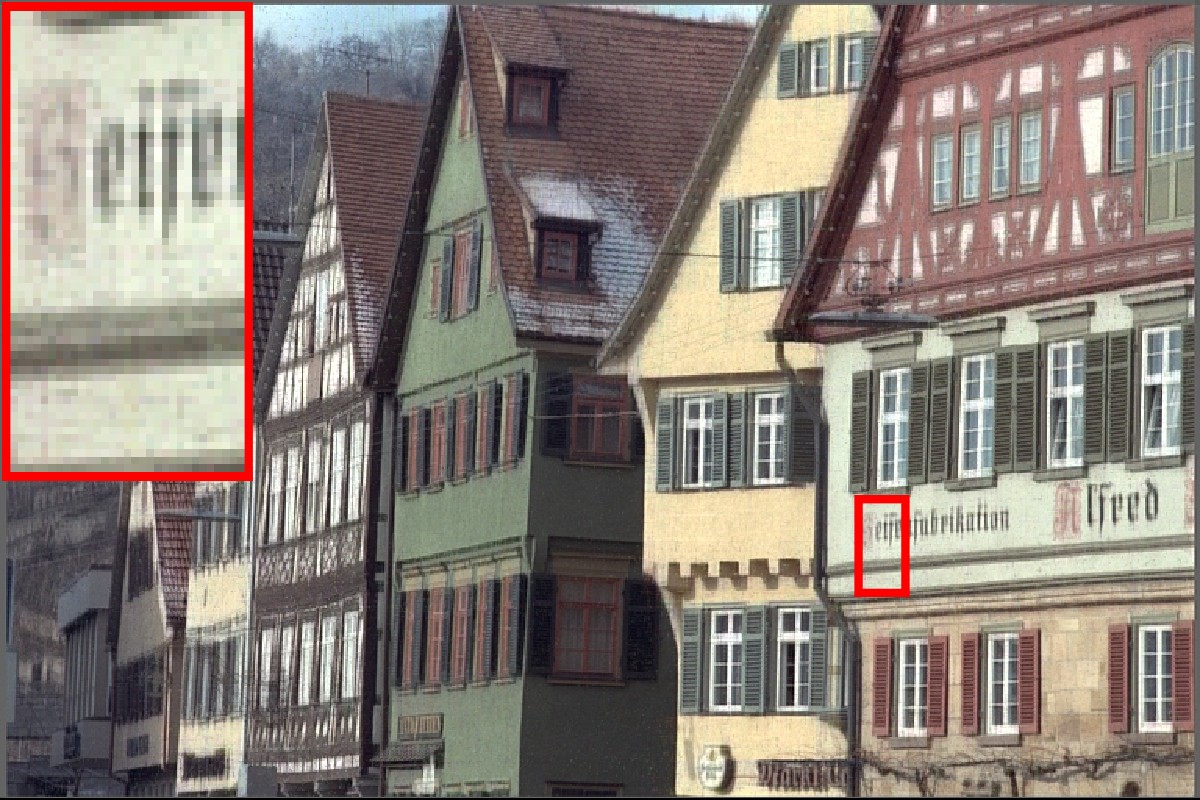} \hspace{-3.9mm} & 
			\includegraphics[width=2.2cm]{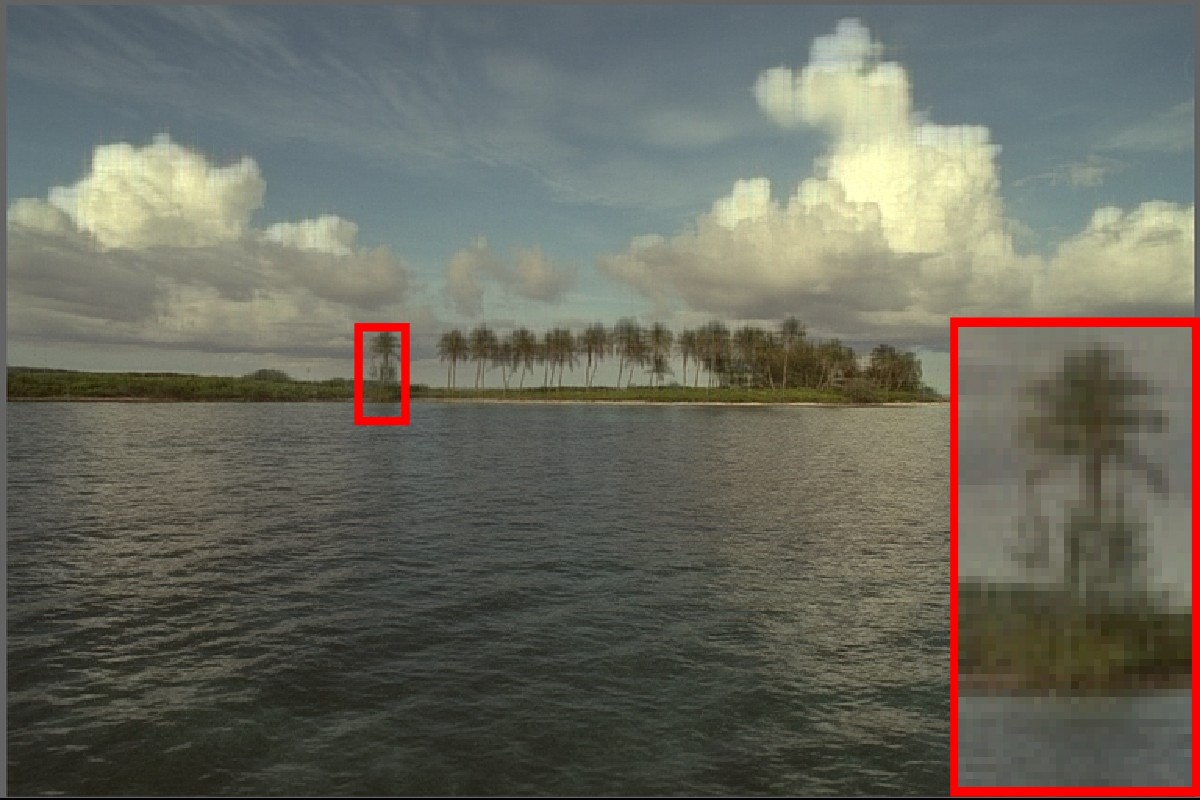} \hspace{-3.9mm} & 
			\includegraphics[width=2.2cm]{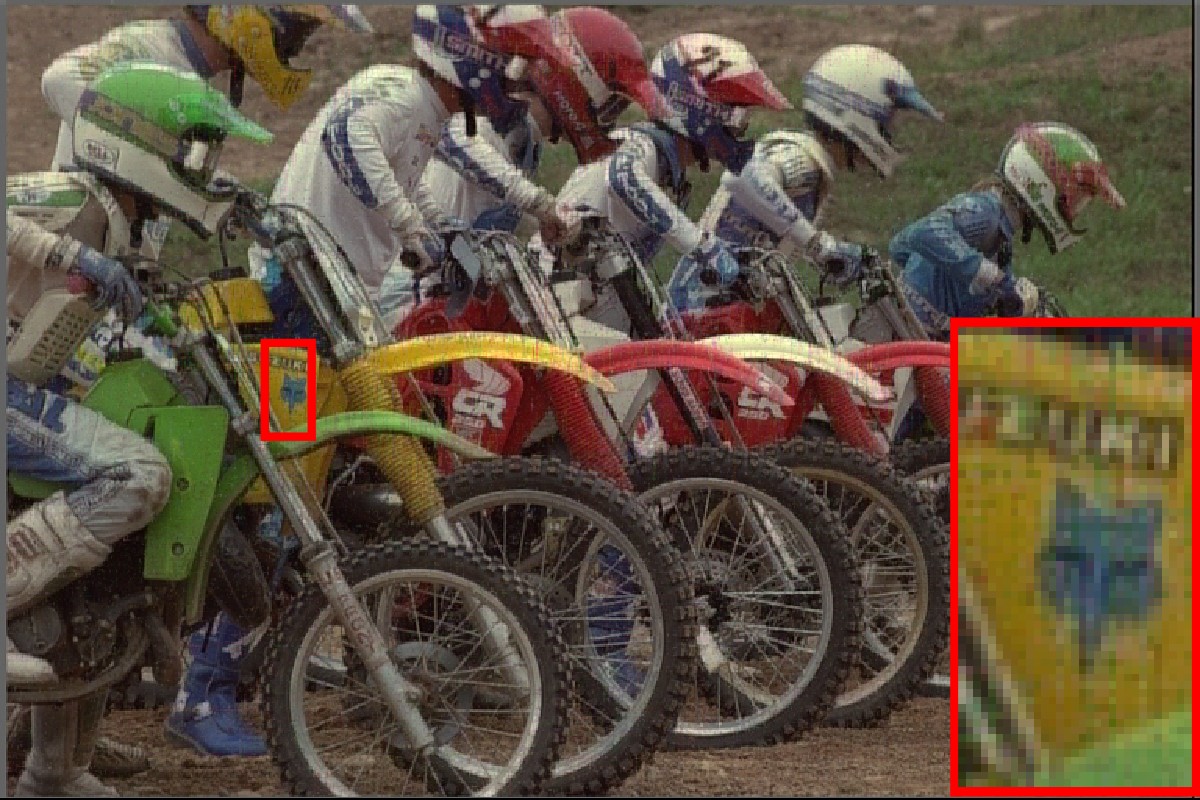} \hspace{-3.9mm} & 
			\includegraphics[width=2.2cm]{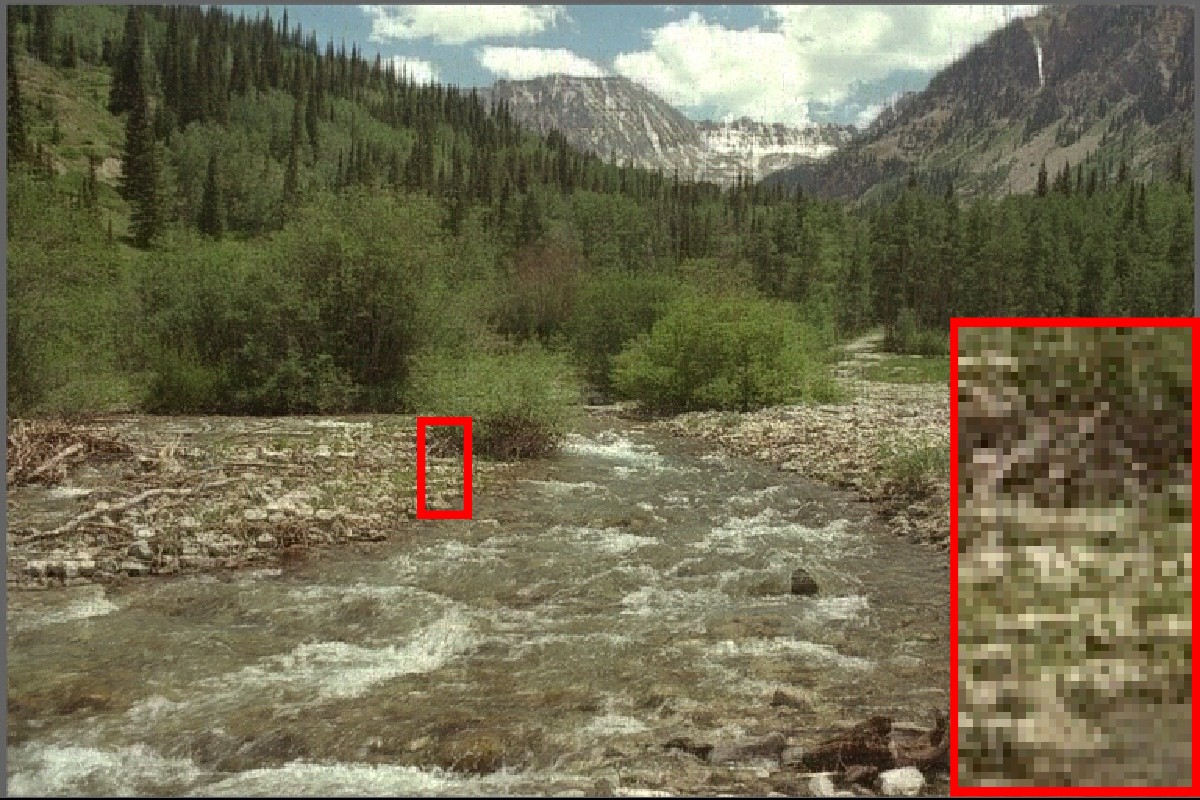} \hspace{-3.9mm} \\ 			
			\hspace{-5mm}(d)\hspace{-3mm} & \includegraphics[width=2.2cm]{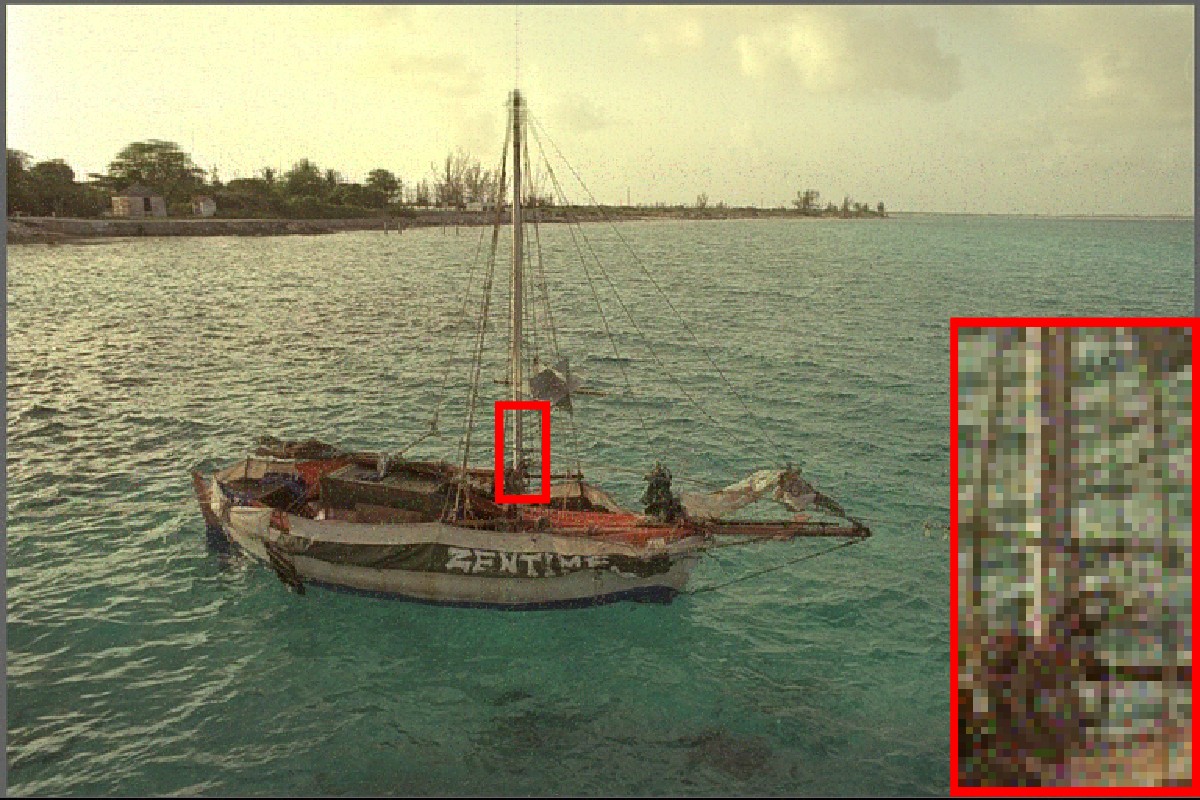} \hspace{-3.9mm} & 
			\includegraphics[width=2.2cm]{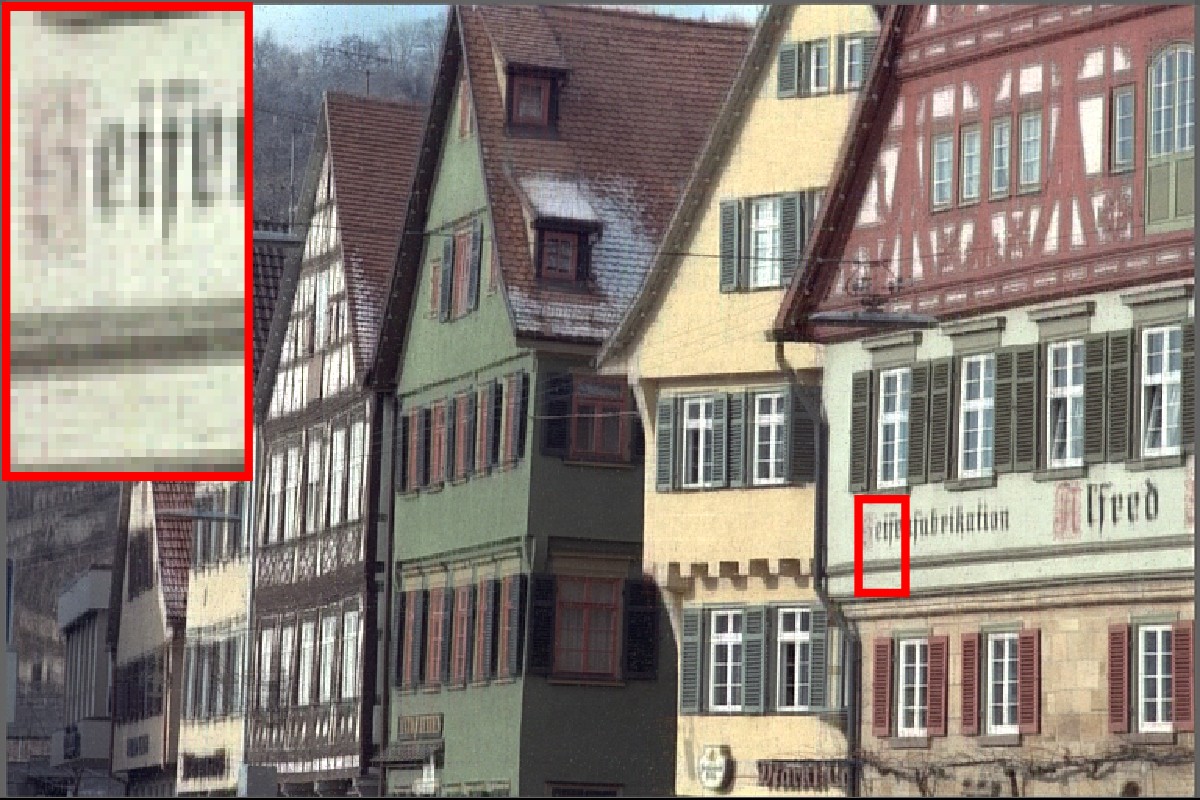} \hspace{-3.9mm} & 
			\includegraphics[width=2.2cm]{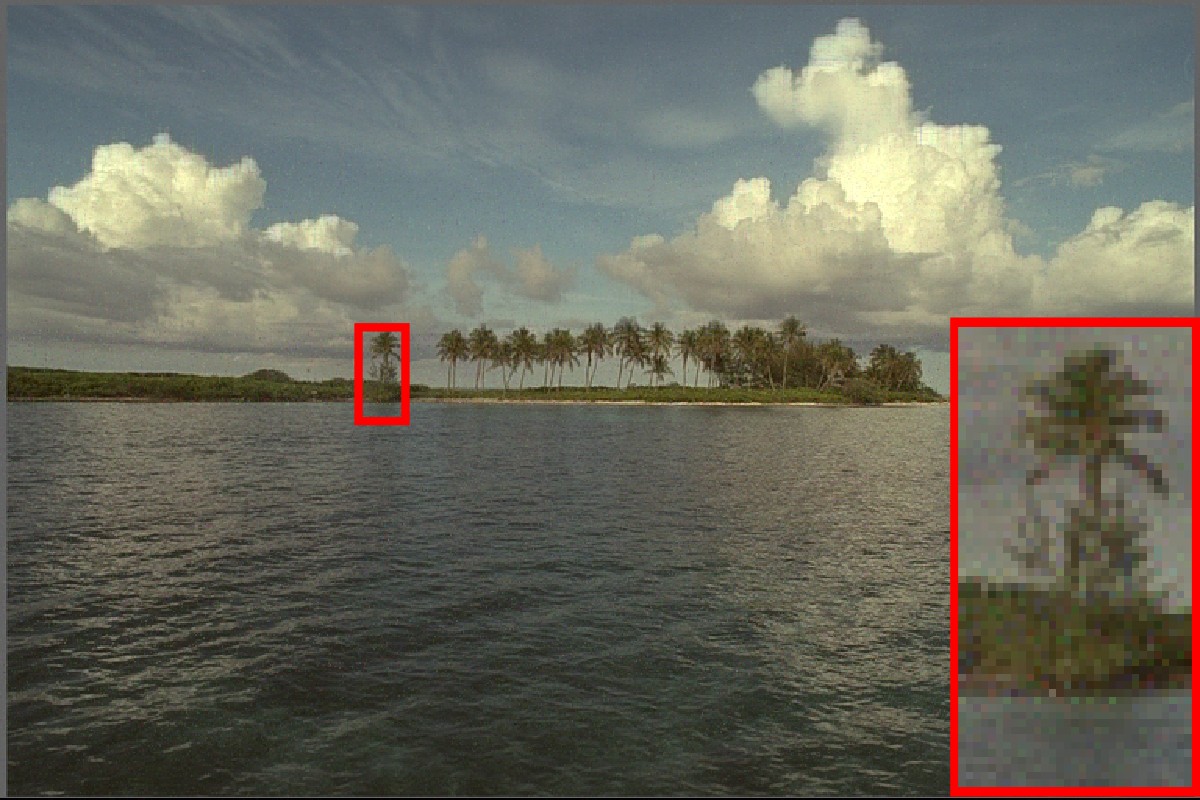} \hspace{-3.9mm} & 
			\includegraphics[width=2.2cm]{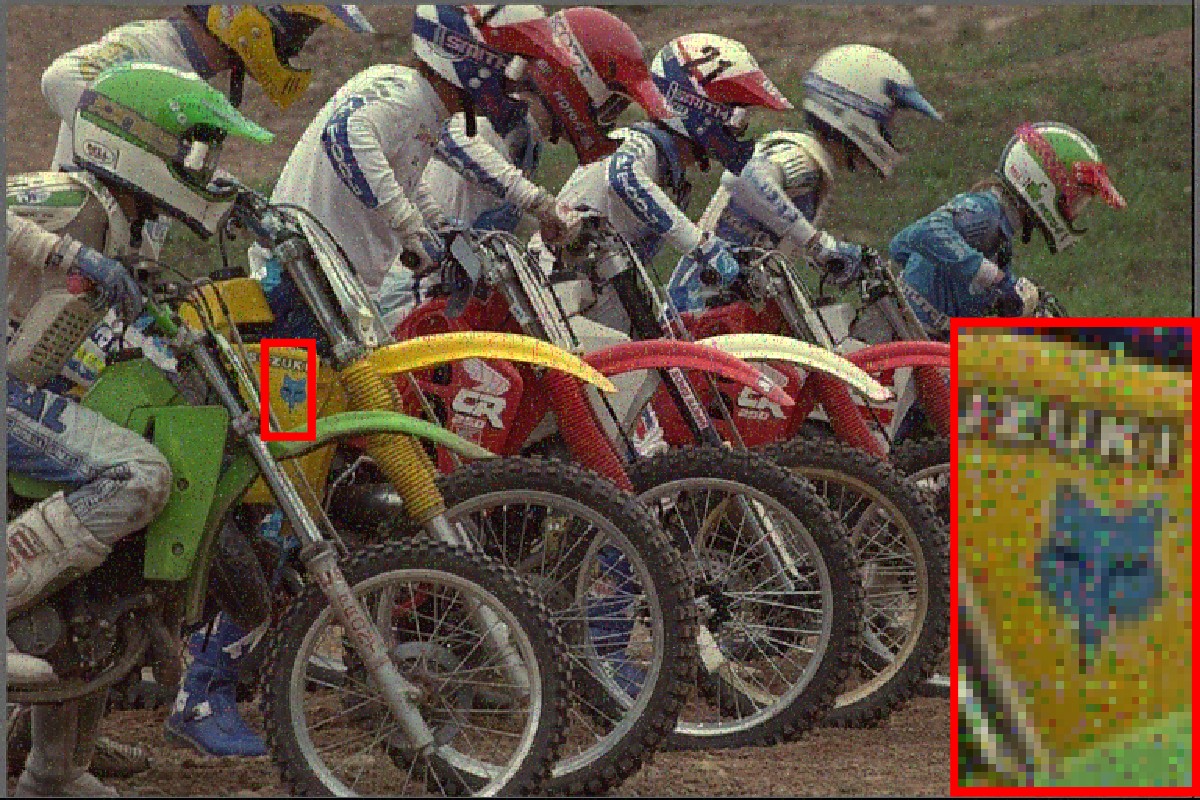} \hspace{-3.9mm} & 
			\includegraphics[width=2.2cm]{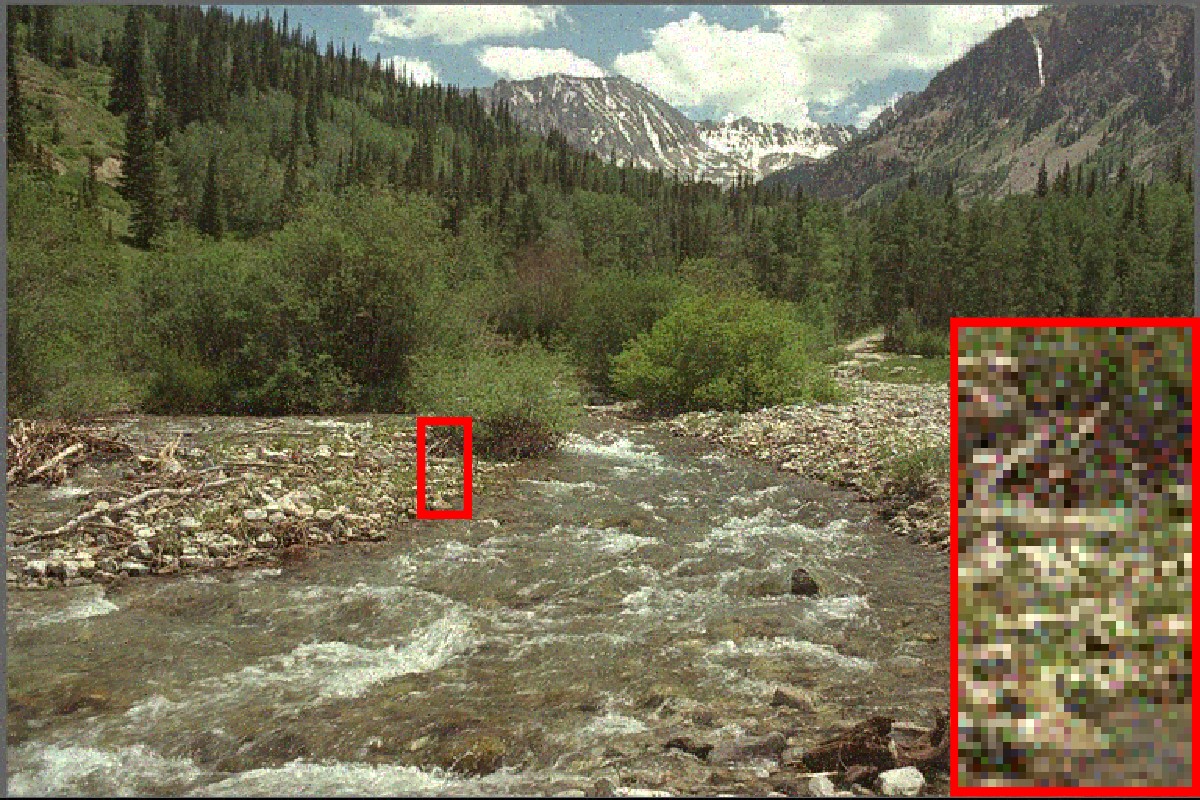} \hspace{-3.9mm} \\ 			
			\hspace{-5mm}(e)\hspace{-3mm} & \includegraphics[width=2.2cm]{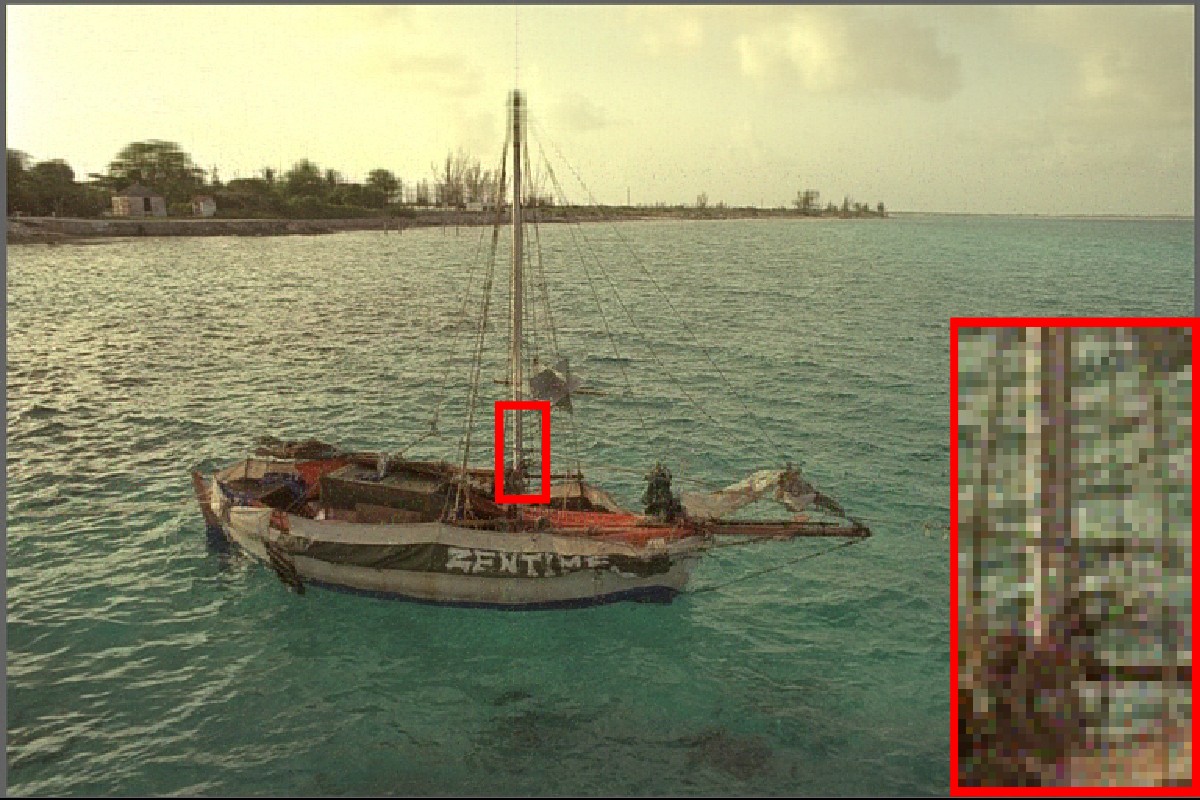} \hspace{-3.9mm} & 
			\includegraphics[width=2.2cm]{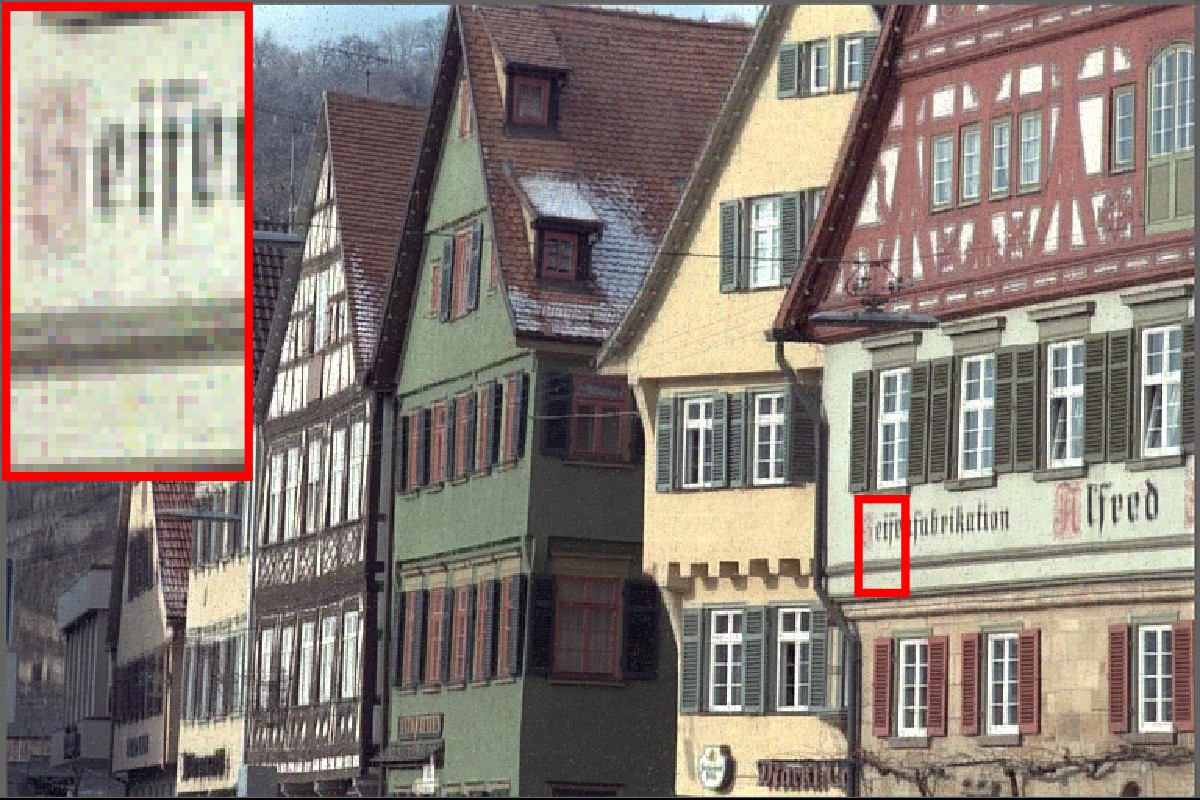} \hspace{-3.9mm} & 
			\includegraphics[width=2.2cm]{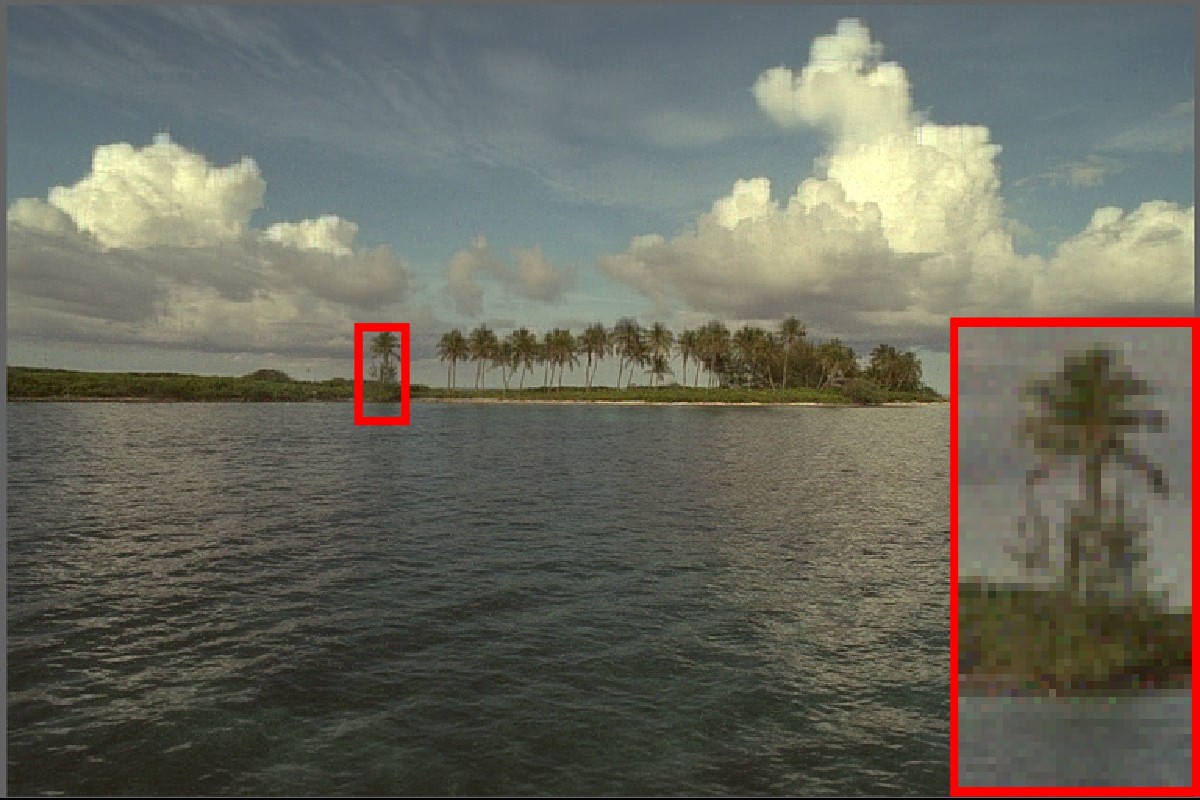} \hspace{-3.9mm} & 
			\includegraphics[width=2.2cm]{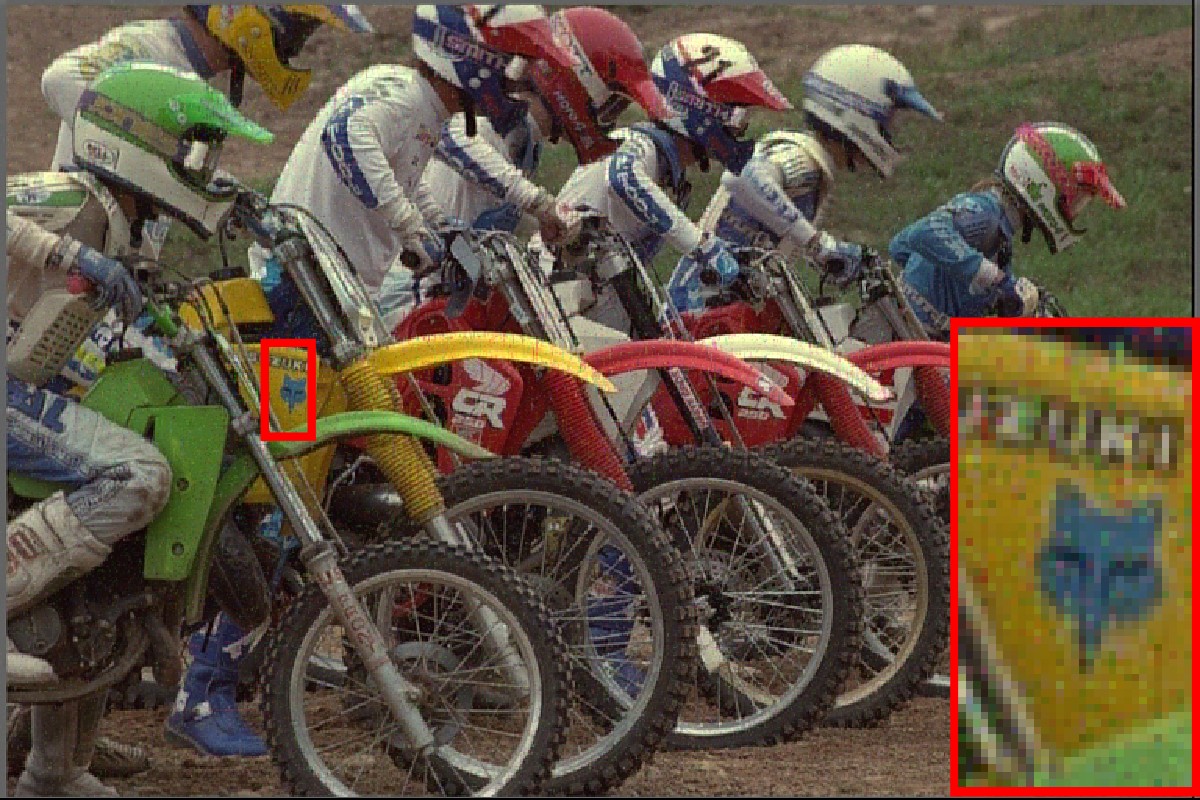} \hspace{-3.9mm} & 
			\includegraphics[width=2.2cm]{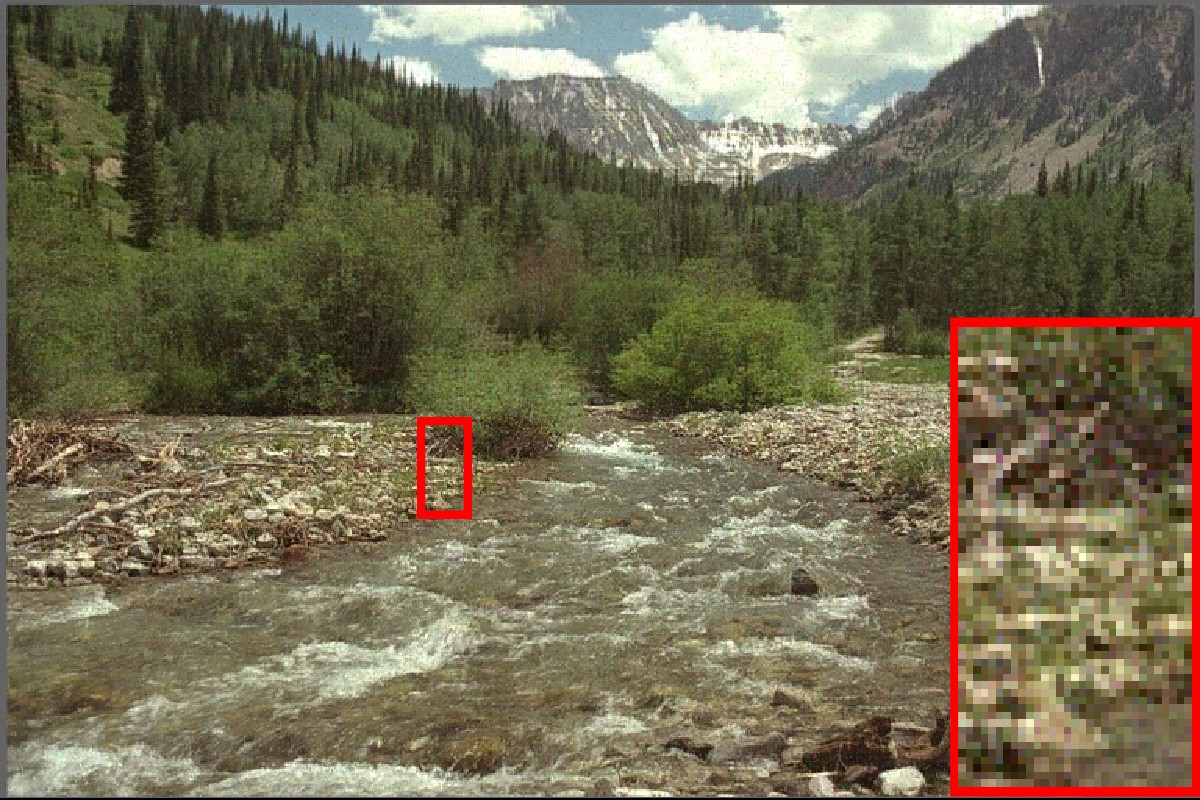} \hspace{-3.9mm} \\ 			
			\hspace{-5mm}(f)\hspace{-3mm} & \includegraphics[width=2.2cm]{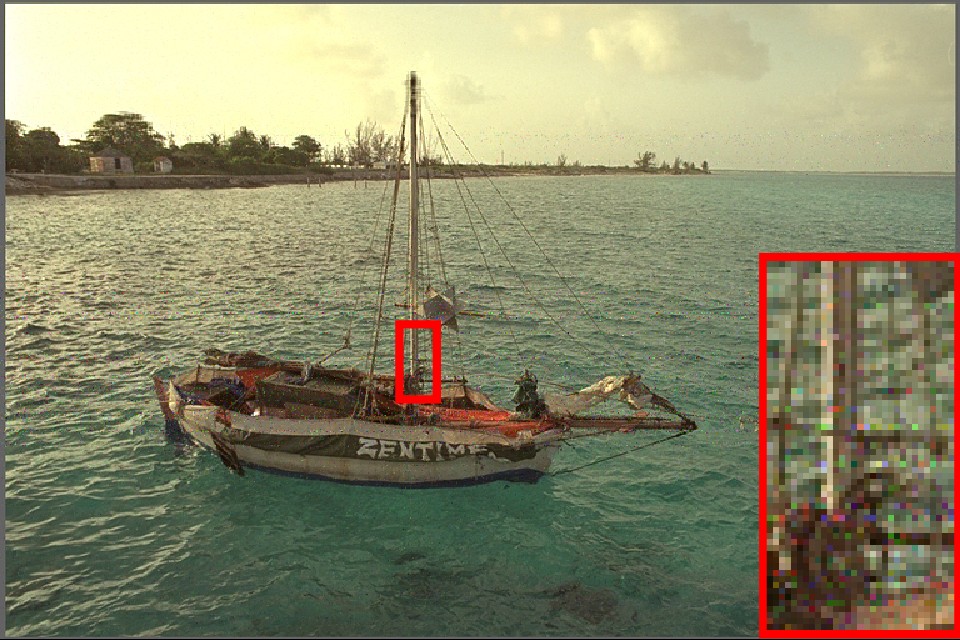} \hspace{-3.9mm} & 
			\includegraphics[width=2.2cm]{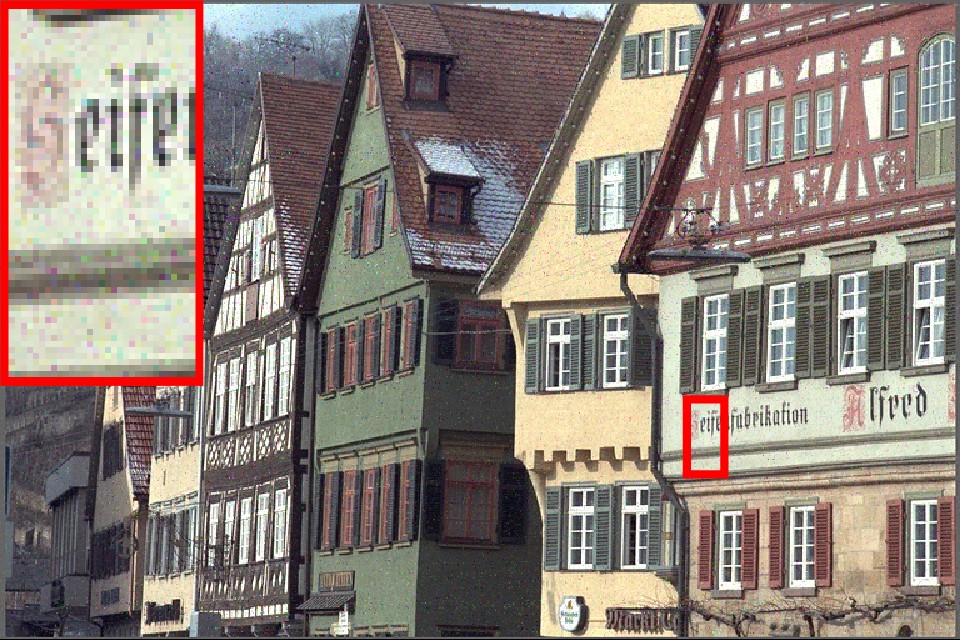} \hspace{-3.9mm} & 
			\includegraphics[width=2.2cm]{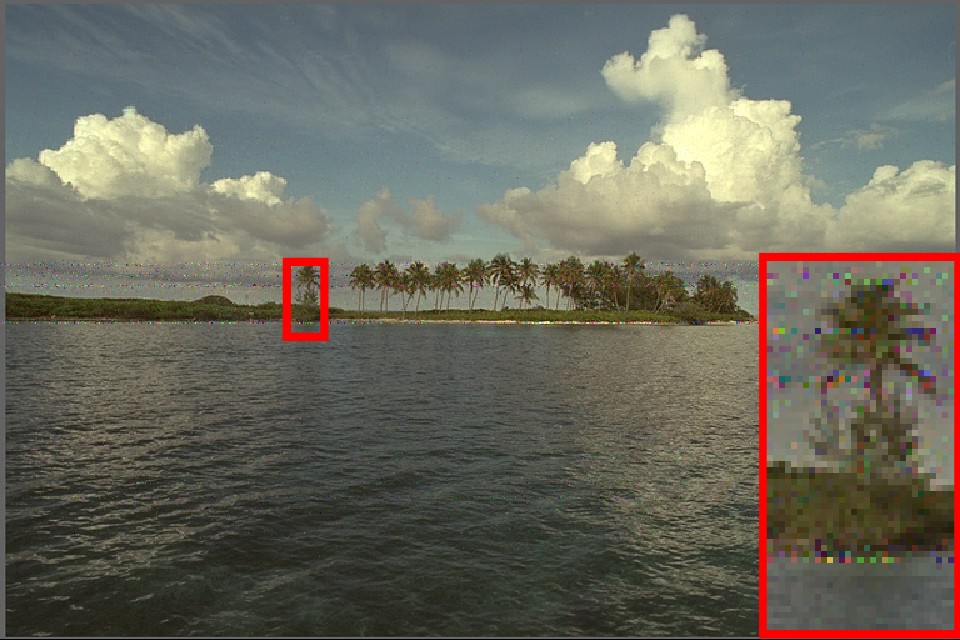} \hspace{-3.9mm} & 
			\includegraphics[width=2.2cm]{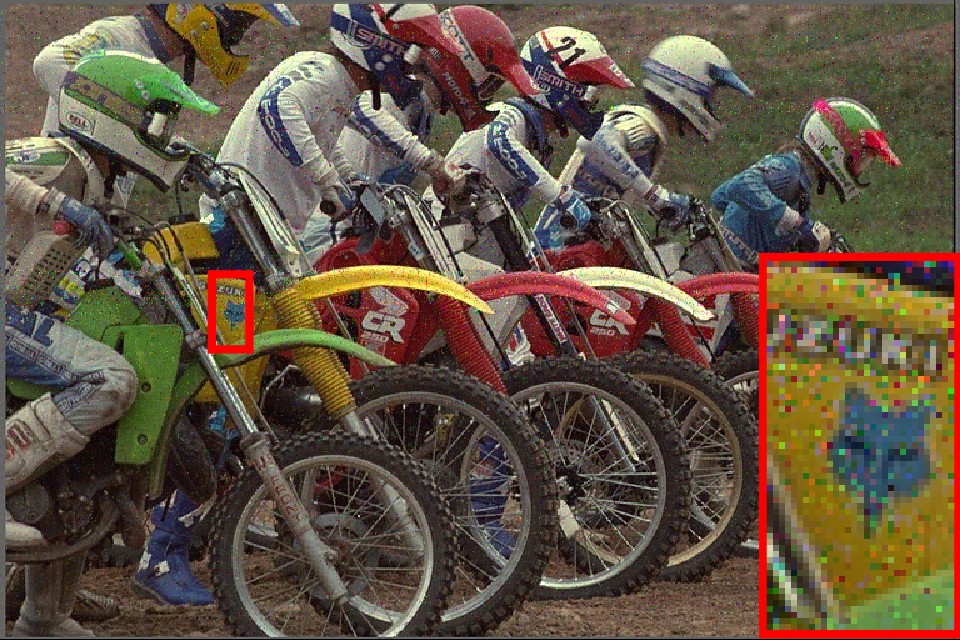} \hspace{-3.9mm} & 
			\includegraphics[width=2.2cm]{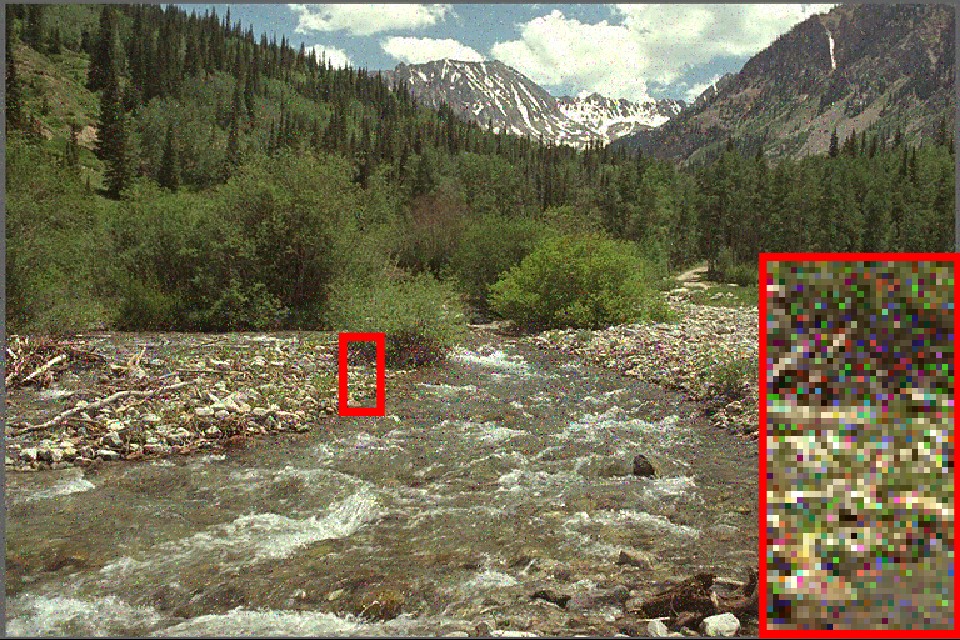} \hspace{-3.9mm} \\ 			
			\hspace{-5mm}(g)\hspace{-3mm} & \includegraphics[width=2.2cm]{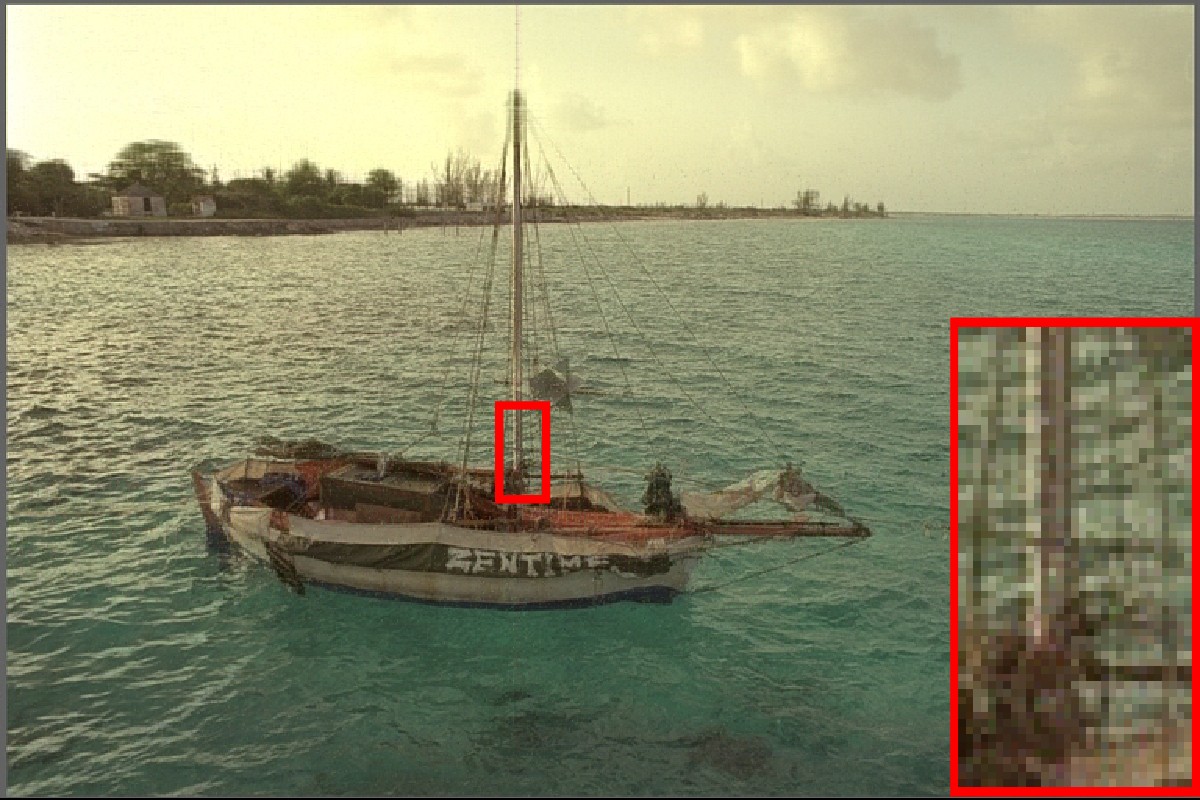} \hspace{-3.9mm} & 
			\includegraphics[width=2.2cm]{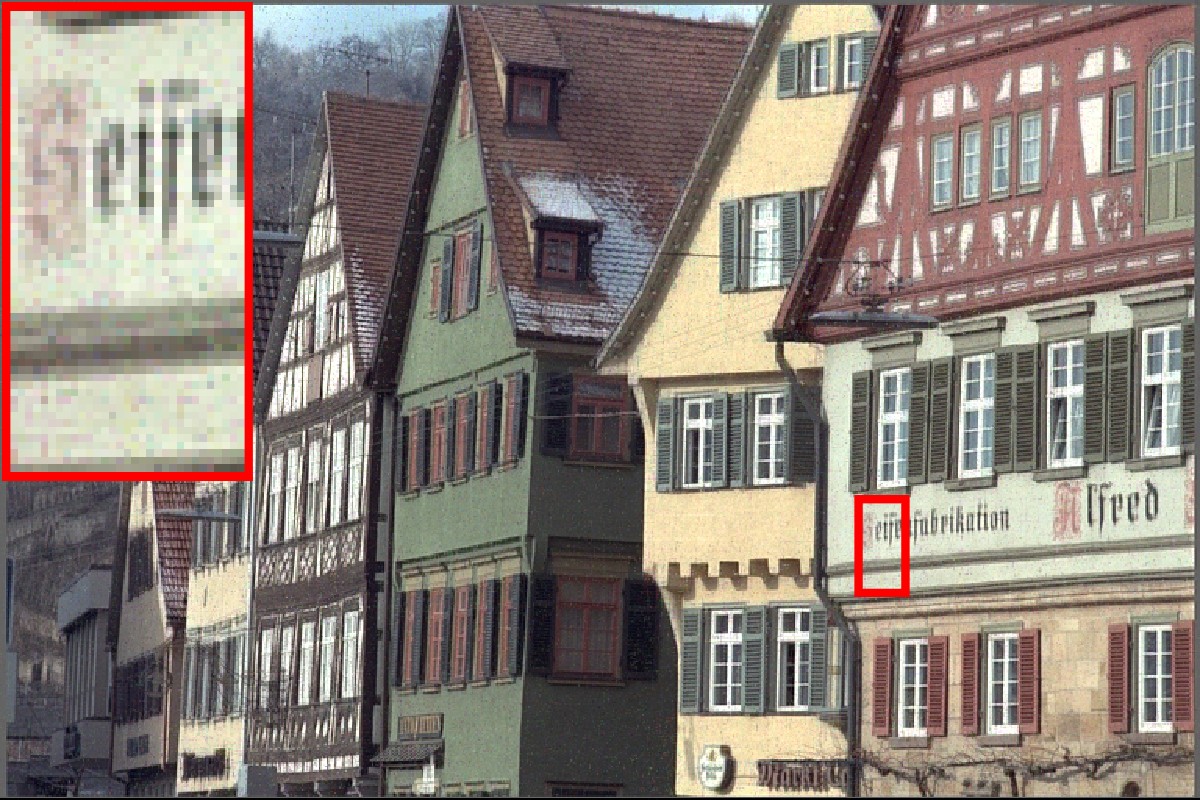} \hspace{-3.9mm} & 
			\includegraphics[width=2.2cm]{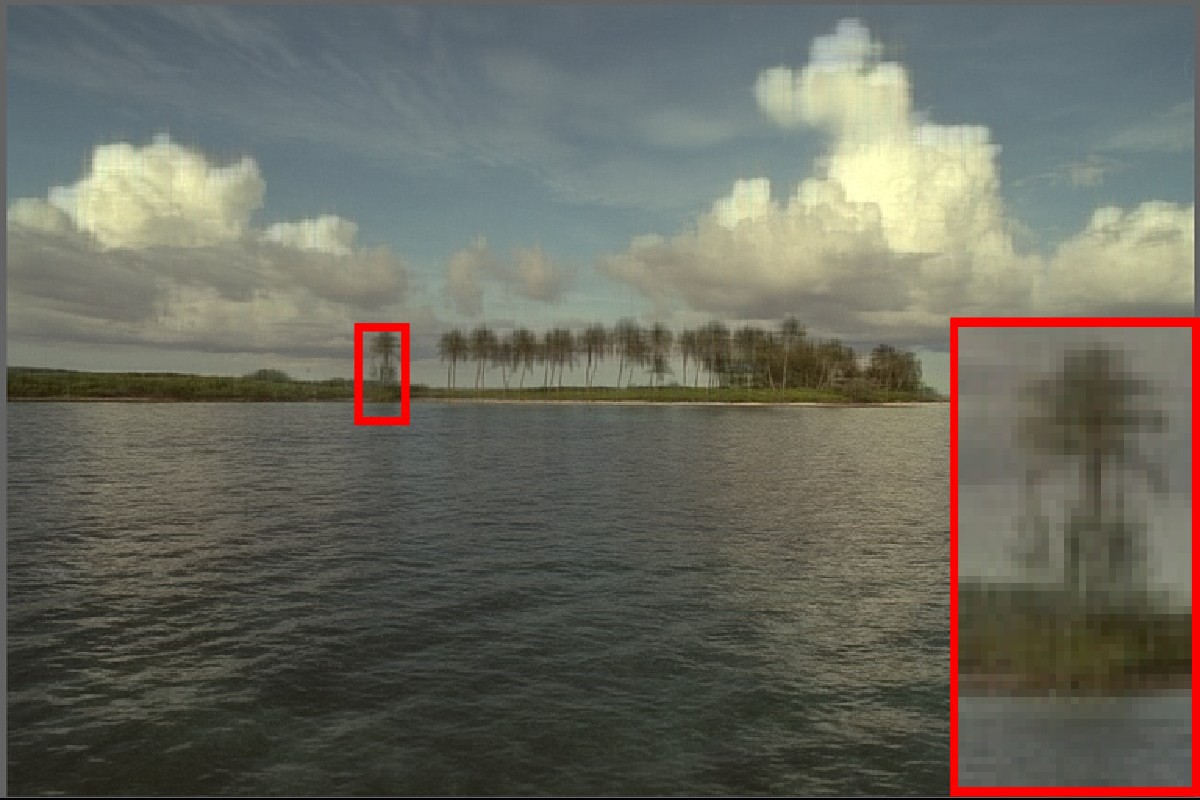} \hspace{-3.9mm} & 
			\includegraphics[width=2.2cm]{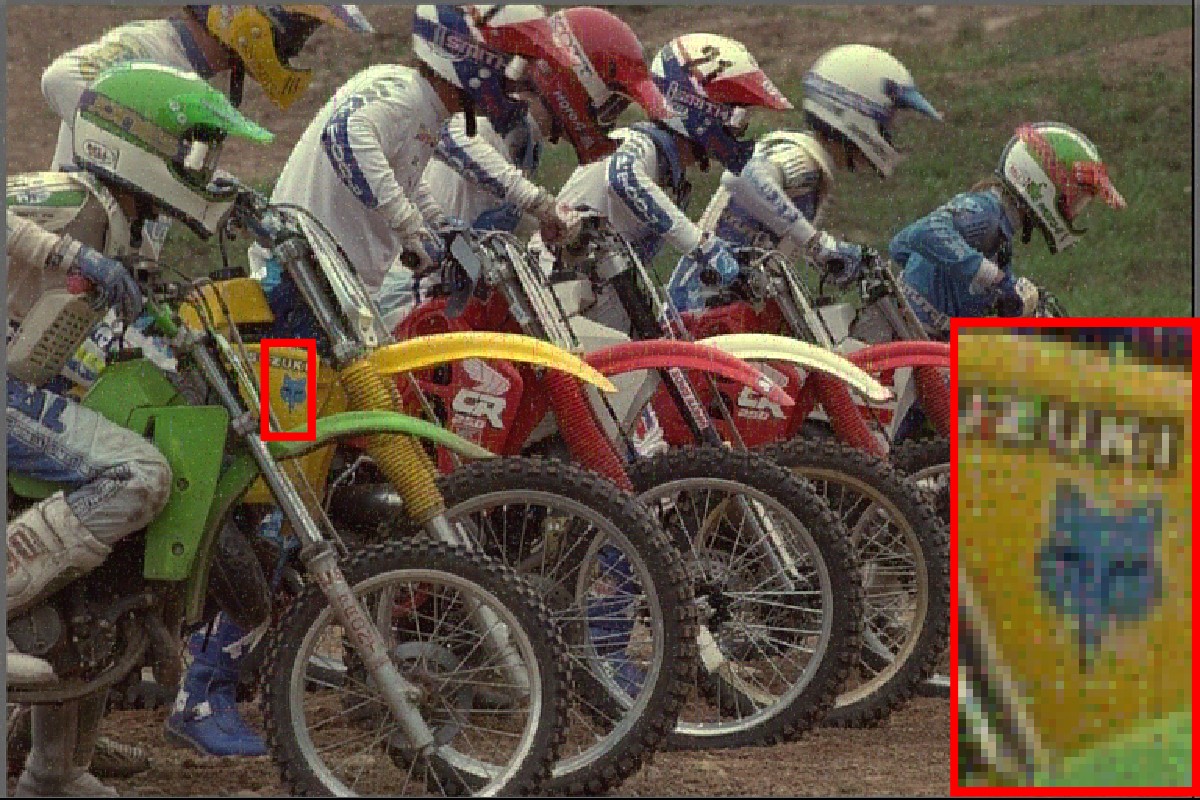} \hspace{-3.9mm} & 
			\includegraphics[width=2.2cm]{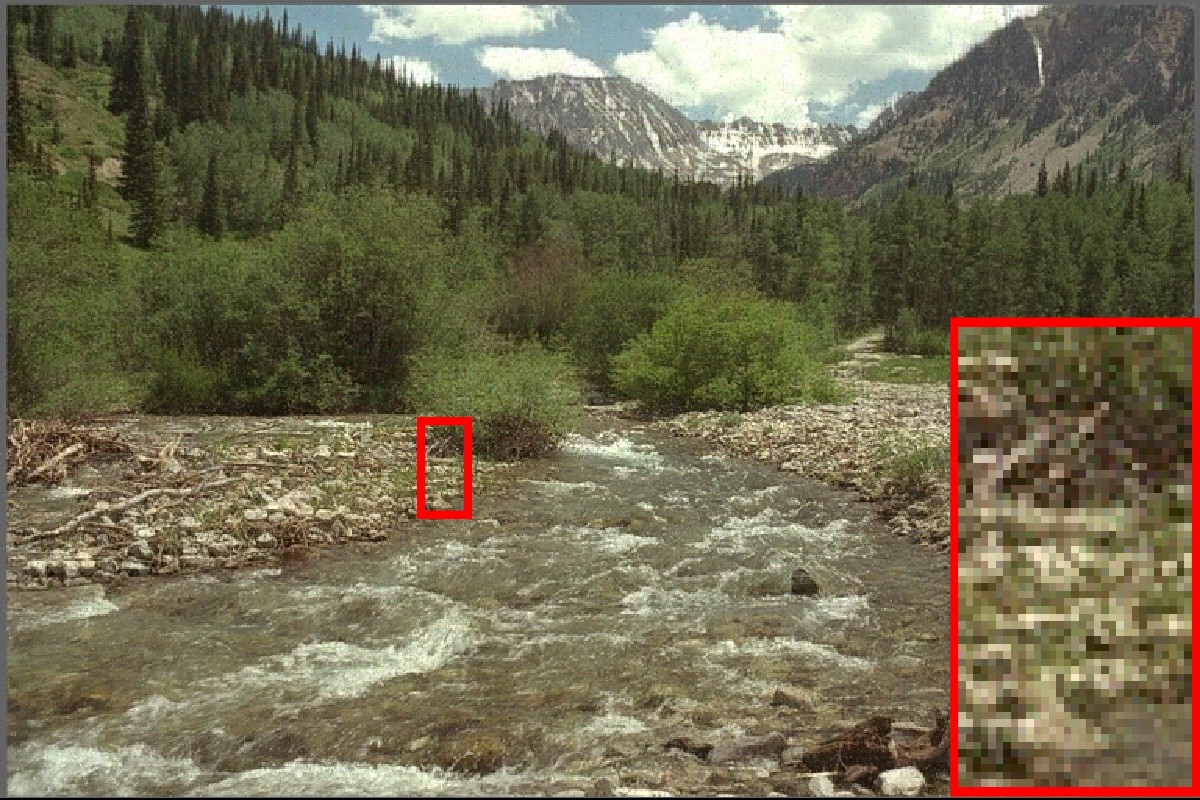} \hspace{-3.9mm} \\ 			
			\hspace{-5mm}(h)\hspace{-3mm} & \includegraphics[width=2.2cm]{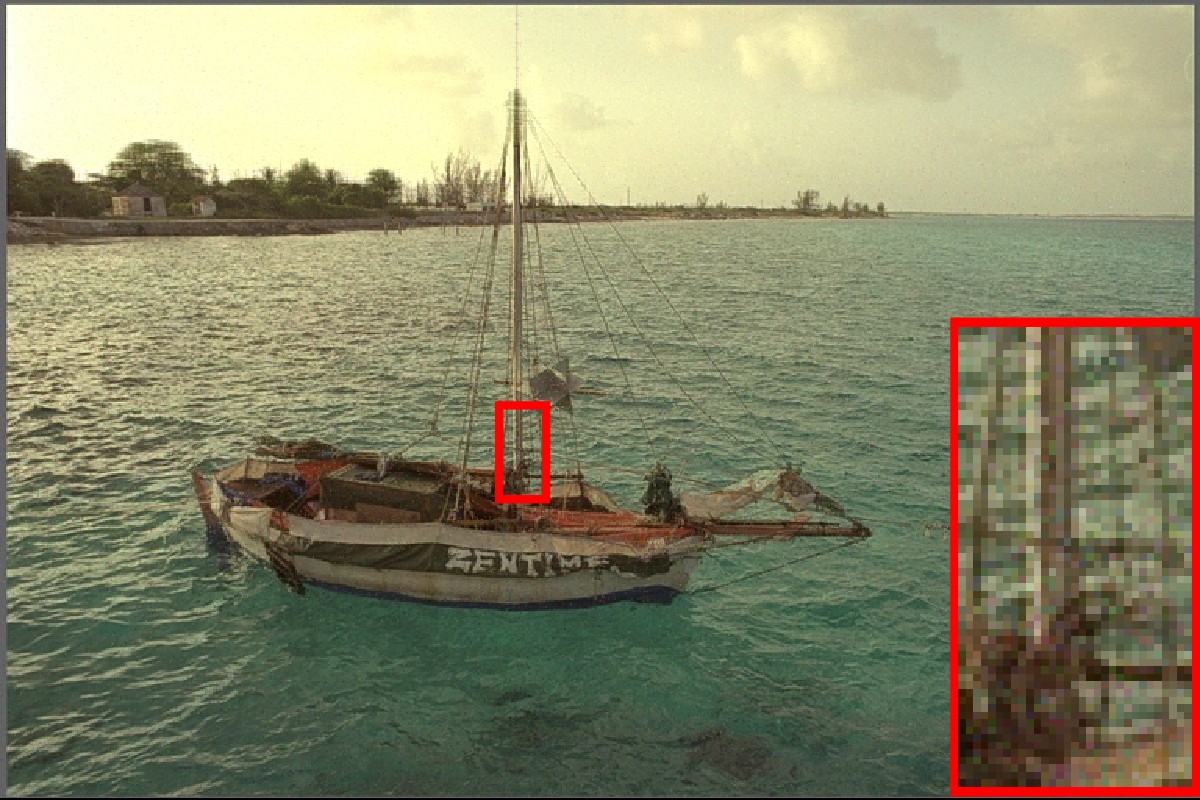} \hspace{-3.9mm} & 
			\includegraphics[width=2.2cm]{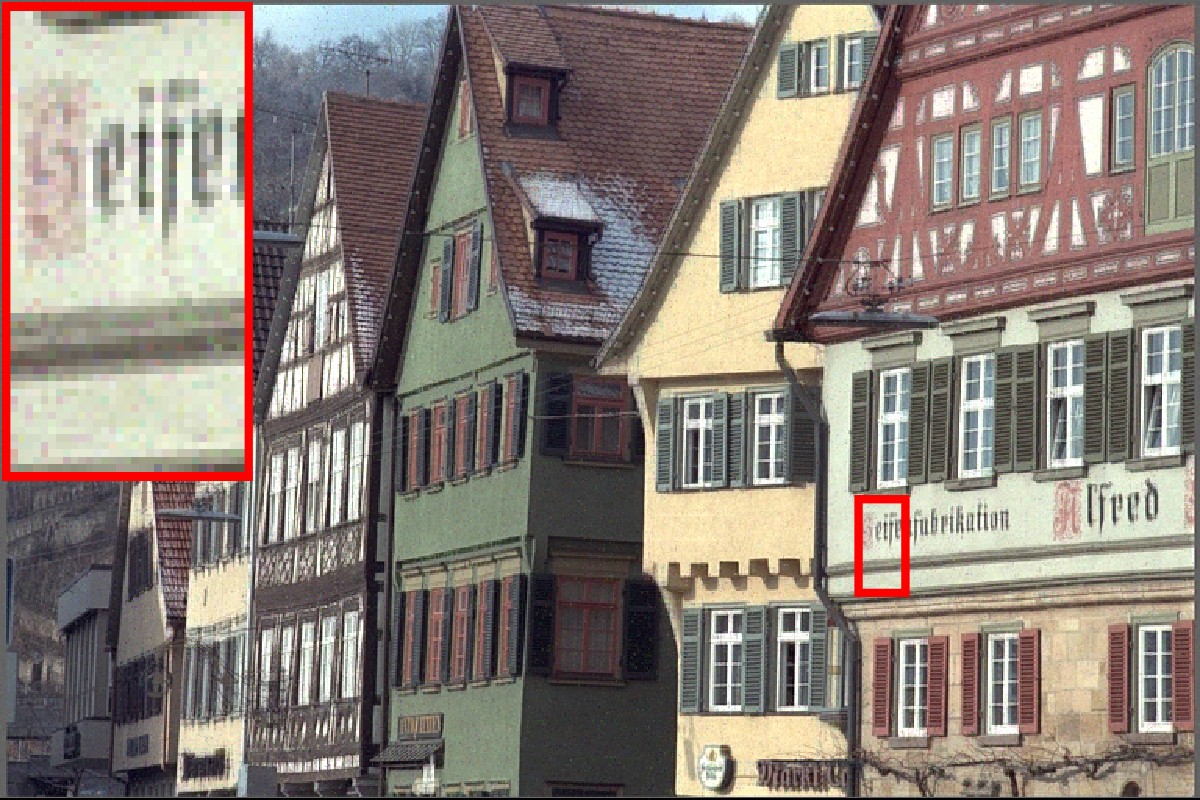} \hspace{-3.9mm} & 
			\includegraphics[width=2.2cm]{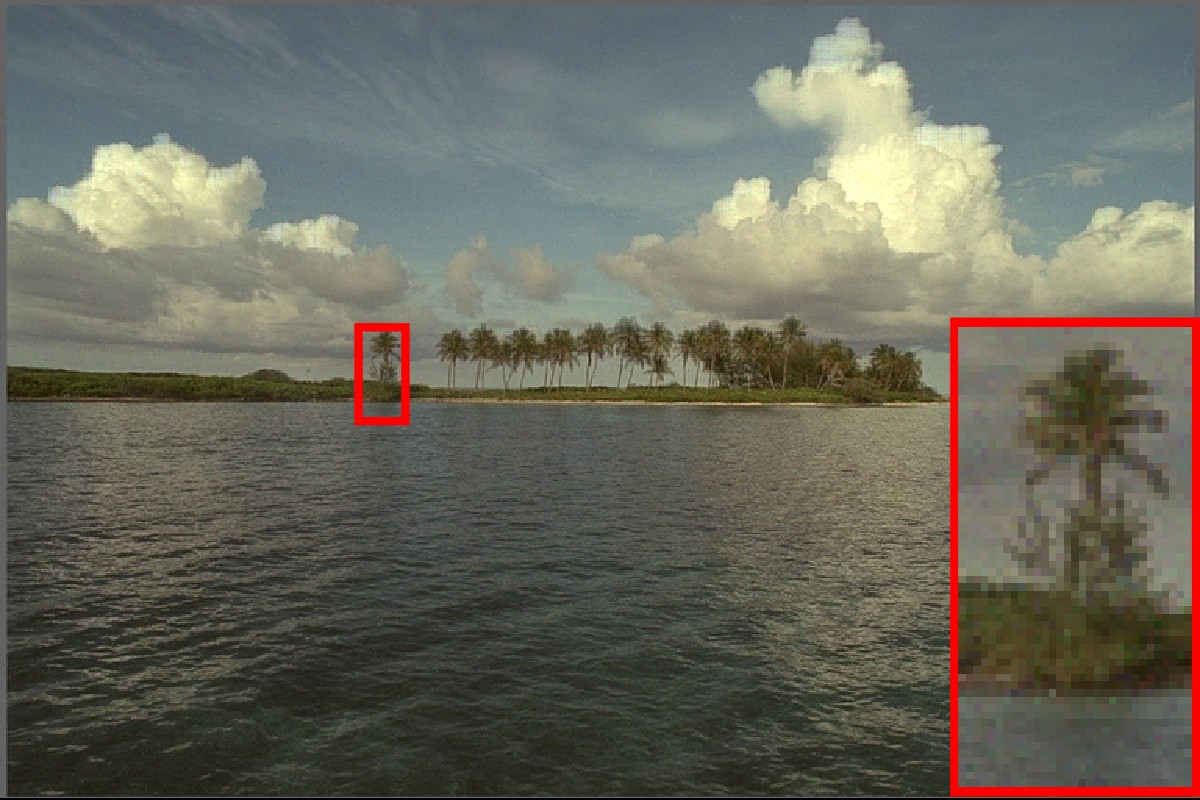} \hspace{-3.9mm} & 
			\includegraphics[width=2.2cm]{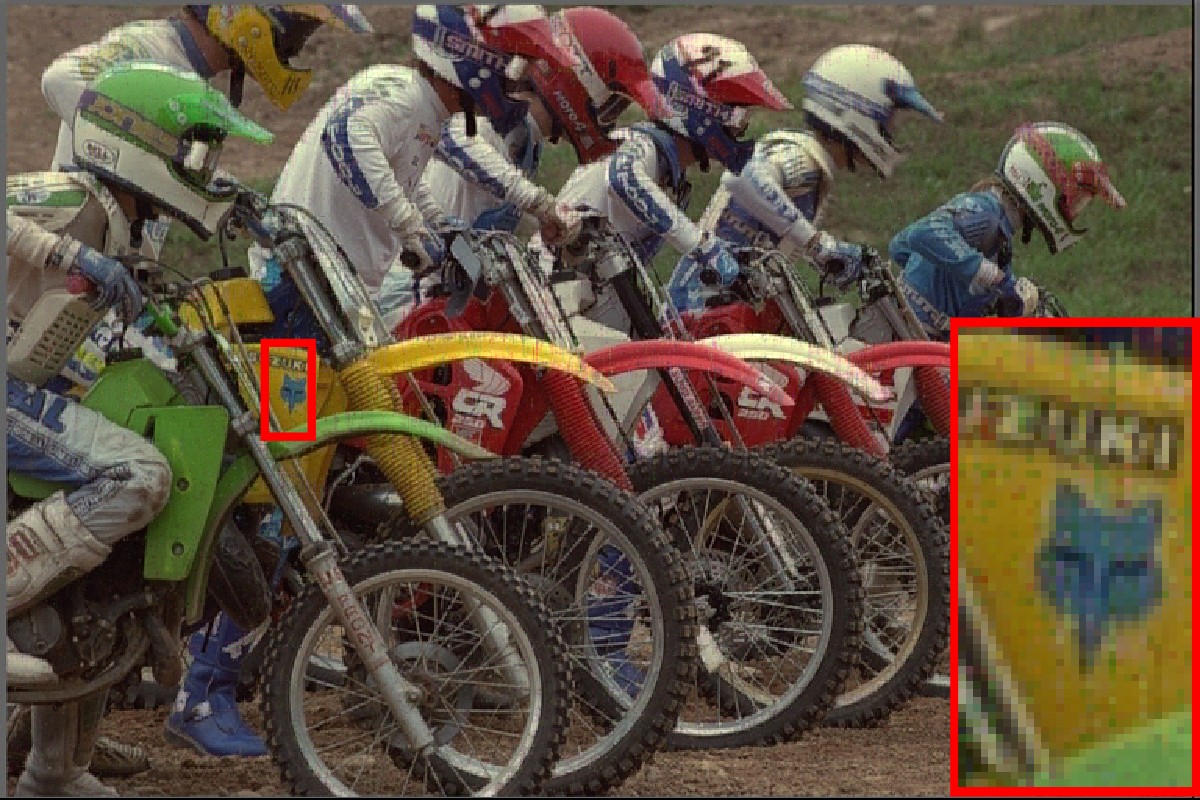} \hspace{-3.9mm} & 
			\includegraphics[width=2.2cm]{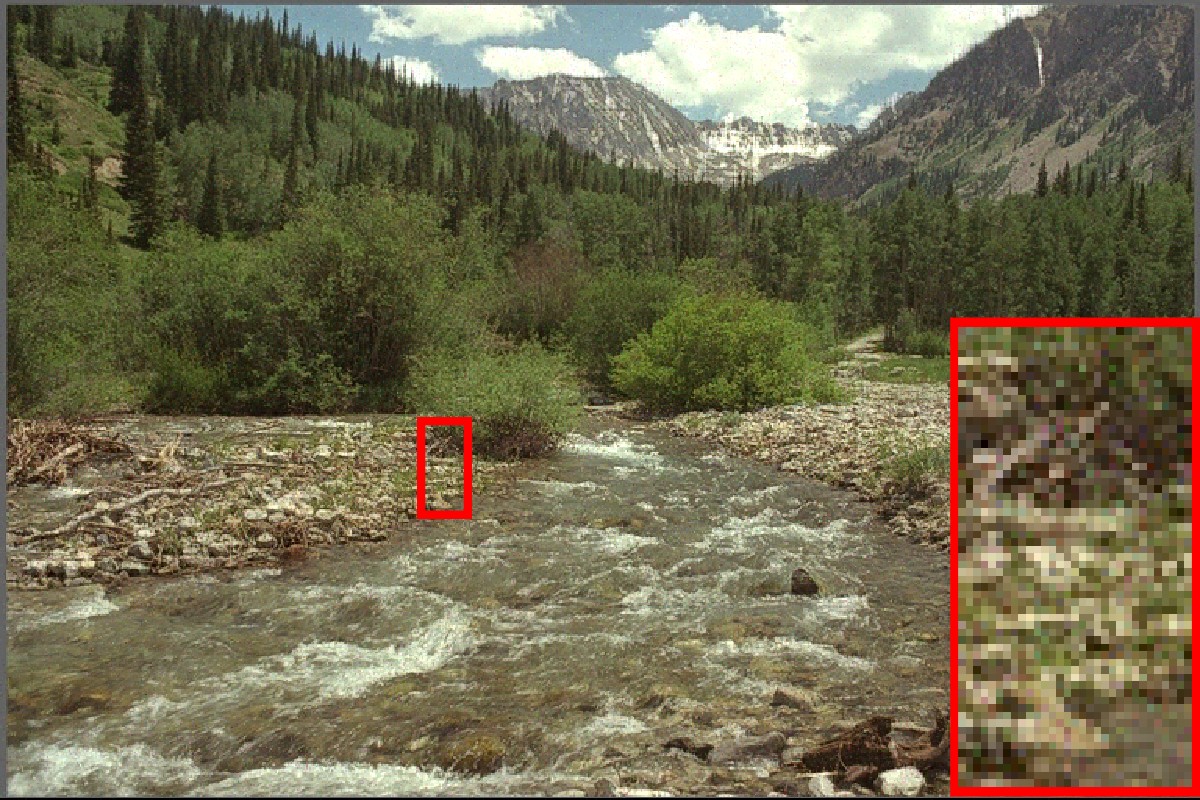} \hspace{-3.9mm} \\ 
		\end{tabular}
	\end{center}
	\vspace{-0.05in}
	\caption{Comparison of denoising performance on five example images. From top to bottom:
		(a) Original image, (b) observed image, recovered images by (c) TNN, (d) Laplace function based nonconvex surrogate, (e) $\mathrm{t}\mbox{-}S_{w, p}(0.9)$, (f) $p$-TRPCA, (g) TNF, and (h) TNF$+$.
		From left to right: five color images (``boat,'' ``house,'' ``seabeach,'' ``bicycle,'' and ``brook'').}\label{fig:TRPCAreal}
	\vspace{-0.05in}
    \end{figure}

\subsubsection{Background modeling.}
The background modeling problem aims to separate foreground objects from the background. In videos, the background is typically approximated as a low-rank tensor since it remains nearly constant across the timeframes. In contrast, moving foreground objects are treated as sparse components. In the context of the TRPCA problem, the background and foreground tensors correspond to the low-rank tensor $\mathcal{L}_0$ and the sparse tensor $\mathcal{S}_0$, respectively.
We conduct experiments using sequences of ``airport" ($144\times176\times400$) and ``bootstrap" ($120\times160\times400$) from the 12R dataset \cite{li2004statistical}, as well as ``shoppingmall" ($220\times352\times400$) from \cite{li2004statistical}. All three videos have slow object movement against different background scenarios.  We compare our TNF and TNF$+$ models with SNN, TNN, PSTNN, Laplace, and $\mathrm{t}\mbox{-}S_{w, p}(0.9)$ models. We set $\lambda=10^{-6}$ and $\mu_1=\mu_2=10^{-5}$ in TNF and $\lambda=1/\sqrt{\max(n_1,n_2)n_3}$, $\mu_1=\mu_3=10^{-5}$ and $\mu_2=10^{-3}$ in TNF$+$.

Fig.~\ref{fig:TRPCAback} presents the visual results of background modeling obtained by various methods. 
For each video, we select one image in the sequence as shown in the first column (a) of Fig.~\ref{fig:TRPCAback}, followed by the background images of the same frame obtained by (b) TNN, (c) Laplace, (d) $\mathrm{t}\mbox{-}S_{w, p}(0.9)$, (e) EAP-TRPCA-FFT \cite{qiu2022fast}, (g) TNF, and (h) TNF$+$. The second row of each video depicts the motion in the scene. 
In the ``airport'' video, the background recovered by both TNF and TNF$+$ contains less ghost silhouette compared to other methods, indicating a better background separation. Similarly, in the ``bootstrap'' and ``shopping mall'' videos, the humans identified by the proposed method look sharper and clearer than the ones by other methods.

\begin{figure*}
\begin{center}
\begin{tabular}{cccccccc}
(a)&(b)&(c)&(d)& (e)&(f)&(g)\\
\includegraphics[width=1.6cm]{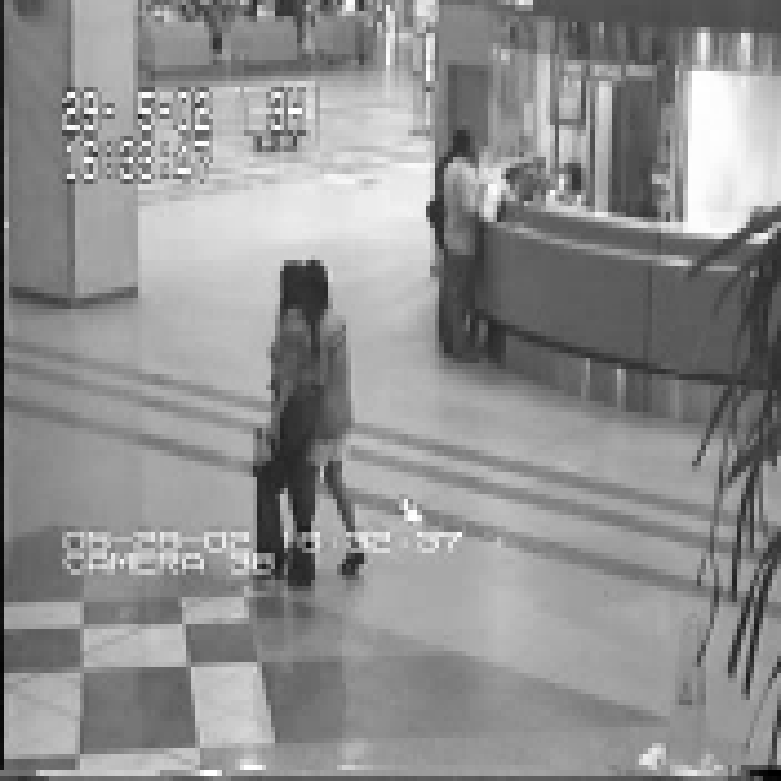} \hspace{-5.2mm} &
\includegraphics[width=1.6cm]{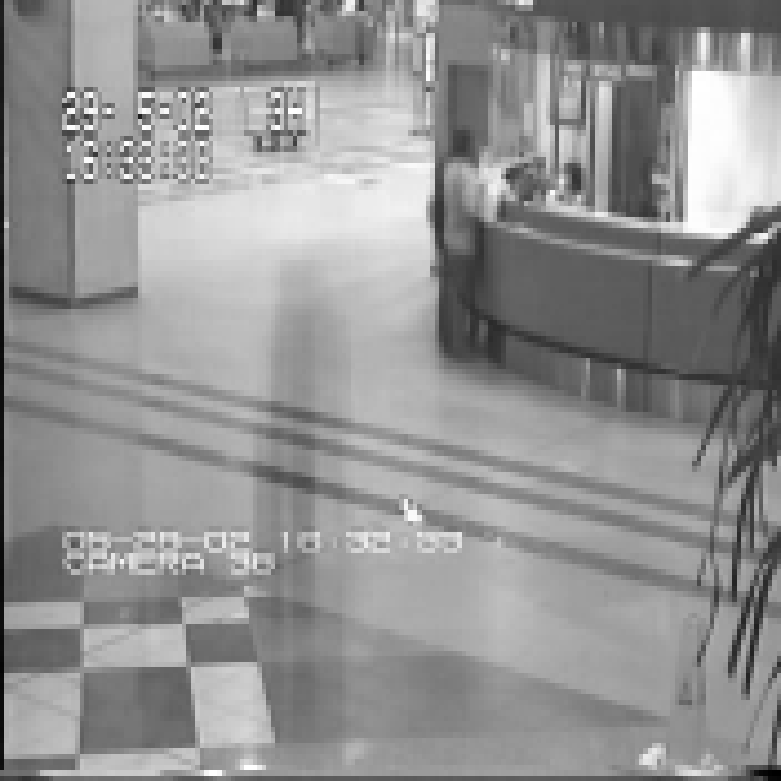} \hspace{-5.2mm} &
\includegraphics[width=1.6cm]{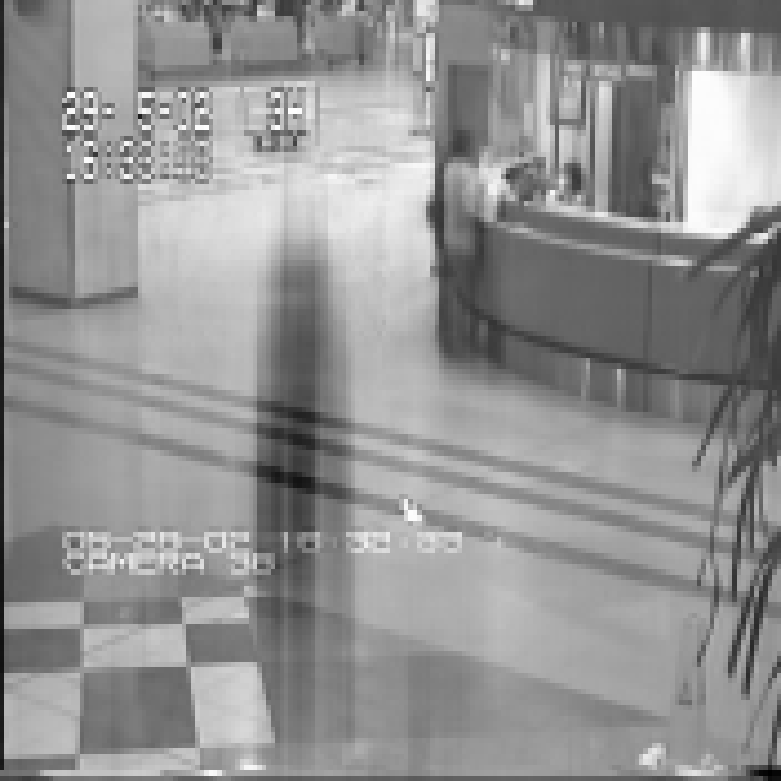} \hspace{-5.2mm} &
\includegraphics[width=1.6cm]{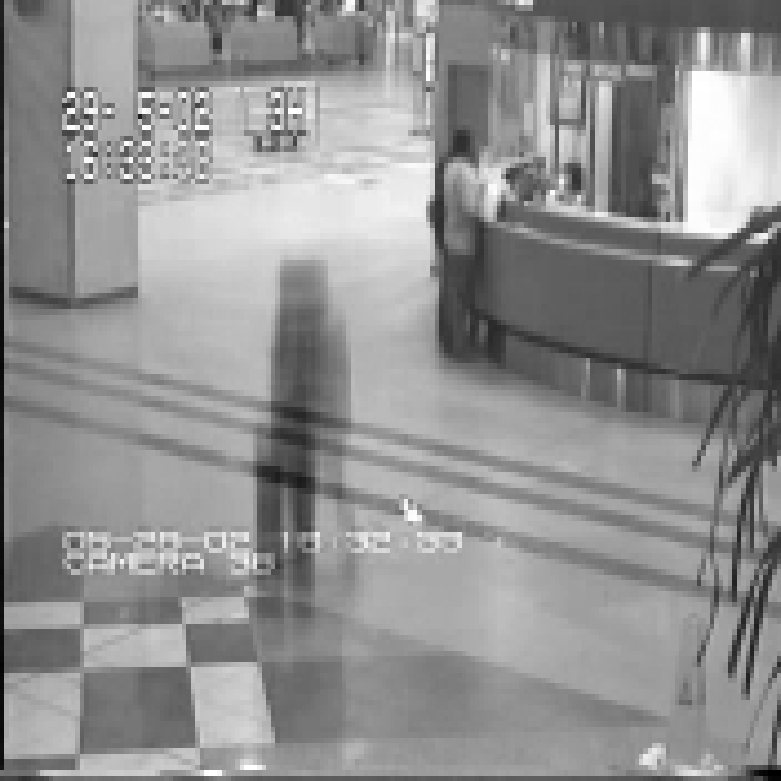} \hspace{-5.2mm} &
\includegraphics[width=1.6cm]{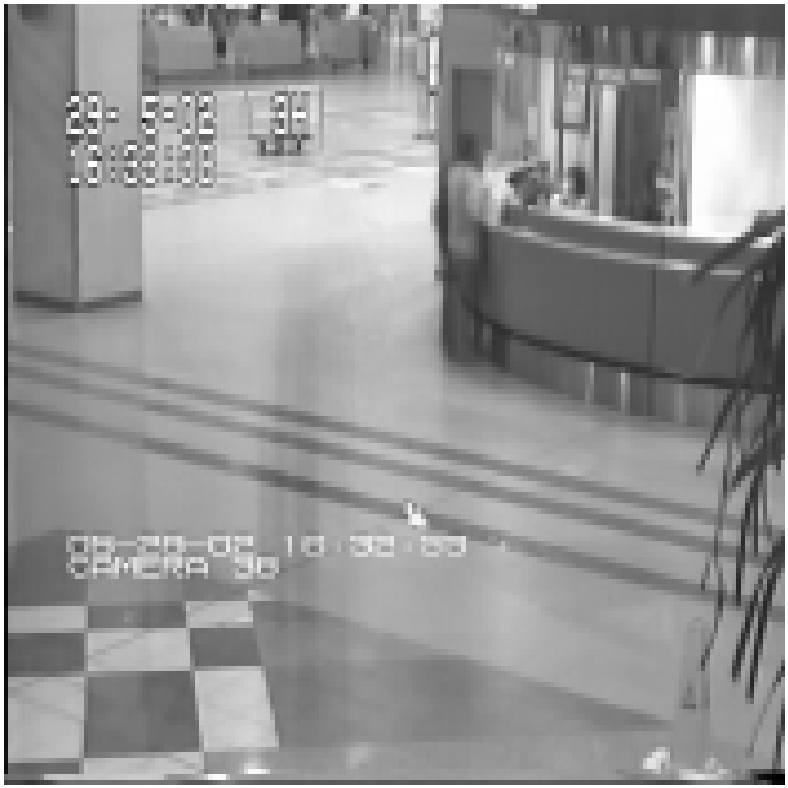} \hspace{-5.2mm} &
\includegraphics[width=1.6cm]{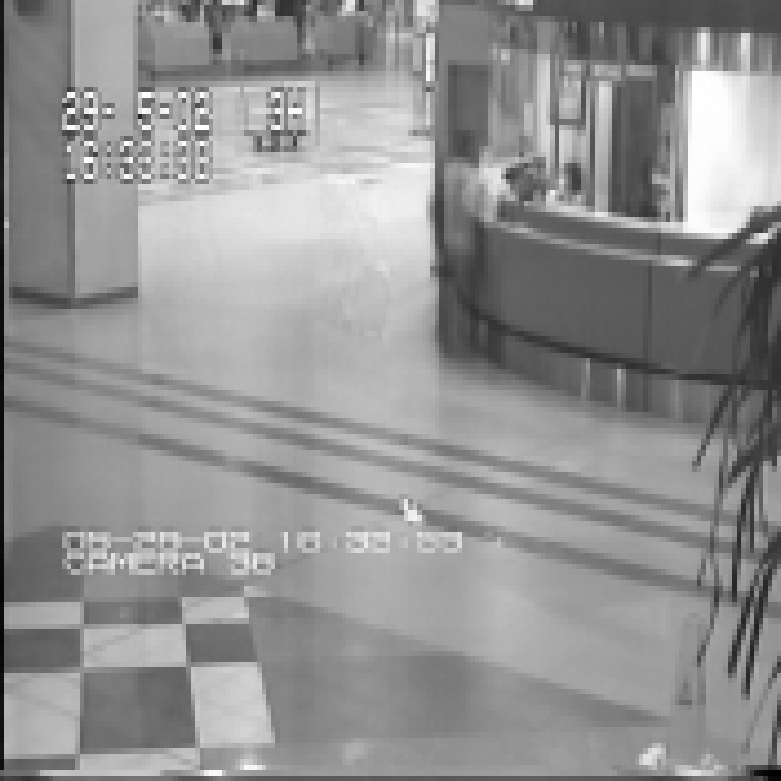} \hspace{-5.2mm} &
\includegraphics[width=1.6cm]{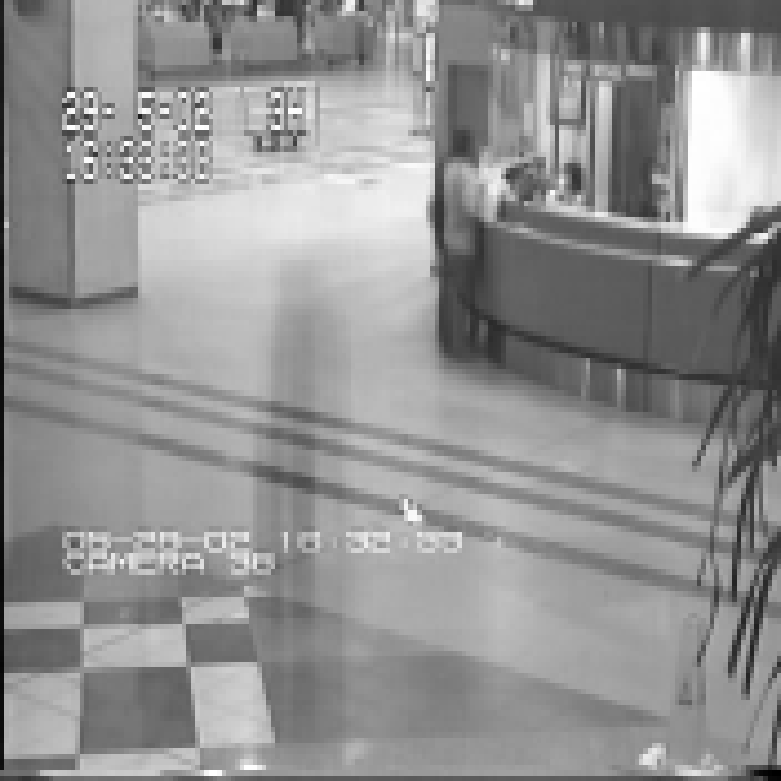} \hspace{-5.2mm} 
 \\  
\includegraphics[width=1.6cm]{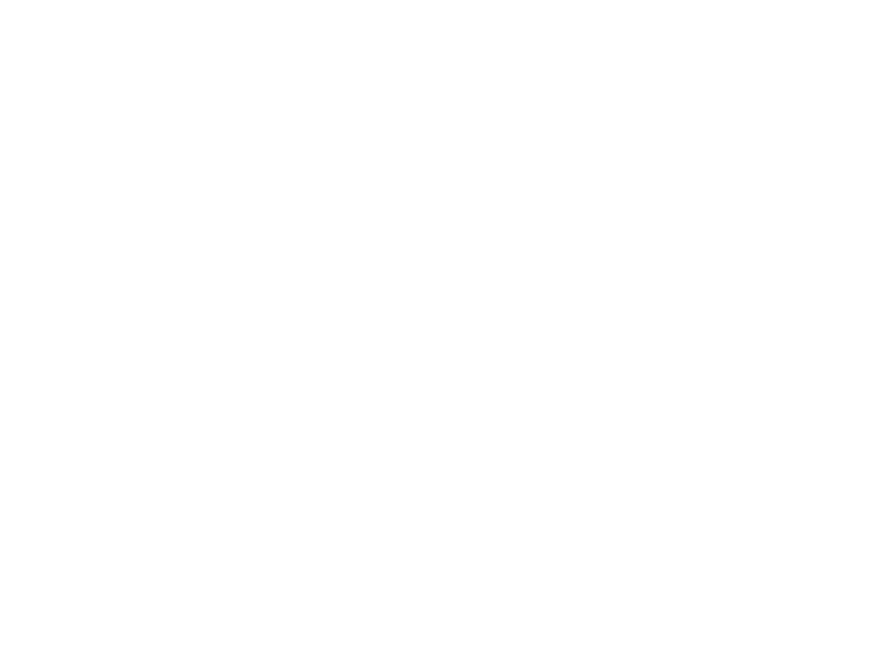} \hspace{-5.2mm} &
\includegraphics[width=1.6cm]{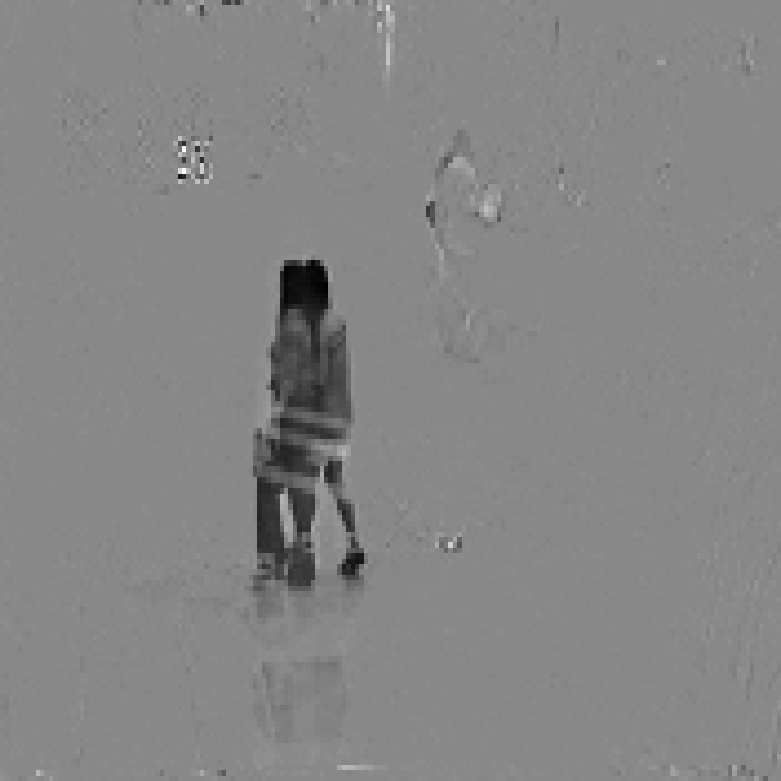} \hspace{-5.2mm} &
\includegraphics[width=1.6cm]{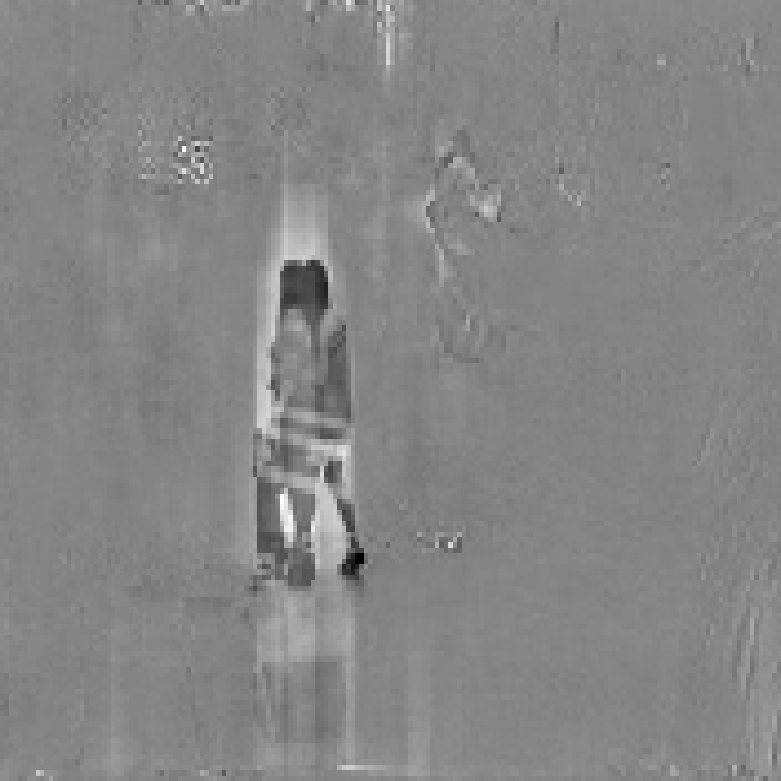} \hspace{-5.2mm} &
\includegraphics[width=1.6cm]{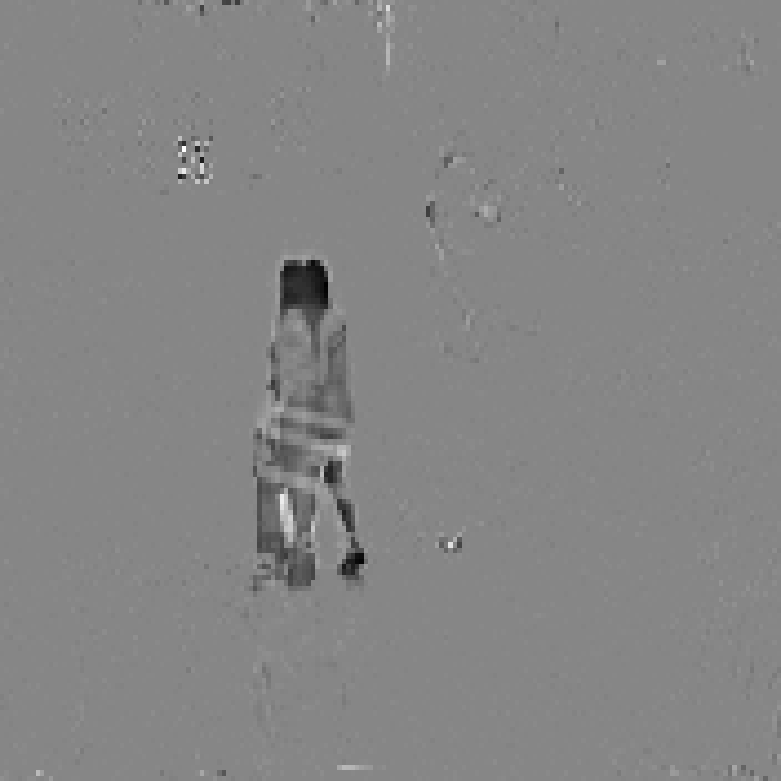} \hspace{-5.2mm} &
\includegraphics[width=1.6cm]{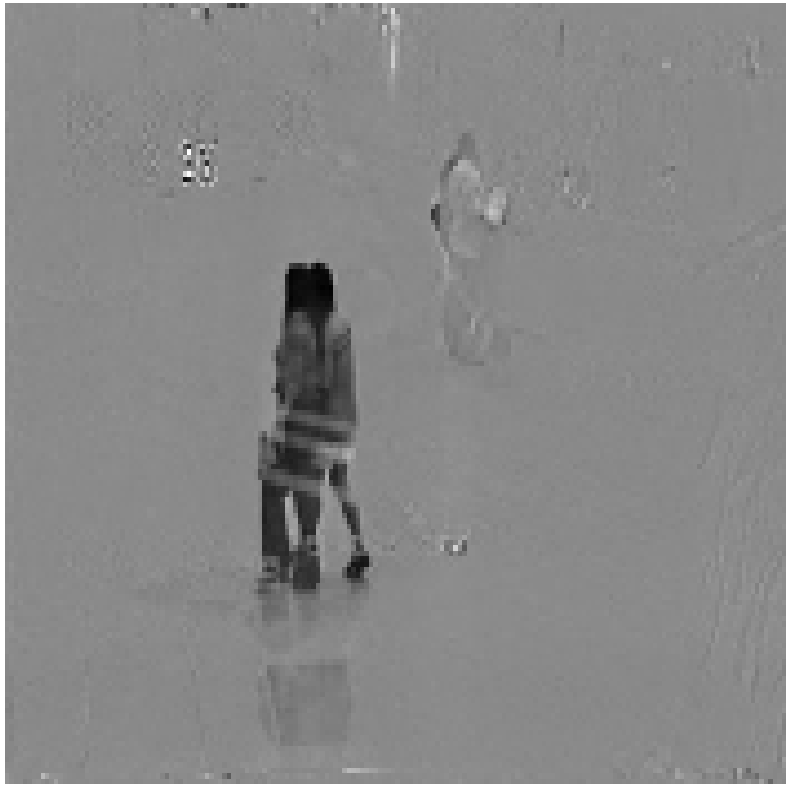} \hspace{-5.2mm} &
\includegraphics[width=1.6cm]{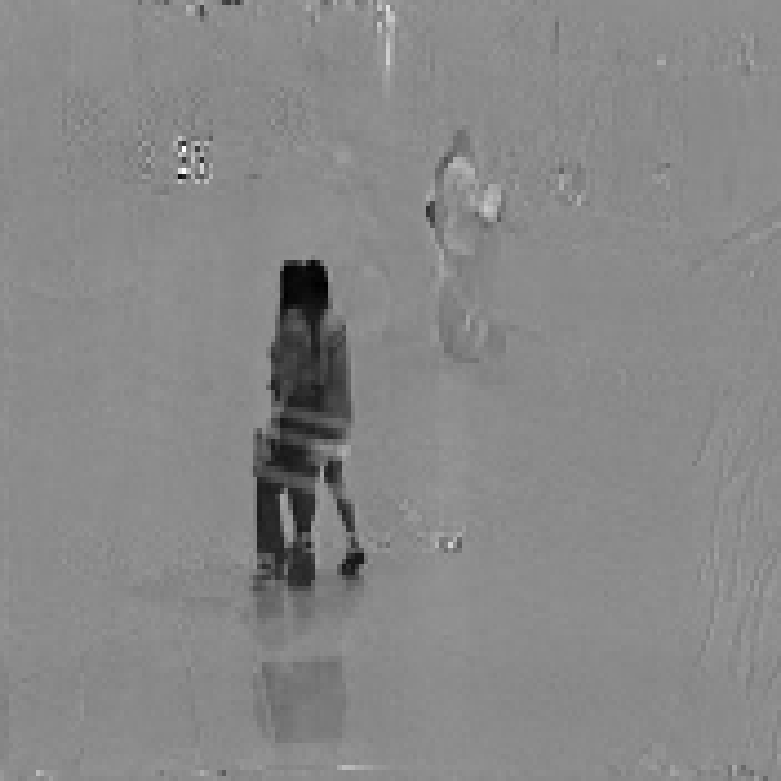} \hspace{-5.2mm} &
\includegraphics[width=1.6cm]{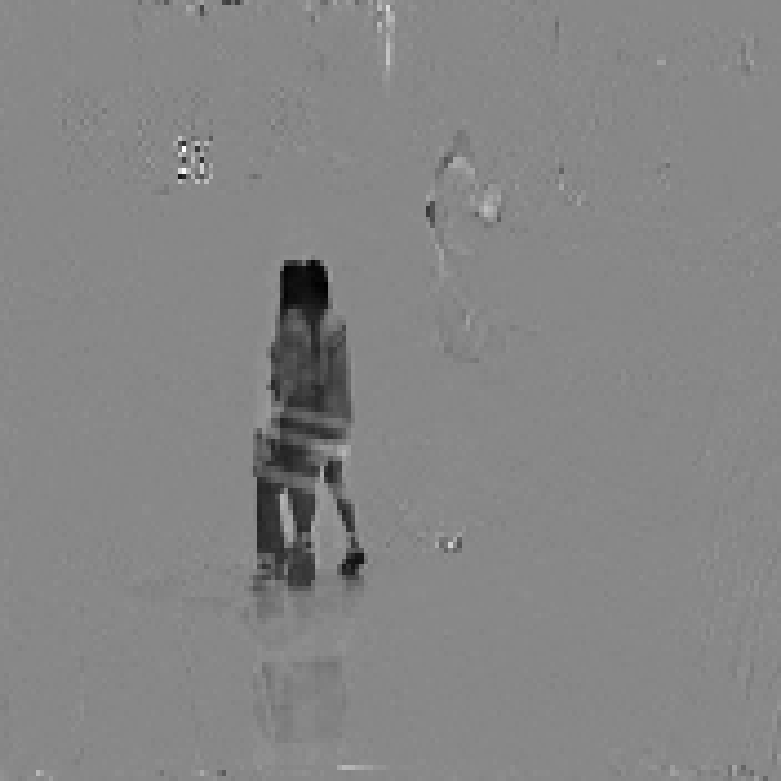} \hspace{-5.2mm} 
 \\  
\includegraphics[width=1.6cm]{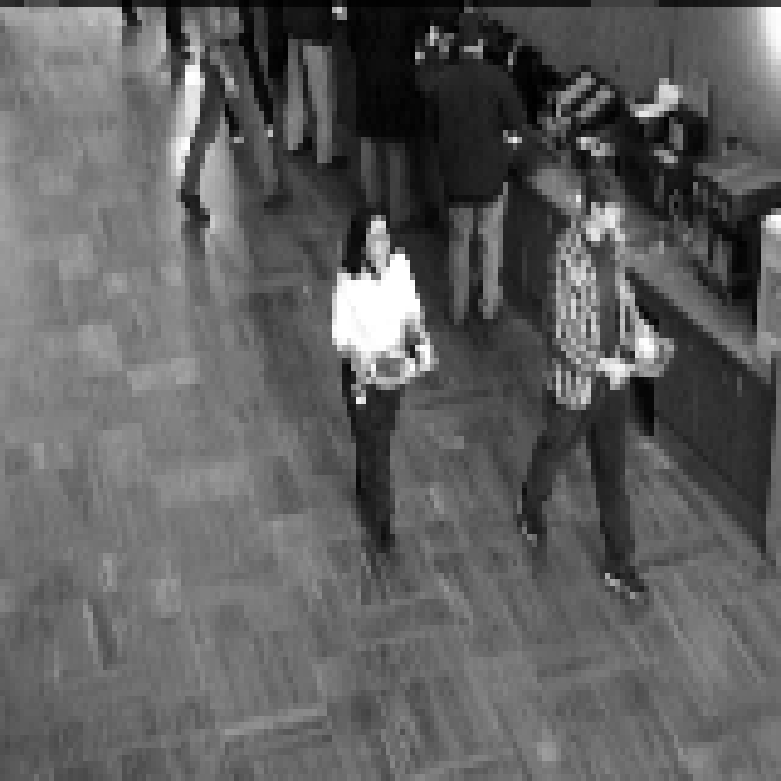} \hspace{-5.2mm} &
\includegraphics[width=1.6cm]{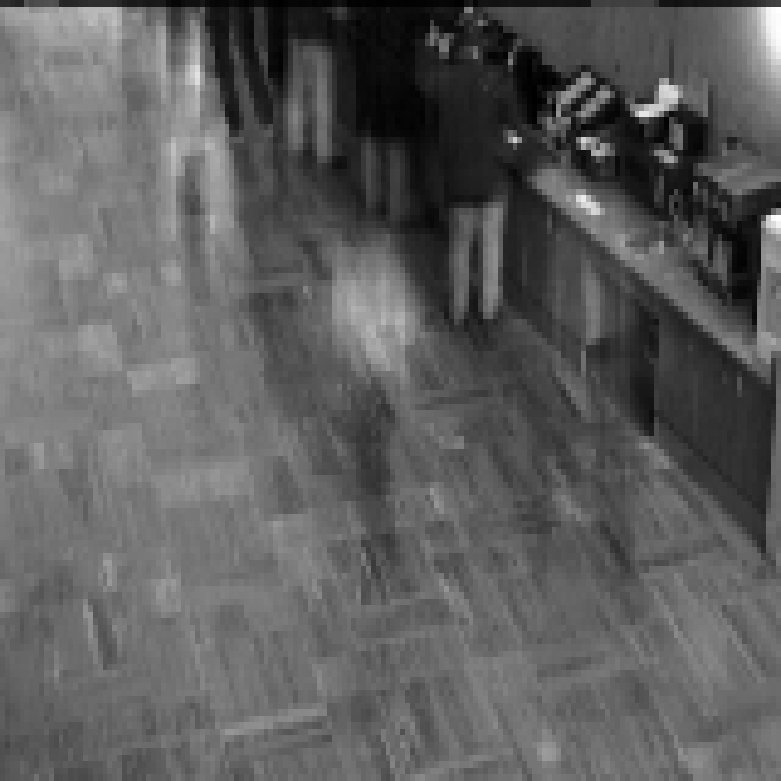} \hspace{-5.2mm} &
\includegraphics[width=1.6cm]{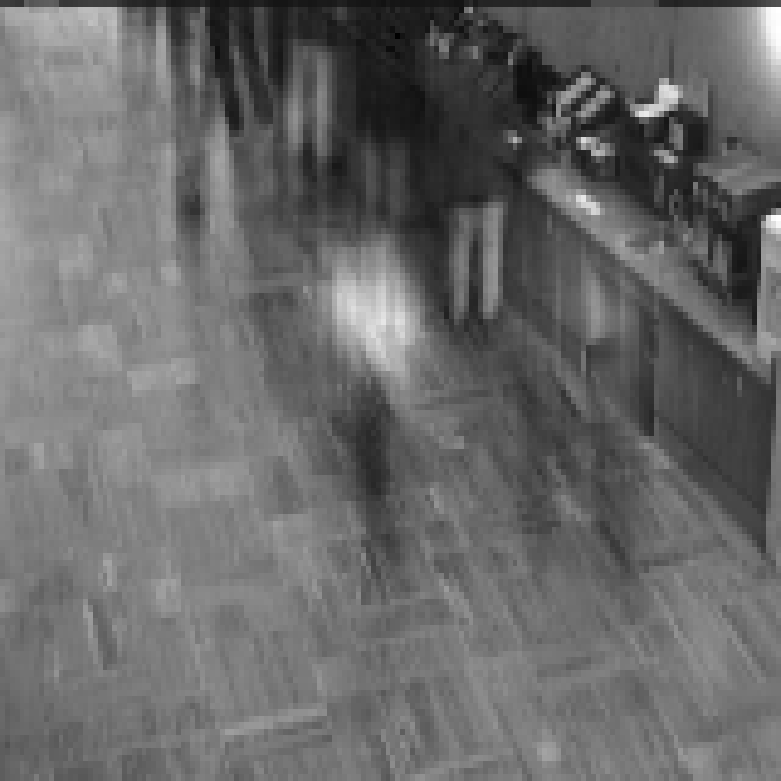} \hspace{-5.2mm} &
\includegraphics[width=1.6cm]{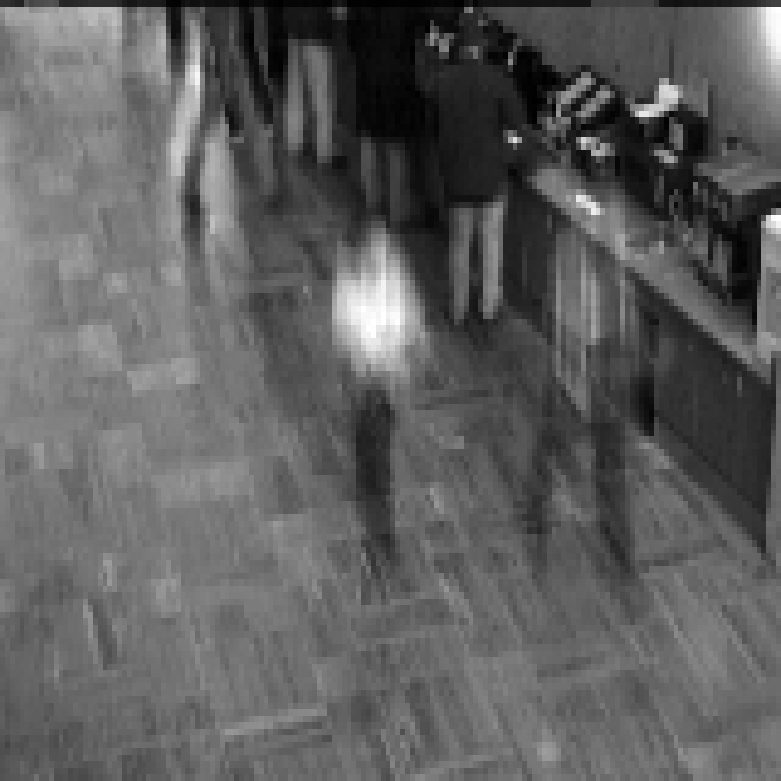} \hspace{-5.2mm} &
\includegraphics[width=1.6cm]{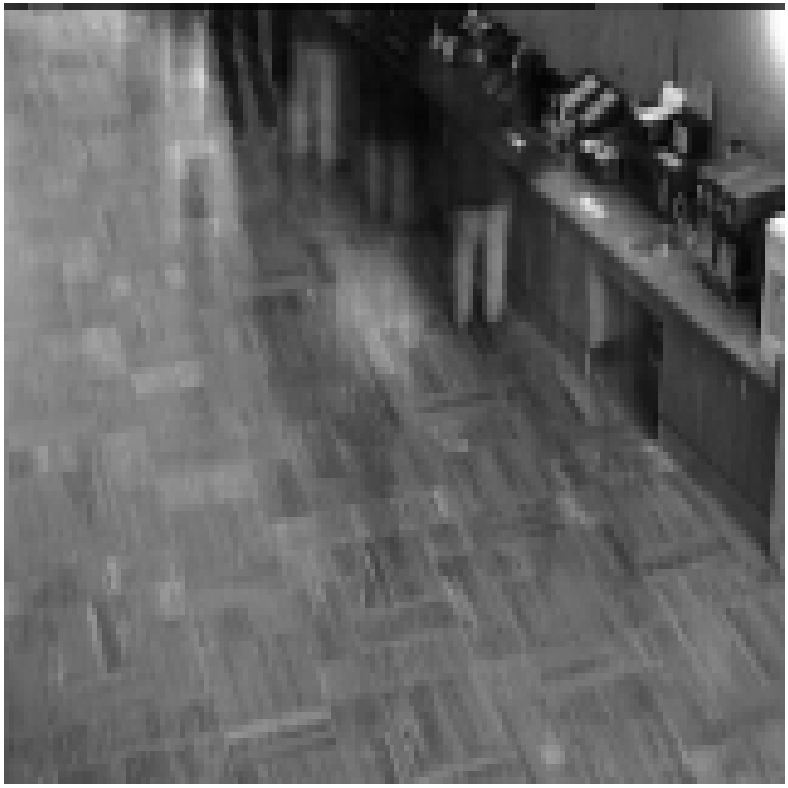} \hspace{-5.2mm} &
\includegraphics[width=1.6cm]{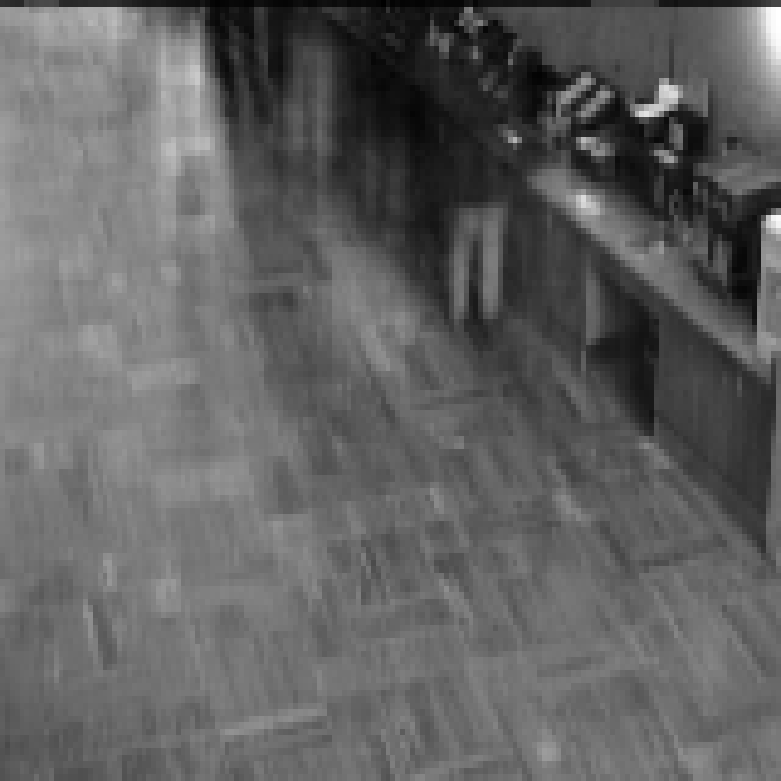} \hspace{-5.2mm} &
\includegraphics[width=1.6cm]{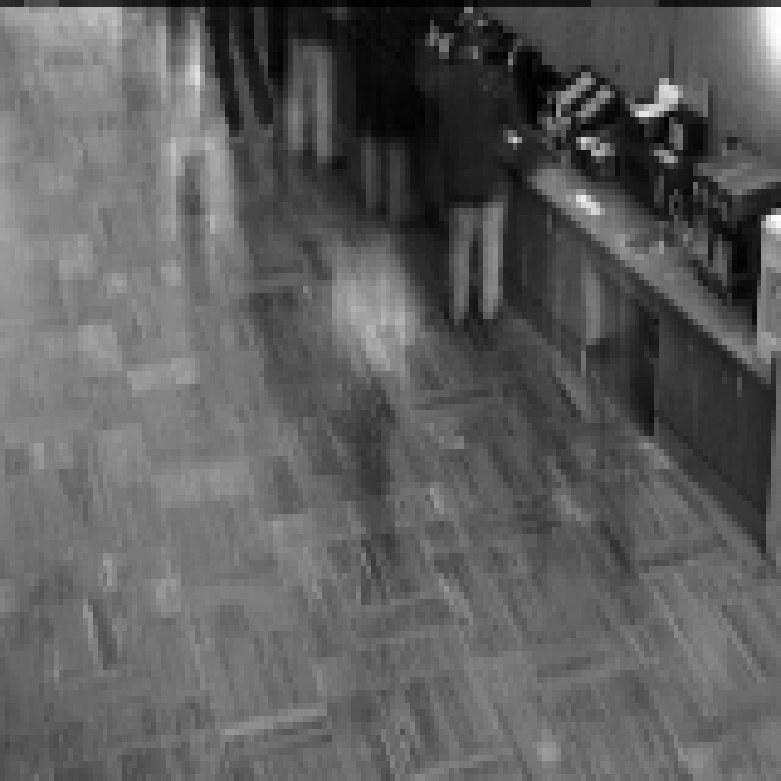} \hspace{-5.2mm} 
 \\  
\includegraphics[width=1.6cm]{figure/white1} \hspace{-5.2mm} &
\includegraphics[width=1.6cm]{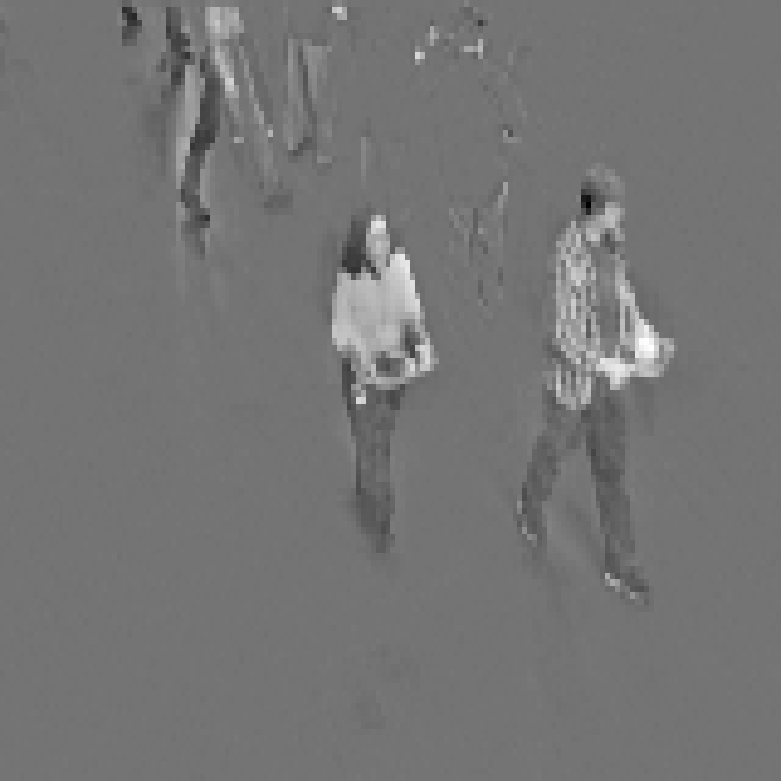} \hspace{-5.2mm} &
\includegraphics[width=1.6cm]{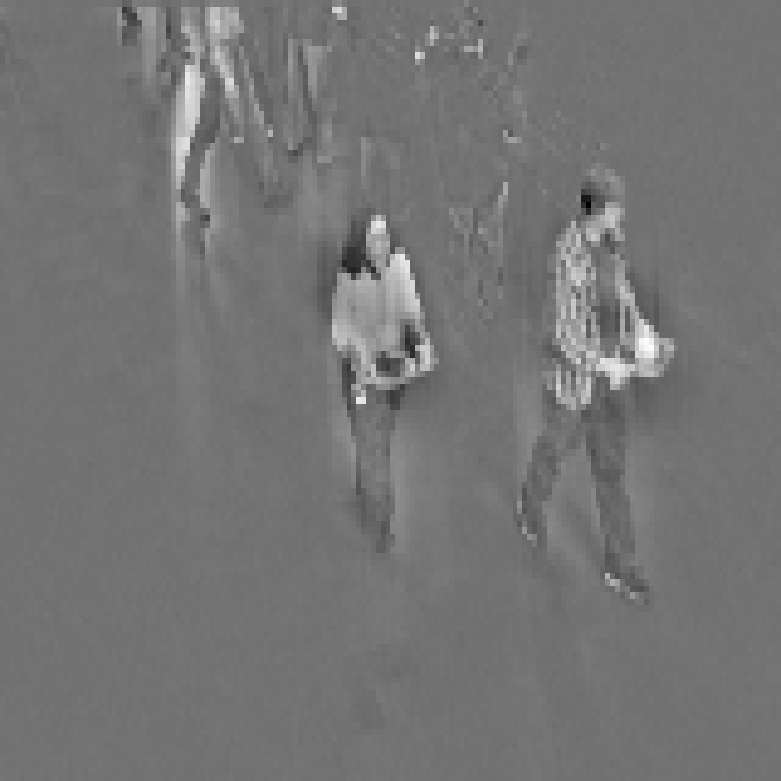} \hspace{-5.2mm} &
\includegraphics[width=1.6cm]{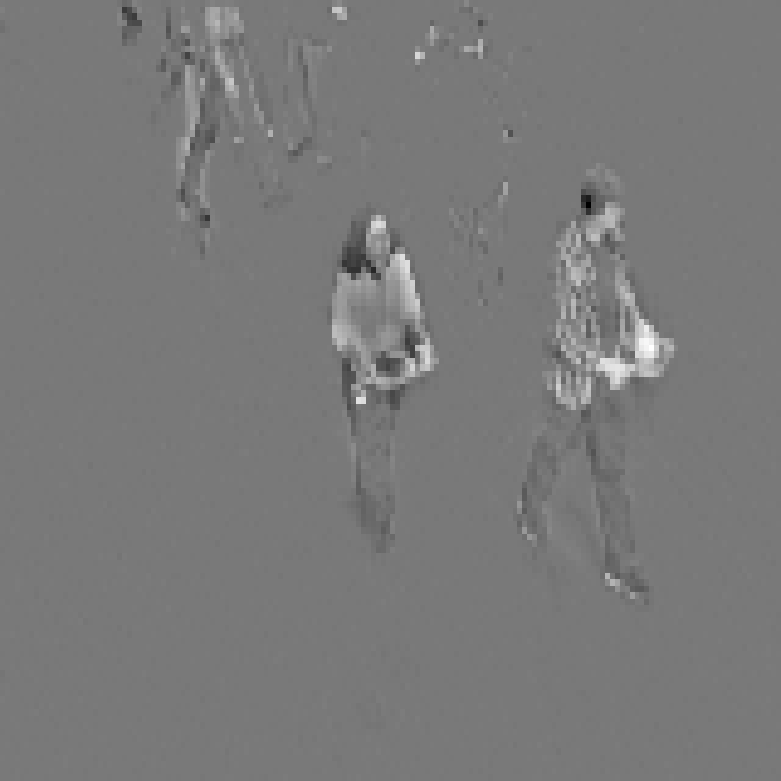} \hspace{-5.2mm} &
\includegraphics[width=1.6cm]{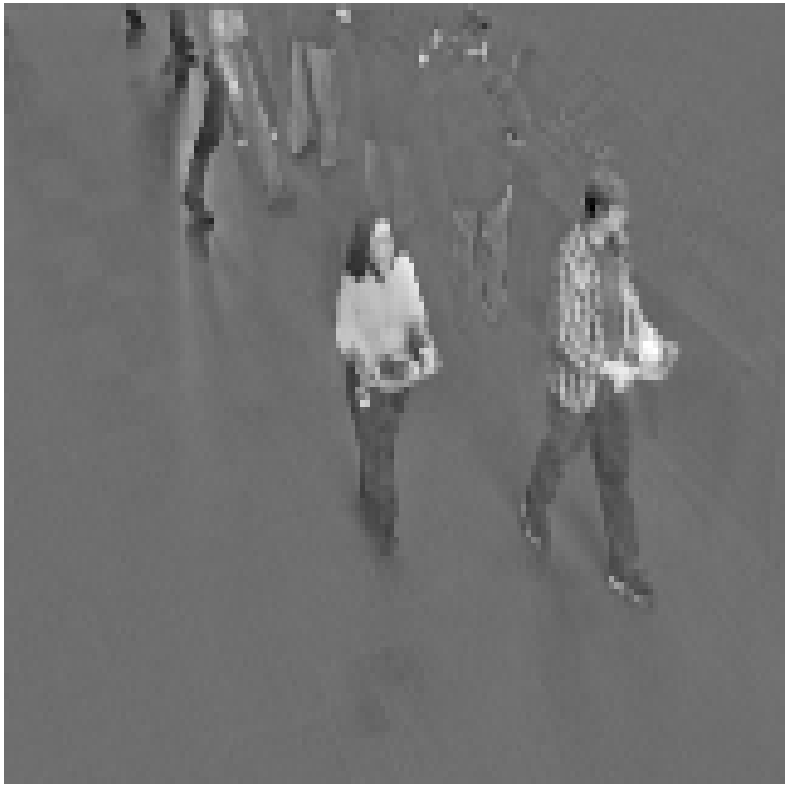} \hspace{-5.2mm} &
\includegraphics[width=1.6cm]{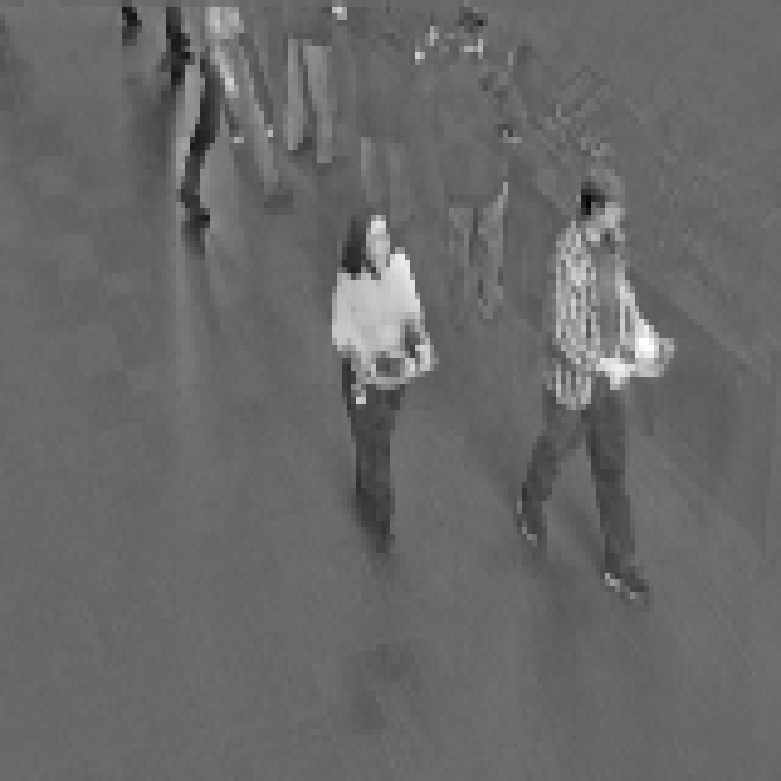} \hspace{-5.2mm} &
\includegraphics[width=1.6cm]{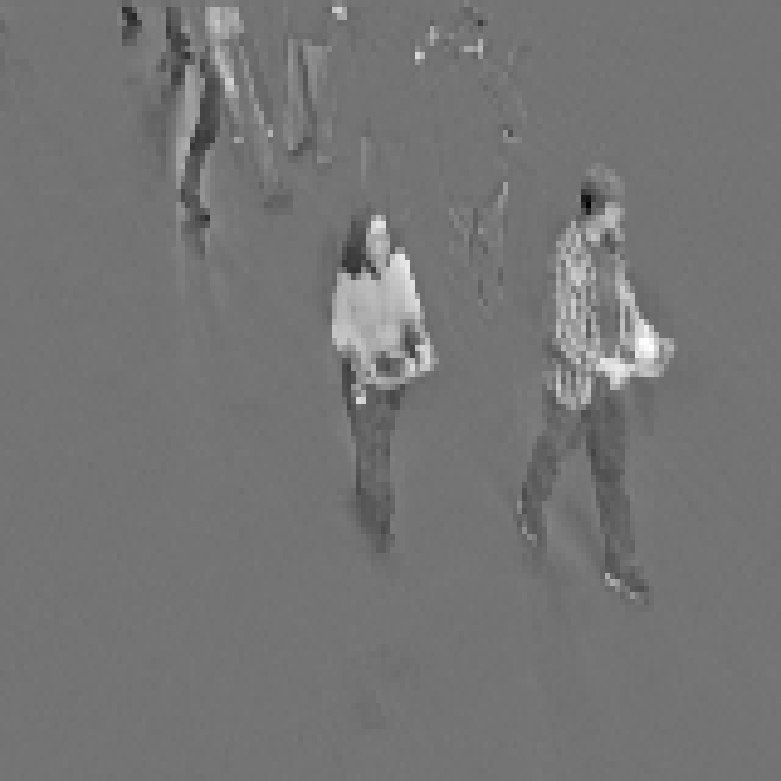} \hspace{-5.2mm} 
\\  		
\includegraphics[width=1.6cm]{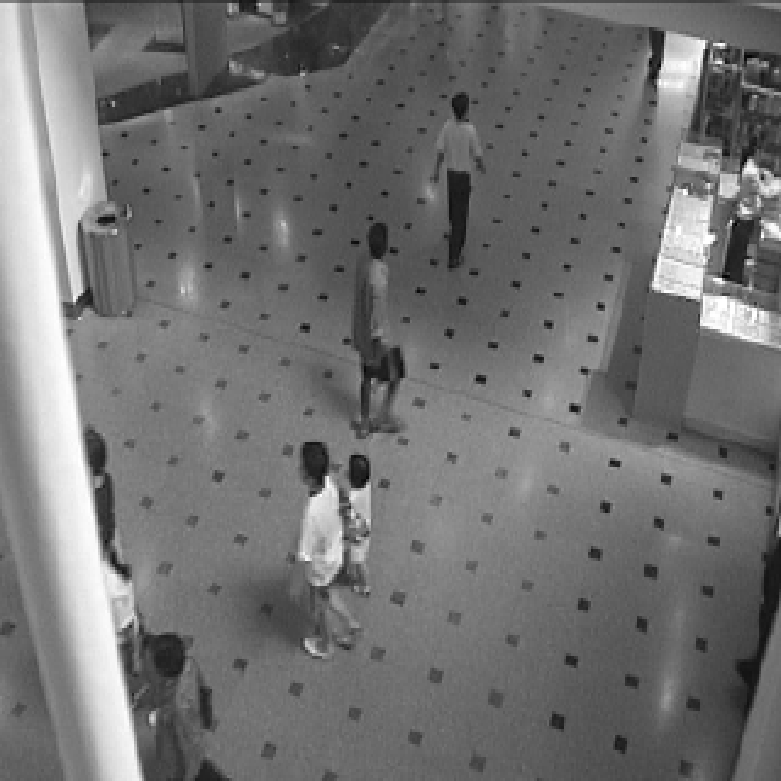} \hspace{-5.2mm} &
\includegraphics[width=1.6cm]{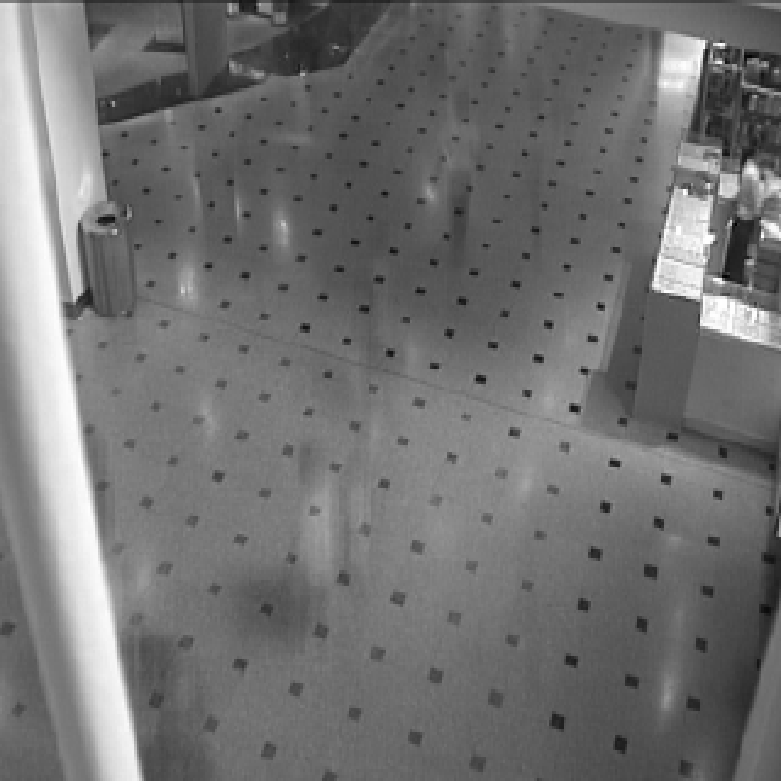} \hspace{-5.2mm} &
\includegraphics[width=1.6cm]{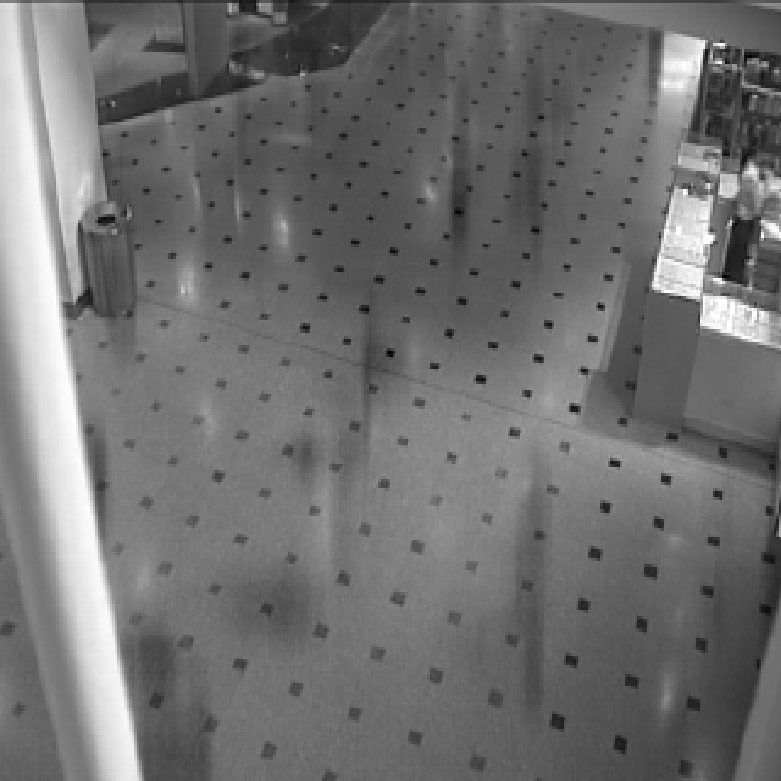} \hspace{-5.2mm} &
\includegraphics[width=1.6cm]{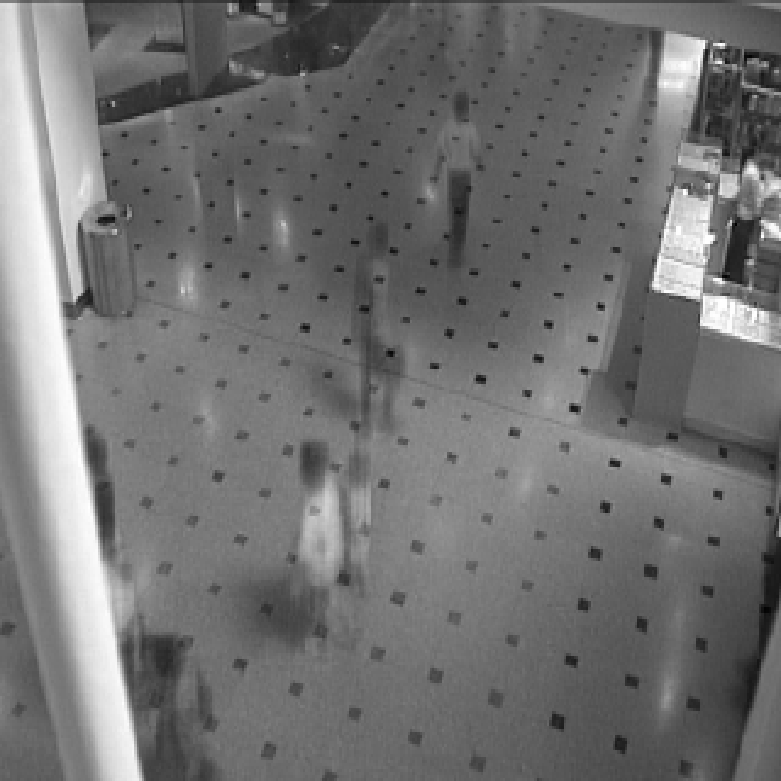} \hspace{-5.2mm} &
\includegraphics[width=1.6cm]{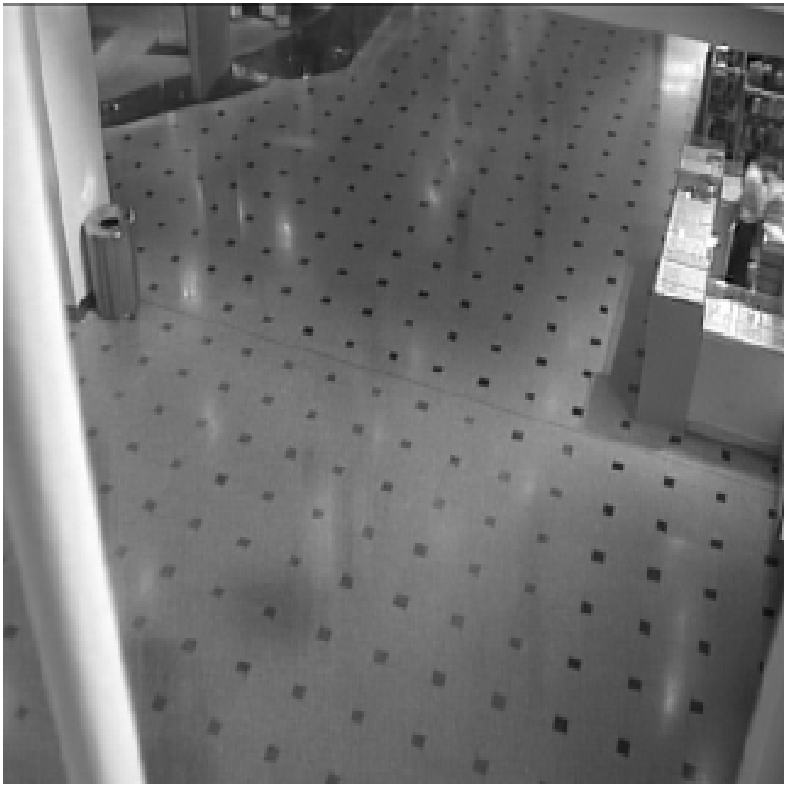} \hspace{-5.2mm} &
\includegraphics[width=1.6cm]{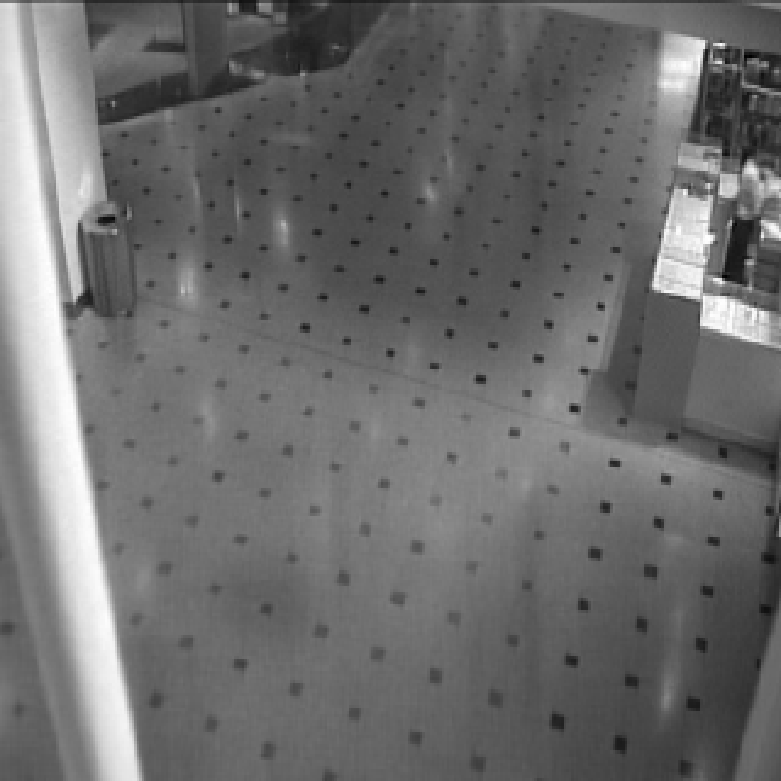} \hspace{-5.2mm} &
\includegraphics[width=1.6cm]{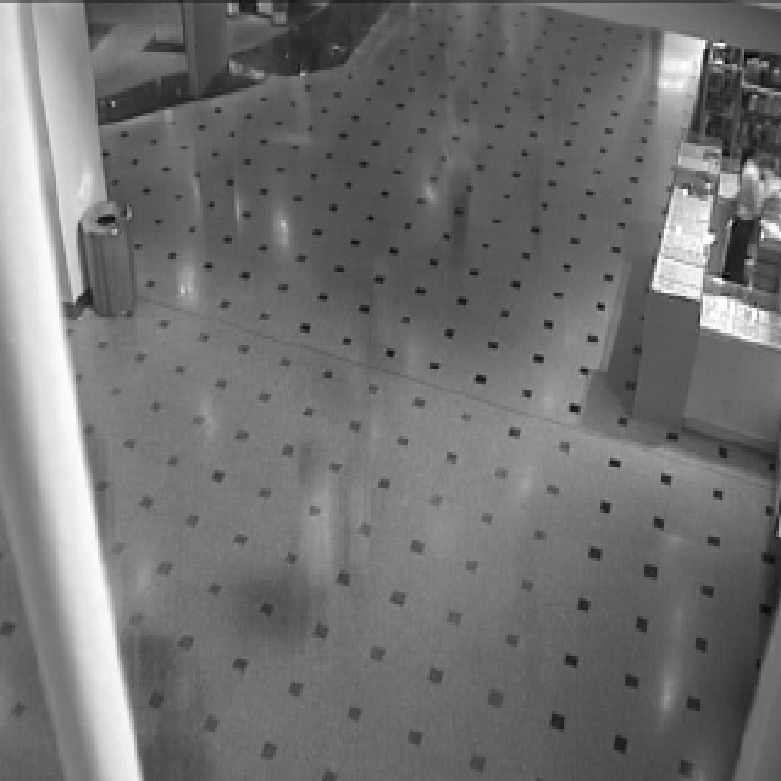} \hspace{-5.2mm} 
\\  		
\includegraphics[width=1.6cm]{figure/white1} \hspace{-5.2mm} &
\includegraphics[width=1.6cm]{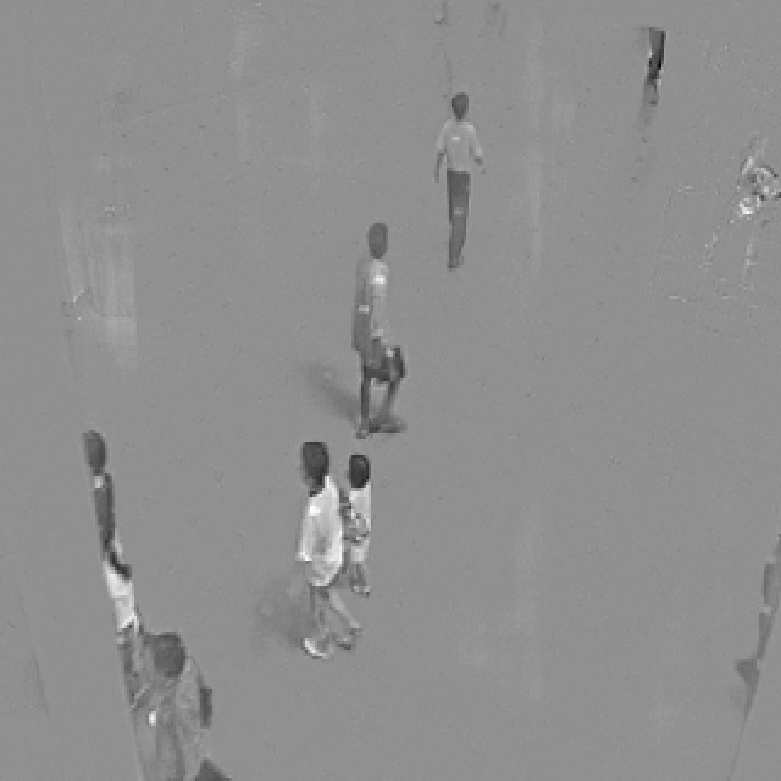} \hspace{-5.2mm} &
\includegraphics[width=1.6cm]{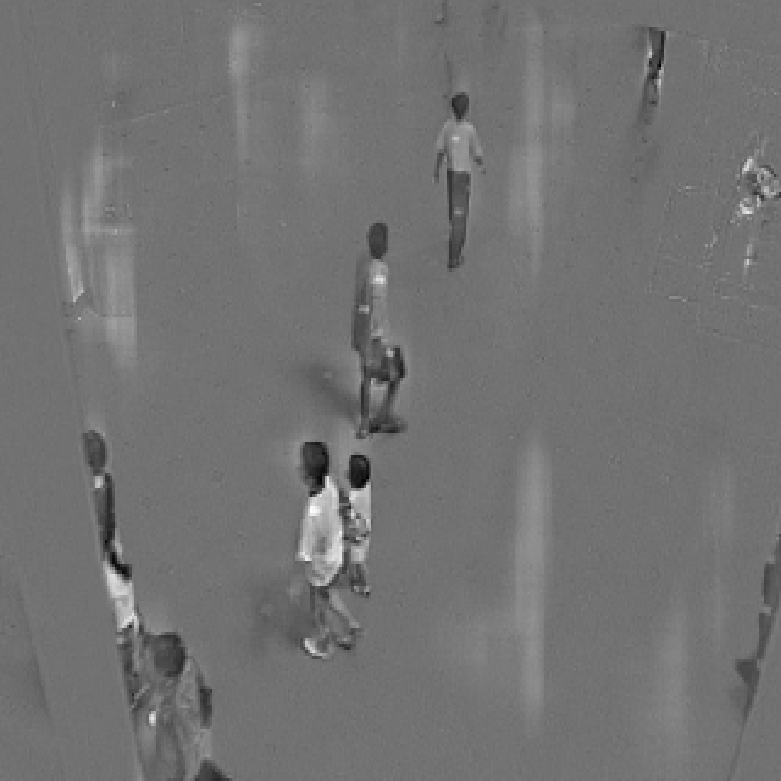} \hspace{-5.2mm} &
\includegraphics[width=1.6cm]{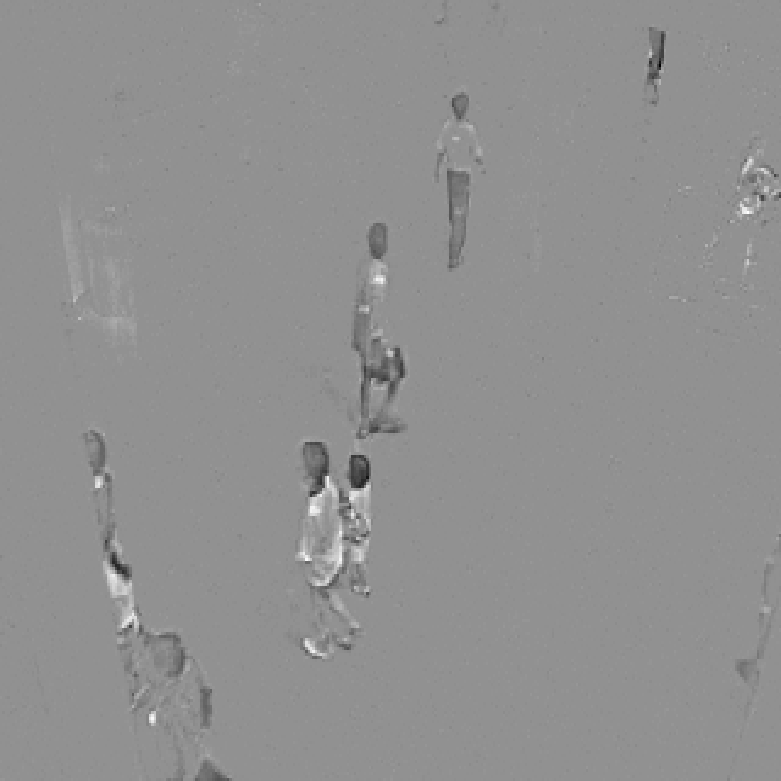} \hspace{-5.2mm} &
\includegraphics[width=1.6cm]{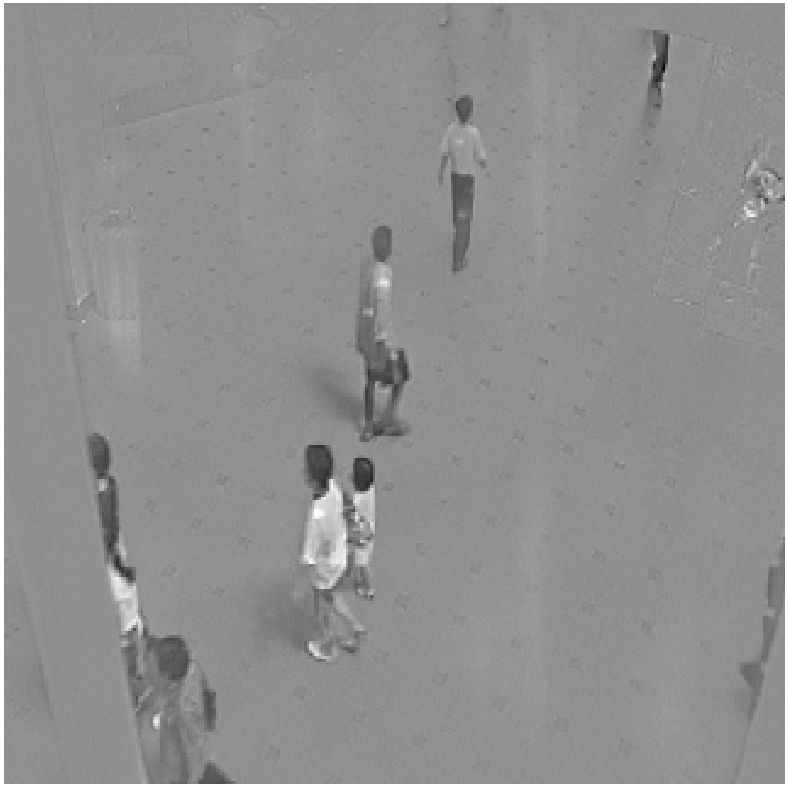} \hspace{-5.2mm} &
\includegraphics[width=1.6cm]{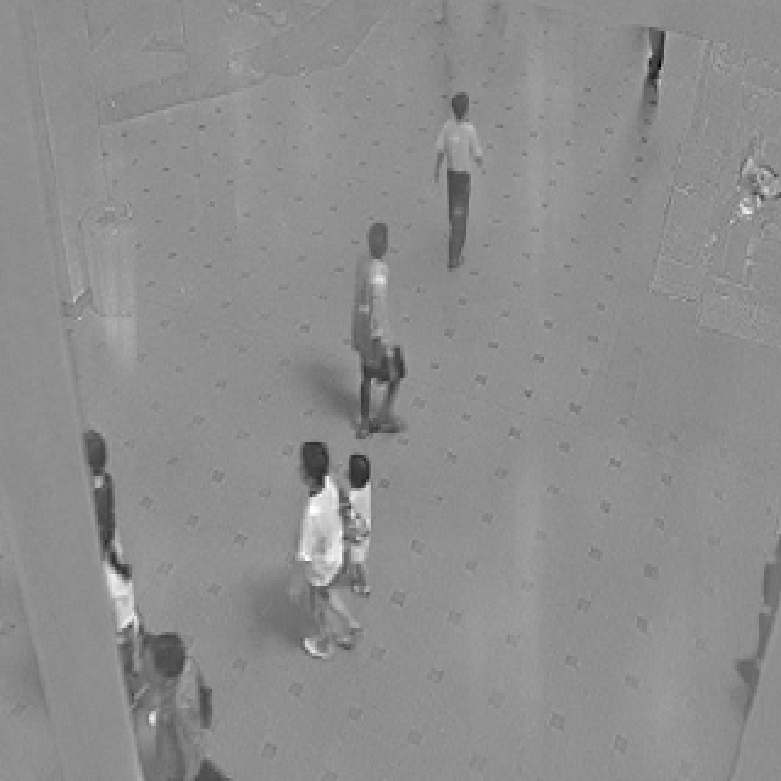} \hspace{-5.2mm} &
\includegraphics[width=1.6cm]{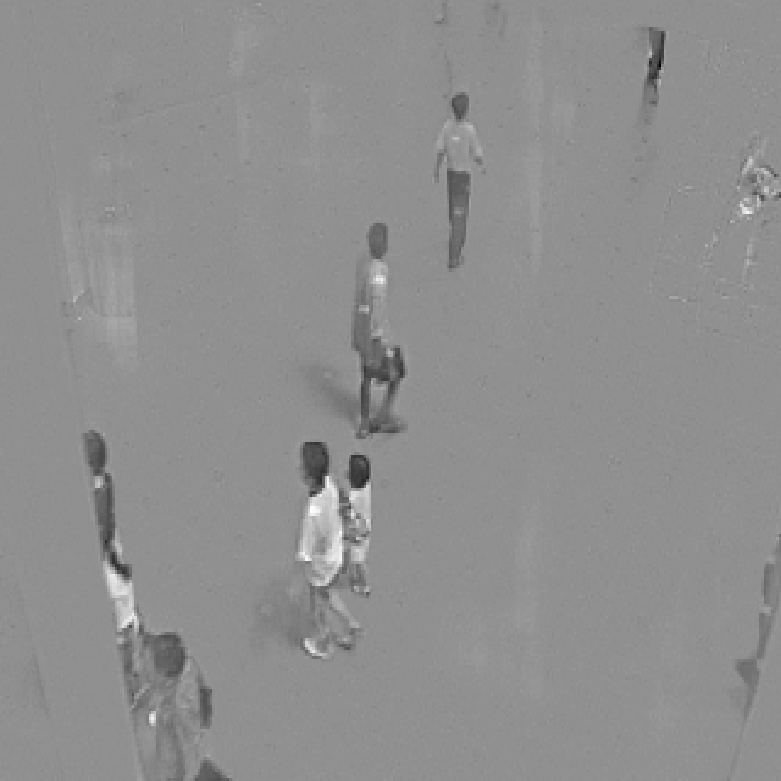} \hspace{-5.2mm} 
\end{tabular}
\end{center}
\vspace{-0.05in}
\caption{Comparison of background model on three example images, labeled by ``airport'' (top two rows), ``bootstrap'' (middle two), and ``shopping mall'' (bottom two). From left to right:
 (a) Original image, background model by  (b) TNN,  (c) Laplace function based nonconvex surrogate, (d) $\mathrm{t}\mbox{-}S_{w, p}(0.9)$, (e) EAP-TRPCA-FFT, (g) TNF, and (h) TNF$+$.
}
     \label{fig:TRPCAback}
     \vspace{-0.05in}
 \end{figure*}

\section {Conclusions}\label{sect:concludePCA}
In this paper, we revisited a nonconvex approximation to the tensor tubal rank, referred to as the tensor nuclear over the Frobenius norms (TNF). Building upon this approximation, we developed two models for the TRPCA problem, where a sparse tensor is identified by minimizing $\ell_1$ and $\ell_1/\ell_F$ regularizations, thus leading to
 TNF and TNF$+$ models, respectively.  We proved that the underlying pair of the low-rank tensor and the sparse tensor is a local minimizer of the proposed TNF model under tensor incoherence conditions. Both TNF and TNF$+$ models can be effectively solved via ADMM with convergence guarantees. Extensive experiments were conducted to showcase the effectiveness of our proposed models compared to state-of-the-art methods. 
 Future endeavors would focus on relaxing the conditions in the theoretical analysis of the two models and adapting the models to various noise distributions. Additionally, we will fill in the gap between  TNF and  TNF$+$ regarding their recovery theory. 

\begin{acknowledgements}
This work was partially supported by the National Science Foundation CAREER Award 2414705, the National Key R\&D Program of China (2023YFA1011400), the Natural Science Foundation of China (No.~12201286), Guangdong Basic and Applied Research Foundation 2024A1515012347, the Shenzhen Science and Technology Program 20231115165836001,  and the Shenzhen Fundamental Research Program JCYJ20220818100602005.
\end{acknowledgements}

\section*{Data Availability}
The MATLAB codes and datasets generated and/or analyzed during the current study will be available after publication.

\section*{Declarations}
The authors have no relevant financial or non-financial interests to disclose.  The authors declare that they have no conflict of interest.


\appendix

\renewcommand{\thedefinition}{\Alph{section}.\arabic{definition}}

\section{Relevant concepts on t-SVD} \label{app:t_svd}
Let $\overline{\mathbf{A}}\in \mathbb{C}^{n_{1}n_{3} \times n_{2}n_{3}}$ be a block diagonal matrix of the tensor $\overline{\mathcal{A}}$, i.e., 
\begin{equation}
\overline{\mathbf{A}}:=\operatorname{bdiag}(\overline{\mathcal{A}})=\left[\begin{array}{llll}
\overline{\mathbf{A}}^{(1)} & & & \\
& \overline{\mathbf{A}}^{(2)} & & \\
& & \ddots & \\
& & & \overline{\mathbf{A}}^{\left(n_{3}\right)}
\end{array}\right]. \label{con:bdiagdef}
\end{equation} 
It follows from \cite{friedland2018nuclear} that $\langle\mathcal{A},\mathcal{B}\rangle=\tfrac{1}{n_{3}}\langle\overline{\mathbf{A}}, \overline{\mathbf{B}}\rangle$ and $ \|\mathcal{A}\|_{F}=\tfrac{1}{\sqrt{n_{3}}}\|\overline{\mathbf{A}}\|_{F}$. 
Using the frontal slices of a tensor $\mathcal A$, we define the block circulant matrix  of ${\mathcal{A}}$ as
\begin{equation}
\operatorname{bcirc}(\mathcal{A}):=\left[\begin{array}{cccc}
\mathbf{A}^{(1)} & \mathbf{A}^{\left(n_{3}\right)} & \cdots & \mathbf{A}^{(2)} \\
\mathbf{A}^{(2)} & \mathbf{A}^{(1)} & \cdots & \mathbf{A}^{(3)} \\
\vdots & \vdots & \ddots & \vdots \\
\mathbf{A}^{\left(n_{3}\right)} & \mathbf{A}^{\left(n_{3}-1\right)} & \cdots & \mathbf{A}^{(1)}
\end{array}\right] \in \mathbb{R}^{n_{1} n_{3} \times n_{2} n_{3}}. 
\end{equation}
We define two operators: 
\begin{equation}
 \operatorname{unfold}(\mathcal{A})=\left[\begin{array}{c}
\mathbf{A}^{(1)} \\
\mathbf{A}^{(2)} \\
\vdots \\
\mathbf{A}^{\left(n_{3}\right)}
\end{array}\right] \quad \text{and}\quad \text {fold }(\operatorname{unfold}(\mathcal{A}))=\mathcal{A},   
\end{equation}
where unfold$(\cdot)$ maps ${\mathcal{A}}$ to a matrix of size ${n_{1} n_{3} \times n_{2}}$ and fold$(\cdot)$ is its inverse operator. 

\begin{definition}[t-product \cite{kilmer2011factorization}]
\label{def:tprod}
Let ${\mathcal{A} \in \mathbb{R}^{n_{1} \times l \times n_{3}}}$ and ${\mathcal{B} \in \mathbb{R}^{l \times n_{2} \times n_{3}}}$, then the t-product ${\mathcal{A} * \mathcal{B}}$ is defined by 
\begin{equation}
{ \mathcal{A} * \mathcal{B}=\operatorname{fold}(\operatorname{bcirc}(\mathcal{A}) \cdot \operatorname{unfold}(\mathcal{B})), }
\end{equation}
resulting a tensor of size ${n_{1} \times n_{2} \times n_{3}}$. Note that $\mathcal{A} * \mathcal{B}= \mathcal{Z} $ if and only if $\overline{\mathbf{A}}\,\overline{\mathbf{B}}=\overline{\mathbf{Z}}.$ 
\end{definition}
\begin{definition}[identity tensor \cite{kilmer2011factorization}]
The identity tensor ${\mathcal{I} \in \mathbb{R}^{n \times n \times n_{3}}}$ is the tensor with its first frontal slice being the ${n \times n}$ identity matrix and other frontal slices being all zeros. It is clear that ${\mathcal{A} * \mathcal{I}=\mathcal{A}}$ and ${\mathcal{I} * \mathcal{A}=\mathcal{A}}$ given the appropriate dimensions. 
\end{definition}
\begin{definition} [tensor conjugate transpose \cite{kilmer2011factorization}]
The conjugate transpose of a tensor ${\mathcal{A}} \in\mathbb{C}^{n_1 \times n_2 \times n_{3}}$ is a tensor $\mathcal{A}^\ast$ obtained by conjugate transposing each of the frontal slices and then reversing the order of transposed frontal slices 2 through $n_3$.
\end{definition}
\begin{definition} [orthogonal tensor  \cite{kilmer2011factorization}]
A tensor ${\mathcal{Q} \in}$ ${\mathbb{R}^{n \times n \times n_{3}}}$ is orthogonal if it satisfies ${\mathcal{Q}^{*}*\mathcal{Q}=\mathcal{Q}*\mathcal{Q}^{*}=\mathcal{I}}$.  
\end{definition}
\begin{definition} [$f$-diagonal tensor \cite{kilmer2011factorization}]
A tensor is called $f$-diagonal if each frontal slice is a diagonal matrix.
\end{definition}


\renewcommand{\thedefinition}{S.\arabic{definition}.}
\renewcommand{\thetheorem}{S.\arabic{theorem}}
\renewcommand{\thelemma}{S.\arabic{lemma}}
\section {Proof of Theorem 1}\label{sect:modelsPCA}

\subsection{Preliminary of definitions and lemmas}

Some assumptions on the low-rank tensor and the sparse tensor are required to avoid the degenerated situations in the TRPCA problem. We assume the signs of the nonzero entries of $\mathcal{E}_0$ are independent symmetric $\pm 1$ random variables, i.e., following the probability $\mathbb{P}(\operatorname{sgn}(\mathcal{E}_0)=1)=\mathbb{P}(\operatorname{sgn}(\mathcal{E}_0)=-1) =\gamma$. 
Let $\boldsymbol{\Omega}$ be the corrupted entries of $\mathcal{L}_{0}$ and $\boldsymbol{\Omega}^{c}$ be locations where data are available and clean, i.e., $\boldsymbol{\Omega}$ is the support set of $\mathcal{E}_0$. For convenience, we define $n_{(1)}:=\text{max}(n_1,n_2)$ and $n_{(2)}:=\text{min}(n_1,n_2)$.  
In addition to the notations introduced in the paper, we require the following definitions for their use in the proofs. 
\begin{definition}[tensor operator \cite{zhang2016exact}]
Suppose ${\mathcal{F}: \mathbb{R}^{n_{1} \times n_{2} \times n_{3}} \rightarrow}{\mathbb{R}^{n_{4} \times n_{2} \times n_{3}}}$ is a tensor operator that maps a tensor ${\mathcal{A}}$ of ${n_{1} \times n_{2} \times n_{3}}$ to a tensor ${\mathcal{B}}$ of ${n_{4} \times n_{2} \times n_{3}}$ , i.e., 
$$ \mathcal{B}=\mathcal{F}(\mathcal{A}). $$ 
\end{definition}
 A special case of tensor operators is through t-product, i.e., $\mathcal{B}=\mathcal{F}(\mathcal{A})=\mathcal{L}*\mathcal{A}$, where ${\mathcal{L}}$ is a tensor of ${n_{4} \times n_{1} \times n_{3}}$. 

\begin{definition}[tensor operator norm \cite{zhang2016exact}]
Suppose $\mathcal{F}$ is a tensor operator, then the operator norm of $\mathcal{F}$ is defined as, $ \|\mathcal{F}\|_{o p}=\sup _{\mathcal{X}:\|\mathcal{X}\|_{F} \leq 1}\|\mathcal{F}(\mathcal{X})\|_{F}$, which is consistent with the matrix case. 
\end{definition}
\begin{definition}[tensor spectral norm \cite{zhang2016exact}]
 The tensor spectral norm of ${\mathcal{X} \in \mathbb{R}^{n_{1} \times n_{2} \times n_{3}}}$, denoted as ${\|\mathcal{X}\|}$, is defined as 
 \begin{equation*}
\|\mathcal{X}\|=\max_{i j}\sigma_{ij},
 \end{equation*} 
 where $\sigma_{ij}$ is the $j$-th singular value of the $i$-th front slice of $\overline{\mathcal{X}}$.
\end{definition}
Note that tensor spectral norm is a special case of the operator norm if the tensor operator $f$ can be represented by t-product in such a way that $\mathcal{X}:\mathcal{F}(\mathcal{X})=\mathcal{L} * \mathcal{X}$, 
then $\|\mathcal{F}\|_{o p}=\|\mathcal{L}\|$. For simple notation, we express the operator norm $\|\cdot\|_{op}=\|\cdot\|$ as the spectral norm in this document.

Given the skinny t-SVD of $\mathcal{L}_0$, i.e.,  $\mathcal{L}_0=\mathcal{U}*\mathcal{S}*\mathcal{V}^*$, where $\mathcal{U}\in \mathbb{R}^{n_{1} \times r \times n_{3}},$ $\mathcal{S}\in \mathbb{R}^{r \times r \times n_{3}},$ and $\mathcal{V}\in \mathbb{R}^{n_{2} \times r \times n_{3}}$, we denote $\mathbf{T}$ by the set 
\begin{equation}
\label{eq:T}
\mathbf{T}=\{\mathcal{U}*\mathcal{Y}^{*}+\mathcal{W}*\mathcal{V}^{*} |\; \mathcal{Y}, \mathcal{W} \in \mathbb{R}^{n_1\times r\times n_3}\}.
\end{equation}
We then define ${\mathbf{T}^{\perp}}$ as the orthogonal complement 
of $\mathbf{T}$. The projections of an arbitrary tensor $\mathcal Z\in \mathbb{R}^{n_1\times n_2\times n_3}$ onto $\mathbf{T}$ and ${\mathbf{T}^{\perp}}$ 
are given by \cite{lu2019tensor} 
\begin{equation}
 \begin{array}{rl} 
 \mathcal{P}_{\mathbf{T}}(\mathcal{Z})&=\mathcal{U}  * \mathcal{U}^{*} * \mathcal{Z}+\mathcal{Z} * \mathcal{V} * \mathcal{V}^{*}-\mathcal{U} * \mathcal{U}^{*} * \mathcal{Z} * \mathcal{V} * \mathcal{V}^{*}, \\ 
 \mathcal{P}_{\mathbf{T}^{\perp}}(\mathcal{Z}) & =\mathcal{Z}-\mathcal{P}_{\mathbf{T}}(\mathcal{Z}) =\left(\mathcal{I}_{n_{1}}-\mathcal{U} * \mathcal{U}^{*}\right) * \mathcal{Z} *\left(\mathcal{I}_{n_{2}}-\mathcal{V} * \mathcal{V}^{*}\right),
 \end{array}  \label{equ:defptptt}
\end{equation}
 where ${\mathcal{I}_{n}}$ denotes the  identity tensor of ${n \times n \times n_{3}}$. It is straightforward that $\left\langle\mathcal{P}_{\mathbf{T}}(\mathcal{A}),\mathcal{P}_{\mathbf{T}^{\perp}}(\mathcal{B})\right\rangle=0$,   $\mathcal{P}_{\mathbf{T}}\mathcal{P}_{\mathbf{T}}(\mathcal{A})=\mathcal{P}_{\mathbf{T}}(\mathcal{A}),$ and $\mathcal{P}_{\mathbf{T}^{\perp}}\mathcal{P}_{\mathbf{T}^{\perp}}(\mathcal{A})=\mathcal{P}_{\mathbf{T}^{\perp}}(\mathcal{A})$ for any tensors $\mathcal{A},\ \mathcal{B} \in \mathbb{R}^{n_1\times n_2\times n_3}$.
To make the paper self-contained, we include the following lemmas; please refer to the respective references for proofs.

\begin{lemma}{\rm\cite[Lemma D.1]{lu2019tensor}}
For the  Bernoulli sign tensor ${\mathcal{M}}\in \mathbb{R}^{n_{1} \times n_{2} \times n_{3}}$ whose entries are distributed as
\begin{equation}
\mathcal{M}_{ijk}=\left\{\begin{aligned}
1 & \quad \text{ w.p. } \quad \gamma,\\
0 & \quad \text{ w.p. } \quad  1-2\gamma, \\
-1 & \quad \text{ w.p. } \quad \gamma,
\end{aligned}\right.
\end{equation}
there exists a function ${\varphi(\gamma)}$ with ${\lim _{\gamma \rightarrow 0^{+}} \varphi(\gamma)=0}$ such that 
$${ \|\mathcal{M}\| \leq \varphi(\gamma) \sqrt{n_{(1)} n_{3}}},$$ 
holds for any $\gamma\in[0,1/2]$ with high probability.   
\label{lemmabernoulli}
\end{lemma}

\begin{lemma} {\rm\cite[Lemma D.2]{lu2019tensor}}
Suppose the sampling follows the Bernoulli distribution with the probability $\rho$ being in $\boldsymbol{\Theta}$  i.e., ${\boldsymbol{\Theta} \sim \operatorname{Ber}(\rho)}$. Define $\mathcal{P}_{\boldsymbol{\Theta}}$ as a linear projection such that the entries in the set $\boldsymbol{\Theta}$ are known while the remaining entries are unknown. For any $0<\epsilon\leq 1$, there exists a constant $C_{0}>0$ such that for any $\rho \geq C_{0} \epsilon^{-2} \tfrac{\mu r \log \left(n_{(1)} n_{3}\right)}{n_{(2)}n_3}$ with tensor incoherence $\mu$ defined in {\rm\textbf{Definition 2.3}}, the following inequality, 
\begin{equation}
\left\|\rho^{-1} \mathcal{P}_{\mathbf{T}} \mathcal{P}_{\boldsymbol{\Theta}} \mathcal{P}_{\mathbf{T}}-\mathcal{P}_{\mathbf{T}}\right\| \leq \epsilon,    
\end{equation}
holds with high probability at least $1-2(n_{(1)}n_3)^{1-\tfrac{3}{16}C_0}$.\label{lemmaomega1}
\end{lemma}

\begin{lemma} {\rm\cite[Lemma D.4]{lu2019tensor}}
Given a tensor $\mathcal{Z} \in \mathbf{T}$  and ${\boldsymbol{\Theta} \sim \operatorname{Ber}(\rho)}$. For any $0<\epsilon\leq 1$, the following inequality 
\begin{equation}
\left\|\mathcal{Z}-\rho^{-1} \mathcal{P}_{\mathbf{T}} \mathcal{P}_{\boldsymbol{\Theta}} \mathcal{Z}\|_{\infty} \leq \epsilon\right\|\mathcal{Z}\|_{\infty},   \end{equation}
holds with high probability at least $1-2(n_{(1)}n_3)^{-\tfrac{3C_0}{16}}$,
provided that 
$${\rho \geq C_{0} \epsilon^{-2} \tfrac{\mu r \log \left(n_{(1)} n_{3}\right)}{n_{(2)}n_3}}$$ 
for some constant ${C_{0}>0}$. 
\label{lemmaomega2}
\end{lemma}

\begin{lemma} {\rm\cite[Lemma D.5]{lu2019tensor}}
Given any tensor ${\mathcal{Z} \in \mathbb{R}^{n_{1} \times n_{2} \times n_{3}}}$ and ${\boldsymbol{\Theta} \sim \operatorname{Ber}(\rho)}$. Then with high probability at least $1-2(n_{(1)}n_3)^{1-\tfrac{3C_0}{8}}$, 
\begin{equation}
\left\|\left(\mathcal{I}-\rho^{-1} \mathcal{P}_{\boldsymbol{\Theta}}\right) \mathcal{Z}\right\| \leq \sqrt{\tfrac{C_{0}n_{(1)} n_{3} \log \left(n_{(1)} n_{3}\right)}{\rho}}\|\mathcal{Z}\|_{\infty},
\end{equation}
provided that $\rho \geq C_{0} \tfrac{\log \left(n_{(1)} n_{3}\right)}{n_{(2)} n_{3}}$ for some numerical constant ${C_{0}>0}$. \label{lemmaomega3}
\end{lemma}

\begin{lemma} {\rm \cite[Theorem 1.6]{tropp2012user}}
Consider a finite sequence ${\left\{\mathbf{Z}_{k}\right\}}$ of independent, random ${n_{1} \times n_{2}}$ matrices that satisfy the assumption  ${\mathbb{E} \mathbf{Z}_{k}=\mathbf{0}}$ and  ${\left\|\mathbf{Z}_{k}\right\| \leq R}$ almost surely,  where $\|\cdot\|$ is the spectral norm. Let 
$$\sigma^{2}=\max \left\{\left\|\sum_{k} \mathbb{E}\left[\mathbf{Z}_{k} \mathbf{Z}_{k}^{*}\right]\right\|, \left\|\sum_{k} \mathbb{E}\left[\mathbf{Z}_{k}^{*} \mathbf{Z}_{k}\right]\right\|\right\}.$$  
Then for any ${0 \leq t \leq \tfrac{\sigma^{2}}{R}}$, we have 
\begin{equation}
 \mathbb{P}\left[\left\|\sum_{k} \mathbf{Z}_{k}\right\| \geq t\right]  \leq\left(n_{1}+n_{2}\right) \exp \left(-\tfrac{t^{2}}{2 \sigma^{2}+\tfrac{2}{3} R t}\right) \leq\left(n_{1}+n_{2}\right) \exp \left(-\tfrac{3 t^{2}}{8 \sigma^{2}}\right). 
 \end{equation} \label{inequality:Ber}
\end{lemma}

\subsection{Proof of recovery guarantee}

In this section, we provide the proof of Theorem 1. The idea is to drive conditions under which $(\mathcal L_0, \mathcal E_0)$ is a local minimizer of the TNF model (6), and then show that these conditions are met with overwhelming probability under the assumptions of Theorem 1.
These conditions are stated in terms of a dual variable $\mathcal{Y}$, as characterized in Theorem~\ref{thm:PtPo}.

\begin{theorem} 
Given a low rank tensor $\mathcal L_0\in\mathbb R^{n_1\times n_2\times n_3}$ with tubal rank $r$ and a sparse tensor $\mathcal E_0\in\mathbb R^{n_1\times n_2\times n_3}$ with its support denoted by $\Omega$ being a feasible solution of the TNF model (6), we further define the skinny t-SVD of $\mathcal{L}_0$, i.e.,  $\mathcal{L}_0=\mathcal{U}*\mathcal{S}*\mathcal{V}^*.$
For sufficiently large $n_1, n_2, n_3$ in the sense that $n_1n_2n_3^2>\sqrt{\tfrac{24}{1-2\gamma}}/\|\mathcal L_0\|_F$, $\tfrac{1}{4}\sqrt{\tfrac{n_{(1)}}{2\mu r}}>\sqrt{rn_3}$ and
$$\max\left(\tfrac{2\sqrt{6r}n_1n_2n_3^2}{n_1n_2n_3^2\sqrt{1-2\gamma}\|\mathcal{L}_{0}\|_F-2\sqrt{6}},\sqrt{\tfrac{\log (n_{(1)}n_3)}{n_{(1)}n_3^2}}\right)<\left(\tfrac{1}{4}\sqrt{\tfrac{n_{(1)}}{2\mu r}}-\sqrt{rn_3}\right)n_1n_2n_3^{3/2}, $$ 
if $\lambda$ satisfies
$$\max\left(\tfrac{2\sqrt{6r}n_1n_2n_3^2}{n_1n_2n_3^2\sqrt{1-2\gamma}\|\mathcal{L}_{0}\|_F-2\sqrt{6}},\sqrt{\tfrac{\log (n_{(1)}n_3)}{n_{(1)}n_3^2}}\right)<\lambda< \left(\tfrac{1}{4}\sqrt{\tfrac{n_{(1)}}{2\mu r}}-\sqrt{rn_3}\right)n_1n_2n_3^{3/2},$$
and there exists a tensor $\mathcal Y\in\mathbb R^{n_1\times n_2\times n_3}$ obeying 
\begin{equation}
\left\{\begin{array}{l}
\left\|\mathcal{P}_\mathbf{T}\left(\mathcal{Y} +\lambda \|\mathcal{L}_0\|_F \operatorname{sgn}\left(\mathcal{E}_0\right)-\mathcal{U}  * \mathcal{V} ^*\right)\right\|_F \leq \tfrac{\lambda}{n_1 n_2 n_3^2} \\
\left\|\mathcal{P}_{\mathbf{T}^{\perp}}\left(\mathcal{Y} +\lambda \|\mathcal{L}_0\|_F\operatorname{sgn}\left(\mathcal{E}_0\right)\right)\right\| \leq \tfrac{1}{2} \\
\left\|\mathcal{P}_{_{\boldsymbol{\Omega}^c}}(\mathcal{Y} )\right\|_{\infty} \leq \tfrac{\lambda}{2}\|\mathcal{L}_0\|_F \\
\mathcal{P}_{\boldsymbol{\Omega}}(\mathcal{Y} )=\mathcal{O},
\end{array}\right. \label{con:ptcondition1}
\end{equation}
where $\mathbf{T}$ is defined in \eqref{eq:T} and the projections $\mathcal{P}_{\mathbf{T}},\mathcal{P}_{\mathbf{T}^{\perp}}$ are defined in \eqref{equ:defptptt}, 
then $(\mathcal L_0, \mathcal E_0)$ is a local minimizer of the TNF model (6). In other words, there exists a constant $\overline t>0$ such that the following inequality,
\begin{equation}
\tfrac{\|\mathcal{L}_0\|_*}{\|\mathcal{L}_0\|_F}+\lambda\|\mathcal{E}_0\|_1\leq \tfrac{\|\mathcal{L}_0+\mathcal{Z}\|_*}{\|\mathcal{L}_0+\mathcal{Z}\|_F}+\lambda\|\mathcal{E}_0-\mathcal{Z}\|_1,\label{equ:LZE}
\end{equation}
 holds for any $\|\mathcal Z\|_F\leq \overline t$.
\label{thm:PtPo}
 \end{theorem}

\subsubsection{Proof of Theorem~\ref{thm:PtPo}.}
Given a tensor $\mathcal Z$ with $\|\mathcal{Z}\|_F=1$, we consider 
a function of a scalar variable $t$, defined by
\begin{equation*}
 F(t)= \tfrac{\|\mathcal{L}_0+t\mathcal{Z}\|_*}{\|\mathcal{L}_0+t\mathcal{Z}\|_F}+\lambda\|\mathcal{E}_0-t\mathcal{Z}\|_1.
\end{equation*}
If $(\mathcal L_0, \mathcal E_0)$ is a feasible solution to the TNF model (6), then so is $(\mathcal{L}_0+\mathcal{Z}, \mathcal{E}_0-\mathcal{Z})$.  
We study the lower bounds of $\|\mathcal{L}_0+t\mathcal{Z}\|_*$, $\|\mathcal{E}_0-t\mathcal{Z}\|_1$ in the following lemmas:

\begin{lemma}
\label{lem:lowerbound_2term}
For any tensor $\mathcal{L}_0,\mathcal{E}_0, \mathcal{Z} \in \mathbb{R}^{n_1 \times  n_2 \times n_3}$ and $t\geq 0$,  we denote $\boldsymbol{\Omega}$ as the support of $\mathcal{E}_0$  and have
\begin{equation}\label{ineq:lowerbound}
    \begin{aligned}
      \|\mathcal{E}_0-t\mathcal{Z}\|_1 & \geq   \|\mathcal{E}_0\|_1 +at \\
      \|\mathcal{L}_0+t\mathcal{Z}\|_*  & \geq \|\mathcal{L}_0\|_* + bt, 
    \end{aligned}
\end{equation}
where  
\begin{equation}
\label{eq:ab}
   \begin{split}
        a & := -\left\langle\operatorname{sgn}(\mathcal{E}_0), \mathcal{Z}\right\rangle+\|\mathcal{P}_{\boldsymbol{\Omega}^c}(\mathcal{Z})\|_1 \\ b & :=\left\langle\mathcal{U} * \mathcal{V}^{*},\mathcal{Z}\right\rangle+ \|\mathcal{P}_{\mathbf{T}^{\perp}}(\mathcal{Z})\|_*.
   \end{split}
\end{equation}
\end{lemma}
\begin{proof}
To prove the first inequality in \eqref{ineq:lowerbound}, we estimate
\begin{equation*}
\begin{aligned}
\|\mathcal{E}_0-t\mathcal{Z}\|_1 & =
\|\mathcal{P}_{\boldsymbol{\Omega}}(\mathcal{E}_0-t\mathcal{Z})\|_1 +\|\mathcal{P}_{\boldsymbol{\Omega}^c}(\mathcal{E}_0-t\mathcal{Z})\|_1  \\
& = \|\mathcal{E}_0-t\mathcal{P}_{\boldsymbol{\Omega}}(\mathcal{Z})\|_1 +\|\mathcal{P}_{\boldsymbol{\Omega}^c}(\mathcal{Z})\|_1 t\\
&\geq \|\mathcal{E}_0\|_1-\left\langle\operatorname{sgn}(\mathcal{E}_0), \mathcal{P}_{\boldsymbol{\Omega}}(\mathcal{Z})\right\rangle t  + \|\mathcal{P}_{\boldsymbol{\Omega}^c}(\mathcal{Z})\|_1 t\\
& = \|\mathcal{E}_0\|_1 -\left\langle\operatorname{sgn}(\mathcal{E}_0), \mathcal{Z}\right\rangle t + \|\mathcal{P}_{\boldsymbol{\Omega}^c}(\mathcal{Z})\|_1 t,
\end{aligned}
\end{equation*}
where we use $\|\mathcal{A}-\mathcal{B}\|_1 \geq \|\mathcal{A}\|_1-\langle\operatorname{sgn}(\mathcal{A}),\mathcal{B}\rangle$ for any $\mathcal{A}$ and $\mathcal{B}$ in the last inequality.

For the second inequality in \eqref{ineq:lowerbound}, we use two identities from \cite{zhang2016exact}: 
$\|\mathcal X\|_*=\langle \mathcal U*\mathcal V^*+\mathcal U_{\perp}*\mathcal V_{\perp}^*, \mathcal X\rangle$ 
and
$\|\mathcal{U}*\mathcal{V}^{*}+\mathcal{U}_{\perp}*\mathcal{V}_{\perp}^{*}\|=1$, where $\mathcal{U}_{\perp}, \mathcal{V}_{\perp}$ are from the skinny t-SVD of 
$\mathcal{P}_{\mathbf{T}^{\perp}}(\mathcal{Z})=\mathcal{U}_{\perp}*\mathcal{S}_{\perp}*\mathcal{V}_{\perp}^{*}$ with $\mathcal{U}_{\perp}\in \mathbb{R}^{n_{1} \times r \times n_{3}}, \mathcal{V}_{\perp}\in \mathbb{R}^{n_{2} \times r \times n_{3}},$ and f-diagonal tensor $\mathcal S_{\perp}\in  \mathbb{R}^{r \times r \times n_{3}}$.
The simple calculations give 
\begin{equation*}
\begin{aligned}
\|\mathcal{L}_0+t\mathcal{Z}\|_* 
& \geq \left\langle\mathcal{U}*\mathcal{V}^{*}+\mathcal{U}_{\perp}*\mathcal{V}_{\perp}^{*},\mathcal{L}_0+t\mathcal{Z}\right\rangle   \\
& = \left\langle\mathcal{U}*\mathcal{V}^{*}+\mathcal{U}_{\perp}*\mathcal{V}_{\perp}^{*},\mathcal{L}_0\right\rangle +  \left\langle\mathcal{U}*\mathcal{V}^{*},\mathcal{Z}\right\rangle t + \left\langle\mathcal{U}_{\perp}*\mathcal{V}_{\perp}^{*},\mathcal{Z}\right \rangle t  \\ 
&=\|\mathcal{L}_0\|_* +\left\langle\mathcal{U}_{\perp}*\mathcal{V}_{\perp}^{*},\mathcal{P}_{\mathbf{T}^{\perp}}(\mathcal{Z})\right\rangle t + \left\langle\mathcal{U}*\mathcal{V}^{*},\mathcal{Z}\right\rangle t \\
& = \|\mathcal{L}_0\|_* +\|\mathcal{P}_{\mathbf{T}^{\perp}}(\mathcal{Z})\|_* t + \left\langle\mathcal{U}*\mathcal{V}^{*},\mathcal{Z}\right\rangle t. 
\end{aligned}
\end{equation*}
\end{proof}
It follows from Lemma~\ref{lem:lowerbound_2term} that 
\begin{equation*}
\begin{aligned}
F(t) &\geq \tfrac{\|\mathcal{L}_0\|_*+bt}{\|\mathcal{L}_0+t\mathcal{Z}\|_F}+\lambda(\|\mathcal{E}_0\|_1+at). 
\end{aligned}
\end{equation*}
Denote $f(t):=\tfrac{\|\mathcal{L}_0\|_*+bt}{\|\mathcal{L}_0+t\mathcal{Z}\|_F}+\lambda(\|\mathcal{E}_0\|_1+at)$ and its derivative is written by
\begin{equation}
\begin{aligned}\label{eq:f-der}
f^{\prime}(t) &=\tfrac{b\|\mathcal{L}_0+t\mathcal{Z}\|_F^2-(\|\mathcal{L}_0\|_*+bt)\left\langle\mathcal{L}_0+t\mathcal{Z},\mathcal{Z}\right\rangle}{\|\mathcal{L}_0+t\mathcal{Z}\|_F^{3}}+\lambda a \\
& = \tfrac{\lambda a \|\mathcal{L}_0+t\mathcal{Z}\|_F^{3} +b\|\mathcal{L}_0+t\mathcal{Z}\|_F^2-(\|\mathcal{L}_0\|_*+bt)\left\langle\mathcal{L}_0+t\mathcal{Z},\mathcal{Z}\right\rangle}{\|\mathcal{L}_0+t\mathcal{Z}\|_F^{3}}.
\end{aligned}
\end{equation}
We denote the numerator of the right-hand side in \eqref{eq:f-der} by 
\begin{equation*}
g(t):= \lambda a \|\mathcal{L}_0+t\mathcal{Z}\|_F^{3} +b\|\mathcal{L}_0+t\mathcal{Z}\|_F^2-(\|\mathcal{L}_0\|_*+bt)\left\langle\mathcal{L}_0+t\mathcal{Z},\mathcal{Z}\right\rangle,
\end{equation*}
which is continuous for $t \geq 0$. 
 Note that 
\begin{equation}
\label{eq:g_0}
\begin{aligned}
g(0) &= \lambda a \|\mathcal{L}_0\|_F^{3}+b \|\mathcal{L}_{0}\|_F^2-\|\mathcal{L}_0\|_*\left\langle\mathcal{L}_0,\mathcal{Z}\right\rangle \\
 &= \lambda a\|\mathcal{L}_0\|_F^{3}+b\|\mathcal{L}_{0}\|_F^{2}-\|\mathcal{L}_0\|_*\left\langle\mathcal{L}_0,\mathcal{P}_{\mathbf{T}}(\mathcal{Z})\right\rangle\\
 &\geq \lambda a\|\mathcal{L}_0\|_F^{3}+b\|\mathcal{L}_{0}\|_F^{2}-\sqrt{r}\|\mathcal{L}_0\|_F^{2}\|\mathcal{P}_{\mathbf{T}}(\mathcal{Z})\|_F \\
 &\geq \|\mathcal{L}_0\|_F^{2}( \lambda a\|\mathcal{L}_0\|_F+b-\sqrt{r}\|\mathcal{P}_{\mathbf{T}}(\mathcal{Z})\|_F),
\end{aligned}
\end{equation}
where  the first equality is from $\langle\mathcal{L}_0,\mathcal{P}_{\mathbf{T}^{\perp}}(\mathcal{Z})\rangle=\mathcal{O}$  and the first inequality utilizes  $\|\mathcal{L}_0\|_*\leq \sqrt{r}\|\mathcal{L}_0\|_F$.  
We introduce Lemmas \ref{le:part_g}-\ref{le:bound_pt1} to obtain a lower bound of $g(0).$
\begin{lemma}
\label{le:part_g}
Given a tensor $\mathcal{L}_0$, and $a,b$ are defined from \eqref{eq:ab}, then we have 
    \begin{equation}
\label{eq:t4}
    \begin{split}
       &  \lambda a\|\mathcal{L}_0\|_F+b  \geq \tfrac{1}{2}\|\mathcal{P}_{\mathbf{T}^{\perp}} (\mathcal{Z})\|_* + \tfrac{\lambda\|\mathcal{L}_{0}\|_F}{2}\|\mathcal{P}_{\boldsymbol{\Omega}^{c}}(\mathcal{Z})\|_{1} - \tfrac{\lambda}{n_1n_2n_3^2}\|\mathcal{P}_{\mathbf{T}} (\mathcal{Z})\|_F. 
    \end{split}
\end{equation}
\end{lemma}

\begin{proof}
Inserting \eqref{eq:ab} into $\lambda a\|\mathcal{L}_0\|_F+b$, we have 
\begin{equation}
\label{eq:t00}
    \begin{split}
        &\lambda a\|\mathcal{L}_0\|_F+b
        \\ =&-\lambda \|\mathcal{L}_{0}\|_F\left\langle\operatorname{sgn}(\mathcal{E}_0), \mathcal{Z}\right\rangle+\lambda \|\mathcal{L}_{0}\|_F\|\mathcal{P}_{\boldsymbol{\Omega}^c}(\mathcal{Z})\|_1 + \left\langle\mathcal{U} * \mathcal{V}^{*},\mathcal{Z}\right\rangle 
 + \|\mathcal{P}_{\mathbf{T}^{\perp}} (\mathcal{Z})\|_* \\
 = &\langle-\lambda \|\mathcal{L}_{0}\|_F\operatorname{sgn}(\mathcal{E}_0)+\mathcal{U} * \mathcal{V}^{*}, \mathcal{Z}\rangle 
 +\lambda \|\mathcal{L}_{0}\|_F\|\mathcal{P}_{\boldsymbol{\Omega}^c}(\mathcal{Z})\|_1+ \|\mathcal{P}_{\mathbf{T}^{\perp}} (\mathcal{Z})\|_*.
    \end{split}
\end{equation}
 Introducing an arbitrary tensor $\mathcal{Y}$, we get: 
\begin{equation}
\label{eq:t0}
    \begin{split}
       & \langle-\lambda \|\mathcal{L}_{0}\|_F\operatorname{sgn}(\mathcal{E}_0)+\mathcal{U} * \mathcal{V}^{*}, \mathcal{Z}\rangle  = \left\langle\mathcal{Y},\mathcal{Z}\right\rangle-\left\langle\mathcal{Y}+\lambda \|\mathcal{L}_{0}\|_F \operatorname{sgn}(\mathcal{E}_0)-\mathcal{U} * \mathcal{V}^{*}, \mathcal{Z}\right\rangle \\
         = & \left\langle\mathcal{Y},\mathcal{Z}\right\rangle - \left\langle \mathcal{P}_{\mathbf{T}}( \mathcal{Y}+\lambda \|\mathcal{L}_{0}\|_F \operatorname{sgn}(\mathcal{E}_0)-\mathcal{U} * \mathcal{V}^{*}), \mathcal{P}_{\mathbf{T}}(\mathcal{Z})\right\rangle\\
       &  -\left\langle \mathcal{P}_{\mathbf{T}^{\perp}}( \mathcal{Y}+\lambda \|\mathcal{L}_{0}\|_F \operatorname{sgn}(\mathcal{E}_0)), \mathcal{P}_{\mathbf{T}^{\perp}}(\mathcal{Z})\right\rangle,
    \end{split}
\end{equation}
where we use $\left\langle\mathcal{P}_{\mathbf{T}}(\mathcal{A}),\mathcal{P}_{\mathbf{T}^{\perp}}(\mathcal{B})\right\rangle=0$.
For any $\mathcal{Y}$ satisfying the conditions in \eqref{con:ptcondition1}, we have 
\begin{equation}
\label{eq:t1}
    \begin{split}
          &\left\langle \mathcal{P}_{\mathbf{T}}( \mathcal{Y}+\lambda \|\mathcal{L}_{0}\|_F \operatorname{sgn}(\mathcal{E}_0)-\mathcal{U} * \mathcal{V}^{*}), \mathcal{P}_{\mathbf{T}}(\mathcal{Z})\right\rangle \\ \leq & \| \mathcal{P}_{\mathbf{T}}( \mathcal{Y}+\lambda \|\mathcal{L}_{0}\|_F \operatorname{sgn}(\mathcal{E}_0)-\mathcal{U} * \mathcal{V}^{*})\|_F \| \mathcal{P}_{\mathbf{T}}(\mathcal{Z})\|_F  \leq \tfrac{\lambda}{n_1n_2n_3^2}\|\mathcal{P}_{\mathbf{T}} (\mathcal{Z})\|_F,
    \end{split}
\end{equation}
and 
\begin{equation}
\label{eq:t2}
    \begin{split}
          &\left\langle \mathcal{P}_{\mathbf{T}^{\perp}}( \mathcal{Y}+\lambda \|\mathcal{L}_{0}\|_F \operatorname{sgn}(\mathcal{E}_0)), \mathcal{P}_{\mathbf{T}^{\perp}}(\mathcal{Z})\right\rangle \\\leq &  \| \mathcal{P}_{\mathbf{T}^{\perp}}( \mathcal{Y}+\lambda \|\mathcal{L}_{0}\|_F \operatorname{sgn}(\mathcal{E}_0))\| \| \mathcal{P}_{\mathbf{T}^\perp}(\mathcal{Z})\|_\ast  \leq \frac{1}{2}\|\mathcal{P}_{\mathbf{T}^{\perp}} (\mathcal{Z})\|_\ast.
    \end{split}
\end{equation} 
In addition,
\begin{equation}
\label{eq:t3}
\begin{split}
\left\langle\mathcal{Y},\mathcal{Z}\right\rangle 
& \geq -\left|\left\langle\mathcal{Y},\mathcal{Z}\right\rangle\right|
 =-\left|\left\langle\mathcal{P}_{\boldsymbol{\Omega}}\mathcal{Y}+\mathcal{P}_{\boldsymbol{\Omega}^{c}}\mathcal{Y},\mathcal{P}_{\boldsymbol{\Omega}}\mathcal{Z}+\mathcal{P}_{\boldsymbol{\Omega}^{c}}\mathcal{Z}\right\rangle\right| 
=-\left|\left\langle\mathcal{P}_{\boldsymbol{\Omega}^c}\mathcal{Y},\mathcal{P}_{\boldsymbol{\Omega}^c}\mathcal{Z}\right\rangle\right|\\ 
& \geq -\|\mathcal{P}_{\boldsymbol{\Omega}^c}(\mathcal{Y})\|_\infty \|\mathcal{P}_{\boldsymbol{\Omega}^c}(\mathcal {Z})\|_1 \geq -\tfrac{\lambda}{2}\|\mathcal{L}_0\|_F \|\mathcal{P}_{\boldsymbol{\Omega}^c}(\mathcal {Z})\|_1,
\end{split}
\end{equation}
where we use $\mathcal{P}_{\boldsymbol{\Omega}}(\mathcal{Y})=\mathcal{O}$ by \eqref{con:ptcondition1} and $\langle\mathcal{P}_{\boldsymbol{\Omega}^c}\mathcal{Y},\mathcal{P}_{\boldsymbol{\Omega}}\mathcal{Z}\rangle=\mathcal{O}$. 
Plugging \eqref{eq:t0}-\eqref{eq:t3} into \eqref{eq:t00}, we complete the proof.
\end{proof}

\begin{lemma}
\label{le:boundpts}
For any $\mathcal{Z}$ with $\|\mathcal{Z}\|_F=1$ and a constant  $\xi\in \left(0,\sqrt{\tfrac{2\mu r}{n_{(1)}}}\right)$,  then the inequality
\begin{equation*}
    \left\|\tfrac{1}{1-2\gamma}\mathcal{P}_{\mathbf{T}}\mathcal{P}_{\boldsymbol{\Omega}^c}\mathcal{P}_{\mathbf{T}^{\perp}} (\mathcal{Z})\right\|\leq\xi,
\end{equation*}
holds with probability at least $1-2(n_{(1)}n_3)^{1-\tfrac{3}{16}C_0}$
for some numerical constant $C_0>0$.
\end{lemma}
\begin{proof}
For any tensor $\mathcal{Z}$, we can write
\begin{equation*}
\tfrac{1}{1-2\gamma}\mathcal{P}_{\mathbf{T}}\mathcal{P}_{\boldsymbol{\Omega}^c}\mathcal{P}_{\mathbf{T}^{\perp}} (\mathcal{Z})=\sum_{ijk}\tfrac{1}{1-2\gamma}\delta_{ijk}\langle\mathcal{Z}, \mathcal{P}_{\mathbf{T}^{\perp}}(\mathbf{e}_{ijk})\rangle\mathcal{P}_{\mathbf{T}}(\mathbf{e}_{ijk}):=\sum_{ijk}\mathcal{Q}_{ijk}(\mathcal Z),
\end{equation*}
where $\delta_{ijk}=I_{(i,j,k)\in \boldsymbol{\Omega}^c }$ for the indicator function $I_{(\cdot)}$. It is straightforward that $\mathcal{Q}_{ijk}: \mathbb{R}^{n_1\times n_2\times n_3}\rightarrow \mathbb{R}^{n_1\times n_2\times n_3}$ is a self-adjoint random operator. 
Since we have
$$\mathbb{E}(\tfrac{1}{1-2\gamma}\mathcal{P}_{\mathbf{T}}\mathcal{P}_{\boldsymbol{\Omega}^c}\mathcal{P}_{\mathbf{T}^{\perp}} (\mathcal{Z}))=\tfrac{1}{1-2\gamma}\mathcal{P}_{\mathbf{T}}\mathbb{E}(\mathcal{P}_{\boldsymbol{\Omega}^c})\mathcal{P}_{\mathbf{T}^{\perp}} (\mathcal{Z})=\mathcal{P}_{\mathbf{T}}\mathcal{P}_{\mathbf{T}^{\perp}}(\mathcal{Z})=\mathcal{O},$$
then the following equality
$$\mathbb{E}[\mathcal{Q}_{ijk}(\mathcal{Z})]=\mathbb{E}[\tfrac{1}{1-2\gamma}\delta_{ijk}\langle\mathcal{Z}, \mathcal{P}_{\mathbf{T}^{\perp}}(\mathbf{e}_{ijk})\rangle\mathcal{P}_{\mathbf{T}}(\mathbf{e}_{ijk})]=0,$$
holds for any $i,j,k$-th element.
Define the matrix operator $\overline{\boldsymbol{Q}}_{i j k}: \mathbb{B} \rightarrow \mathbb{B}$, where ${\mathbb{B}=\left\{\overline{\boldsymbol{B}}: \mathcal{B} \in \mathbb{R}^{n_1 \times n_2 \times n_{3}}\right\}}$ denotes a set of block diagonal matrices $\overline{\boldsymbol{B}}$ with the blocks as the frontal slices of ${\overline{\mathcal{B}}}$, then we get
\begin{equation*}
\overline{\boldsymbol{Q}}_{i j k}(\overline{\mathbf{Z}})=\tfrac{1}{1-2\gamma}\delta_{ijk}\langle\mathcal{Z}, \mathcal{P}_{\mathbf{T}^{\perp}}(\mathbf{e}_{ijk})\rangle\overline{\mathbf{P}_{ijk}},
\end{equation*}
where $\overline{\mathbf{P}_{ijk}}=\text{bdiag}\left(\overline{\mathcal{P}_\mathbf{T}\left(\mathbf{e}_{ijk}\right)}\right) \in \mathbb{R}^{n_1n_3\times n_2n_3} $ for a given coordinate $(i,j,k)$.
We can estimate an upper bound for
\begin{equation}
\begin{aligned}
\left\|\overline{\boldsymbol{Q}}_{i j k}\right\| & =\sup _{\|\overline{\mathbf{Z}}\|_F=1}\left\|\overline{\boldsymbol{Q}}_{i j k}(\overline{\mathbf{Z}})\right\|_F \\
& \leq \sup _{\|\overline{\mathbf{Z}}\|_F=1} \tfrac{1}{1-2\gamma}\left\|\mathcal{P}_{\mathbf{T}^{\perp}}\left(\mathbf{e}_{i j k}\right)\right\|_F\left\|\overline{\mathbf{P}_{ijk}}\right\|_F\|\mathcal{Z}\|_F \\
& =\sup _{\|\overline{\mathbf{Z}}\|_F=1} \tfrac{1}{1-2\gamma}\left\|\mathcal { P }_{\mathbf{T}^{\perp}}\left(\mathbf{e}_{i j k}\right)\right\|_F\left\|\mathcal { P }_{\mathbf{T}}\left(\mathbf{e}_{i j k}\right)\right\|_F\|\overline{\mathbf{Z}}\|_F \\
& \leq \tfrac{1}{1-2\gamma}\sqrt{\tfrac{2\mu r}{n_{(1)}n_3}},
\end{aligned} \label{in:Qspectral}
\end{equation}
where the last inequality is from $\|\mathcal{P}_{\mathbf{T}}(\mathbf{e}_{ijk})\|_F^2 \leq  \frac{2\mu r}{n_{(1)}n_3}$ \cite{lu2019tensor}. On the other hand, we compute
\begin{equation*}
\overline{\boldsymbol{Q}}_{i j k}(\overline{\mathbf{Z}})^**\overline{\boldsymbol{Q}}_{i j k}(\overline{\mathbf{Z}})= (\tfrac{1}{1-2\gamma}\delta_{ijk})^2\langle\mathcal{Z}, \mathcal{P}_{\mathbf{T}^{\perp}}(\mathbf{e}_{ijk})\rangle^2\overline{\mathbf{P}_{ijk}}^**\overline{\mathbf{P}_{ijk}}.
\end{equation*}
Using $\|\mathbf{A}^**\mathbf{A}\|_F=\|\mathbf{A}*\mathbf{A}^*\|_F$ for any matrix  $\mathbf A$,  we have
\begin{equation*}
\begin{aligned}
&\left\|\sum_{ijk}\mathbb{E}\left[\overline{\boldsymbol{Q}}_{i j k}(\overline{\mathbf{Z}})^**\overline{\boldsymbol{Q}}_{i j k}(\overline{\mathbf{Z}})\right]\right\|_F =\left\|\sum_{ijk}\mathbb{E}\left[\overline{\boldsymbol{Q}}_{i j k}(\overline{\mathbf{Z}})*\overline{\boldsymbol{Q}}_{i j k}(\overline{\mathbf{Z}})^*\right]\right\|_F \\
= &\tfrac{1}{(1-2\gamma)^2} \left\|\mathbb{E}(\delta_{ijk})^2\sum_{ijk} \langle\mathcal{Z}, \mathcal{P}_{\mathbf{T}^{\perp}}(\mathbf{e}_{ijk})\rangle^2\overline{\mathbf{P}_{ijk}}^**\overline{\mathbf{P}_{ijk}}\right\|_F \\
= & \tfrac{1}{1-2\gamma} \left\|\sum_{ijk} \langle\mathcal{Z}, \mathcal{P}_{\mathbf{T}^{\perp}}(\mathbf{e}_{ijk})\rangle^2\overline{\mathbf{P}_{ijk}}^**\overline{\mathbf{P}_{ijk}}\right\|_F \\
\leq & \tfrac{n_3}{1-2\gamma}\left\|\mathcal { P }_{\mathbf{T}}\left(\mathbf{e}_{i j k}\right)\right\|_F^2\left\| \sum_{ijk} \langle \mathcal{P}_{\mathbf{T}^{\perp}}(\mathcal{Z}), \mathbf{e}_{ijk}\rangle^2 
  \right\|_F 
\leq  \tfrac{n_3}{1-2\gamma}\left\|\mathcal { P }_{\mathbf{T}}\left(\mathbf{e}_{i j k}\right)\right\|_F^2\left\|\mathcal { P }_{\mathbf{T}^{\perp}}\left(\mathcal{Z}\right)\right\|_F^2 \\
\leq & \tfrac{n_3}{1-2\gamma}\left\|\mathcal { P }_{\mathbf{T}}\left(\mathbf{e}_{i j k}\right)\right\|_F^2\left\|\mathcal{Z}\right\|_F^2
= \tfrac{\sqrt{n_3}}{1-2\gamma}\left\|\mathcal { P }_{\mathbf{T}}\left(\mathbf{e}_{i j k}\right)\right\|_F^2\left\|\overline{\mathbf{Z}}\right\|_F \leq \tfrac{2\mu r}{(1-2\gamma)n_{(1)}\sqrt{n_3}},
\end{aligned}
\end{equation*}
which implies 
\begin{equation}
\begin{aligned}
&\max\left\{\left\|\sum_{ijk}\mathbb{E}\left[\overline{\boldsymbol{Q}}_{i j k}(\overline{\mathbf{Z}})^**\overline{\boldsymbol{Q}}_{i j k}(\overline{\mathbf{Z}})\right]\right\|, \left\|\sum_{ijk}\mathbb{E}\left[\overline{\boldsymbol{Q}}_{i j k}(\overline{\mathbf{Z}})*\overline{\boldsymbol{Q}}_{i j k}(\overline{\mathbf{Z}})^*\right]\right\|\right\} \\
\leq & \tfrac{2\mu r}{(1-2\gamma)n_{(1)}\sqrt{n_3}}.       
\end{aligned}\label{in:maxEQQ} 
\end{equation}

Motivated by \eqref{in:Qspectral}, \eqref{in:maxEQQ}, and Lemma~\ref{inequality:Ber}, we choose 
$$\xi\leq  \tfrac{2\mu r}{(1-2\gamma)n_{(1)}\sqrt{n_3}} (\tfrac{1}{1-2\gamma}\sqrt{\tfrac{2\mu r}{n_{(1)}n_3}})^{-1}= \sqrt{\tfrac{2\mu r}{n_{(1)}}}$$ to get
\begin{equation*}
\begin{aligned}
 &\mathbb{P}\left[\left\|\tfrac{1}{1-2\gamma}\mathcal{P}_{\mathbf{T}}\mathcal{P}_{\boldsymbol{\Omega}^c}\mathcal{P}_{\mathbf{T}^{\perp}} (\mathcal{Z})\right\|>\xi\right]=\mathbb{P}\left[\left\|\sum_{ijk}\mathcal{Q}_{ijk}(\mathcal{Z})\right\|>\xi\right] \\
 =& \mathbb{P}\left[\left\|\sum_{ijk}\overline{\mathbf{Q}}_{ijk}(\overline{\mathbf{Z}})\right\|>\xi\right] \leq (n_1+n_2)n_3\exp(-\tfrac{3}{8}\cdot\tfrac{\xi^2}{2\mu r/((1-2\gamma)n_{(1)}\sqrt{n_3})}) \\
 =&(n_1+n_2)n_3\exp\left(-\tfrac{3\xi^2(1-2\gamma)n_{(1)}\sqrt{n_3}}{16\mu r}\right).\\
 \leq&2(n_{(1)}n_3)^{1-\tfrac{3}{16}C_0},
\end{aligned}
\end{equation*}
where $C_0\leq \tfrac{\xi^2(1-2\gamma)n_{(1)}n_{3}^{1/2}}{nr\log(n_{(1)}n_3)}.$
In other words, if $0<\xi\leq \sqrt{\tfrac{2\mu r}{n_{(1)}}}$, the following estimate $$\left\|\tfrac{1}{1-2\gamma}\mathcal{P}_{\mathbf{T}}\mathcal{P}_{\boldsymbol{\Omega}^c}\mathcal{P}_{\mathbf{T}^{\perp}} (\mathcal{Z})\right\|\leq\xi,$$ holds with probability at least $1-2(n_{(1)}n_3)^{1-\tfrac{3}{16}C_0}$
for some numerical constant $C_0>0$.
\end{proof}

\begin{lemma}
\label{le:bound_pt1}
 For any $\mathcal{Z}$ with $\|\mathcal{Z}\|_F=1$ and $\xi\in \left(0,\sqrt{\tfrac{2\mu r}{n_{(1)}}}\right)$,  we have 
\begin{equation*}
  \|\mathcal{P}_{\mathbf{T}}(\mathcal{Z})\|_F  
  \leq \sqrt{\tfrac{6}{1-2\gamma}}\|\mathcal{P}_{\boldsymbol{\Omega}^{c}}(\mathcal{Z})\|_1+2\xi\|\mathcal{P}_{\mathbf{T}}(\mathcal{Z})\|_F, 
\end{equation*}
holds with high probability, provided $\gamma \leq \tfrac{1}{2}- 2C_0\frac{\mu r \log \left(n_{(1)} n_{3}\right)}{n_{(2)}n_3}$ for a constant $C_0$.
\end{lemma}
\begin{proof}
Since ${\boldsymbol{\Omega}^{c} \sim \operatorname{Ber}(1-2\gamma)}$, we get from Lemma~\ref{lemmaomega1} by setting  $\epsilon=\frac{1}{2}$ and $\rho=1-2\gamma$ that 
\begin{equation}
\left\|\tfrac{1}{1-2\gamma}\mathcal{P}_{\mathbf{T}}\mathcal{P}_{\boldsymbol{\Omega}^c}\mathcal{P}_{\mathbf{T}}-\mathcal{P}_{\mathbf{T}}\right\| \leq \frac{1}{2}, \label{ine:pt1d2}
\end{equation}
holds with high probability provided $\gamma \leq \tfrac{1}{2}- 2C_0\frac{\mu r \log \left(n_{(1)} n_{3}\right)}{n_{(2)}n_3}$ for a constant $C_0$.
Then we have
\begin{equation*}
    \begin{split}
       \tfrac{1}{2} \|\mathcal{P}_{\mathbf{T}}(\mathcal{Z})\|_F^2 & \geq  \|\mathcal{P}_{\mathbf{T}}(\mathcal{Z})\|_F^2 \left\|\tfrac{1}{1-2\gamma}\mathcal{P}_{\mathbf{T}}\mathcal{P}_{\boldsymbol{\Omega^c}}\mathcal{P}_{\mathbf{T}}-\mathcal{P}_{\mathbf{T}}\right\|  \\ 
       & \geq \|\mathcal{P}_{\mathbf{T}}(\mathcal{Z})\|_F \left\|\left(\tfrac{1}{1-2\gamma}\mathcal{P}_{\mathbf{T}}\mathcal{P}_{\boldsymbol{\Omega^c}}\mathcal{P}_{\mathbf{T}}-\mathcal{P}_{\mathbf{T}}\right)\mathcal{P}_{\mathbf{T}}(\mathcal{Z})\right\|_F \\
       & \geq  \left|\left\langle \mathcal{P}_{\mathbf{T}}(\mathcal{Z}), (\tfrac{1}{1-2\gamma}\mathcal{P}_{\mathbf{T}}\mathcal{P}_{\boldsymbol{\Omega^c}}\mathcal{P}_{\mathbf{T}}-\mathcal{P}_{\mathbf{T}})\mathcal{P}_{\mathbf{T}}(\mathcal{Z}) \right\rangle\right|,
    \end{split}
\end{equation*}
which directly leads to 
\begin{equation}
\label{eq: pt1}
    \left\langle \mathcal{P}_{\mathbf{T}}(\mathcal{Z}), \left(\tfrac{1}{1-2\gamma}\mathcal{P}_{\mathbf{T}}\mathcal{P}_{\boldsymbol{\Omega}^c}\mathcal{P}_{\mathbf{T}}-\mathcal{P}_{\mathbf{T}}\right)\mathcal{P}_{\mathbf{T}}(\mathcal{Z}) \right\rangle \geq -\frac{1}{2}\|\mathcal{P}_{\mathbf{T}}(\mathcal{Z})\|_F^2. 
\end{equation}
On the other hand, \begin{equation}
\label{eq: pt2}
\begin{aligned}
&\left\langle \mathcal{P}_{\mathbf{T}}(\mathcal{Z}), \left(\tfrac{1}{1-2\gamma}\mathcal{P}_{\mathbf{T}}\mathcal{P}_{\boldsymbol{\Omega}^c}\mathcal{P}_{\mathbf{T}}-\mathcal{P}_{\mathbf{T}}\right)\mathcal{P}_{\mathbf{T}}(\mathcal{Z}) \right\rangle \\
&= \left\langle \mathcal{P}_{\mathbf{T}}(\mathcal{Z}), \tfrac{1}{1-2\gamma}\mathcal{P}_{\mathbf{T}}\mathcal{P}_{\boldsymbol{\Omega}^c}\mathcal{P}_{\mathbf{T}}(\mathcal{Z})\right\rangle -\|\mathcal{P}_{\mathbf{T}}(\mathcal{Z})\|_F^2 \\
& \leq \|\mathcal{P}_{\mathbf{T}}(\mathcal{Z})\|_F \left\|\tfrac{1}{1-2\gamma}\mathcal{P}_{\mathbf{T}}\mathcal{P}_{\boldsymbol{\Omega}^c}\mathcal{P}_{\mathbf{T}}(\mathcal{Z})\right\|_F-\|\mathcal{P}_{\mathbf{T}}(\mathcal{Z})\|_F^2.
\end{aligned} 
\end{equation}
Combining the inequalities \eqref{eq: pt1} and \eqref{eq: pt2} yields
$$-\tfrac{1}{2}\|\mathcal{P}_{\mathbf{T}}(\mathcal{Z})\|_F^2\leq \|\mathcal{P}_{\mathbf{T}}(\mathcal{Z})\|_F \left\|\tfrac{1}{1-2\gamma}\mathcal{P}_{\mathbf{T}}\mathcal{P}_{\boldsymbol{\Omega}^c}\mathcal{P}_{\mathbf{T}}(\mathcal{Z})\right\|_F-\|\mathcal{P}_{\mathbf{T}}(\mathcal{Z})\|_F^2,$$
which implies that
\begin{equation}
\label{eq:term1}
    \begin{split}
     \|\mathcal{P}_{\mathbf{T}}(\mathcal{Z})\|_F & \leq 2\left\|\tfrac{1}{1-2\gamma}\mathcal{P}_{\mathbf{T}}\mathcal{P}_{\boldsymbol{\Omega}^c}\mathcal{P}_{\mathbf{T}}(\mathcal{Z})\right\|_F \\
    &\leq 2\left(\left\|\tfrac{1}{1-2\gamma}\mathcal{P}_{\mathbf{T}}\mathcal{P}_{\boldsymbol{\Omega}^c}(\mathcal{Z})\right\|_F+\left\|\tfrac{1}{1-2\gamma}\mathcal{P}_{\mathbf{T}}\mathcal{P}_{\boldsymbol{\Omega}^c}\mathcal{P}_{\mathbf{T}^{\perp}}(\mathcal{Z})\right\|_F\right)
    \end{split}
\end{equation}
For the first term in \eqref{eq:term1}, we use the identity $\|\mathcal{P}_{\mathbf{T}}\mathcal{P}_{\boldsymbol{\Omega}^c}\|^2=\|\mathcal{P}_{\mathbf{T}}\mathcal{P}_{\boldsymbol{\Omega}^c}\mathcal{P}_{\mathbf{T}}\|$ in \cite{lu2019tensor} to obtain
\begin{equation}
    \begin{split}
        \left\|\tfrac{1}{\sqrt{1-2\gamma}}\mathcal{P}_{\mathbf{T}}\mathcal{P}_{\boldsymbol{\Omega}^c}\right\|^2 & = \left\|\tfrac{1}{1-2\gamma}\mathcal{P}_{\mathbf{T}}\mathcal{P}_{\boldsymbol{\Omega}^c}\mathcal{P}_{\mathbf{T}}\right\|\\
        & = \left\|\tfrac{1}{1-2\gamma}\mathcal{P}_{\mathbf{T}}\mathcal{P}_{\boldsymbol{\Omega}^c}\mathcal{P}_{\mathbf{T}}\right\| -\left\|\mathcal{P}_{\mathbf{T}}\right\| + \left\|\mathcal{P}_{\mathbf{T}}\right\| \\
        & \leq 
        \left\|\tfrac{1}{1-2\gamma}\mathcal{P}_{\mathbf{T}}\mathcal{P}_{\boldsymbol{\Omega}^c}\mathcal{P}_{\mathbf{T}}-\mathcal{P}_{\mathbf{T}}\right\| + \left\|\mathcal{P}_{\mathbf{T}}\right\|
        \leq \tfrac{3}{2},
    \end{split}
    \label{ineq:32}
\end{equation}
where we further use $ \|\mathcal{P}_{\mathbf{T}}\| \leq 1$ and \eqref{ine:pt1d2}. 
Plugging \eqref{ineq:32} into \eqref{eq:term1}, we obtain:  
\begin{equation}
\begin{aligned}
 \tfrac{2}{1-2\gamma}\|\mathcal{P}_{\mathbf{T}} \mathcal{P}_{\boldsymbol{\Omega}^{c}}(\mathcal{Z})\|_F &= \tfrac{2}{1-2\gamma}\|\mathcal{P}_{\mathbf{T}} \mathcal{P}_{\boldsymbol{\Omega}^{c}}\mathcal{P}_{\boldsymbol{\Omega}^{c}}(\mathcal{Z})\|_F
\leq \tfrac{2}{1-2\gamma} \|\mathcal{P}_{\mathbf{T}} \mathcal{P}_{\boldsymbol{\Omega}^{c}}\| \|\mathcal{P}_{\boldsymbol{\Omega}^{c}}(\mathcal{Z})\|_F\\
&\leq \sqrt{\tfrac{6}{1-2\gamma}}\|\mathcal{P}_{\boldsymbol{\Omega}^{c}}(\mathcal{Z})\|_F
 \leq \sqrt{\tfrac{6}{1-2\gamma}}\|\mathcal{P}_{\boldsymbol{\Omega}^{c}}(\mathcal{Z})\|_1,  
\end{aligned} \label{firstpart}
\end{equation}
where the last inequality is due to $\|\mathcal{P}_{\boldsymbol{\Omega}^{c}}(\mathcal{Z})\|_F \leq \|\mathcal{P}_{\boldsymbol{\Omega}^{c}}(\mathcal{Z})\|_1 $.
For the second term in \eqref{eq:term1}, we use Lemma~\ref{le:boundpts} to get
\begin{equation}
\begin{aligned}
&\left\|\tfrac{1}{1-2\gamma}\mathcal{P}_{\mathbf{T}}\mathcal{P}_{\boldsymbol{\Omega}^c}\mathcal{P}_{\mathbf{T}^{\perp}} (\mathcal{Z})\right\|_F \\
\leq &\left\|\tfrac{1}{1-2\gamma}\mathcal{P}_{\mathbf{T}}\mathcal{P}_{\boldsymbol{\Omega}^c}\mathcal{P}_{\mathbf{T}^{\perp}} (\mathcal{Z})\right\| \left\|\mathcal{P}_{\mathbf{T}^{\perp}} (\mathcal{Z})\right\|_F\leq 2\xi \left\|\mathcal{P}_{\mathbf{T}^{\perp}} (\mathcal{Z})\right\|_F.     
\end{aligned}
 \label{secondpart}
\end{equation}
 Combining \eqref{firstpart} and \eqref{secondpart}, we get the following inequality 
 $$ \|\mathcal{P}_{\mathbf{T}}(\mathcal{Z})\|_F  \leq \sqrt{\tfrac{6}{1-2\gamma}}\|\mathcal{P}_{\boldsymbol{\Omega}^{c}}(\mathcal{Z})\|_1+2\xi\|\mathcal{P}_{\mathbf{T}^{\perp}}(\mathcal{Z})\|_F,$$
holds with high probability.
\end{proof}

By incorporating the lower bounds obtained by Lemma~\ref{le:part_g} and Lemma~\ref{le:bound_pt1} into \eqref{eq:g_0}, we get 
\begin{equation*}
    \begin{split}
 g(0)  &= \lambda a \|\mathcal{L}_0\|_F^{3} +b\|\mathcal{L}_0\|_F^2-\|\mathcal{L}_0\|_*\left\langle\mathcal{L}_0,\mathcal{Z}\right\rangle \\
&\geq  \|\mathcal{L}_0\|_F^{2} \left(\tfrac{1}{2}\|\mathcal{P}_{\mathbf{T}^{\perp}} (\mathcal{Z})\|_* + \tfrac{\lambda\|\mathcal{L}_{0}\|_F}{2}\|\mathcal{P}_{\boldsymbol{\Omega}^{c}}(\mathcal{Z})\|_{1}\right) \\
&\quad - \|\mathcal{L}_0\|_F^{2} \left(\tfrac{\lambda}{n_1n_2n_3^2}\|\mathcal{P}_{\mathbf{T}} (\mathcal{Z})\|_F-\sqrt{r}\|\mathcal{P}_{\mathbf{T}}(\mathcal{Z})\|_F\right)\\
&\geq  \|\mathcal{L}_0\|_F^{2} \left(\tfrac{1}{2}\|\mathcal{P}_{\mathbf{T}^{\perp}} (\mathcal{Z})\|_* + \tfrac{\lambda\|\mathcal{L}_{0}\|_F}{2}\|\mathcal{P}_{\boldsymbol{\Omega}^{c}}(\mathcal{Z})\|_{1} \right)\\
&\quad-\|\mathcal{L}_{0}\|_F^2\left(\tfrac{\lambda}{n_1n_2n_3^2} + \sqrt{r}\right)\left(\sqrt{\tfrac{6}{1-2\gamma}} \|\mathcal{P}_{\boldsymbol{\Omega}^{c}}(\mathcal{Z})\|_1+2\xi\|\mathcal{P}_{\mathbf{T}^{\perp}}(\mathcal{Z})\|_F\right)\\
&\geq \|\mathcal{L}_{0}\|_F^2\left[\left(\tfrac{1}{2} - \tfrac{2\lambda\xi}{n_1n_2n_3^{3/2}}-2\xi\sqrt{rn_3}\right)\|\mathcal{P}_{\mathbf{T}^{\perp}} (\mathcal{Z})\|_* \right]\\
&\quad + \|\mathcal{L}_{0}\|_F^2\left[\left(\tfrac{\lambda\|\mathcal{L}_{0}\|_F}{2}-\tfrac{\lambda}{n_1n_2n_3^2}\sqrt{\tfrac{6\lambda^2}{1-2\gamma}}-\sqrt{\tfrac{6r}{1-2\gamma}}\right)\|\mathcal{P}_{\boldsymbol{\Omega}^{c}}(\mathcal{Z})\|_{1}\right]. 
\end{split}
\end{equation*}
If the coefficients in front of $\|\mathcal{P}_{\mathbf{T}^{\perp}}(\mathcal{Z})\|_*$ and $\|\mathcal{P}_{\boldsymbol{\Omega}^{c}}(\mathcal{Z})\|_1$ are positive, then $g(0)>0$. 
Specifically for $\|\mathcal{P}_{\mathbf{T}^{\perp}} (\mathcal{Z})\|_*$, we have
$$ \lambda< \left(\tfrac{1}{2}-2\xi\sqrt{rn_3}\right)\left(\tfrac{2\xi}{n_1n_2n_3^{3/2}}\right)^{-1}=\left(\tfrac{1}{4\xi}-\sqrt{rn_3}\right)n_1n_2n_3^{3/2},$$ 
which can be relaxed to 
$$\lambda<\left(\tfrac{1}{4}\sqrt{\tfrac{n_{(1)}}{2\mu r}}-\sqrt{rn_3}\right)n_1n_2n_3^{3/2},$$
by $\xi \leq \sqrt{\tfrac{2\mu r}{n_{(1)}}}$.

As for $\|\mathcal{P}_{\boldsymbol{\Omega}^{c}}(\mathcal{Z})\|_1$, we have
\begin{equation}
\lambda >\sqrt{\tfrac{6r}{1-2\gamma}}\left(\tfrac{\|\mathcal{L}_{0}\|_F}{2}-\tfrac{\sqrt{6}}{n_1n_2n_3^2\sqrt{1-2\gamma}} \right)^{-1}=\tfrac{2\sqrt{6r}n_1n_2n_3^2}{n_1n_2n_3^2\sqrt{1-2\gamma}\|\mathcal{L}_{0}\|_F-2\sqrt{6}}.    
\end{equation}
Therefore, $g(0)>0$ if
\begin{equation}
\tfrac{2\sqrt{6r}n_1n_2n_3^2}{n_1n_2n_3^2\sqrt{1-2\gamma}\|\mathcal{L}_{0}\|_F-2\sqrt{6}}<\lambda <\left(\tfrac{1}{4}\sqrt{\tfrac{n_{(1)}}{2\mu r}}-\sqrt{rn_3}\right)n_1n_2n_3^{3/2}. \label{ine:lambdabound}   
\end{equation}

Now, we finalize the proof of Theorem~\ref{thm:PtPo}.
For any tensor $\mathcal{Z}$ with $\|\mathcal{Z}\|_F=1$, let 
 \begin{equation}
 l(t, \mathcal{Z})=\lambda a \|\mathcal{L}_0+t\mathcal{Z}\|_F^{3} +b\|\mathcal{L}_0+t\mathcal{Z}\|_F^2-(\|\mathcal{L}_0\|_*+bt)\left\langle\mathcal{L}_0+t\mathcal{Z},\mathcal{Z}\right\rangle,  \label{equ:ltzTR}  
 \end{equation}
which is continuous.
We also let 
\begin{equation}
\begin{aligned}
&h(\mathcal{Z})\\
=&\|\mathcal{L}_{0}\|_F^2\left(\tfrac{1}{2}\|\mathcal{P}_{\mathbf{T}^{\perp}} (\mathcal{Z})\|_*  + (\tfrac{\lambda\|\mathcal{L}_{0}\|_F}{2}-\sqrt{\tfrac{6\lambda^2}{(1-2\gamma)n_1^2n_2^2n_3^4}}-\sqrt{\tfrac{6r}{1-2\gamma}})\|\mathcal{P}_{\boldsymbol{\Omega}^{c}}(\mathcal{Z})\|_{1}\right). \label{equ:hzTR}
\end{aligned}
\end{equation}
Clearly, $h(\mathcal{Z})>0, \forall \mathcal{Z}\in \mathbb{R}^{n_1\times n_2\times n_3}$, if $\lambda$ satisfies \eqref{ine:lambdabound}.
Combining \eqref{equ:ltzTR} and \eqref{equ:hzTR}, we get
 \begin{equation}
  l(0,\mathcal{Z})\geq h(\mathcal{Z})>0 \quad \forall \mathcal{Z}\in \mathbb{R}^{n_1\times n_2\times n_3}, \|\mathcal{Z}\|_F=1,
  \label{ine:l}
 \end{equation}
which implies that $\lim_{t\to 0^+}l(t,\mathcal{Z})>0$. 
By the continuity of $l(t,\mathcal{Z})$ with respect to $t$, there exists a constant $\overline t>0$ (independent of $\mathcal Z$), such that $l(t,\mathcal{Z})>0$ for $t \in \left[0, \overline t\right)$ and any $\mathcal Z$ with $\|\mathcal Z\|_F=1$.  Consequently, there exists $\overline t>0$, 
such that $g(t)>0$ for $t \in \left[0,\overline t\right)$. Then, we get $f^{\prime}(t)>0$ when $t \in \left[0,\overline t\right)$. Hence $f(0)\leq f(t) $ for any $t \in \left[0,\overline t\right)$, i.e.,
\begin{equation*}
\tfrac{\|\mathcal{L}_0+t\mathcal{Z}\|_*}{\|\mathcal{L}_0+t\mathcal{Z}\|_F}+\lambda\|\mathcal{E}_0-t\mathcal{Z}\|_1 \geq \tfrac{\|\mathcal{L}_0\|_*}{\|\mathcal{L}_0\|_F}+\lambda\|\mathcal{E}_0\|_1 \quad t\in \left[0,\overline t\right).
\end{equation*}
So, there exists a positive $\overline t$, such that when $\|\mathcal{Z}\|_F \leq \overline t$, ${\left(\mathcal{L}_{0}, \mathcal{E}_{0}\right)}$ satisfies \eqref{equ:LZE}.
We complete the proof of Theorem~\ref{thm:PtPo}.

\subsubsection{The construction of tensor $\mathcal{Y}$.}
We apply the golfing scheme that was used in \cite{lu2019tensor} to construct the dual tensor $\mathcal{Y}$, whose support is $\boldsymbol{\Omega}^{c}$. Let  the distribution of ${\boldsymbol{\Omega}^{c}}$ be the same as that of ${\boldsymbol{\Omega}^{c}=\boldsymbol{\Omega}_{1} \cup}{\boldsymbol{\Omega}_{2} \cup \cdots \cup \boldsymbol{\Omega}_{J}}$, where each ${\boldsymbol{\Omega}_{j}}$ follows the Bernoulli model with parameter ${q}$ and $J=\lceil3\log_2(n_{(1)}n_3)\rceil$. 
Hence we have $2\gamma=(1-q)^{J}$. 

Now we construct a sequence of tensors $\{\mathcal Y_j\}_{j=0}^J$ as follows, 
\begin{equation}
\begin{split}
    \mathcal{Y}_0 & =\mathcal{P}_{\mathbf{T}}( \mathcal{U} * \mathcal{V}^{*}-\lambda \|\mathcal{L}_{0}\|_F \operatorname{sgn}(\mathcal{E}_0)), \\
\mathcal{Y}_j & = \left(\mathcal{P}_{\mathbf{T}}-\tfrac{1}{q}\mathcal{P}_{\mathbf{T}}\mathcal{P}_{\boldsymbol{\Omega}_{j}}\mathcal{P}_{\mathbf{T}}\right)\mathcal{Y}_{j-1}, \ j = 1, 2, \cdots, J. 
\end{split}
\label{eq:definitionyiterate}
\end{equation}
We intend to show that a tensor defined by
\begin{equation}
\mathcal{Y}:=\sum_{j=1}^{J} \tfrac{1}{q} \mathcal{P}_{\boldsymbol{\Omega}_j}\left(\mathcal{Y}_{j-1}\right), \label{eq:definitionylast}
\end{equation}
satisfies all the conditions in \eqref{con:ptcondition1}.
Obviously, ${\mathcal{P}_{\boldsymbol{\Omega}}(\mathcal{Y})=\mathcal{O}}$.
 Before verifying the remaining condition of \eqref{con:ptcondition1}, 
we first give the upper bounds of $\left\|\mathcal{Y}_0\right\|_{\infty}$ and $\left\|\mathcal{Y}_0\right\|_{F}$.
\begin{lemma}
\label{lem:y0_bound}
For $\mathcal{Y}_0$ defined as \eqref{eq:definitionyiterate}, there exists a constant $C$ such that
\begin{equation}
\left\|\mathcal{Y}_0\right\|_{\infty} \leq C\lambda \sqrt{\tfrac{\mu r \log(n_{(1)}n_3)}{n_{(2)}}}, \label{inequality:infinityY}
\end{equation}
holds with high probability.
\end{lemma}
\begin{proof}
Note that
the $(u, v, w)$-th element of $\mathcal{P}_\mathbf{T}\left(\operatorname{sgn}\left(\mathcal{E}_0\right)\right)$ can be obtained by,
\begin{equation*}
\begin{aligned}
\left\langle\mathcal{P}_\mathbf{T}\left(\operatorname{sgn}\left(\mathcal{E}_0\right)\right), \mathbf{e}_{uvw}\right\rangle 
& = \left\langle
\sum_{i j k}\left[\operatorname{sgn}\left(\mathcal{E}_0\right)\right]_{i j k}\mathcal{P}_\mathbf{T}\left(\mathbf{e}_{ijk}\right),\mathbf{e}_{uvw} \right\rangle \\
&=\sum_{i j k}\left[\operatorname{sgn}\left(\mathcal{E}_0\right)\right]_{i j k}\left\langle\mathcal{P}_\mathbf{T}\left(\mathbf{e}_{ijk}\right), \mathbf{e}_{uvw}\right\rangle  
\\
&=\sum_{i j k}\left[\operatorname{sgn}\left(\mathcal{E}_0\right)\right]_{i j k}\left\langle\mathcal{P}_\mathbf{T}\left(\mathbf{e}_{ijk}\right), \mathcal{P}_\mathbf{T}\left(\mathbf{e}_{uvw}\right)\right\rangle\\
&=\sum_{i j k}\tfrac{1}{n_3}\left[\operatorname{sgn}\left(\mathcal{E}_0\right)\right]_{i j k}\langle\overline{\mathbf{P}_{ijk}}, \overline{\mathbf{P}_{uvw}} \rangle,
\end{aligned}
\end{equation*}
where $\overline{\mathbf{P}_{uvw}}=\text{bdiag}\left(\overline{\mathcal{P}_\mathbf{T}\left(\mathbf{e}_{uvw}\right)}\right) \in \mathbb{R}^{n_1n_3\times n_2n_3} $.  
By $\mathbb{P}(\mathcal{E}_0=1)=\mathbb{P}(\mathcal{E}_0=-1)$, it is straightforward to get $$\mathbb{E}\left( \tfrac{1}{n_3}\left[\operatorname{sgn}\left(\mathcal{E}_0\right)\right]_{i j k}\langle\overline{\mathbf{P}_{ijk}}, \overline{\mathbf{P}_{uvw}} \rangle\right)=0.$$
Additionally, we have
\begin{equation}
\begin{aligned}
&\left\|\tfrac{1}{n_3}\left[\operatorname{sgn}\left(\mathcal{E}_0\right)\right]_{i j k}\langle\overline{\mathbf{P}_{ijk}}, \overline{\mathbf{P}_{uvw}} \rangle\right\|=\left\|\left[\operatorname{sgn}\left(\mathcal{E}_0\right)\right]_{i j k}\left\langle\mathcal{P}_\mathbf{T}\left(\mathbf{e}_{ijk}\right), \mathcal{P}_\mathbf{T}\left(\mathbf{e}_{uvw}\right)\right\rangle \right\| \\
& \leq \left|\left[\operatorname{sgn}\left(\mathcal{E}_0\right)\right]_{i j k}\right|\left\|\mathcal{P}_\mathbf{T}\left(\mathbf{e}_{ijk}\right)\right\|_F\left\|\mathcal{P}_\mathbf{T}\left(\mathbf{e}_{uvw}\right)\right\|_F \leq  \tfrac{2\mu r}{n_{(1)}n_3},
\end{aligned} \label{ine:spectralnorm}
\end{equation}
where the last inequality is from $\|\mathcal{P}_{\mathbf{T}}(\mathbf{e}_{ijk})\|_F^2 \leq  \tfrac{2\mu r}{n_{(1)}n_3}$ \cite{lu2019tensor}. Denoting $R=\tfrac{2 \mu r}{n_{(2)}n_3}$, we  get
\begin{equation}
\begin{split}
\mathbb{P}\left(\left|\left\langle\mathcal{P}_\mathbf{T}\left(\operatorname{sgn}\left(\mathcal{E}_0\right)\right), \mathbf{e}_{uvw}\right\rangle\right| \geq t\right) &
= \mathbb{P} (|\sum_{i j k}\tfrac{1}{n_3}\left[\operatorname{sgn}\left(\mathcal{E}_0\right)\right]\langle\overline{\mathbf{P}_{ijk}}, \overline{\mathbf{P}_{uvw}} \rangle | \geq t ) \\ 
& \leq 2\exp \left(-\tfrac{t^2}{2\sigma^2+\tfrac{2}{3}Rt}\right),
\end{split}
\label{ineq:prob-t}
\end{equation}
where $\sigma^2$ is calculated by Lemma~\ref{inequality:Ber}, or more specifically,
\begin{equation*}
\begin{aligned}
\sigma^2&=\sum_{ijk} \mathbb{E}\left[\operatorname{sgn}\left(\mathcal{E}_0\right)\right]^2\left\langle\mathcal{P}_\mathbf{T}\left(\mathbf{e}_{ijk}\right), \mathcal{P}_\mathbf{T}(\mathbf{e}_{uvw})\right\rangle^2 \\ &=\mathbb{E}\left[\operatorname{sgn}\left(\mathcal{E}_0\right)\right]^2\sum_{ijk}\left\langle\mathcal{P}_\mathbf{T}\left(\mathbf{e}_{ijk}\right), \mathcal{P}_\mathbf{T}(\mathbf{e}_{uvw})\right\rangle^2 
\\ & =\mathbb{E}\left[\operatorname{sgn}\left(\mathcal{E}_0\right)\right]^2\sum_{ijk}\left\langle\mathbf{e}_{ijk}, \mathcal{P}_\mathbf{T}(\mathbf{e}_{uvw})\right\rangle^2
=\mathbb{E}\left[\operatorname{sgn}\left(\mathcal{E}_0\right)\right]^2\left\|\mathcal{P}_\mathbf{T}\left(\mathbf{e}_{uvw}\right)\right\|_F^2 \\
& =2\gamma\left\|\mathcal{P}_\mathbf{T}\left(\mathbf{e}_{uvw}\right)\right\|_F^2 \leq
\tfrac{4 \gamma \mu r}{n_{(1)}n_3}.
\end{aligned}
\end{equation*}

Considering that the entries of $\mathcal{P}_\mathbf{T}\left(\operatorname{sgn}\left(\mathcal{E}_0\right)\right)$ can be understood as i.i.d. copies of the $(u, v, w)$-th entry and setting $t=\tfrac{C^{\prime}}{\|\mathcal{L}_0\|_F}\sqrt{\tfrac{\mu r \log(n_{(1)}n_3)}{n_{(2)}}}$ in \eqref{ineq:prob-t} with some positive constant $C^{\prime}$, we get
\begin{equation}
\|\mathcal{P}_\mathbf{T}\left(\operatorname{sgn}\left(\mathcal{E}_0\right)\right)\|_{\infty} \leq \tfrac{C^{\prime}}{\|\mathcal{L}_0\|_F}\sqrt{\tfrac{\mu r \log(n_{(1)}n_3)}{n_{(2)}}} , \label{ine:ptsgne} 
\end{equation}
with the probability $\mathbb{P}$ at least by
\begin{equation*}
\begin{aligned}
\mathbb{P} &\geq 1-2 \exp \left(-\tfrac{t^2}{2\sigma^2+\tfrac{2}{3}Rt}\right) 
\geq 1-2 \exp (-\tfrac{C^{\prime2}}{\|\mathcal{L}_0\|_F^2}  \tfrac{\mu r\log(n_{(1)}n_3)}{n_{(2)}}  \tfrac{3n_{(2)}n_3}{24 \gamma \mu r+4\mu rt})\\
&\geq  1-2 \exp (-\tfrac{C^{\prime2}}{\|\mathcal{L}_0\|_F^2}  \tfrac{n_3\log(n_{(1)}n_3)}{8\gamma}).  
\end{aligned}
\end{equation*}
On the other hand, according to tensor incoherence conditions (7), we have 
\begin{equation}
\begin{aligned}
\left\|\mathcal{U} * \mathcal{V}^*\right\|_{\infty} &\leq \sqrt{\tfrac{\mu r}{n_1 n_2 n_3^{2}}} =\tfrac{1}{\sqrt{n_{(1)}n_3^2 \log(n_{(1)}n_3 )}}\sqrt{\tfrac{\mu r \log(n_{(1)}n_3)}{n_{(2)}}}\\
&\leq \tfrac{\lambda}{\log(n_{(1)}n_3)}\sqrt{\tfrac{\mu r \log(n_{(1)}n_3)}{n_{(2)}}},    
\end{aligned}
  \label{ine:uvinfinity}
\end{equation}
if $\lambda \geq\sqrt{\tfrac{\log (n_{(1)}n_3)}{n_{(1)}n_3^2}}$. 
Then, using \eqref{ine:ptsgne} and \eqref{ine:uvinfinity}, we get
\begin{equation*}
\begin{aligned}
\left\|\mathcal{Y}_0\right\|_{\infty}&=\left\|\mathcal{P}_{\mathbf{T}}( \mathcal{U} * \mathcal{V}^{*}-\lambda \|\mathcal{L}_{0}\|_F \operatorname{sgn}(\mathcal{E}_0))\right\|_{\infty} \\
& \leq \left\|\mathcal{P}_{\mathbf{T}}( \mathcal{U} * \mathcal{V}^{*})\right\|_{\infty}+\lambda \|\mathcal{L}_{0}\|_F \left\|\mathcal{P}_{\mathbf{T}}(\operatorname{sgn}(\mathcal{E}_0))\right\|_{\infty}\\
&=\left\| \mathcal{U} * \mathcal{V}^{*}\right\|_{\infty}+\lambda \|\mathcal{L}_{0}\|_F \left\|\mathcal{P}_{\mathbf{T}}(\operatorname{sgn}(\mathcal{E}_0))\right\|_{\infty}\\
&\leq   (\tfrac{1}{\log(n_{(1)}n_3)}+  C^{\prime }) \lambda \sqrt{\tfrac{\mu r \log(n_{(1)}n_3)}{n_{(2)}}} = C\lambda \sqrt{\tfrac{\mu r \log(n_{(1)}n_3)}{n_{(2)}}} ,
\end{aligned}
\end{equation*}
where $C=\tfrac{1}{\log(n_{(1)}n_3)}+ C^{\prime }$.
\end{proof}
From Lemma~\ref{lem:y0_bound}, we can easily get:
\begin{equation}
\left\|\mathcal{Y}_0\right\|_{F} \leq \sqrt{n_1n_2n_3}\left\|\mathcal{Y}_0\right\|_{\infty} \leq C\lambda \sqrt{\mu r n_{(1)}n_3 \log(n_{(1)}n_3)}. \label{ine:yfupper}
\end{equation}

Next, we  prove $$\left\|\mathcal{P}_\mathbf{T}\left(\mathcal{Y} +\lambda \|\mathcal{L}_0\|_F \operatorname{sgn}\left(\mathcal{E}_0\right)-\mathcal{U}  * \mathcal{V} ^*\right)\right\|_F \leq \tfrac{\lambda}{n_1 n_2 n_3^2}. $$ 
By setting $\xi=\tfrac{1}{2}$ in Lemma~\ref{lemmaomega1} and assuming  $q\geq 4C_{0} \tfrac{\mu r \log \left(n_{(1)} n_{3}\right)}{n_{(2)}n_3}$, we obtain  
\begin{equation}
\left\|\mathcal{Y}_j\right\|_F \leq \left\|\mathcal{P}_{\mathbf{T}}-\tfrac{1}{q}\mathcal{P}_{\mathbf{T}}\mathcal{P}_{\boldsymbol{\Omega}_{j}}\mathcal{P}_{\mathbf{T}}\right\| \left\|\mathcal{Y}_{j-1} \right\|_F
 \leq \tfrac{1}{2}\left\|\mathcal{Y}_{j-1}\right\|_F. \label{ine:yjyjFTR}
\end{equation}
By the definition of $\mathcal{Y}_0$ in \eqref{eq:definitionyiterate}, $\mathcal{Y}$ in \eqref{eq:definitionylast}, and using \eqref{ine:yjyjFTR}, we get
\begin{equation*}
\begin{aligned}
& \|\mathcal{P}_\mathbf{T}\left(\mathcal{Y} +\lambda \|\mathcal{L}_0\|_F \operatorname{sgn}\left(\mathcal{E}_0\right)-\mathcal{U}  * \mathcal{V} ^*\right) \|_F \\
= &\left\|\mathcal{P}_\mathbf{T}(\mathcal{Y})-\mathcal{P}_\mathbf{T}\left(\mathcal{U} * \mathcal{V}^*-\lambda \|\mathcal{L}_0\|_F\operatorname{sgn}\left(\mathcal{E}_0\right)\right)\right\|_F \\
  = &  \|\sum\limits_{j=1}^{J} \tfrac{1}{q} \mathcal{P}_\mathbf{T} \mathcal{P}_{\boldsymbol{\Omega}_j}\left(\mathcal{Y}_{j-1}\right)-\mathcal{Y}_0\|_F \\
=& \|\sum_{j=2}^{J} \tfrac{1}{q} \mathcal{P}_\mathbf{T} \mathcal{P}_{\boldsymbol{\Omega}_j}\left(\mathcal{Y}_{j-1}\right)-\left(\mathcal{P}_{\mathbf{T}}-\tfrac{1}{q}\mathcal{P}_{\mathbf{T}}\mathcal{P}_{\boldsymbol{\Omega}_{1}}\mathcal{P}_{\mathbf{T}}\right)\mathcal{Y}_{0}\|_F\\
= &  \|\sum_{j=2}^{J} \tfrac{1}{q} \mathcal{P}_\mathbf{T} \mathcal{P}_{\boldsymbol{\Omega}_j}\left(\mathcal{Y}_{j-1}\right)-\mathcal Y_1\|_F = \|\mathcal{Y}_{J}\|_F\leq 2^{-J} \|\mathcal Y_0\|_F \\
\leq& C  2^{-J} \lambda \sqrt{\mu r n_{(1)}n_3 \log(n_{(1)}n_3)} \\
 \leq & \tfrac{C \lambda \sqrt{ \mu r n_{(1)} n_3 \log \left(n_{(1)} n_3\right)} }{n_{(1)}^3 n_3^{3}}  
 \leq \tfrac{\lambda}{n_{(1)}^2 n_3^2} \leq \tfrac{\lambda}{n_1 n_2 n_3^2} ,
\end{aligned}
\end{equation*}
where $J=\lceil3\log_2(n_{(1)}n_3)\rceil\geq 3\log_2(n_{(1)}n_3)$   and $r\leq \tfrac{n_{(1)}n_3}{\mu C^2\log (n_{(1)}n_3)}. $ 
Therefore, $\left\|\mathcal{P}_\mathbf{T}\left(\mathcal{Y} +\lambda \|\mathcal{L}_0\|_F \operatorname{sgn}\left(\mathcal{E}_0\right)-\mathcal{U}  * \mathcal{V} ^*\right)\right\|_F \leq \tfrac{\lambda}{n_1 n_2 n_3^2}$.

In order to prove $\left\|\mathcal{P}_{\mathbf{T}^{\perp}}\left(\mathcal{Y}\right)\right\| \leq \tfrac{1}{4}$, we use the construction of $\mathcal{Y}$ \eqref{eq:definitionylast} and $\mathcal{P}_{\mathbf{T}^{\perp}}\left(\mathcal{Y}_{j-1}\right)=\mathcal{O}$ to get
\begin{equation*}
\begin{aligned}
\left\|\mathcal{P}_{\mathbf{T}^{\perp}}(\mathcal{Y})\right\| 
& =\left\|\mathcal{P}_{\mathbf{T}^{\perp}} \left( \sum_{j=1}^{J} \tfrac{1}{q} \mathcal{P}_{\boldsymbol{\Omega}_j}\left(\mathcal{Y}_{j-1}\right)\right)\right\| 
 \leq \sum_{j=1}^{J}\left\|\tfrac{1}{q} \mathcal{P}_{\mathbf{T}^{\perp}} \mathcal{P}_{\boldsymbol{\Omega}_j}\left(\mathcal{Y}_{j-1}\right)\right\| \\
& =\sum_{j=1}^{J}\left\|\mathcal{P}_{\mathbf{T}^{\perp}}\left(\tfrac{1}{q} \mathcal{P}_{\boldsymbol{\Omega}_j}\left(\mathcal{Y}_{j-1}\right)-\mathcal{Y}_{j-1}\right)\right\| 
 \leq \sum_{j=1}^{J}\left\|\tfrac{1}{q} \mathcal{P}_{\boldsymbol{\Omega}_j}\left(\mathcal{Y}_{j-1}\right)-\mathcal{Y}_{j-1}\right\|, 
\end{aligned}
\end{equation*}
where the last inequality is from $\|\mathcal{P}_{\mathbf{T}^{\perp}}(\mathcal{Z})\|\leq \|\mathcal{Z}\|$ \cite{lu2019tensor}. 
By Lemma~\ref{lemmaomega3}, with an assumption $q \geq  C_{0} \tfrac{\log \left(n_{(1)} n_{3}\right)}{n_{(2)} n_{3}}$, we get
\begin{equation}
\left\|\tfrac{1}{q} \mathcal{P}_{\boldsymbol{\Omega}_j}\left(\mathcal{Y}_{j-1}\right)-\mathcal{Y}_{j-1}\right\|  = \left\|\left(\mathcal{I}-q^{-1} \mathcal{P}_{\boldsymbol{\Omega}_j}\right) \mathcal{Y}_{j-1}\right\| \leq \sqrt{\tfrac{C_{0} n_{(1)} n_{3} \log \left(n_{(1)} n_{3}\right)}{q}}\|\mathcal{Y}_{j-1}\|_{\infty}.  \label{ine: qyjinfinityTR}  
\end{equation}
On the other hand, by setting $\epsilon=\tfrac{1}{2 \sqrt{\log \left(n_{(1)} n_3\right)}}$ in  \eqref{lemmaomega2}
and assuming $q\geq 4C_0\tfrac{\mu r(\log(n_{(1)}n_3))^2}{n_{(2)}n_3}$, we have 
\begin{equation}
\begin{aligned}
\left\|\mathcal{Y}_j\right\|_{\infty}&= \left\|\left(\mathcal{P}_{\mathbf{T}}-\tfrac{1}{q}\mathcal{P}_{\mathbf{T}}\mathcal{P}_{\boldsymbol{\Omega}_{j}}\mathcal{P}_{\mathbf{T}}\right)\mathcal{Y}_{j-1}\right\|_{\infty} \\
&\leq \tfrac{1}{2 \sqrt{\log \left(n_{(1)} n_3\right)}}\left\|\mathcal{P}_{\mathbf{T}}(\mathcal{Y}_{j-1})\right\|_{\infty}\leq \tfrac{1}{2 \sqrt{\log \left(n_{(1)} n_3\right)}}\left\|\mathcal{Y}_{j-1}\right\|_{\infty}.    
\end{aligned}
 \label{ine:yjyjinfinityTR}   
\end{equation}
Hence, combining \eqref{ine: qyjinfinityTR}, \eqref{ine:yjyjinfinityTR} and using (8), it is easy to get
\begin{equation}
\begin{aligned}
 &\sum_{j=1}^{J}\left\|\tfrac{1}{q} \mathcal{P}_{\boldsymbol{\Omega}_j}\left(\mathcal{Y}_{j-1}\right)-\mathcal{Y}_{j-1}\right\| 
 \leq \sum_{j=1}^{J} C_0 \sqrt{\tfrac{n_{(1)} n_3 \log \left(n_{(1)} n_3\right)}{q}}\left\|\mathcal{Y}_{j-1}\right\|_{\infty} \\
 \leq& C_0 \sqrt{\tfrac{n_{(1)} n_3 \log \left(n_{(1)} n_3\right)}{q}} \sum_{j=1}^{J} \left(\tfrac{1}{2 \sqrt{\log(n_{(1)}n_3)}}\right)^{j-1}\left\|\mathcal{Y}_{0}\right\|_{\infty}\\
 \leq & C_0 \sqrt{\tfrac{n_{(1)} n_3 \log \left(n_{(1)} n_3\right)}{q}}  C\lambda \sqrt{\tfrac{\mu r \log(n_{(1)}n_3)}{n_{(2)}}} \sum_{j=1}^{J} \left(\tfrac{1}{2}\right)^{j-1} \\
 \leq & 2C_0C\lambda\log(n_{(1)}n_3)\sqrt{\tfrac{n_{(1)}n_3\mu r}{n_{(2)}q}} \leq \tfrac{1}{4}, 
\end{aligned} \label{ine:ptprepystep2TR}
\end{equation}
where the last inequality is satisfied if $r$ is sufficient small, e.g., $$r\leq \tfrac{n_{(2)}q}{64C_0^2C^2\lambda^2\mu n_{(1)}n_3(\log(n_{(1)}n_3))^2}. $$

To prove $\lambda \|\mathcal{L}_0\|_F\left\|\mathcal{P}_{\mathbf{T}^{\perp}}(\operatorname{sgn}\left(\mathcal{E}_0\right))\right\| \leq \tfrac{1}{4} $, we apply
 Lemma~\ref{lemmabernoulli} that leads to a function ${\varphi(\gamma)}$ satisfying ${\lim _{\gamma \rightarrow 0^{+}} \varphi(\gamma)=0}$, such that
\begin{equation*}
\|\operatorname{sgn}(\mathcal{E}_0)\| \leq \varphi(\gamma) \sqrt{n_{(1)} n_{3}}. 
\end{equation*}
Therefore, if $\gamma$ is sufficiently small,
\begin{equation*}
\lambda \|\mathcal{L}_0\|_F\left\|\mathcal{P}_{\mathbf{T}^{\perp}}(\operatorname{sgn}\left(\mathcal{E}_0\right))\right\| \leq \lambda \|\mathcal{L}_0\|_F \varphi(\gamma) \sqrt{n_{(1)} n_{3}} \leq \tfrac{1}{4},    
\end{equation*}
holds with a high probability.

Lastly, we prove $\left\|\mathcal{P}_{_{\boldsymbol{\Omega}^c}}(\mathcal{Y} )\right\|_{\infty} \leq \tfrac{\lambda}{2}\|\mathcal{L}_0\|_F$. By the construction of $\mathcal{Y}$ \eqref{eq:definitionylast} and $\boldsymbol{\Omega}^{c}=\boldsymbol{\Omega}_{1} \cup\boldsymbol{\Omega}_{2} \cup \cdots \cup \boldsymbol{\Omega}_{J}$, we get
\begin{equation*}
\begin{aligned}
&\left\|\mathcal{P}_{\boldsymbol{\Omega}^c}(\mathcal{Y})\right\|_{\infty} 
 =\left\|\mathcal{P}_{\boldsymbol{\Omega}^c}\left( \sum_{j=1}^{J} \tfrac{1}{q} \mathcal{P}_{\boldsymbol{\Omega}_j}\left(\mathcal{Y}_{j-1}\right)\right)\right\|_{\infty} \\
=&\left\| \sum_{j=1}^{J} \tfrac{1}{q} \mathcal{P}_{\boldsymbol{\Omega}_j}\left(\mathcal{Y}_{j-1}\right)\right\|_{\infty} \leq \sum_{j=1}^{J} \tfrac{1}{q}\left\|\mathcal{Y}_{j-1}\right\|_{\infty}.
\end{aligned}
\end{equation*}
We further use \eqref{ine:ptprepystep2TR}, thus getting
\begin{equation}
    \begin{split}
        &\quad\sum_{j=1}^{J} \tfrac{1}{q}\left\|\mathcal{Y}_{j-1}\right\|_{\infty}\\ 
& \leq (n_{(1)} n_3 \log \left(n_{(1)} n_3\right) q C_0^2)^{-\tfrac{1}{2}}
\sum_{j=1}^{J} C_0 \sqrt{\tfrac{n_{(1)} n_3 \log \left(n_{(1)} n_3\right)}{q}}\left\|\mathcal{Y}_{j-1}\right\|_{\infty} \\
& \leq \tfrac{1}{4} (n_{(1)} n_3 \log \left(n_{(1)} n_3\right) q C_0^2)^{-\tfrac{1}{2}} \leq \tfrac{\lambda}{2}\|\mathcal{L}_0\|_F ,\end{split} \label{ine:thelastyinfinity}
\end{equation}
where the last inequality holds if $\lambda$ is sufficient large such that $$q \geq \tfrac{1}{2\|\mathcal{L}_0\|_F C_0\sqrt{n_{(1)} n_3 \log \left(n_{(1)} n_3\right) \lambda }}. $$

By using $2\gamma=(1-q)^{J}$, 
we set 
\begin{equation*}
c_{\gamma} = \tfrac{1}{2}(1-q_0)^{\lceil3\log_2(n_{(1)}n_3)\rceil}, \end{equation*}
where $q_0=\max\left( C_{0} \tfrac{\log \left(n_{(1)} n_{3}\right)}{n_{(2)} n_{3}}, 4C_0\tfrac{\mu r\log^2(n_{(1)}n_3)}{n_{(2)}n_3}, \tfrac{1}{2\|\mathcal{L}_0\|_F C_0\sqrt{n_{(1)} n_3 \log \left(n_{(1)} n_3\right) \lambda }}\right).$
Meanwhile, we set 
\begin{equation*}
c_r = \max\left(\tfrac{n_{(1)}\log(n_{(1)}n_3)}{C^2}, \tfrac{q_0}{64C_0^2C^2\lambda^2n_{(1)}n_3} \right).     
\end{equation*}
Consequently, if $r \leq \tfrac{c_{r} n_{(2)}n_3}{\mu\left(\log \left(n_{(1)} n_{3}\right)\right)^{2}} \quad \text{and} \quad \gamma \leq c_{\gamma}$, we can construct a tensor $\mathcal Y$ defined in \eqref{eq:definitionylast} 
that satisfies
\eqref{con:ptcondition1}. 

\section{Convergence proof} \label{sec:algorithm}
This section is devoted to the convergence analysis of the two Algorithms for solving the TNF and the TNF$+$ models. We need the following lemma.
\begin{lemma}{\rm\cite[Lemma B.4]{zheng2024scale}}
Given a function $g(\mathcal{X})=\tfrac{1}{\|\mathcal{X}\|_{F}}$ and a set ${\mathcal{M}_{\delta}:=\left\{\mathcal{X}|\| \mathcal{X}\|_{F} \geq \delta\right\}}$ with a positive constant $\delta > 0$, we have
\begin{equation*}
 \left\|\nabla g(\mathcal{X})-\nabla g(\mathcal{Y})\right\|_{F} 
\leq \tfrac{2}{\delta^{3}}\|\mathcal{X}-\mathcal{Y}\|_{F}, \quad  \forall \mathcal{X}, \mathcal{Y} \in \mathcal{M}_{\delta}.    
\end{equation*}
 \label{con:lemmaA4}
\end{lemma}

\subsection {Convergence analysis of TNF}\label{sect:RPCA1}
\paragraph{Proof of Lemma 1}
\begin{proof}
The optimality condition of the $\mathcal{H}$ subproblem in (11) indicates that
\begin{equation}
-\tfrac{\left\|\mathcal{L}^{(k+1)}\right\|_{*}}{\left\|\mathcal{H}^{(k+1)}\right\|_{F}^{3}} \mathcal{H}^{(k+1)}+\mu_1\left(\mathcal{H}^{(k+1)}-\mathcal{L}^{(k+1)}-\tfrac{\mathcal{Y}^{(k)}}{\mu_1}\right)=\mathcal{O}, \label{Hoptimal}
\end{equation}
where $\mathcal{O}\in \mathbb{R}^{n_1\times n_2\times n_3}$ is the zero tensor. Using the dual update $\mathcal{Y}^{(k+1)}=\mathcal{Y}^{(k)}+\mu_1\left(\mathcal{L}^{(k+1)}-\mathcal{H}^{(k+1)}\right)$, we have
\begin{equation}
\mathcal{Y}^{(k+1)}=-\tfrac{\left\|\mathcal{L}^{(k+1)}\right\|_{*}}{\left\|\mathcal{H}^{(k+1)}\right\|_{F}^{3}} \mathcal{H}^{(k+1)},  \label{con:Yk1} 
\end{equation}
which directly deduces to
\begin{equation}
\mathcal{Y}^{(k)}=-\tfrac{\left\|\mathcal{L}^{(k)}\right\|_{*}}{\left\|\mathcal{H}^{(k)}\right\|_{F}^{3}} \mathcal{H}^{(k)}. \label{con:Yk}   
\end{equation}
Then it is straightforward to have
\begin{equation}
\begin{aligned}
&\left\|\mathcal{Y}^{(k+1)}-\mathcal{Y}^{(k)}\right\|_F
=\left\|\tfrac{\left\|\mathcal{L}^{(k+1)}\right\|_*}{\left\|\mathcal{H}^{(k+1)}\right\|_F^3} \mathcal{H}^{(k+1)}-\tfrac{\left\|\mathcal{L}^{(k)}\right\|_*}{\left\|\mathcal{H}^{(k)}\right\|_F^3} \mathcal{H}^{(k)}\right\|_F \\
 \leq&\left\|\tfrac{\mathcal{H}^{(k+1)}}{\left\|\mathcal{H}^{(k+1)}\right\|_F^3}\left(\left\|\mathcal{L}^{(k+1)}\right\|_*-\left\|\mathcal{L}^{(k)}\right\|_*\right)\right\|_F+\left\|\mathcal{L}^{(k)}\right\|_*\left\|\tfrac{\mathcal{H}^{(k+1)}}{\left\|\mathcal{H}^{(k+1)}\right\|_F^3}-\tfrac{\mathcal{H}^{(k)}}{\left\|\mathcal{H}^{(k)}\right\|_F^3}\right\|_F \\
=&\tfrac{1}{\left\|\mathcal{H}^{(k+1)}\right\|_F^2}\left| \| \mathcal{L}^{(k+1)}\left\|_*-\right\| \mathcal{L}^{(k)}\|_{*}\right|+\left\|\mathcal{L}^{(k)}\right\|_{*}\left\|\tfrac{\mathcal{H}^{(k+1)}}{\left\|\mathcal{H}^{(k+1)}\right\|_F^3}-\tfrac{\mathcal{H}^{(k)}}{\left\|\mathcal{H}^{(k)}\right\|_F^3}\right\|_F.  
\end{aligned} \label{con:proofAstep}
\end{equation}
Note that $\left\|\mathcal{L}\right\|_{*} \leq \sqrt{r}\left\|\mathcal{L}\right\|_{F} \leq \sqrt{n_{(2)}}\left\|\mathcal{L}\right\|_{F}$, where $r$ is the tubal rank of tensor $\mathcal{L} \in \mathbb{R}^{ n_1\times n_2\times n_3}$.
Then, the first term in \eqref{con:proofAstep} turns to
\begin{equation}
\left| \|\mathcal{L}^{(k+1)}\|_*-\|\mathcal{L}^{(k)}\|_*\right| \leq \left\|\mathcal{L}^{(k+1)}-\mathcal{L}^{(k)}\right\|_* \leq \sqrt{n_{(2)}} \left\|\mathcal{L}^{(k+1)}-\mathcal{L}^{(k)}\right\|_{F}.\label{con:proofL1}
\end{equation} 
It follows from Lemma~\ref{con:lemmaA4} that
\begin{equation}
\left\|\mathcal{L}^{(k)}\right\|_{*}\left\|\tfrac{\mathcal{H}^{(k+1)}}{\left\|\mathcal{H}^{(k+1)}\right\|_F^3}-\tfrac{\mathcal{H}^{(k)}}{\left\|\mathcal{H}^{(k)}\right\|_F^3}\right\|_F \leq \tfrac{2\|\mathcal{L}^{(k)}\|_{*}}{\delta^{3}}\left\|\mathcal{H}^{(k+1)}-\mathcal{H}^{(k)}\right\|_{F}. \label{con:proofL2}
\end{equation}
Putting \eqref{con:proofL1}-\eqref{con:proofL2} together with Cauchy-Schwarz inequality leads to the desired inequality (17).
\end{proof}

\paragraph{Proof of Lemma 2}
\begin{proof}
 The function $L_{1}\left(\mathcal{L},\mathcal{H}^{(k)},\mathcal{E}^{(k)},\mathcal{Y}^{(k)},\mathcal{Z}^{(k)}\right)$, derived from (10) with fixed $\mathcal{H}^{(k)}$, $\mathcal{E}^{(k)}$, $\mathcal{Y}^{(k)}$ and $\mathcal{Z}^{(k)}$, consists of a TNN term and two quadratic terms, hence it is strongly convex in terms of $\mathcal{L}$ with constant $ \mu_{1}+\mu_{2}$.
 The convex property leads to
 \begin{equation}
 \begin{aligned}
& L_{1}\left(\mathcal{L}^{(k+1)},\mathcal{H}^{(k)},\mathcal{E}^{(k)},\mathcal{Y}^{(k)},\mathcal{Z}^{(k)}\right)\\ \leq &L_{1}\left(\mathcal{L}^{(k)},\mathcal{H}^{(k)},\mathcal{E}^{(k)},\mathcal{Y}^{(k)},\mathcal{Z}^{(k)}\right) - \tfrac{\mu_1+\mu_2}{2}\left\|\mathcal{L}^{(k+1)}-\mathcal{L}^{(k)}\right\|_{F}^{2}.   
\end{aligned}\label{proofPCAL}
\end{equation}

Now we examine the change in $L_1$ caused by the $\mathcal H$, that is,
\begin{equation}
\begin{aligned}
&L_{1}\left(\mathcal{L}^{(k+1)},\mathcal{H}^{(k+1)},\mathcal{E}^{(k)},\mathcal{Y}^{(k)},\mathcal{Z}^{(k)}\right)-L_{1}\left(\mathcal{L}^{(k+1)},\mathcal{H}^{(k)},\mathcal{E}^{(k)},\mathcal{Y}^{(k)},\mathcal{Z}^{(k)}\right) \\
=& \tfrac{\|\mathcal{L}^{(k+1)}\|_{*}}{\|\mathcal{H}^{(k+1)}\|_{F}}-\tfrac{\|\mathcal{L}^{(k+1)}\|_{*}}{\|\mathcal{H}^{(k)}\|_{F}} + \tfrac{\mu_{1}}{2}\left\|\mathcal{L}^{(k+1)}-\mathcal{H}^{(k+1)}\right\|_{F}^{2}-\tfrac{\mu_{1}}{2}\left\|\mathcal{L}^{(k+1)}-\mathcal{H}^{(k)}\right\|_{F}^{2} \\
 &\quad +\left\langle\mathcal{Y}^{(k)}, \mathcal{L}^{(k+1)}-\mathcal{H}^{(k+1)}\right\rangle-\left\langle\mathcal{Y}^{(k)}, \mathcal{L}^{(k+1)}-\mathcal{H}^{(k)}\right\rangle. 
\end{aligned} \label{eq:A10TRPCA}
\end{equation}

It follows from Lemma~\ref{con:lemmaA4} along with the assumption {\rm A2} that $\tfrac{1}{\|\mathcal{H}\|_{F}}$ is Lipschitz continuous with parameter $\tfrac{2}{\delta^{3}}.$ Hence, we obtain 
\begin{equation}
\begin{aligned}
\tfrac{\|\mathcal{L}^{(k+1)}\|_{*}}{\|\mathcal{H}^{(k+1)}\|_{F}} &\leq \tfrac{\|\mathcal{L}^{(k+1)}\|_{*}}{\|\mathcal{H}^{(k)}\|_{F}}- \left\langle\tfrac{\|\mathcal{L}^{(k+1)}\|_{*}\mathcal{H}^{(k)}}{\left\|\mathcal{H}^{(k)}\right\|_{F}^3}, \mathcal{H}^{(k+1)}-\mathcal{H}^{(k)}\right\rangle \\
&\quad+ \tfrac{\|\mathcal{L}^{(k+1)}\|_{*}}{\delta^{3}}\left\|\mathcal{H}^{(k+1)}-\mathcal{H}^{(k)}\right\|_{F}^{2}. 
\end{aligned} \label{eq:B10TRPCA}
\end{equation}
Simple calculations of the third and the fourth terms in \eqref{eq:A10TRPCA} yield 
\begin{equation}
\begin{aligned}
&\tfrac{\mu_{1}}{2}\left\|\mathcal{L}^{(k+1)}-\mathcal{H}^{(k+1)}\right\|_{F}^{2}-\tfrac{\mu_{1}}{2}\left\|\mathcal{L}^{(k+1)}-\mathcal{H}^{(k)}\right\|_{F}^{2} \\
=&\tfrac{\mu_{1}}{2}\left\|\mathcal{H}^{(k+1)}\right\|_{F}^{2}-\tfrac{\mu_{1}}{2}\left\|\mathcal{H}^{(k)}\right\|_{F}^{2}-\mu_{1}\left\langle\mathcal{L}^{(k+1)}, \mathcal{H}^{(k+1)}-\mathcal{H}^{(k)}\right\rangle \\
=& \tfrac{\mu_{1}}{2}\left\|\mathcal{H}^{(k+1)}\right\|_{F}^{2}-\tfrac{\mu_{1}}{2}\left\|\mathcal{H}^{(k)}\right\|_{F}^{2}-\left\langle\mu_{1}\mathcal{H}^{(k+1)}+\mathcal{Y}^{(k+1)}-\mathcal{Y}^{(k)}, \mathcal{H}^{(k+1)}-\mathcal{H}^{(k)}\right\rangle \\
=&-\tfrac{\mu_{1}}{2}\left\|\mathcal{H}^{(k+1)}-\mathcal{H}^{(k)}\right\|_{F}^{2}-\left\langle\mathcal{Y}^{(k+1)}-\mathcal{Y}^{(k)}, \mathcal{H}^{(k+1)}-\mathcal{H}^{(k)}\right\rangle,
\end{aligned} \label{eq:B11TRPCA}
\end{equation}
where the second equality is from the $\mathcal{Y}$-update.
Putting together \eqref{eq:A10TRPCA}, \eqref{eq:B10TRPCA}, \eqref{eq:B11TRPCA}, we have
\begin{equation}
\begin{aligned}
&L_{1}\left(\mathcal{L}^{(k+1)},\mathcal{H}^{(k+1)},\mathcal{E}^{(k)},\mathcal{Y}^{(k)},\mathcal{Z}^{(k)}\right)-L_{1}\left(\mathcal{L}^{(k+1)},\mathcal{H}^{(k)},\mathcal{E}^{(k)},\mathcal{Y}^{(k)},\mathcal{Z}^{(k)}\right)  \\
\leq& - \left\langle\tfrac{\|\mathcal{L}^{(k+1)}\|_{*}\mathcal{H}^{(k)}}{\|\mathcal{H}^{(k)}\|_{F}^3}, \mathcal{H}^{(k+1)}-\mathcal{H}^{(k)}\right\rangle 
+ \tfrac{\|\mathcal{L}^{(k+1)}\|_{*}}{ \delta^{3}}\left\|\mathcal{H}^{(k+1)}-\mathcal{H}^{(k)}\right\|_{F}^{2}\\
&\quad-\tfrac{\mu_{1}}{2}\left\|\mathcal{H}^{(k+1)}-\mathcal{H}^{(k)}\right\|_{F}^{2}-\left\langle\mathcal{Y}^{(k+1)}-\mathcal{Y}^{(k)}, \mathcal{H}^{(k+1)}-\mathcal{H}^{(k)}\right\rangle\\
&\quad-\left\langle\mathcal{Y}^{(k)}, \mathcal{H}^{(k+1)}-\mathcal{H}^{(k)}\right\rangle \\
=&  \left\langle\tfrac{\|\mathcal{L}^{(k+1)}\|_{*}}{\|\mathcal{H}^{(k+1)}\|_{F}^3}\mathcal{H}^{(k+1)}-\tfrac{\|\mathcal{L}^{(k+1)}\|_{*}}{\|\mathcal{H}^{(k)}\|_{F}^3}\mathcal{H}^{(k)}, \mathcal{H}^{(k+1)}-\mathcal{H}^{(k)}\right\rangle \\
&\quad+ \tfrac{\tfrac{2}{\delta^{3}}\|\mathcal{L}^{(k+1)}\|_{*}-\mu_{1}}{2}\left\|\mathcal{H}^{(k+1)}-\mathcal{H}^{(k)}\right\|_{F}^{2} \\
\leq& -(\tfrac{\mu_{1}}{2}-\tfrac{3M}{\delta^3})\left\|\mathcal{H}^{(k+1)}-\mathcal{H}^{(k)}\right\|_{F}^{2},
\end{aligned}\label{proofPCAH} 
\end{equation}
where the equality is from \eqref{con:Yk1}.

The function $L_{1}\left(\mathcal{L}^{(k+1)},\mathcal{H}^{(k+1)},\mathcal{E},\mathcal{Y}^{(k)}, \mathcal{Z}^{(k)}\right)$ with fixed $\mathcal{L}^{(k+1)}$, $\mathcal{H}^{(k+1)}$, $\mathcal{Y}^{(k)}$ and $\mathcal{Z}^{(k)}$ consists of a $\ell_1$-norm term and one quadratic term, hence it is strongly convex in terms of $\mathcal{E}$ with constant $\mu_{2}$, thus leading to 
\begin{equation}
\begin{aligned}
&L_{1}\left(\mathcal{L}^{(k+1)},\mathcal{H}^{(k+1)},\mathcal{E}^{(k+1)},\mathcal{Y}^{(k)},\mathcal{Z}^{(k)}\right)-L_{1}\left(\mathcal{L}^{(k+1)},\mathcal{H}^{(k+1)},\mathcal{E}^{(k)},\mathcal{Y}^{(k)},\mathcal{Z}^{(k)}\right) \\
\leq& -\tfrac{\mu_{2}}{2}\left\|\mathcal{E}^{(k+1)}-\mathcal{E}^{(k)}\right\|_{F}^{2}. 
\end{aligned}\label{proofPCAZ}
\end{equation}
Lastly, we have 
\begin{equation}
\begin{aligned}
&L_{1}\left(\mathcal{L}^{(k+1)},\mathcal{H}^{(k+1)},\mathcal{E}^{(k+1)},\mathcal{Y}^{(k+1)},\mathcal{Z}^{(k+1)}\right)\\
\quad&-L_{1}\left(\mathcal{L}^{(k+1)},\mathcal{H}^{(k+1)},\mathcal{E}^{(k+1)},\mathcal{Y}^{(k)}, \mathcal{Z}^{(k)}\right)   \\
= & \left\langle\mathcal{Y}^{(k+1)}-\mathcal{Y}^{(k)}, \mathcal{L}^{(k+1)}-\mathcal{H}^{(k+1)}\right\rangle+\left\langle\mathcal{Z}^{(k+1)}-\mathcal{Z}^{(k)}, \mathcal{L}^{(k+1)}+\mathcal{E}^{(k+1)}-\mathcal{X}\right\rangle \\
= &\tfrac{1}{\mu_{1}} \left\|\mathcal{Y}^{(k+1)}-\mathcal{Y}^{(k)}\right\|_{F}^{2}+\tfrac{1}{\mu_{2}} \left\|\mathcal{Z}^{(k+1)}-\mathcal{Z}^{(k)}\right\|_{F}^{2}.
\end{aligned}\label{proofPCAY}
\end{equation}

By putting together \eqref{proofPCAL}, \eqref{proofPCAH}, \eqref{proofPCAZ},  \eqref{proofPCAY}, 
and (17), we have
\begin{equation}
\begin{aligned}
& L_{1}\left(\mathcal{L}^{(k+1)},\mathcal{H}^{(k+1)},\mathcal{E}^{(k+1)},\mathcal{Y}^{(k+1)},\mathcal{Z}^{(k+1)}\right)  \\
\leq &L_{1}\left(\mathcal{L}^{(k)},\mathcal{H}^{(k)},\mathcal{E}^{(k)},\mathcal{Y}^{(k)},\mathcal{Z}^{(k)}\right) 
-c_1\left\| \mathcal{L}^{(k+1)}-\mathcal{L}^{(k)} \right\|_{F}^{2}-c_2\left\|\mathcal{H}^{(k+1)}-\mathcal{H}^{(k)}\right\|_{F}^{2} \\
&-c_3 \left\|\mathcal{E}^{(k+1)}-\mathcal{E}^{(k)}\right\|_{F}^{2}+c_4 \left\|\mathcal{Z}^{(k+1)}-\mathcal{Z}^{(k)}\right\|_{F}^{2},\\
\end{aligned} \label{LTRPCA}
\end{equation}
where $c_{1}=\tfrac{\mu_{1}+\mu_{2}}{2}-\tfrac{2  n_{(2)} }{\mu_1\delta^{4}}$, $c_{2}=\tfrac{\mu_{1}}{2}-\tfrac{3M}{\delta^{3}}-\tfrac{4M^2}{\mu_{1}\delta^{6}}$,  $c_{3}=\tfrac{\mu_2}{2}$ and $c_4=\tfrac{1}{\mu_2}$. 
\end{proof}

\paragraph{Proof of Lemma 3 }
\begin{proof}
By the optimal condition of $\mathcal{L}$ in  iterative steps (11), there exits  $\mathcal{Q}^{(k+1)} \in \partial(\left\|(\mathcal{L}^{(k+1)})\right\|_{*})$ such that
\begin{equation}
\tfrac{\mathcal{Q}^{(k+1)}}{\left\|\mathcal{H}^{(k)}\right\|_F} +\mu_1(\mathcal{L}^{(k+1)}-\mathcal{H}^{(k)})+\mathcal{Y}^{(k)}+\mu_2(\mathcal{L}^{(k+1)}+\mathcal{E}^{(k)}-\mathcal{X})+\mathcal{Z}^{(k)}=\mathcal{O}. \label{con:Q1}
\end{equation}
We denote
\begin{equation}
\begin{aligned}
&\mathcal{V}_{1}^{(k+1)}:=\tfrac{\mathcal{Q}^{(k+1)}}{\left\|\mathcal{H}^{(k+1)}\right\|_F} +\mu_1\left(\mathcal{L}^{(k+1)}-\mathcal{H}^{(k+1)}\right)+\mathcal{Y}^{(k+1)}\\
 \quad&+\mu_2(\mathcal{L}^{(k+1)}+\mathcal{E}^{(k+1)}-\mathcal{X})+\mathcal{Z}^{(k+1)},    
\end{aligned}
  \label{con:Q2}
\end{equation}
which belongs to $\partial_{\mathcal{L}}L_{1}\left(\mathcal{L}^{(k+1)}, \mathcal{H}^{(k+1)}, \mathcal{E}^{(k+1)}, \mathcal{Y}^{(k+1)}, \mathcal{Z}^{(k+1)}\right)$ by the definition of subgradient.
Combining \eqref{con:Q1} and \eqref{con:Q2} yields 
\begin{equation*}
\begin{aligned}
 \mathcal{V}_{1}^{(k+1)} 
 =&\left(\tfrac{1}{\left\|\mathcal{H}^{(k+1)}\right\|_F}-\tfrac{1}{\left\|\mathcal{H}^{(k)}\right\|_F}\right) \mathcal{Q}^{(k+1)}
 +\mu_1(-\mathcal{H}^{(k+1)}+\mathcal{H}^{(k)})+\mathcal{Y}^{(k+1)}
 \\
 &\quad-\mathcal{Y}^{(k)} +\mu_2(\mathcal{E}^{(k+1)}-\mathcal{E}^{(k)})+\mathcal{Z}^{(k+1)}-\mathcal{Z}^{(k)}. 
\end{aligned}
\end{equation*}
When expressed by skinny t-SVD, $\mathcal{A}=\mathcal{U}*\mathcal{S}*\mathcal{V}^{*}$ has its subgradient defined by
\begin{equation}
    \partial(\left\|\mathcal{A}\right\|_{*})=\left\{\mathcal{U}*\mathcal{V}^{*}+\mathcal{J} \ | \ \mathcal{U}^**\mathcal{J}=\mathcal{O},\mathcal{J}*\mathcal{V}=\mathcal{O}, \|\mathcal{J}\| \leq 1\right\}.
    \label{eq:partial_nuclnTRPCA}
\end{equation}

Additionally,  for any tensor  $\mathcal A \in \mathbb{R}^{n_1\times n_2\times n_3}$, we have
 \begin{equation}
    \|\mathcal{A}\|_F = \tfrac{1}{\sqrt{n_3}}\|\overline{\mathbf{A}}\|_F\leq \sqrt{\tfrac{\mathrm{rank}(\overline{\mathbf{A}})}{n_3}}\|\overline{\mathbf{A}}\|\leq 
    \sqrt{n_{(2)}}\|\mathcal{A}\|.
    \label{FinequalityTRPCA}
\end{equation}
Based on \eqref{eq:partial_nuclnTRPCA} and \eqref{FinequalityTRPCA},
we define
  the skinny t-SVD of $\mathcal{X}^{(k+1)}$ by $\mathcal{X}^{(k+1)} = \mathcal{U}^{(k+1)}*\mathcal{S}^{(k+1)}*(\mathcal{V}^{(k+1)})^{*}$, then we get
\begin{equation*}
\begin{aligned}
\|\mathcal{Q}^{(k+1)}\|_F & = \|\mathcal{U}^{(k+1)}*(\mathcal{V}^{(k+1)})^{*}+\mathcal{J}^{(k+1)}\|_F \\
 & \leq \sqrt{n_{(2)}}\|\mathcal{U}^{(k+1)}*(\mathcal{V}^{(k+1)})^{*}+\mathcal{J}^{(k+1)}\| \\
&\leq   \sqrt{n_{(2)}}\|\mathcal{U}^{(k+1)}*(\mathcal{V}^{(k+1)})^{*}\|+\sqrt{n_{(2)}}\|\mathcal{J}^{(k+1)}\|\\ 
&\leq  2\sqrt{n_{(2)}} \leq 2\sqrt{n_{(2)}}.   
\end{aligned}
\end{equation*}

By the property of Frobenius norm, we get
\begin{equation}
\begin{aligned}
\left \|\mathcal{V}_{1}^{(k+1)}\right\|_{F} 
\leq &  \left\|\tfrac{1}{\left\|\mathcal{H}^{(k+1)}\right\|_F}-\tfrac{1}{\left\|\mathcal{H}^{(k)}\right\|_F}\right\|_F \| \mathcal{Q}^{(k+1)}\|_F + \mu_{1} \left\|\mathcal{H}^{(k+1)}-\mathcal{H}^{(k)}\right\|_{F} \\
&+\left\|\mathcal{Y}^{(k+1)}-\mathcal{Y}^{(k)}\right\|_{F} +\mu_{2}\left\|\mathcal{E}^{(k+1)}-\mathcal{E}^{(k)}\right\|_{F} 
 +\left\|\mathcal{Z}^{(k+1)}-\mathcal{Z}^{(k)}\right\|_{F}\\
 =&\left\|\tfrac{\left\|\mathcal{H}^{(k+1)}\right\|_F-\left\|\mathcal{H}^{(k)}\right\|_F}{\left\|\mathcal{H}^{(k+1)}\right\|_F\left\|\mathcal{H}^{(k)}\right\|_F}\right\|_{F}\left\| \mathcal{P}^{(k+1)}  \right\|_{F} 
  + \mu_{1} \left\|\mathcal{H}^{(k+1)}-\mathcal{H}^{(k)}\right\|_{F}\\
  &+\left\|\mathcal{Y}^{(k+1)}-\mathcal{Y}^{(k)}\right\|_{F} +\mu_{2}\left\|\mathcal{E}^{(k+1)}-\mathcal{E}^{(k)}\right\|_{F} 
 +\left\|\mathcal{Z}^{(k+1)}-\mathcal{Z}^{(k)}\right\|_{F}\\
\leq &(\tfrac{2\sqrt{n_{(2)}}}{\delta^{2}}+\mu_{1}) \left\|\mathcal{H}^{(k+1)}-\mathcal{H}^{(k)}\right\|_{F}+ \left\|\mathcal{Y}^{(k+1)}-\mathcal{Y}^{(k)}\right\|_{F} \\
 &+\mu_{2}\left\|\mathcal{E}^{(k+1)}-\mathcal{E}^{(k)}\right\|_{F} 
 +\left\|\mathcal{Z}^{(k+1)}-\mathcal{Z}^{(k)}\right\|_{F}.
\end{aligned} \label{eq:v1}
\end{equation}

Choosing $\mathcal{U}^{(k+1)}\in \partial (\|\mathcal{E}^{(k+1)}\|_{1})$, we define $\mathcal{V}_{2}^{(k+1)}$, $\mathcal{V}_{3}^{(k+1)}$, $\mathcal{V}_{4}^{(k+1)}$, $\mathcal{V}_{5}^{(k+1)}$ as follows:
 
\begin{equation*}
    \begin{split}
        \mathcal{V}_{2}^{(k+1)}& :=  -\tfrac{\left\|\mathcal{L}^{(k+1)}\right\|_{*}}{\left\|\mathcal{H}^{(k+1)}\right\|_{F}^{3}} \mathcal{H}^{(k+1)}-\mu_1\left(\mathcal{L}^{(k+1)}-\mathcal{H}^{(k+1)}\right)-\mathcal{Y}^{(k+1)}, \\
        \mathcal{V}_{3}^{(k+1)}& := \lambda\mathcal{U}^{(k+1)}+\mu_2(\mathcal{L}^{(k+1)}+\mathcal{E}^{(k+1)}-\mathcal{X}) +\mathcal{Z}^{(k+1)}, \\
        \mathcal{V}_{4}^{(k+1)}& := \mathcal{L}^{(k+1)}-\mathcal{H}^{(k+1)},\\
         \mathcal{V}_{5}^{(k+1)}& := \mathcal{L}^{(k+1)}+\mathcal{E}^{(k+1)}-\mathcal{X}.
    \end{split}
\end{equation*}

By the optimal condition of $L_{1}$ in (10), we know that
\begin{equation}
    \begin{split}
         \mathcal{V}_{2}^{(k+1)}& =\mathcal{Y}^{(k)}-\mathcal{Y}^{(k+1)}, \quad \quad \quad \quad 
        \mathcal{V}_{3}^{(k+1)} =\mathcal{Z}^{(k+1)}-\mathcal{Z}^{(k)}, \\
        \mathcal{V}_{4}^{(k+1)}& =\tfrac{1}{\mu_1}(\mathcal{Y}^{(k+1)}-\mathcal{Y}^{(k)}), \quad \quad 
        \mathcal{V}_{5}^{(k+1)} =\tfrac{1}{\mu_2}(\mathcal{Z}^{(k+1)}-\mathcal{Z}^{(k)}).
    \end{split}\label{eq:v2-v5}
\end{equation}

By the subgradient definition, we get 
\begin{equation*}
    \begin{split}
         \mathcal{V}_{2}^{(k+1)} & \in \partial_{\mathcal{H}}L_{1}\left(\mathcal{L}^{(k+1)}, \mathcal{H}^{(k+1)}, \mathcal{E}^{(k+1)}, \mathcal{Y}^{(k+1)}, \mathcal{Z}^{(k+1)}\right),\\
       \mathcal{V}_{3}^{(k+1)} &\in \partial_{\mathcal{E}}L_{1}\left(\mathcal{L}^{(k+1)}, \mathcal{H}^{(k+1)}, \mathcal{E}^{(k+1)}, \mathcal{Y}^{(k+1)}, \mathcal{Z}^{(k+1)}\right) \\
        \mathcal{V}_{4}^{(k+1)} &\in \partial_{\mathcal{Y}}L_{1}\left(\mathcal{L}^{(k+1)}, \mathcal{H}^{(k+1)}, \mathcal{E}^{(k+1)}, \mathcal{Y}^{(k+1)}, \mathcal{Z}^{(k+1)}\right),\\
        \mathcal{V}_{5}^{(k+1)} &\in \partial_{\mathcal{Z}}L_{1}\left(\mathcal{L}^{(k+1)}, \mathcal{H}^{(k+1)}, \mathcal{E}^{(k+1)}, \mathcal{Y}^{(k+1)}, \mathcal{Z}^{(k+1)}\right).
    \end{split}
\end{equation*}

Let $\mathcal{V}^{(k+1)}=\left(\mathcal{V}_{1}^{(k+1)}, \mathcal{V}_{2}^{(k+1)}, \mathcal{V}_{3}^{(k+1)}, \mathcal{V}_{4}^{(k+1)}, \mathcal{V}_{5}^{(k+1)}\right)^{T}$, we get 
\begin{equation}
\mathcal{V}^{(k+1)}\in \partial L_{1}\left(\mathcal{L}^{(k+1)}, \mathcal{H}^{(k+1)}, \mathcal{E}^{(k+1)}, \mathcal{Y}^{(k+1)}, \mathcal{Z}^{(k+1)}\right).  
\end{equation}

Incorporating Lemma 1, \eqref{eq:v1} and \eqref{eq:v2-v5}, we have
\begin{equation*}
\begin{aligned}
\left\|\mathcal{V}^{(k+1)}\right\|_{F}^{2}=&\left\|\mathcal{V}_{1}^{(k+1)}\right\|_{F}^{2}+\left\|\mathcal{V}_{2}^{(k+1)}\right\|_{F}^{2}+\left\|\mathcal{V}_{3}^{(k+1)}\right\|_{F}^{2}+\left\|\mathcal{V}_{4}^{(k+1)}\right\|_{F}^{2}+\left\|\mathcal{V}_{5}^{(k+1)}\right\|_{F}^{2} \\
\leq&\kappa_{3}\left\| \mathcal{L}^{(k+1)}-\mathcal{L}^{(k)}\right \|_{F}^{2} +\kappa_4\left\|\mathcal{H}^{(k+1)}-\mathcal{H}^{(k)}\right\|_{F}^{2}\\
+&\kappa_5\left\|\mathcal{E}^{(k+1)}-\mathcal{E}^{(k)}\right\|_{F}^{2}+\kappa_6\left\|\mathcal{Z}^{(k+1)}-\mathcal{Z}^{(k)}\right\|_{F}^{2},
\end{aligned}
\end{equation*}
where $\kappa_{3}=\tfrac{2n_{(2)}(3\mu_1^2+1)}{\mu_1^2\delta^4}$, $\kappa_{4}=\tfrac{16n_{(2)}}{\delta^{4}}+4\mu_{1}^{2}+\tfrac{4M^{2}(3\mu_1^2+1)}{\mu_1^2\delta^6}$, $\kappa_{5}=2\mu_2^2$ and $\kappa_{6}=3+\tfrac{1}{\mu_2^2}$.
 we can choose a proper value $\kappa>0$ such that the desired inequality (19)) holds.
\end{proof}

\paragraph{Proof of Theorem 2}
\begin{proof}
 {\rm \textrm{(i)}}  We first show $\{\mathcal{Y}^{(k)}\}$ is bounded. From \eqref{con:Yk}, we have
$$\|\mathcal{Y}^{(k)}\|_{F}=\left\|-\tfrac{\left\|\mathcal{L}^{(k)}\right\|_{*}}{\left\|\mathcal{H}^{(k)}\right\|_{F}^{3}} \mathcal{H}^{(k)}\right\|_{F}=\tfrac{\left\|\mathcal{L}^{(k)}\right\|_{*}}{\left\|\mathcal{H}^{(k)}\right\|_{F}^{2}}, $$
which suggests that $\{\mathcal{Y}^{(k)}\}$ is bounded under assumptions (C1) and (C2). Therefore, $\{\mathcal{H}^{(k)}\}$ is also bounded due to the $\mathcal{H}$-update in (11).   
 
 We further use the optimality condition of $\mathcal{E}$ in (11) to show that $\{\mathcal{E}^{(k)}\}$ is bounded.
In particular, $\mathcal{E}^{(k+1)}$ is updated by 
\begin{equation*}
 \mathcal{E}^{(k+1)}_{ijl}=\left\{\begin{array}{cc}
\mathcal{X}_{ijl}-\mathcal{L}^{(k+1)}_{ijl}-\tfrac{\mathcal{Z}^{(k)}_{ijl}}{\mu_2}-\tfrac{\lambda}{\mu_2}, & \mathcal{X}_{ijl}-\mathcal{L}^{(k+1)}_{ijl}-\tfrac{\mathcal{Z}^{(k)}_{ijl}}{\mu_2} \geq \tfrac{\lambda}{\mu_2} \\
0, & -\tfrac{\lambda}{\mu_2} \leq \mathcal{X}_{ijl}-\mathcal{L}^{(k+1)}_{ijl}-\tfrac{\mathcal{Z}^{(k)}_{ijl}}{\mu_2} \leq \tfrac{\lambda}{\mu_2} \\
\mathcal{X}_{ijl}-\mathcal{L}^{(k+1)}_{ijl}-\tfrac{\mathcal{Z}^{(k)}_{ijl}}{\mu_2}+\tfrac{\lambda}{\mu_2}, & \mathcal{X}_{ijl}-\mathcal{L}^{(k+1)}_{ijl}-\tfrac{\mathcal{Z}^{(k)}_{ijl}}{\mu_2}<-\tfrac{\lambda}{\mu_2},
\end{array}\right.   
\end{equation*}
where $\mathcal{X}_{ijl}$, $\mathcal{L}_{ijl},$ and $\mathcal{Z}_{ijl}$ is the $(i,j,l)$ element of third-order tensor $\mathcal{X}$, $\mathcal{L}$ and $\mathcal{Z}$ respectively. 
In other words, we have
\begin{equation*}
\begin{aligned}
 &\mathcal{E}^{(k+1)}_{ijl}-\mathcal{X}_{ijl}+\mathcal{L}^{(k+1)}_{ijl}+\tfrac{\mathcal{Z}^{(k)}_{ijl}}{\mu_2}\\
=&\left\{\begin{array}{cc}
-\tfrac{\lambda}{\mu_2}, & \mathcal{X}_{ijl}-\mathcal{L}^{(k+1)}_{ijl}-\tfrac{\mathcal{Z}^{(k)}_{ijl}}{\mu_2} \geq \tfrac{\lambda}{\mu_2} \\
-\mathcal{X}_{ijl}+\mathcal{L}^{(k+1)}_{ijl}+\tfrac{\mathcal{Z}^{(k)}_{ijl}}{\mu_2},& -\tfrac{\lambda}{\mu_2} \leq \mathcal{X}_{ijl}-\mathcal{L}^{(k+1)}_{ijl}-\tfrac{\mathcal{Z}^{(k)}_{ijl}}{\mu_2} \leq \tfrac{\lambda}{\mu_2} \\
\tfrac{\lambda}{\mu_2}. & \mathcal{X}_{ijl}-\mathcal{L}^{(k+1)}_{ijl}-\tfrac{\mathcal{Z}^{(k)}_{ijl}}{\mu_2}<-\tfrac{\lambda}{\mu_2}.
\end{array}\right.     
\end{aligned}
\end{equation*}
Therefore, $\left(\mathcal{E}^{(k+1)}_{ijl}-\mathcal{X}_{ijl}+\mathcal{L}^{(k+1)}_{ijl}+\tfrac{\mathcal{Z}^{(k)}_{ijl}}{\mu_2}\right)^{2} \leq \tfrac{\lambda^{2}}{\mu_{2}^{2}}$, which implies that 
\begin{equation*}
    \begin{split}
    \left\|\mathcal{Z}^{(k+1)}\right\|_{F}^{2} = \mu_2^2\left\|\mathcal{L}^{(k+1)}+\mathcal{E}^{(k+1)}-\mathcal{X}+\tfrac{\mathcal{Z}^{(k)}}{\mu_2}\right\|_{F}^{2}\leq n_{1}n_{2}n_{3}\lambda^{2},
    \end{split}
    \label{con:Zlimit}
\end{equation*} is bounded. It further follows from the assumption A1 that $\mathcal E$ is bounded due to the boundedness of $\mathcal L$ and $\mathcal Z$.

\rm{\textrm{(ii)}} 
Since the sequence $\{\mathcal{L}^{(k)},\mathcal{H}^{(k)},\mathcal{E}^{(k)},\mathcal{Y}^{(k)},\mathcal{Z}^{(k)}\}$ is bounded, Bolzano-Weierstrass theorem states that there exists a subsequence defined as 
$$\{\mathcal{L}^{(k_{i})}, \mathcal{H}^{(k_{i})}, \mathcal{E}^{(k_{i})}, \mathcal{Y}^{(k_{i})}, \mathcal{Z}^{(k_{i})}\} \rightarrow \{\mathcal{L}^{*}, \mathcal{H}^{*}, \mathcal{E}^{*}, \mathcal{Y}^{*}, \mathcal{Z}^{*}\}.$$ 

It follows from Lemma 2 that $L_{1}\left(\mathcal{L}^{(k+1)},\mathcal{H}^{(k+1)},\mathcal{E}^{(k+1)},\mathcal{Y}^{(k+1)},\mathcal{Z}^{(k+1)} \right)$ has upper bound if $\|\mathcal{Z}^{(k+1)}-\mathcal{Z}^{(k)} \|_{F}^{2}\rightarrow 0$.
Using (10), we get
\begin{equation*}
    \begin{aligned}
&L_{1}\left(\mathcal{L},\mathcal{H},\mathcal{E},\mathcal{Y},\mathcal{Z}\right)\\
    =&\tfrac{\|\mathcal{L}\|_{*}}{\left\|\mathcal{H}^{(k)}\right\|_{F}}+\tfrac{\mu_{1}}{2}\left\|\mathcal{L}-\mathcal{H}^{(k)}+\tfrac{\mathcal{Y}^{(k)}}{\mu_1}\right\|_{F}^{2} 
       +\tfrac{\mu_{2}}{2}\left\|\mathcal{L}+\mathcal{E}^{(k)}
       -\mathcal{X}+\tfrac{\mathcal{Z}^{(k)}}{\mu_2}\right\|_{F}^{2}\\
       &-\tfrac{1}{2\mu_1}\|\mathcal{Y}^{(k)}\|_F^2-\tfrac{1}{2\mu_2}\|\mathcal{Z}^{(k)}\|_F^2 \\
       \geq &-\tfrac{1}{2\mu_1}\|\mathcal{Y}^{(k)}\|_F^2-\tfrac{1}{2\mu_2}\|\mathcal{Z}^{(k)}\|_F^2,
    \end{aligned}
\end{equation*}
due to the boundness of $\{\mathcal{Y}^{(k)}\}$ and $\{\mathcal{Z}^{(k)}\}$. 
Let $k\rightarrow \infty$, by $L_{1}\left(\mathcal{L},\mathcal{H},\mathcal{E},\mathcal{Y},\mathcal{Z}\right)$ having a lower bound, we know that $\sum_{j=0}^{k}\| \mathcal{L}^{(j+1)}-\mathcal{L}^{(j)} \|_{F}^{2}$, $\sum_{j=0}^{k}\| \mathcal{H}^{(j+1)}-\mathcal{H}^{(j)} \|_{F}^{2}$ and $\sum_{j=0}^{k}\| \mathcal{E}^{(j+1)}-\mathcal{E}^{(j)} \|_{F}^{2}$ are finite, which implies that $\| \mathcal{L}^{(k+1)}-\mathcal{L}^{(k)} \|_{F}^{2} \rightarrow 0$, $\| \mathcal{H}^{(k+1)}-\mathcal{H}^{(k)} \|_{F}^{2} \rightarrow 0$ and $\| \mathcal{E}^{(k+1)}-\mathcal{E}^{(k)} \|_{F}^{2} \rightarrow 0$ as $k\rightarrow \infty$. Then we can  get $\| \mathcal{Y}^{(k+1)}-\mathcal{Y}^{(k)} \|_{F}^{2} \rightarrow 0$ due to (17). 

Therefore, we have 
$$\{\mathcal{L}^{(k_{i}+1)}, \mathcal{H}^{(k_{i}+1)}, \mathcal{E}^{(k_{i}+1)}, \mathcal{Y}^{(k_{i}+1)}, \mathcal{Z}^{(k_{i}+1)}\} \rightarrow \{\mathcal{L}^{*}, \mathcal{H}^{*}, \mathcal{E}^{*}, \mathcal{Y}^{*}, \mathcal{Z}^{*}\},$$ 
which implies that 
$\| \mathcal{H}^{(k_i+1)}-\mathcal{H}^{(k_i)} \|_{F} \rightarrow 0$, $\| \mathcal{E}^{(k_i+1)}-\mathcal{E}^{(k_i)} \|_{F} \rightarrow 0$, $\| \mathcal{Y}^{(k_i+1)}-\mathcal{Y}^{(k_i)} \|_{F} \rightarrow 0$ and $\| \mathcal{Z}^{(k_i+1)}-\mathcal{Z}^{(k_i)} \|_{F} \rightarrow 0$. 
Hence Lemma 3 guarantees that  the zero tensor $\mathcal{O}\in \partial  L_{1}(\mathcal{X}^{\ast}, \mathcal{H}^{\ast}, \mathcal{A}^{\ast}). $
\end{proof}

\subsection{Convergence analysis of TNF$+$}
\paragraph{Proof of Lemma 4}
\begin{proof}
As the proof of (27) is the same as (17), we omit it. We only prove for (28).
The optimality condition of the $\mathcal{D}$ subproblem in (23)  indicates that
\begin{equation}
-\tfrac{\left\|\mathcal{E}^{(k+1)}\right\|_{1}}{\left\|\mathcal{D}^{(k+1)}\right\|_{F}^{3}} \mathcal{D}^{(k+1)}+\mu_3\left(\mathcal{D}^{(k+1)}-\mathcal{E}^{(k+1)}-\tfrac{\mathcal{U}^{(k)}}{\mu_3}\right)=\mathcal{O}, \label{Doptimal}
\end{equation}
where $\mathcal{O}\in \mathbb{R}^{n_1\times n_2\times n_3}$ is the zero tensor.
\end{proof}
Using the dual update $\mathcal{U}^{(k+1)}=\mathcal{U}^{(k)}+\mu_{3}\left(\mathcal{E}^{(k+1)}-\mathcal{D}^{(k+1)}\right)$, we have
\begin{equation}
\mathcal{U}^{(k+1)}=-\lambda\tfrac{\left\|\mathcal{E}^{(k+1)}\right\|_{1}}{\left\|\mathcal{D}^{(k+1)}\right\|_{F}^{3}} \mathcal{D}^{(k+1)},  \label{con:Uk1} 
\end{equation}
which directly deduces 
\begin{equation}
\mathcal{U}^{(k)}=-\tfrac{\left\|\mathcal{E}^{(k)}\right\|_{1}}{\left\|\mathcal{D}^{(k)}\right\|_{F}^{3}} \mathcal{D}^{(k)}. \label{con:Uk}   
\end{equation}
It is straightforward to have
\begin{equation}
\begin{aligned}
&\left\|\mathcal{U}^{(k+1)}-\mathcal{U}^{(k)}\right\|_F
=\lambda\left\|\tfrac{\left\|\mathcal{E}^{(k+1)}\right\|_1}{\left\|\mathcal{D}^{(k+1)}\right\|_F^3} \mathcal{D}^{(k+1)}-\tfrac{\left\|\mathcal{E}^{(k)}\right\|_1}{\left\|\mathcal{D}^{(k)}\right\|_F^3} \mathcal{D}^{(k)}\right\|_F \\
 \leq&\lambda\left\|\tfrac{\mathcal{D}^{(k+1)}}{\left\|\mathcal{D}^{(k+1)}\right\|_F^3}\left(\left\|\mathcal{E}^{(k+1)}\right\|_1-\left\|\mathcal{E}^{(k)}\right\|_*\right)\right\|_F+\lambda\left\|\mathcal{E}^{(k)}\right\|_1\left\|\tfrac{\mathcal{D}^{(k+1)}}{\left\|\mathcal{D}^{(k+1)}\right\|_F^3}-\tfrac{\mathcal{D}^{(k)}}{\left\|\mathcal{D}^{(k)}\right\|_F^3}\right\|_F \\
=&\tfrac{\lambda}{\left\|\mathcal{D}^{(k+1)}\right\|_F^2}\left| \| \mathcal{E}^{(k+1)}\left\|_1-\right\| \mathcal{E}^{(k)}\|_{1}\right|+\lambda\left\|\mathcal{E}^{(k)}\right\|_{1}\left\|\tfrac{\mathcal{D}^{(k+1)}}{\left\|\mathcal{D}^{(k+1)}\right\|_F^3}-\tfrac{\mathcal{D}^{(k)}}{\left\|\mathcal{D}^{(k)}\right\|_F^3}\right\|_F.  
\end{aligned} \label{con:proofUstep}
\end{equation}
Notice that 
\begin{equation}
\left| \|\mathcal{E}^{(k+1)}\|_1-\|\mathcal{E}^{(k)}\|_1\right| \leq \left\|\mathcal{E}^{(k+1)}-\mathcal{E}^{(k)}\right\|_1 \leq \sqrt{n_1n_2n_3} \left\|\mathcal{E}^{(k+1)}-\mathcal{E}^{(k)}\right\|_{F} \label{con:proofE1}
\end{equation} 
and
\begin{equation}
\left\|\mathcal{E}^{(k)}\right\|_{1}\left\|\tfrac{\mathcal{D}^{(k+1)}}{\left\|\mathcal{D}^{(k+1)}\right\|_F^3}-\tfrac{\mathcal{D}^{(k)}}{\left\|\mathcal{D}^{(k)}\right\|_F^3}\right\|_F \leq \tfrac{2\|\mathcal{E}^{(k)}\|_{*}}{\delta^{3}}\left\|\mathcal{D}^{(k+1)}-\mathcal{D}^{(k)}\right\|_{F}. \label{con:proofE2}
\end{equation}
Clearly, we have
 \begin{equation*}
    \left\|\mathcal{U}^{(k+1)}-\mathcal{U}^{(k)}\right\|_{F}^{2} \leq \tfrac{2 \lambda^2 n_1n_{2}n_3}{\delta_2^{4}}\left\|\mathcal{E}^{(k+1)}-\mathcal{E}^{(k)}\right\|_F^2+\tfrac{4\lambda^2 m^{2}}{\delta_2^{6}}\left\|\mathcal{D}^{(k+1)}-\mathcal{D}^{(k)}\right\|_{F}^{2},
 \end{equation*}
where we use $\sup_k \{\|\mathcal{E}^{(k)}\|_{1}\}\leq m$ from A3.

\paragraph{Proof of Lemma 5}
\begin{proof}
Similar to proof of Lemma 2, we can get
\begin{equation}
 \begin{aligned}
& L_{2}\left(\mathcal{L}^{(k+1)},\mathcal{H}^{(k)},\mathcal{E}^{(k)},\mathcal{D}^{(k)},\mathcal{Y}^{(k)},\mathcal{Z}^{(k)},\mathcal{U}^{(k)}\right)\\ \leq &L_{2}\left(\mathcal{L}^{(k)},\mathcal{H}^{(k)},\mathcal{E}^{(k)},\mathcal{D}^{(k)},\mathcal{Y}^{(k)},\mathcal{Z}^{(k)},\mathcal{U}^{(k)}\right) - \tfrac{\mu_1+\mu_2}{2}\left\|\mathcal{L}^{(k+1)}-\mathcal{L}^{(k)}\right\|_{F}^{2}.   
\end{aligned}\label{proofPCAL2}
\end{equation}

\begin{equation}
\begin{aligned}
&L_{2}\left(\mathcal{L}^{(k+1)},\mathcal{H}^{(k+1)},\mathcal{E}^{(k)},\mathcal{D}^{(k)},\mathcal{Y}^{(k)},\mathcal{Z}^{(k)},\mathcal{U}^{(k)}\right) \\
\leq&L_{2}\left(\mathcal{L}^{(k+1)},\mathcal{H}^{(k)},\mathcal{E}^{(k)},\mathcal{D}^{(k)},\mathcal{Y}^{(k)},\mathcal{Z}^{(k)},\mathcal{U}^{(k)}\right)  
 -(\tfrac{\mu_{1}}{2}-\tfrac{3M}{\delta_1^3})\left\|\mathcal{H}^{(k+1)}-\mathcal{H}^{(k)}\right\|_{F}^{2},
\end{aligned}\label{proofPCAH2} 
\end{equation}

The function $L_{2}\left(\mathcal{L}^{(k+1)},\mathcal{H}^{(k+1)},\mathcal{E},\mathcal{D}^{(k)},\mathcal{Y}^{(k)},\mathcal{Z}^{(k)},\mathcal{U}^{(k)}\right)$ defined in (22)  with fixed $\mathcal{L}^{(k+1)}$, $\mathcal{H}^{(k+1)}$, $\mathcal{D}^{(k)}$, $\mathcal{Y}^{(k)}$, $\mathcal{Z}^{(k)}$ and $\mathcal{U}^{(k)}$ have a $\ell_1$-norm term with two quadratic terms, hence it is strongly convex in terms of $\mathcal{E}$ with constant $\mu_{2}+\mu_3$, thus leading to
\begin{equation}
\begin{aligned}
&L_{2}\left(\mathcal{L}^{(k+1)},\mathcal{H}^{(k+1)},\mathcal{E}^{(k+1)},\mathcal{D}^{(k)},\mathcal{Y}^{(k)},\mathcal{Z}^{(k)},\mathcal{U}^{(k)}\right) \\
\leq&L_{2}\left(\mathcal{L}^{(k+1)},\mathcal{H}^{(k+1)},\mathcal{E}^{(k)},\mathcal{D}^{(k)},\mathcal{Y}^{(k)},\mathcal{Z}^{(k)},\mathcal{U}^{(k)}\right) -\tfrac{\mu_2+\mu_{3}}{2}\left\|\mathcal{E}^{(k+1)}-\mathcal{E}^{(k)}\right\|_{F}^{2}. 
\end{aligned}\label{proofPCAZ2}
\end{equation}
As for the function $L_{2}\left(\mathcal{L}^{(k+1)},\mathcal{H}^{(k+1)},\mathcal{E}^{(k+1)},\mathcal{D},\mathcal{Y}^{(k)},\mathcal{Z}^{(k)},\mathcal{U}^{(k)}\right)$  (22), we use the similar computation of \eqref{eq:A10TRPCA} to get
\begin{equation}
\begin{aligned}
&L_{2}\left(\mathcal{L}^{(k+1)},\mathcal{H}^{(k+1)},\mathcal{E}^{(k+1)},\mathcal{D}^{(k+1)},\mathcal{Y}^{(k)},\mathcal{Z}^{(k)},\mathcal{U}^{(k)}\right) \\
\leq&L_{2}\left(\mathcal{L}^{(k+1)},\mathcal{H}^{(k+1)},\mathcal{E}^{(k+1)},\mathcal{D}^{(k)},\mathcal{Y}^{(k)},\mathcal{Z}^{(k)},\mathcal{U}^{(k)}\right) \\ 
 &\quad-(\tfrac{\mu_{3}}{2}-\tfrac{3m}{\delta_2^3})\left\|\mathcal{D}^{(k+1)}-\mathcal{D}^{(k)}\right\|_{F}^{2},
\end{aligned}\label{proofPCAH22} 
\end{equation}
In addition, we get
\begin{equation}
\begin{aligned}
&L_{2}\left(\mathcal{L}^{(k+1)},\mathcal{H}^{(k+1)},\mathcal{E}^{(k+1)},\mathcal{D}^{(k+1)},\mathcal{Y}^{(k+1)},\mathcal{Z}^{(k+1)},\mathcal{U}^{(k+1)}\right)\\
&-L_{2}\left(\mathcal{L}^{(k+1)},\mathcal{H}^{(k+1)},\mathcal{E}^{(k+1)},,\mathcal{D}^{(k+1)},\mathcal{Y}^{(k)},\mathcal{Z}^{(k)},\mathcal{U}^{(k)}\right)   \\
= & \left\langle\mathcal{Y}^{(k+1)}-\mathcal{Y}^{(k)}, \mathcal{L}^{(k+1)}-\mathcal{H}^{(k+1)}\right\rangle+\left\langle\mathcal{Z}^{(k+1)}-\mathcal{Z}^{(k)}, \mathcal{L}^{(k+1)}+\mathcal{E}^{(k+1)}-\mathcal{X}\right\rangle \\
&+\left\langle\mathcal{U}^{(k+1)}-\mathcal{U}^{(k)}, \mathcal{E}^{(k+1)}-\mathcal{D}^{(k+1)}\right\rangle\\
= &\tfrac{1}{\mu_{1}} \left\|\mathcal{Y}^{(k+1)}-\mathcal{Y}^{(k)}\right\|_{F}^{2}+\tfrac{1}{\mu_{2}} \left\|\mathcal{Z}^{(k+1)}-\mathcal{Z}^{(k)}\right\|_{F}^{2}+\tfrac{1}{\mu_{3}} \left\|\mathcal{U}^{(k+1)}-\mathcal{U}^{(k)}\right\|_{F}^{2}.\\
\end{aligned}\label{proofPCAY2}
\end{equation}
By putting together \eqref{proofPCAL2}, \eqref{proofPCAZ2}, \eqref{proofPCAH2}, \eqref{proofPCAY2} 
with (27), (28), we obtain
\begin{equation*}
\begin{aligned}
& L_{2}\left(\mathcal{L}^{(k+1)},\mathcal{H}^{(k+1)},\mathcal{E}^{(k+1)},\mathcal{D}^{(k+1)},\mathcal{Y}^{(k+1)},\mathcal{Z}^{(k+1)},\mathcal{U}^{(k+1)} \right)\\
\leq &L_{2}\left(\mathcal{L}^{(k)},\mathcal{H}^{(k)},\mathcal{E}^{(k)},\mathcal{D}^{(k)},\mathcal{Y}^{(k)},\mathcal{Z}^{(k)},\mathcal{U}^{(k)},\right)-c_{5} \| \mathcal{L}^{(k+1)}-\mathcal{L}^{(k)} \|_{F}^{2} \\
&-c_{6} \|\mathcal{H}^{(k+1)}-\mathcal{H}^{(k)} \|_{F}^{2} 
-c_{7} \|\mathcal{E}^{(k+1)}-\mathcal{E}^{(k)} \|_{F}^{2}
-c_{8} \|\mathcal{D}^{(k+1)}-\mathcal{D}^{(k)} \|_{F}^{2}\\
&+c_{9} \|\mathcal{Z}^{(k+1)}-\mathcal{Z}^{(k)} \|_{F}^{2} 
\end{aligned} 
\end{equation*}
where $c_{5}=\tfrac{\mu_{1}+\mu_{2}}{2}-\tfrac{2  n_{(2)} }{\mu_1\delta^{4}}$, $c_{6}=\tfrac{\mu_{1}}{2}-\tfrac{3M}{\delta_1^{3}}-\tfrac{4M^2}{\mu_{1}\delta^{6}}$, $c_{7}=\tfrac{\mu_2+\mu_3}{2}-\tfrac{2 \lambda^2 n_1n_{2}n_3}{\mu_3\delta_2^{4}}$,
$c_8=\tfrac{\mu_{3}}{2}-\tfrac{3m}{\delta^{3}}-\tfrac{4\lambda^2m^2}{\mu_{3}\delta_2^{6}},$ and
$c_9=\tfrac{1}{\mu_2}$.
\end{proof}

\bibliographystyle{spmpsci}      
\bibliography{references} 

\begin{thebibliography}{10}
\providecommand{\url}[1]{{#1}}
\providecommand{\urlprefix}{URL }
\expandafter\ifx\csname urlstyle\endcsname\relax
  \providecommand{\doi}[1]{DOI~\discretionary{}{}{}#1}\else
  \providecommand{\doi}{DOI~\discretionary{}{}{}\begingroup
  \urlstyle{rm}\Url}\fi

\bibitem{anandkumar2014tensor}
Anandkumar, A., Ge, R., Hsu, D., Kakade, S.M., Telgarsky, M.: Tensor
  decompositions for learning latent variable models.
\newblock Journal of Machine Learning Research \textbf{15}, 2773--2832 (2014)

\bibitem{boyd2011distributed}
Boyd, S., Parikh, N., Chu, E., Peleato, B., Eckstein, J., et~al.: Distributed
  optimization and statistical learning via the alternating direction method of
  multipliers.
\newblock Foundations and Trends{\textregistered} in Machine Learning
  \textbf{3}(1), 1--122 (2011)

\bibitem{brandoni2020tensor}
Brandoni, D., Simoncini, V.: Tensor-train decomposition for image recognition.
\newblock Calcolo \textbf{57}, 1--24 (2020)

\bibitem{candes2011robust}
Cand{\`e}s, E.J., Li, X., Ma, Y., Wright, J.: Robust principal component
  analysis?
\newblock Journal of the ACM (JACM) \textbf{58}(3), 1--37 (2011)

\bibitem{chen2016direct}
Chen, C., He, B., Ye, Y., Yuan, X.: The direct extension of admm for
  multi-block convex minimization problems is not necessarily convergent.
\newblock Mathematical Programming \textbf{155}(1), 57--79 (2016)

\bibitem{chen2018sharing}
Chen, Y., Jin, X., Kang, B., Feng, J., Yan, S.: Sharing residual units through
  collective tensor factorization to improve deep neural networks.
\newblock In: International Joint Conference on Artificial Intelligence, pp.
  635--641 (2018)

\bibitem{de2008tensor}
De~Silva, V., Lim, L.H.: Tensor rank and the ill-posedness of the best low-rank
  approximation problem.
\newblock SIAM Journal on Matrix Analysis and Applications \textbf{30}(3),
  1084--1127 (2008)

\bibitem{du2021unifying}
Du, S., Xiao, Q., Shi, Y., Cucchiara, R., Ma, Y.: Unifying tensor factorization
  and tensor nuclear norm approaches for low-rank tensor completion.
\newblock Neurocomputing \textbf{458}, 204--218 (2021)

\bibitem{ely20155d}
Ely, G., Aeron, S., Hao, N., Kilmer, M.E.: {5D} seismic data completion and
  denoising using a novel class of tensor decompositions.
\newblock Geophysics \textbf{80}(4), V83--V95 (2015)

\bibitem{fan2001variable}
Fan, J., Li, R.: Variable selection via nonconcave penalized likelihood and its
  oracle properties.
\newblock Journal of the American statistical Association \textbf{96}(456),
  1348--1360 (2001)

\bibitem{friedland2018nuclear}
Friedland, S., Lim, L.H.: Nuclear norm of higher-order tensors.
\newblock Mathematics of Computation \textbf{87}(311), 1255--1281 (2018)

\bibitem{gandy2011tensor}
Gandy, S., Recht, B., Yamada, I.: Tensor completion and low-$n$-rank tensor
  recovery via convex optimization.
\newblock Inverse Problems \textbf{27}(2), 025010 (2011)

\bibitem{gao2023tensor}
Gao, K., Huang, Z.H.: Tensor robust principal component analysis via tensor
  fibered rank and minimization.
\newblock SIAM Journal on Imaging Sciences \textbf{16}(1), 423--460 (2023)

\bibitem{gu2017tvscad}
Gu, G., Jiang, S., Yang, J.: A {TVSCAD} approach for image deblurring with
  impulsive noise.
\newblock Inverse Problems \textbf{33}(12), 125008 (2017)

\bibitem{guo2011tensor}
Guo, W., Kotsia, I., Patras, I.: Tensor learning for regression.
\newblock IEEE Transactions on Image Processing \textbf{21}(2), 816--827 (2011)

\bibitem{han2022tensor}
Han, Y., Zhang, C.H.: Tensor principal component analysis in high dimensional
  {CP} models.
\newblock IEEE Transactions on Information Theory \textbf{69}(2), 1147--1167
  (2022)

\bibitem{huang2008asymptotic}
Huang, J., Horowitz, J.L., Ma, S.: Asymptotic properties of bridge estimators
  in sparse high-dimensional regression models.
\newblock Institute of Mathematical Statistics \textbf{36}(2), 587--613 (2008)

\bibitem{8884203}
Ji, Y., Wang, Q., Li, X., Liu, J.: A survey on tensor techniques and
  applications in machine learning.
\newblock IEEE Access \textbf{7}, 162950--162990 (2019)

\bibitem{jiang2020multi}
Jiang, T.X., Huang, T.Z., Zhao, X.L., Deng, L.J.: Multi-dimensional imaging
  data recovery via minimizing the partial sum of tubal nuclear norm.
\newblock Journal of Computational and Applied Mathematics \textbf{372}, 112680
  (2020)

\bibitem{kiers2000towards}
Kiers, H.A.: Towards a standardized notation and terminology in multiway
  analysis.
\newblock Journal of Chemometrics: A Journal of the Chemometrics Society
  \textbf{14}(3), 105--122 (2000)

\bibitem{kilmer2013third}
Kilmer, M.E., Braman, K., Hao, N., Hoover, R.C.: Third-order tensors as
  operators on matrices: A theoretical and computational framework with
  applications in imaging.
\newblock SIAM Journal on Matrix Analysis and Applications \textbf{34}(1),
  148--172 (2013)

\bibitem{kilmer2011factorization}
Kilmer, M.E., Martin, C.D.: Factorization strategies for third-order tensors.
\newblock Linear Algebra and its Applications \textbf{435}(3), 641--658 (2011)

\bibitem{kolda2009tensor}
Kolda, T.G., Bader, B.W.: Tensor decompositions and applications.
\newblock SIAM Review \textbf{51}(3), 455--500 (2009)

\bibitem{kossaifi2020tensor}
Kossaifi, J., Lipton, Z.C., Kolbeinsson, A., Khanna, A., Furlanello, T.,
  Anandkumar, A.: Tensor regression networks.
\newblock The Journal of Machine Learning Research \textbf{21}(1), 4862--4882
  (2020)

\bibitem{li2004statistical}
Li, L., Huang, W., Gu, I.Y.H., Tian, Q.: Statistical modeling of complex
  backgrounds for foreground object detection.
\newblock IEEE Transactions on Image Processing \textbf{13}(11), 1459--1472
  (2004)

\bibitem{li2018low}
Li, Y.F., Shang, K., Huang, Z.H.: Low tucker rank tensor recovery via admm
  based on exact and inexact iteratively reweighted algorithms.
\newblock Journal of Computational and Applied Mathematics \textbf{331}, 64--81
  (2018)

\bibitem{liu2012tensor}
Liu, J., Musialski, P., Wonka, P., Ye, J.: Tensor completion for estimating
  missing values in visual data.
\newblock IEEE Transactions on Pattern Analysis and Machine Intelligence
  \textbf{35}(1), 208--220 (2012)

\bibitem{liu2021tensors}
Liu, Y.: Tensors for Data Processing: Theory, Methods, and Applications.
\newblock Academic Press (2021)

\bibitem{lu2016tensor}
Lu, C., Feng, J., Chen, Y., Liu, W., Lin, Z., Yan, S.: Tensor robust principal
  component analysis: Exact recovery of corrupted low-rank tensors via convex
  optimization.
\newblock In: Proceedings of the IEEE Conference on Computer Vision and Pattern
  Recognition, pp. 5249--5257 (2016)

\bibitem{lu2019tensor}
Lu, C., Feng, J., Chen, Y., Liu, W., Lin, Z., Yan, S.: Tensor robust principal
  component analysis with a new tensor nuclear norm.
\newblock IEEE Transactions on Pattern Analysis and Machine Intelligence
  \textbf{42}(4), 925--938 (2020)

\bibitem{miron2020tensor}
Miron, S., Zniyed, Y., Boyer, R., Lima Ferrer~de Almeida, A., Favier, G., Brie,
  D., Comon, P.: Tensor methods for multisensor signal processing.
\newblock IET Signal Processing \textbf{14}(10), 693--709 (2020)

\bibitem{mu2020weighted}
Mu, Y., Wang, P., Lu, L., Zhang, X., Qi, L.: Weighted tensor nuclear norm
  minimization for tensor completion using tensor-{SVD}.
\newblock Pattern Recognition Letters \textbf{130}, 4--11 (2020)

\bibitem{natarajan1995sparse}
Natarajan, B.K.: Sparse approximate solutions to linear systems.
\newblock SIAM Journal on Computing \textbf{24}(2), 227--234 (1995)

\bibitem{nikolova2008efficient}
Nikolova, M., Ng, M.K., Zhang, S., Ching, W.K.: Efficient reconstruction of
  piecewise constant images using nonsmooth nonconvex minimization.
\newblock SIAM Journal on Imaging Sciences \textbf{1}(1), 2--25 (2008)

\bibitem{popa2021improved}
Popa, J., Minkoff, S.E., Lou, Y.: An improved seismic data completion algorithm
  using low-rank tensor optimization: Cost reduction and optimal data
  orientationimproved seismic data completion.
\newblock Geophysics \textbf{86}(3), V219--V232 (2021)

\bibitem{qin2021robust}
Qin, W., Wang, H., Zhang, F., Dai, M., Wang, J.: Robust low-rank tensor
  reconstruction using high-order t-{SVD}.
\newblock Journal of Electronic Imaging \textbf{30}(6), 063016--063016 (2021)

\bibitem{qiu2022fast}
Qiu, H., Wang, Y., Tang, S., Meng, D., Yao, Q.: Fast and provable nonconvex
  tensor {RPCA}.
\newblock In: International Conference on Machine Learning, pp. 18211--18249.
  PMLR (2022)

\bibitem{rahimi2019scale}
Rahimi, Y., Wang, C., Dong, H., Lou, Y.: A scale-invariant approach for sparse
  signal recovery.
\newblock SIAM Journal on Scientific Computing \textbf{41}(6), A3649--A3672
  (2019)

\bibitem{semerci2014tensor}
Semerci, O., Hao, N., Kilmer, M.E., Miller, E.L.: Tensor-based formulation and
  nuclear norm regularization for multienergy computed tomography.
\newblock IEEE Transactions on Image Processing \textbf{23}(4), 1678--1693
  (2014)

\bibitem{7891546}
Sidiropoulos, N.D., De~Lathauwer, L., Fu, X., Huang, K., Papalexakis, E.E.,
  Faloutsos, C.: Tensor decomposition for signal processing and machine
  learning.
\newblock IEEE Transactions on Signal Processing \textbf{65}(13), 3551--3582
  (2017)

\bibitem{tibshirani1996regression}
Tibshirani, R.: Regression shrinkage and selection via the lasso.
\newblock Journal of the Royal Statistical Society Series B: Statistical
  Methodology \textbf{58}(1), 267--288 (1996)

\bibitem{tropp2012user}
Tropp, J.A.: User-friendly tail bounds for sums of random matrices.
\newblock Foundations of computational mathematics \textbf{12}, 389--434 (2012)

\bibitem{tucker1966some}
Tucker, L.R.: Some mathematical notes on three-mode factor analysis.
\newblock Psychometrika \textbf{31}(3), 279--311 (1966)

\bibitem{wang2022minimizing}
Wang, C., Tao, M., Chuah, C.N., Nagy, J., Lou, Y.: Minimizing {$L_1$} over
  {$L_2$} norms on the gradient.
\newblock Inverse Problems \textbf{38}(6), 065011 (2022)

\bibitem{wang2021limited}
Wang, C., Tao, M., Nagy, J.G., Lou, Y.: Limited-angle {CT} reconstruction via
  the {$L_1/L_2$} minimization.
\newblock SIAM Journal on Imaging Sciences \textbf{14}(2), 749--777 (2021)

\bibitem{wang2020accelerated}
Wang, C., Yan, M., Rahimi, Y., Lou, Y.: Accelerated schemes for the {$L_1/L_2
  $} minimization.
\newblock IEEE Transactions on Signal Processing \textbf{68}, 2660--2669 (2020)

\bibitem{wang2019global}
Wang, Y., Yin, W., Zeng, J.: Global convergence of admm in nonconvex nonsmooth
  optimization.
\newblock Journal of Scientific Computing \textbf{78}, 29--63 (2019)

\bibitem{wang2004image}
Wang, Z., Bovik, A.C., Sheikh, H.R., Simoncelli, E.P.: Image quality
  assessment: from error visibility to structural similarity.
\newblock IEEE Transactions on Image Processing \textbf{13}(4), 600--612 (2004)

\bibitem{wold1987principal}
Wold, S., Esbensen, K., Geladi, P.: Principal component analysis.
\newblock Chemometrics and Intelligent Laboratory Systems \textbf{2}(1-3),
  37--52 (1987)

\bibitem{xia2024tensor}
Xia, S., Qiu, D., Zhang, X.: Tensor factorization via transformed tensor-tensor
  product for image alignment.
\newblock Numerical Algorithms \textbf{95}(3), 1251--1289 (2024)

\bibitem{xu2019laplace}
Xu, W.H., Zhao, X.L., Ji, T.Y., Miao, J.Q., Ma, T.H., Wang, S., Huang, T.Z.:
  Laplace function based nonconvex surrogate for low-rank tensor completion.
\newblock Signal Processing: Image Communication \textbf{73}, 62--69 (2019)

\bibitem{yan2024tensor}
Yan, T., Guo, Q.: Tensor robust principal component analysis via dual lp
  quasi-norm sparse constraints.
\newblock Digital Signal Processing \textbf{150}, 104520 (2024)

\bibitem{yang2020low}
Yang, J.H., Zhao, X.L., Ji, T.Y., Ma, T.H., Huang, T.Z.: Low-rank tensor train
  for tensor robust principal component analysis.
\newblock Applied Mathematics and Computation \textbf{367}, 124783 (2020)

\bibitem{yang2020multiview}
Yang, M., Luo, Q., Li, W., Xiao, M.: Multiview clustering of images with tensor
  rank minimization via nonconvex approach.
\newblock SIAM Journal on Imaging Sciences \textbf{13}(4), 2361--2392 (2020)

\bibitem{YANG2022108311}
Yang, M., Luo, Q., Li, W., Xiao, M.: Nonconvex {3D} array image data recovery
  and pattern recognition under tensor framework.
\newblock Pattern Recognition \textbf{122}, 108311 (2022)

\bibitem{yang2015robust}
Yang, Y., Feng, Y., Suykens, J.A.: Robust low-rank tensor recovery with
  regularized redescending m-estimator.
\newblock IEEE Transactions on Neural Networks and Learning Systems
  \textbf{27}(9), 1933--1946 (2015)

\bibitem{zhang2010nearly}
Zhang, C.H.: Nearly unbiased variable selection under minimax concave penalty.
\newblock The Annals of Statistics \textbf{38}(2), 894--942 (2010)

\bibitem{zhang2016exact}
Zhang, Z., Aeron, S.: Exact tensor completion using t-{SVD}.
\newblock IEEE Transactions on Signal Processing \textbf{65}(6), 1511--1526
  (2017)

\bibitem{zhao2015bayesian}
Zhao, Q., Zhang, L., Cichocki, A.: Bayesian {CP} factorization of incomplete
  tensors with automatic rank determination.
\newblock IEEE Transactions on Pattern Analysis and Machine Intelligence
  \textbf{37}(9), 1751--1763 (2015)

\bibitem{zhao2022robust}
Zhao, X., Bai, M., Sun, D., Zheng, L.: Robust tensor completion: Equivalent
  surrogates, error bounds, and algorithms.
\newblock SIAM Journal on Imaging Sciences \textbf{15}(2), 625--669 (2022)

\bibitem{zheng2024scale}
Zheng, H., Lou, Y., Tian, G., Wang, C.: A scale-invariant relaxation in
  low-rank tensor recovery with an application to tensor completion.
\newblock SIAM Journal on Imaging Sciences \textbf{17}(1), 756--783 (2024)

\bibitem{zhou2017outlier}
Zhou, P., Feng, J.: Outlier-robust tensor {PCA}.
\newblock In: Proceedings of the IEEE Conference on Computer Vision and Pattern
  Recognition, pp. 2263--2271 (2017)

\end{thebibliography}

\end{document}